\documentclass[letterpaper,10pt]{article}
\usepackage{fullpage}

\usepackage{graphicx}%
\usepackage{multirow}%
\usepackage{amsmath,amssymb,amsfonts}%
\usepackage{amsthm}%
\usepackage{mathrsfs}%
\usepackage[title]{appendix}%
\usepackage{xcolor}%
\usepackage{textcomp}%
\usepackage{manyfoot}%
\usepackage{booktabs}%
\usepackage{algpseudocode}%
\usepackage{listings}%
\usepackage{caption}
\usepackage{hyperref}
\usepackage[capitalize,nameinlink,noabbrev]{cleveref}
\usepackage{mathtools}
\usepackage{tikz}
\usepackage{tikz-cd}
\usepackage{cite}

\usepackage{comment}
\usepackage{siunitx}
\usepackage{relsize}
\usepackage{ifthen}
\usepackage[colorinlistoftodos]{todonotes}

\usepackage[caption=false]{subfig}

\usepackage[vlined,ruled,linesnumbered]{algorithm2e}
\usepackage{graphics} 
\usepackage{rotating}
\usepackage{color}
\usepackage{enumerate}
\usepackage[T1]{fontenc}
\usepackage{psfrag}
\usepackage{epsfig} 
\usepackage{booktabs}
\usepackage{graphicx,url}
\usepackage{multirow}
\usepackage{array}
\usepackage{latexsym}
\usepackage{amsfonts}
\usepackage{amsmath}
\usepackage{amssymb}
\usepackage{mathtools}
\usepackage{xstring}
\usepackage{multirow}
\usepackage{xcolor}
\usepackage{prettyref}
\usepackage{flexisym}
\usepackage{bigdelim}
\usepackage{breqn} 
\usepackage{listings}

\usepackage{enumitem}
\usepackage{xspace}
\usepackage{bm}
\graphicspath{{./figures/}}
\usepackage{tikz}
\usetikzlibrary{matrix,calc}

\usepackage{mdwlist}

\makecompactlist{itemize}{stditemize}

\newrefformat{prob}{Problem\,\ref{#1}}
\newrefformat{def}{Definition\,\ref{#1}}
\newrefformat{sec}{Section\,\ref{#1}}
\newrefformat{sub}{Section\,\ref{#1}}
\newrefformat{prop}{Proposition\,\ref{#1}}
\newrefformat{app}{Appendix\,\ref{#1}}
\newrefformat{alg}{Algorithm\,\ref{#1}}
\newrefformat{cor}{Corollary\,\ref{#1}}
\newrefformat{thm}{Theorem\,\ref{#1}}
\newrefformat{lem}{Lemma\,\ref{#1}}
\newrefformat{fig}{Fig.\,\ref{#1}}
\newrefformat{tab}{Table\,\ref{#1}}

\newtheorem{corollary}{Corollary}
\newtheorem{lemma}{Lemma}
\newtheorem{assumption}{Assumption}

\newcommand{\bdmath}{\begin{dmath}}
\newcommand{\edmath}{\end{dmath}}
\newcommand{\beq}{\begin{equation}}
\newcommand{\eeq}{\end{equation}}
\newcommand{\bdm}{\begin{displaymath}}
\newcommand{\edm}{\end{displaymath}}
\newcommand{\bea}{\begin{eqnarray}}
\newcommand{\eea}{\end{eqnarray}}
\newcommand{\beal}{\beq \begin{array}{ll}}
\newcommand{\eeal}{\end{array} \eeq}
\newcommand{\beas}{\begin{eqnarray*}}
\newcommand{\eeas}{\end{eqnarray*}}
\newcommand{\ba}{\begin{array}}
\newcommand{\ea}{\end{array}}
\newcommand{\bit}{\begin{itemize}}
\newcommand{\eit}{\end{itemize}}
\newcommand{\ben}{\begin{enumerate}}
\newcommand{\een}{\end{enumerate}}

\newcommand{\calA}{{\cal A}}

\newcommand{\calF}{{\cal F}}
\newcommand{\calG}{{\cal G}}

\newcommand{\calI}{{\cal I}}

\newcommand{\calK}{{\cal K}}
\newcommand{\calL}{{\cal L}}
\newcommand{\calM}{{\cal M}}
\newcommand{\calN}{{\cal N}}
\newcommand{\calO}{{\cal O}}
\newcommand{\calP}{{\cal P}}

\newcommand{\calR}{{\cal R}}
\newcommand{\calS}{{\cal S}}
\newcommand{\calT}{{\cal T}}

\newcommand{\calV}{{\cal V}}
\newcommand{\calW}{{\cal W}}
\newcommand{\calX}{{\cal X}}

\newcommand{\calZ}{{\cal Z}}

\newcommand{\eg}{\emph{e.g.,}\xspace}
\newcommand{\ie}{\emph{i.e.,}\xspace}

\newcommand{\hide}[1]{}

\newcommand{\hiddenText}{{\color{gray} hidden text.}}
\newcommand{\hideWithText}[1]{\hiddenText}

\newcommand{\norm}[1]{\left\| #1 \right\|}

\newcommand{\E}{{\mathbb{E}}}

\newcommand{\tran}{^{\mathsf{T}}}

\newcommand{\diag}[1]{\mathrm{diag}\left(#1\right)}

\newcommand{\rank}[1]{\mathrm{rank}(#1)}

\newcommand{\Real}[1]{ { {\mathbb R}^{#1} } }

\newcommand{\scenario}[1]{{\smaller \sf#1}\xspace}

\newcommand{\MOSEK}{\scenario{MOSEK}}

\newcommand{\blue}[1]{{\color{blue}#1}}

\newcommand{\linkToPdf}[1]{\href{#1}{\blue{(pdf)}}}
\newcommand{\linkToPpt}[1]{\href{#1}{\blue{(ppt)}}}
\newcommand{\linkToCode}[1]{\href{#1}{\blue{(code)}}}
\newcommand{\linkToWeb}[1]{\href{#1}{\blue{(web)}}}
\newcommand{\linkToVideo}[1]{\href{#1}{\blue{(video)}}}
\newcommand{\linkToMedia}[1]{\href{#1}{\blue{(media)}}}
\newcommand{\award}[1]{\xspace} 

\newcommand{\Sn}{\mathbb{S}^n}
\newcommand{\R}{\mathbb{R}}

\renewcommand{\norm}[1]{\lVert #1 \rVert}
\newcommand{\inprod}[2]{\left\langle #1, #2 \right\rangle}
\newcommand{\vectorize}[1]{\mathrm{vec}\parentheses{#1}}

\newcommand{\mymid}{\ \middle\vert\ }

\newcommand{\sym}[1]{\mathbb{S}^{#1}}

\newcommand{\bmat}{\left[ \begin{array}}
\newcommand{\emat}{\end{array}\right]}
\newcommand{\psd}[1]{\sym{#1}_{+}}
\newcommand{\parentheses}[1]{\left(#1\right)}

\newcommand{\abs}[1]{\left|#1\right|}

\newcommand{\bbN}{\mathbb{N}}

\newcommand{\ceil}[1]{\left\lceil #1 \right\rceil}

\newcommand{\PiSnp}[1]{\Pi_{\psd{n}} (#1)}

\newcommand{\Sym}[1]{\mathbb{S}^{#1}}
\newcommand{\Symp}[1]{\mathbb{S}_{+}^{#1}}
\newcommand{\Sympp}[1]{\mathbb{S}_{++}^{#1}}
\newcommand{\Symn}[1]{\mathbb{S}_{-}^{#1}}
\newcommand{\Symnn}[1]{\mathbb{S}_{--}^{#1}}

\newcommand{\minface}[2]{\mathrm{minface}\left( #1, #2 \right)}

\newcommand{\Id}{\mathrm{Id}}

\newcommand{\Xopt}{\calX_\star}
\newcommand{\Sopt}{\calS_\star}
\newcommand{\Zopt}{\calZ_\star}

\DeclareRobustCommand{\mymat}[1]{ \begin{bmatrix} #1 \end{bmatrix} }
\newcommand{\dist}{\mathrm{dist}}

\DeclareDocumentCommand{\H}{o}{
  \IfNoValueTF{#1}{H}{H_{#1}}
}
\DeclareDocumentCommand{\Z}{o}{
  \IfNoValueTF{#1}{Z}{Z_{#1}}
}

\newcommand{\Zk}{\Z^{(k)}}
\newcommand{\Hk}{H^{(k)}}
\newcommand{\Xk}{X^{(k)}}
\newcommand{\yk}{y^{(k)}}
\newcommand{\Sk}{S^{(k)}}

\newcommand{\Zkpo}{\Z^{(k+1)}}
\newcommand{\Hkpo}{H^{(k+1)}}
\newcommand{\Xkpo}{X^{(k+1)}}
\newcommand{\ykpo}{y^{(k+1)}}
\newcommand{\Skpo}{S^{(k+1)}}

\newcommand{\Xs}{X_{\star}}
\newcommand{\ys}{y_{\star}}
\newcommand{\Ss}{S_{\star}}
\newcommand{\Zs}{\Z_{\star}}

\newcommand{\Qs}{Q_{\star}}

\newcommand{\dXk}{\Xk - \Xs}

\newcommand{\dZk}{\Zk - \Zs}

\newcommand{\sigC}{\sigma C}
\newcommand{\sigS}{\sigma S}
\newcommand{\sigSk}{\sigma \Sk}
\newcommand{\sigSs}{\sigma \Ss}

\newcommand{\HOk}{\HO^{(k)}}

\newcommand{\ak}{a^{(k)}}
\newcommand{\bk}{b^{(k)}}
\newcommand{\akpo}{a^{(k+1)}}
\newcommand{\bkpo}{b^{(k+1)}}

\newcommand{\psik}{\Psi^{(k)}}

\newcommand{\Proj}{\Pi}
\newcommand{\PA}{\calP}
\newcommand{\PAp}{\PA^{\perp}}

\newcommand{\Omep}{\Omega^\perp}

\newcommand{\TheHO}{\Theta \circ \HO}
\newcommand{\TheHOT}{\Theta\tran \circ \HO\tran}
\newcommand{\Thep}{\Theta^\perp}
\newcommand{\ThepHO}{\Theta^\perp \circ \HO}
\newcommand{\ThepHOT}{(\Theta^\perp)\tran \circ \HO\tran}

\DeclareDocumentCommand{\lam}{m o}{
  \IfNoValueTF{#2}{\lambda_{#1}}{\lambda_{#1,#2}}
}
\DeclareDocumentCommand{\gam}{m o}{
  \IfNoValueTF{#2}{\gamma_{#1}}{\gamma_{#1,#2}}
}
\newcommand{\sigmin}[1]{\sigma_{\min} \left( #1 \right)}
\newcommand{\lammin}[1]{\lambda_{\min} \left( #1 \right)}
\newcommand{\lammax}[1]{\lambda_{\max} \left( #1 \right)}

\newcommand{\deltalam}[2]{\lammin{#1} - \lammax{#2}}

\DeclareDocumentCommand{\X}{o}{
  \IfNoValueTF{#1}{X}{X_{#1}}
}

\DeclareDocumentCommand{\QX}{o}{
  \IfNoValueTF{#1}{Q_X}{Q_{X,#1}}
}
\DeclareDocumentCommand{\QS}{o}{
  \IfNoValueTF{#1}{Q_S}{Q_{S,#1}}
}
\DeclareDocumentCommand{\QO}{o}{
  \IfNoValueTF{#1}{Q_O}{Q_{O,#1}}
}

\DeclareDocumentCommand{\W}{o}{
  \IfNoValueTF{#1}{W}{W_{#1}}
}
\DeclareDocumentCommand{\WO}{o}{
  \IfNoValueTF{#1}{W_O}{W_{O,#1}}
}

\DeclareDocumentCommand{\E}{o}{
  \IfNoValueTF{#1}{E}{E_{#1}}
}
\DeclareDocumentCommand{\I}{o}{
  \IfNoValueTF{#1}{I}{I_{#1}}
}
\DeclareDocumentCommand{\Y}{o}{
  \IfNoValueTF{#1}{Q}{Q_{#1}}
}
\DeclareDocumentCommand{\V}{o}{
  \IfNoValueTF{#1}{V}{V_{#1}}
}
\DeclareDocumentCommand{\etanew}{o}{
  \IfNoValueTF{#1}{\eta}{\eta_{#1}}
}

\DeclareDocumentCommand{\ZX}{o}{
  \IfNoValueTF{#1}{Z_X}{Z_{X,#1}}
}
\DeclareDocumentCommand{\ZS}{o}{
  \IfNoValueTF{#1}{Z_S}{Z_{S,#1}}
}
\DeclareDocumentCommand{\ZO}{o}{
  \IfNoValueTF{#1}{Z_O}{Z_{O,#1}}
}

\DeclareDocumentCommand{\dA}{o}{
  \IfNoValueTF{#1}{\Delta A}{\Delta A_{X,#1}}
}
\DeclareDocumentCommand{\HX}{o}{
  \IfNoValueTF{#1}{H_X}{H_{X,#1}}
}
\DeclareDocumentCommand{\HS}{o}{
  \IfNoValueTF{#1}{H_S}{H_{S,#1}}
}
\DeclareDocumentCommand{\HO}{o}{
  \IfNoValueTF{#1}{H_O}{H_{O,#1}}
}
\DeclareDocumentCommand{\dH}{o}{
  \IfNoValueTF{#1}{\Delta H}{\Delta H_{#1}}
}
\DeclareDocumentCommand{\dHX}{o}{
  \IfNoValueTF{#1}{\Delta H_X}{\Delta H_{X,#1}}
}
\DeclareDocumentCommand{\dHS}{o}{
  \IfNoValueTF{#1}{\Delta H_S}{\Delta H_{S,#1}}
}
\DeclareDocumentCommand{\dHO}{o}{
  \IfNoValueTF{#1}{\Delta H_O}{\Delta H_{O,#1}}
}

\DeclareDocumentCommand{\dX}{o}{
  \IfNoValueTF{#1}{\Delta X}{\Delta X_{#1}}
}
\DeclareDocumentCommand{\dS}{o}{
  \IfNoValueTF{#1}{\Delta S}{\Delta S_{#1}}
}

\newcommand{\normtwo}[1]{\norm{#1}_2}
\newcommand{\normF}[1]{\norm{#1}_\mathsf{F}}
\newcommand{\normop}[1]{\norm{#1}_{\mathrm{op}}}

\newcommand{\vvec}{\text{vec}}

\newcommand{\LamXS}{\diag{\lam{1}, \ldots, \lam{r}, \lam{r+1}, \ldots, \lam{n}}}

\newcommand{\LamX}{\Lambda_X}

\newcommand{\LamS}{\Lambda_S}

\newcommand{\Asdp}{\calA}
\newcommand{\AsdpT}{\calA^*}
\newcommand{\Madmm}{\calM}

\newcommand{\Fix}{\mathrm{Fix}}

\DeclareDocumentCommand{\svec}{o}{
  \IfNoValueTF{#1}{\mathrm{svec}}{\svec \left( #1 \right)}
}
\DeclareDocumentCommand{\smat}{o}{
  \IfNoValueTF{#1}{\mathrm{smat}}{\smat \left( #1 \right)}
}

\DeclareDocumentCommand{\tH}{o}{
  \IfNoValueTF{#1}{\widetilde{H}}{\widetilde{H}_{1}}
}
\DeclareDocumentCommand{\tHX}{o}{
  \IfNoValueTF{#1}{\widetilde{H}_X}{\widetilde{H}_{X, 1}}
}
\DeclareDocumentCommand{\tHS}{o}{
  \IfNoValueTF{#1}{\widetilde{H}_S}{\widetilde{H}_{S, 1}}
}
\DeclareDocumentCommand{\tHO}{o}{
  \IfNoValueTF{#1}{\widetilde{H}_O}{\widetilde{H}_{O, 1}}
}

\newcommand{\rmax}{r_{\mathrm{max}}}

\newcommand{\Dd}[1]{\widetilde \Omega (#1)}
\newcommand{\Ddp}[1]{\widetilde \Omega^\perp (#1)}

\hypersetup{colorlinks, hypertexnames=false, pageanchor=true,
            linkcolor=blue, citecolor={green!50!black}, urlcolor=cyan,
            pdftitle={Local Linear Convergence of the Alternating Direction Method of Multipliers for Semidefinite Programming under Strict Complementarity}
            }
\makeatletter
\AddToHook{cmd/appendix/before}{\def\cref@section@alias{appendix}}
\makeatother
\mathtoolsset{centercolon}
\crefname{assumption}{Assumption}{Assumptions}
\newcommand{\range}{\calR}
\newcommand{\nullspace}{\calN}
\newcommand{\pfm}{\Pi_{\Fix(\Madmm)}} 

\newif\ifrebuttal
\rebuttaltrue 
\DeclareRobustCommand{\rebuttal}[1]{%
    \ifrebuttal
        {#1}
    \fi
}

\theoremstyle{thmstyleone}%
\newtheorem{theorem}{Theorem}%
\newtheorem{proposition}{Proposition}%

\theoremstyle{thmstyletwo}%
\newtheorem{example}{Example}%
\newtheorem{remark}{Remark}%

\theoremstyle{thmstylethree}%

\usepackage{etoolbox}
\newcommand{\acceptednotice}{%
  \begin{center}\small
  This paper has been accepted for publication in \emph{Mathematical Programming}. \\
  Please cite the paper as: S.~Kang, X.~Jiang, and H.~Yang, \\
  ``Local Linear Convergence of the Alternating Direction Method of Multipliers \\
  for Semidefinite Programming under Strict Complementarity'', \\
  \emph{Mathematical Programming}, 2026.
  \end{center}%
  \vspace{1em}%
}
\makeatletter
\patchcmd{\@maketitle}{\newpage}{\newpage\acceptednotice}{}{\PackageError{main}{could not patch \string\@maketitle}{}}
\makeatother

\title{Local Linear Convergence of the Alternating Direction Method of Multipliers for Semidefinite Programming \\
under Strict Complementarity}
\author{Shucheng Kang%
\thanks{School of Engineering and Applied Sciences, Harvard University. Email: \texttt{skang1@g.harvard.edu}}
\and Xin Jiang%
\thanks{Department of Industrial and Systems Engineering, University of Houston. Email: \texttt{xinjiang@uh.edu}}
\and Heng Yang%
\thanks{School of Engineering and Applied Sciences, Harvard University. Email: \texttt{hankyang@seas.harvard.edu}}
}
\date{November 10, 2026}

\begin{document}

\maketitle

\begin{abstract}
    We investigate the local linear convergence properties of the Alternating Direction Method of Multipliers (ADMM) when applied to Semidefinite Programming (SDP). A longstanding belief suggests that ADMM is only capable of solving SDPs to moderate accuracy, primarily due to its sublinear worst-case complexity and empirical observations of slow convergence. We challenge this notion by introducing a new sufficient condition for local linear convergence: as long as the converged primal--dual optimal solutions satisfy strict complementarity, ADMM attains local linear convergence, independent of nondegeneracy conditions. Our proof is based on a direct local linearization of the ADMM operator and a refined error bound for the projection onto the positive semidefinite cone, improving previous bounds and revealing the anisotropic nature of projection residuals. 
    Extensive numerical experiments confirm the significance of our theoretical results, demonstrating that ADMM achieves local linear convergence and computes high-accuracy solutions in a variety of SDP instances, including those where nondegeneracy fails. Furthermore, we identify cases where ADMM struggles, linking these difficulties with near violations of strict complementarity—a phenomenon that parallels recent findings in linear programming. 
\end{abstract}

\newpage

{\small
\tableofcontents
}

\newpage


\section{Introduction}
\label{sec:intro}

Consider the semidefinite programs (SDPs) in the standard form:
\begin{equation} \label{eq:intro-sdp}
    \begin{array}{llllll}
        \text{Primal:} \;\; & \text{minimize} & \inprod{C}{X} & \qquad \qquad \quad \text{Dual:} \;\; & \text{maximize} & b\tran y \\
        & \text{subject to} & \Asdp X = b & & \text{subject to} & \AsdpT y + S = C \\
        & & X \in \Symp{n} & & & S \in \Symp{n},
    \end{array}
\end{equation}
with primal variable $X \in \Sn$ and dual variables $S \in \Sn$, $y \in \Real{m}$, where $\Sn$ is the set of real symmetric $n \times n$ matrices and $\Symp{n}$ is the set of positive semidefinite (PSD) matrices in $\Sn$. The linear operator $\Asdp : \Sn \to \Real{m}$ is defined as
\[
    \Asdp X := \left( \inprod{A_1}{X}, \cdots, \inprod{A_m}{X} \right)
\]
and $\AsdpT y = \sum_{i=1}^m y_i A_i$ is its adjoint operator. The coefficients $C,A_1,\ldots,A_m$ are symmetric $n \times n$ matrices. It is assumed that $\{A_i\}_{i=1}^m$ are linearly independent so that $\Asdp \AsdpT$ is an invertible operator.

With the growing demand for solving large-scale SDPs, particularly those arising from moment and sums-of-squares (SOS) relaxations in polynomial optimization~\cite{lasserre2001siopt-global,nie2023siopt-moment-momentpolynomialopt,yang2022pami-outlier-robust-geometric-perception,kang2024wafr-strom,kang2025global,fantuzzi2024siam-global-pop-integral-functionals,magron23book-sparse,wang2022certifying,gertler2025many,huang2024arxiv-sparsehomogenization}, first-order methods (FOMs) have gained increasing attention due to their low per-iteration cost and ability to exploit problem structure. Among these, the Alternating Direction Method of Multipliers (ADMM) has emerged as a widely adopted approach, with numerous implementations, applications, and variations~\cite{wen2010mp-admmsdp,chen2017mp-sgsadmm,zheng2017ifac-cdcs-sdpsolver,odonoghue2023-scs-sdpsolver,yang2015mp-sdpnalplus-sdpsolver,kang2024wafr-strom}.

\paragraph{ADMM for SDP.}
Starting from $(X^{(0)}, y^{(0)}, S^{(0)})$, the classical three-step ADMM iteration for the SDP~\eqref{eq:intro-sdp} reads as~\cite{wen2010mp-admmsdp}:
\begin{subequations}
    \label{eq:intro:admm-three-step}
    \begin{align}
        & \ykpo = (\Asdp \AsdpT)^{-1} \left(\sigma^{-1} b - \Asdp \left( \sigma^{-1} \Xk + \Sk - C \right) \right) \\
        & \Skpo = \Pi_{\Symp{n}} \left( C - \AsdpT \ykpo - \sigma^{-1} \Xk \right) \\
        & \Xkpo = \Xk + \sigma \left( \Skpo + \AsdpT \ykpo - C \right)
    \end{align}
\end{subequations}
where $\Pi_{\Symp{n}}(\cdot)$ denotes the orthogonal projection onto the PSD cone $\Symp{n}$ and $\sigma>0$ is the penalty parameter. Under mild conditions, $(\Xk, \yk, \Sk )$ is convergent to $(\Xs, \ys, \Ss)$, one of the optimal solution pairs satisfying the Karush--Kuhn--Tucker (KKT) conditions~\cite[Theorem 2]{wen2010mp-admmsdp}:
\begin{equation}
    \label{eq:intro:kkt}
    \Asdp \Xs = b, \qquad \AsdpT \ys + \Ss = C, \qquad \inprod{\Xs}{\Ss} = 0, \qquad \Xs \in \Symp{n}, \qquad \Ss \in \Symp{n}.
\end{equation}
\rebuttal{
The ADMM iteration~\eqref{eq:intro:admm-three-step} applied to the dual SDP is equivalent to the Douglas-Rachford splitting (DRS) method applied to the primal SDP~\cite{li2018siamsc-semismooth-newton-sdp}:
}
\begin{equation} \label{eq:intro:admm-one-step}
\begin{aligned}
    \Zkpo &= \AsdpT (\Asdp \AsdpT)^{-1} \Asdp (-2 \Pi_{\Symp{n}}(\Zk) + \Zk) + \Pi_{\Symp{n}}(\Zk) \\
    &\phantom{=} \mbox{} + \AsdpT (\Asdp \AsdpT)^{-1} b + \sigma \AsdpT (\Asdp \AsdpT)^{-1} \Asdp C - \sigC, 
\end{aligned}
\end{equation}
where we make the change of variables $\Z := \X - \sigS$ (and $\Zs := \Xs - \sigSs$). From~$\Zk$, we can extract the primal variable and the (scaled) dual variable as
\begin{align}
    \label{eq:intro:extract-XS}
    \Xk = \Pi_{\Symp{n}} (\Zk), \qquad \sigSk = \Pi_{\Symp{n}} (-\Zk).
\end{align}

In this paper, we investigate the local convergence properties of ADMM for solving SDPs. To this end, we begin by recalling two important regularity conditions in SDP.

\paragraph{Nondegeneracy and strict complementarity.}
First introduced in~\cite{alizadeh1997mp-complementarity-nondegeneracy-sdp}, nondegeneracy and strict complementarity have been two fundamental regularity conditions in SDP~\cite{alizadeh1998siopt-aho,zhao2010siopt-newton-cg-alm}. Since the primal--dual optimal solutions are simultaneously diagonalizable~\cite[pp. 308]{wolkowicz2000book-sdp-handbook}, we assume they admit the following decomposition:
\begin{subequations}
    \label{eq:intro:Xs-sigSs}
    \begin{align}
        & \Xs = \Qs \mymat{\LamX & 0 \\ 0 & 0} \Qs\tran, \quad \LamX := \diag{\lam{1}, \dots, \lam{r}} \\
        & \sigSs = \Qs \mymat{0 & 0 \\ 0 & \LamS} \Qs\tran, \quad \LamS := -\diag{\lam{n-s+1}, \dots, \lam{n}},
    \end{align}    
\end{subequations}
where $\diag{\cdot}$ assembles a vector into a diagonal matrix, $\Qs \in \Real{n \times n}$ is an orthogonal matrix, and the eigenvalues satisfy
\begin{align*}
    \lam{1} \ge \cdots \ge \lam{r} > 0 > \lam{n-s+1} \ge \cdots \ge \lam{n}
\end{align*}
with $r + s \le n$. Then, we define four important subspaces~\cite{alizadeh1997mp-complementarity-nondegeneracy-sdp}:
\begin{subequations}
    \label{eq:intro:orthogonal-complement}
    \begin{align}
        \calT_\Xs := & \left\{ \Qs \mymat{\HX & \HO\tran \\ \HO & 0} \Qs\tran \mymid \HX \in \Sym{r}, \, \HO \in \Real{(n-r) \times r} \right\} \\
        \calN_\Xs := \calT_\Xs^\perp = & \left\{ \Qs \mymat{0 & 0 \\ 0 & \HS} \Qs\tran \mymid \HS \in \Sym{n-r} \right\} \\
        \calT_\Ss := & \left\{ \Qs \mymat{0 & \HO\tran \\ \HO & \HS} \Qs\tran \mymid \HS \in \Sym{s}, \, \HO \in \Real{s \times (n-s)} \right\} \\
        \calN_\Ss := \calT_\Ss^\perp = & \left\{ \Qs \mymat{\HX & 0 \\ 0 & 0} \Qs\tran \mymid \HX \in \Sym{n-s} \right\},
    \end{align}
\end{subequations}
where $\calL^\perp$ represents the orthogonal complement of the linear subspace $\calL$. Further, we denote $\calR(\AsdpT) := \{\sum_{i=1}^m A_i y_i \mid y \in \Real{m} \}$ as the range space of $\AsdpT$, and $\calN(\Asdp)$ as $\calR(\AsdpT)^\perp = \{ X \in \Sn \mid \Asdp X = 0 \}$. The nondegeneracy (ND) and strict complementarity (SC) conditions---two generic properties for SDPs---are defined as~\cite{alizadeh1997mp-complementarity-nondegeneracy-sdp}:
\begin{align}
    \label{eq:intro:primal-nondegeneracy}
    \text{Primal Nondegeneracy: } & \calT_\Xs + \calN(\Asdp) = \Sn \Longleftrightarrow \calN_\Xs \cap \calR(\AsdpT) = \left\{ 0 \right\} \\
    \label{eq:intro:dual-nondegeneracy}
    \text{Dual Nondegeneracy: } & \calT_\Ss + \calR(\AsdpT) = \Sn \Longleftrightarrow \calN_\Ss \cap \calN(\Asdp) = \left\{ 0 \right\} \\
    \label{eq:intro:strict-complementarity}
    \text{Strict Complementarity: } & \rank{\Xs} + \rank{\Ss} = n \Longleftrightarrow r + s = n.
\end{align}
When strict complementarity holds, primal (resp., dual) nondegeneracy is equivalent to the uniqueness of dual (resp., primal) optimal solution~\cite{alizadeh1997mp-complementarity-nondegeneracy-sdp}. 

\subsection{ADMM for SDP: Sublinear Rate and Moderate Accuracy?}
A common perception for ADMM among practitioners is that it generally cannot solve the SDP \eqref{eq:intro-sdp} to high accuracy (\eg max KKT residual below $10^{-10}$). Indeed, this perception arises from both theoretical challenges and empirical observations.

Theoretically, a well-accepted convergence rate for ADMM applied to solving SDPs is $\calO\left( 1/\epsilon \right)$, which matches the general convex optimization case~\cite{ryu2023book-largescale-convexopt}. Establishing linear convergence $\calO(\log (1/\epsilon))$ is considered hard, since known sufficient conditions, such as strong convexity~\cite{nishihara2015pmlr-general-analysis-admm}, local polyhedrality~\cite{liang2017jota-local-convergence-admm,hong2017mp-linear-convergence-admm} and certain growth conditions \cite{deng2016global,hong2017mp-linear-convergence-admm,davis2017fasterconvergence,liu2018partial,yuan2020discerning,zamani2024exact}, either fail or remain unclear for SDP. As pioneering works, \cite{chan2008siopt-nd-sr-nonsigularity-sdp,han2018mor-linear-rate-admm} established local linear convergence of ADMM for solving SDPs when primal nondegeneracy~\eqref{eq:intro:primal-nondegeneracy} and dual nondegeneracy~\eqref{eq:intro:dual-nondegeneracy} both hold (regardless of strict complementarity~\eqref{eq:intro:strict-complementarity}), by showing the metric subregularity (or calmness) of the KKT operator. However, two challenges remain: (a) two-side nondegeneracy conditions are hard to check numerically, and (b) important subclasses of SDP, such as those arising from the moment-SOS relaxations with finite convergence~\cite{lasserre2001siopt-global}, are known to be degenerate. It remains an open question in optimization theory whether alternative, and perhaps simple-to-verify, sufficient conditions can ensure the linear convergence of ADMM for solving SDPs.
    
Empirically, the slow convergence of ADMM and its variants when solving SDPs is widely reported~\cite{garstka2021jota-cosmo,kang2024wafr-strom,zheng2017ifac-cdcs-sdpsolver,yang2023mp-stride}. For a typical SDP, the interior point method~\cite{alizadeh1998siopt-aho} (IPM) implemented in \MOSEK~\cite{aps2019ugrm-mosek-sdpsolver} can solve the problem to machine precision (if memory permits), while ADMM often struggles to achieve moderate accuracy (e.g., max KKT residual $10^{-4}$). Consequently, ADMM and its variants are primarily used to warmstart downstream solvers~\cite{yang2015mp-sdpnalplus-sdpsolver,yang2023mp-stride} or as efficient methods to obtain coarse solutions~\cite{kang2024wafr-strom}. In contrast, recently in the linear programming (LP) literature, first-order methods are observed to exhibit significantly improved local linear convergence rates after initially traversing a prolonged phase of slow convergence~\cite{lu2024mp-geometry-pdhg-lp}. To the best of our knowledge, no comparable numerical evidence has been documented in the SDP literature.

Thus, the central question driving this paper is:
\begin{quotation}
    \noindent
    Can ADMM exhibit empirically observable linear convergence when solving SDPs? If so, can we establish numerically verifiable sufficient conditions to guarantee this behavior?
\end{quotation}

\subsection{Contributions}

We affirm this question both empirically and theoretically. We establish local linear convergence of ADMM for SDP under a mild strict complementarity (SC) assumption. Moreover, comprehensive numerical results are reported to demonstrate such a prevalent linear rate of convergence.

\paragraph{Theoretical contribution.}
We establish a new sufficient condition that guarantees the local linear convergence of ADMM for solving SDPs.
\begin{theorem}[Informal: Local Linear Convergence under Strict Complementarity]
If ADMM converges to an optimal solution $(\Xs, \ys, \Ss)$ of the SDP \eqref{eq:intro-sdp} that satisfies strict complementarity~\eqref{eq:intro:strict-complementarity}, then ADMM attains local linear convergence after finite iterations.
\end{theorem}
Our sufficient condition holds independently of nondegeneracy (ND). Verifying strict complementarity~\eqref{eq:intro:strict-complementarity} requires only checking the numerical rank of $\Xs$ and $\Ss$, a significantly more tractable procedure compared with verifying nondegeneracy, which involves examining the intersection of two subspaces.

At a high level, our proof framework is built upon a detailed local linearization analysis of the ADMM operator, incorporating several key contributions.
\begin{itemize}
    \item \textbf{A refined error bound for the PSD cone projector.}
    We improve the classic linearization result of the PSD cone projector in~\cite[Theorem 4.6]{sun2002mor-semismooth-matrix-valued} by tightening the bound on the linearization residual from $\calO (\norm{\H}^2)$ to $\calO( \norm{\HO} \cdot \norm{\H})$, where~$\H$ represents a small perturbation and $\HO$ denotes its off-block-diagonal part. This refinement is a cornerstone of our proof, revealing the anisotropic nature of the residual and offering potential applications beyond our setting.
    
    \item \textbf{Characterization of the local behavior of ADMM.}
    Our analysis relies on a local linearization of the fixed-point iteration \eqref{lem:error-bound:error-control-one-step}. In particular, when near optimum, the iterate error $\Zkpo-\Zs$ can be written as a linear transformation of $\Zk-\Zs$ plus a quadratic residual term. When both ND and SC hold, we prove that the linear transformation is contractive and that the linearization residual becomes negligible, directly leading to local linear convergence.

    \item \textbf{Handling cases without nondegeneracy.}
    When ND fails but SC holds, we first show that the several ``partial'' sequences, \eg the off-block-diagonal part of $\Zk-\Zs$, vanish at a local linear rate. To establish the convergence of the full sequence, we further extend the regularized backward error bound for spectrahedra~\cite[Lemma 2.3]{sturm2000siopt-error-bounds-linear-matrix-inequalities} to the (scaled) KKT system of SDP. This helps build a local ``conditional sharpness'' property, analogous to sharpness in LP~\cite{applegate2023mp-faster-lp-sharpness} but with an additional term related to the minimal faces of the PSD cone. This ``conditional sharpness'', together with the convergence of ``partial'' sequences, finally yields the R-linear rate of convergence for the ADMM iterates.
\end{itemize}
\rebuttal{
Notably, our proof framework differs from that of~\cite{han2018mor-linear-rate-admm} in that it does not directly rely on metric subregularity of the KKT operator, which is generally intractable to verify numerically. At the end of the paper, we also briefly discuss the relationship between our framework and several local error bounds in the ADMM literature. In particular, in the ADMM-for-SDP setting, we show that strict complementarity at the limiting KKT point implies a weak local error bound for the DRS operator.
}

\paragraph{Empirical contribution.}
In parallel with the theoretical findings, we present extensive numerical evidence demonstrating the prevalent local linear convergence of ADMM when solving SDPs. To systematically analyze this phenomenon, we categorize SDP problems into four families based on whether the nondegeneracy (ND) and strict complementarity (SC) conditions hold or fail. A representative subset of our numerical results is shown in \cref{fig:intro}, where ADMM consistently enters a linear convergence regime across all four cases. The full set of numerical experiments is detailed in \cref{sec:exp}, covering a broad range of SDP instances. Our test suite includes standard benchmark datasets~\cite{abor1999dataset-dimacs,mittelmann2006dataset-sparse-sdp-problems} as well as newly generated SDP problems in real-world applications~\cite{han2025arxiv-xm,yang2022pami-outlier-robust-geometric-perception}. These instances span classical MAXCUT-style SDPs and more challenging problems from the moment-SOS relaxations with finite convergence~\cite{lasserre2001siopt-global,wang2023arxiv-manisdp}. All the SDP problems are available at 
\url{https://github.com/ComputationalRobotics/admmsdp-linearconv}.


\begin{figure}[htbp]
    \centering

    \begin{minipage}{\textwidth}
        \centering
        \begin{tabular}{cc}
            \begin{minipage}{0.43\textwidth}
                \centering
                \includegraphics[width=\columnwidth]{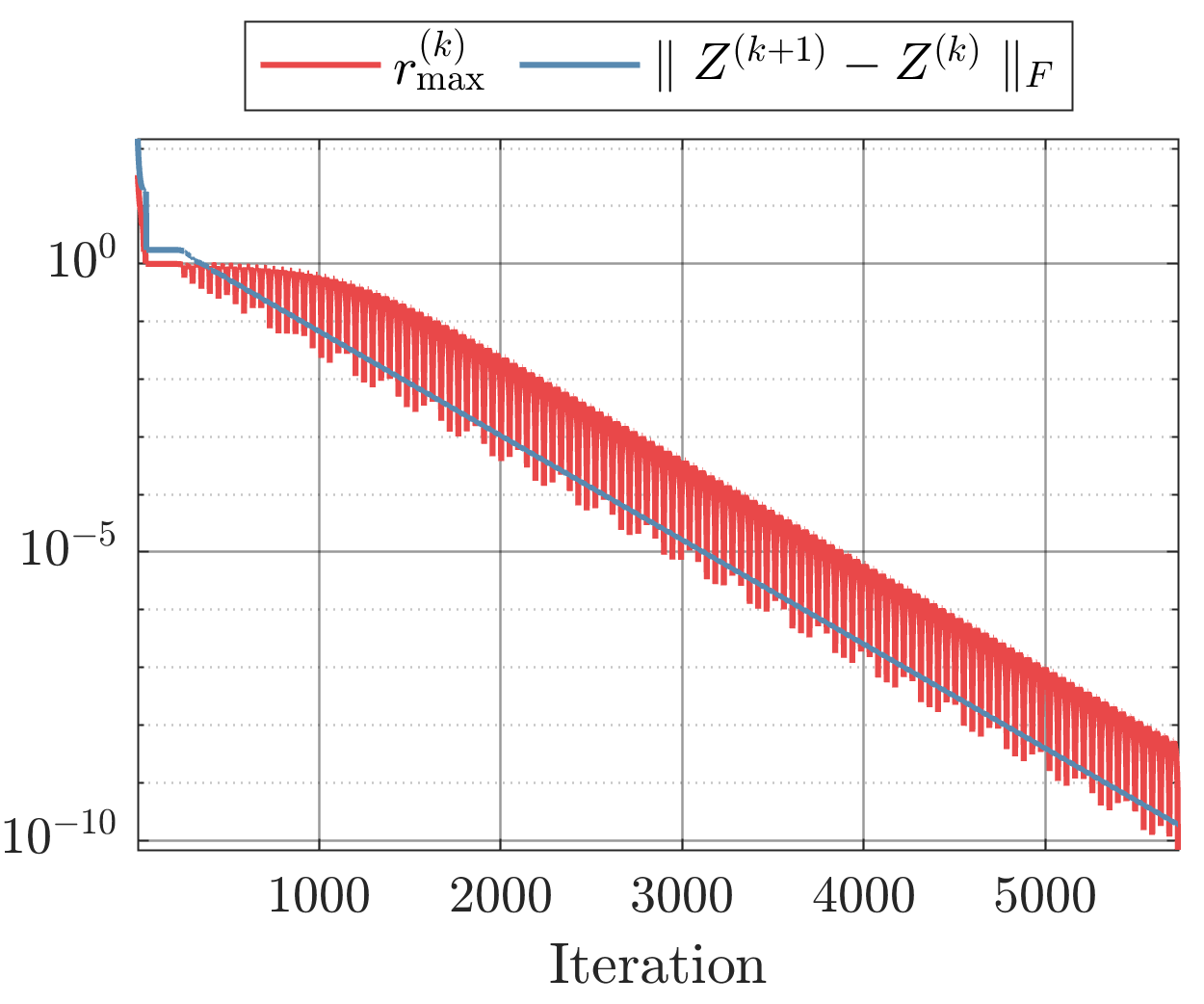}
                (a) ND holds and SC holds
            \end{minipage}

            \begin{minipage}{0.43\textwidth}
                \centering
                \includegraphics[width=\columnwidth]{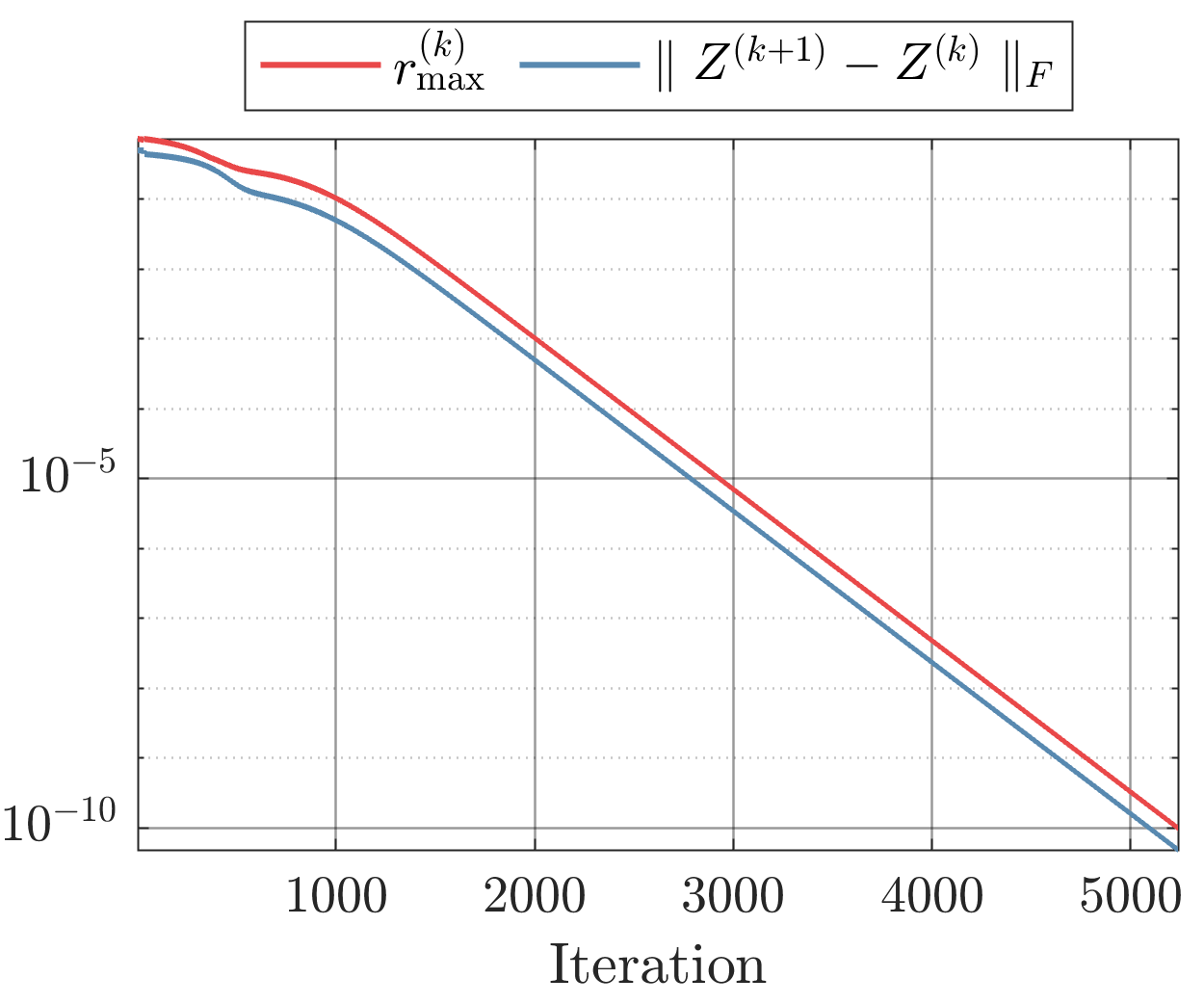}
                (b) ND holds and SC fails
            \end{minipage}
        \end{tabular}
    \end{minipage}

    \begin{minipage}{\textwidth}
        \centering
        \begin{tabular}{cc}
            \begin{minipage}{0.43\textwidth}
                \centering
                \includegraphics[width=\columnwidth]{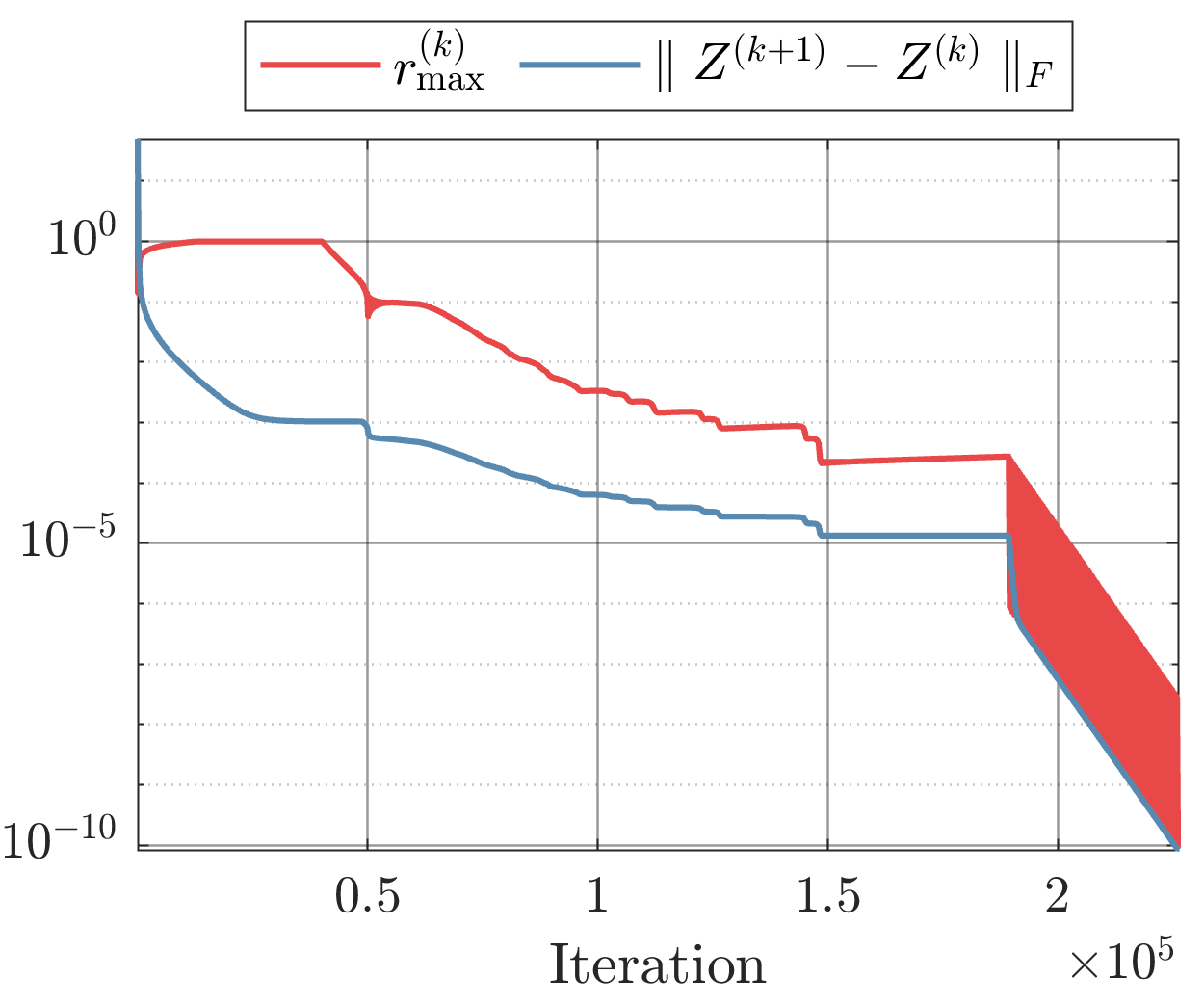}
                (c) ND fails and SC holds
            \end{minipage}

            \begin{minipage}{0.43\textwidth}
                \centering
                \includegraphics[width=\columnwidth]{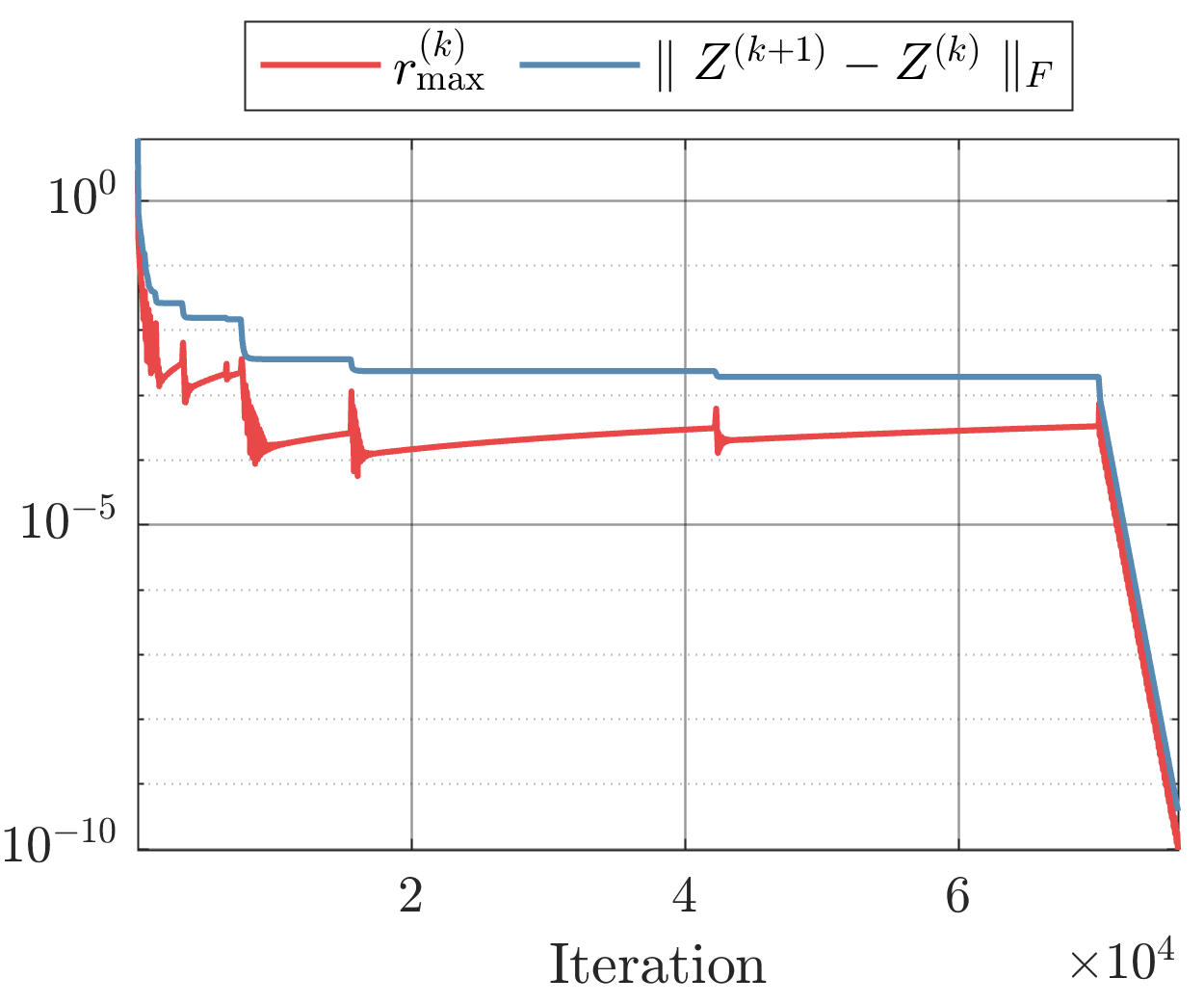}
                (d) ND fails and SC fails
            \end{minipage}
        \end{tabular}
    \end{minipage}

    \caption{Four representative SDP instances. (a) A toy structure-from-motion problem from~\cite{han2025arxiv-xm}; (b) A toy example from~\cite[pp. 44]{wolkowicz2000book-sdp-handbook}; (c) A Quasar problem from~\cite{yang2019quaternion} with random initialization; (d) Second-order relaxation for a random BQP problem~\cite{yang2023mp-stride} with all-zeros initialization. Here, $\rmax$ denotes the maximum KKT residual. In all the cases, ADMM with fixed $\sigma$ parameter eventually exhibits local linear convergence.} 
    \label{fig:intro}
\end{figure}

In addition, we document ``failure'' cases of ADMM in which the maximum KKT residual remains above $10^{-10}$ even after reaching the iteration limit ($10^6$ iterations) or the time limit (100 hours). These cases exhibit a common feature: the minimum positive eigenvalue of the converged $\Xs$ or $\Ss$ is near zero. This behavior closely resembles difficult cases in LP~\cite{lu2024mp-geometry-pdhg-lp} and can be partly explained via our proof framework.

\paragraph{Open questions: rank identification and beyond.}
In first-order methods for LP~\cite{lu2024mp-geometry-pdhg-lp}, local linear convergence is achieved alongside with basis identification. An analogous result in ADMM for SDP (with SC) would be \textit{rank identification}; \ie after a finite number of iterations, ADMM identifies the solution rank, and all subsequent iterates maintain at the same rank. Though this result could be readily drawn from the partial smoothness theory \cite{wright1993identifiable,lewis2002siopt-activesets-nonsmoothness-sensitivity,drusvyatskiy2014optimality}, we provide a more direct proof that only leverages the algorithmic properties of ADMM. However, unlike the LP case, it remains unclear that whether rank identification and linear convergence occur \textit{simultaneously} in ADMM for SDP. In this work, we provide a partial answer through a numerical example and leave a full investigation for future work.

Relating our discussion back to \cref{fig:intro}, this work establishes (R-)linear convergence guarantees of ADMM for SDP, which covers cases (a)--(c) in \cref{fig:intro}. (In comparison, \cite{han2018mor-linear-rate-admm} explains cases (a) and (b).) However, whether ADMM provably attains local linear convergence in case (d), where both ND and SC fail, remains an open question, a gap between theory and practice that warrants further investigation.

\subsection{Notations}

We use $\Real{n}$ to denote the set of $n$-dimensional real vectors and $\R^{n}_+$ (resp., $\R^{n}_{++}$) the set of nonnegative (resp., positive) vectors in $\Real{n}$. Denote $\Real{m \times n}$ as the set of $m \times n$ matrices. Denote $\Sn$ as the set of real symmetric $n \times n$ matrices and $\Symp{n}$/$\Sympp{n}$ (resp. $\Symn{n}$/$\Symnn{n}$) as the set of positive semidefinite/positive definite matrices (resp., negative semidefinite/negative definite matrices) in $\Sn$. Denote $\bbN$ the set of nonnegative integers and for any integer $n \in \bbN$, define $[n] := 1, 2, \dots, n$. Denote $\Id$ as the identity operator, denote $\I[n]$ as the $n \times n$ identity matrix and $E_n$ (resp., $E_{m \times n}$) as the $n \times n$ (resp. $m \times n$) all-ones matrix. For $A \in \Sn$, $\lammin{A}$ (resp., $\lammax{A})$ represents its minimal eigenvalue (resp., maximal eigenvalue). For $x \in \Real{n}$, we denote $\normtwo{x}$ as its Euclidean norm. For $X \in \Real{m \times n}$, $\normtwo{X}$ represents its spectral norm, $\normF{X}$ its Frobenius norm, and $\norm{X}$ an arbitrary norm. For a linear operator $\calM : \Sn \to \Sn$, we use $\normop{\calM}$ to denote its operator norm: $\normop{\calM} := \sup\{\normF{\calM X} \mid \normF{X} = 1\}$.

Denote $A \circ B$ as the Hadamard product between two matrices $A$ and $B$ of the same size; \ie $(A \circ B)_{ij} = A_{ij} B_{ij}$. Denote $A \otimes B$ as the Kronecker product between $A \in \Real{m \times n}$ and $B \in \Real{p \times q}$, and denote $A \oplus B = A \otimes \I[m] + \I[n] \otimes B$ as the Kronecker sum between $A \in \Real{n \times n}$ and $B \in \Real{m \times m}$. Denote $\vectorize{X}$ has the column-major vectorization of an arbitrary matrix $X$, and denote $\svec[A]: \Sn \to \Real{t(n)}$ as the symmetric vectorization of $A$, where $t(n) := \frac{n(n+1)}{2}$. Denote $\smat$ as the inverse operator of $\svec$. For an arbitrary matrix $A$, $A_{a:b, c:d}$ represents the submatrix of $A$ indexed from row~$a$ to row $b$ and from column~$c$ to column $d$. 

The distance from a point $X \in \Sn$ to a set $\calX \subseteq \Sn$ is defined as
\[
    \dist(X, \calX) := \inf_{\widetilde X \in \calX} \normF{X - \widetilde X}.
\]
We will use the same notation to denote the distance to a Cartesian product of sets:
\[
    \dist((X,S), \calX \times \calS) := \inf_{(\widetilde X, \widetilde S) \in \calX \times \calS} \sqrt{\normF{X-\widetilde X}^2 + \normF{S-\widetilde S}^2}.
\]
We denote the orthogonal projection onto a set $\calX$ as $\Proj_\calX$; in particular, $\PiSnp{\cdot}$ means the orthogonal projection onto the PSD cone.

\subsection{Outline}

After a brief review of related work in \cref{sec:related-work}, we introduce our refined error bound in \cref{sec:error-bound-intro}, a fundamental result that underpins our proof framework and holds independent interest. We then examine the local linearization of ADMM in \cref{sec:linearization} and establish its local linear convergence both with and without nondegeneracy in \cref{sec:conv-with-nd} and \cref{sec:conv-without-nd}, respectively. Due to its mathematical complexity, the full proof of our refined error bound is deferred to \cref{sec:error-bound}. In \cref{sec:exp}, we conduct extensive numerical experiments to support our theoretical findings. \Cref{sec:rank-id} briefly discusses the rank identification phenomenon as well as its relationship with local linear convergence. \rebuttal{\cref{sec:metric} briefly discusses the relationship between strict complementarity and metric subregularity, local error bounds, and hidden polyhedrality.} Finally, \cref{sec:conclusion} includes concluding remarks.


\section{Related Work}
\label{sec:related-work}

With the rapid development of data science and all fields in engineering, SDP problems are growing in scale. Efficient and scalable algorithms have been developed, analyzed and implemented for solving large-scale SDPs.

\paragraph{Augmented Lagrangian method (ALM).}
Originally introduced to enhance the performance of penalty methods~\cite{powell1969ac-nonlinear-constraints-min}, ALM has shown promise in tackling large-scale SDPs~\cite{yang2015mp-sdpnalplus-sdpsolver}. Under mild conditions (strong duality and the existence of a strictly complementary solution pair), linear convergence of ALM~\cite{zhao2010siopt-newton-cg-alm,cui2016arxiv-superlinear-alm-sdp,liao2024nips-inexact-alm-linear-convergence} is established by leveraging its connection to proximal point methods~\cite{rockafellar1976mor-alm-ppm} (PPM) and quadratic growth properties~\cite{sturm2000siopt-error-bounds-linear-matrix-inequalities,ding2023ol-strict-complementarity-error-bound}.

\paragraph{Burer--Monteiro (BM) factorization method.}
The BM method~\cite{burer2003springer-bm} replaces the conic constraint by $X = RR\tran$, reducing the problem to a lower-dimensional nonlinear program. Under specific rank and regularity conditions, it recovers the global optimum of the original SDP~\cite{boumal2016arxiv-bm-smoothsdp,tang2023arxiv-feasible-lowranksdp}. Owing to its efficiency, the BM method has achieved significant empirical success in real-world problems with low-rank solutions~\cite{han2025arxiv-xm}. It can also be combined with ALM~\cite{wang2023arxiv-manisdp} or ADMM~\cite{han2024arxiv-lowrank-admm}.

\paragraph{Spectral bundle method (SBM).}
First proposed in~\cite{helmberg2000siopt-spectral-bundle-sdp}, SBM has gained attention for its low per-iteration cost. It enjoys sublinear convergence under mild assumptions~\cite{ding2023siopt-revisit-spectral-bundle}. Furthermore, if a strictly complementary solution pair exists and the surrogate function captures the correct rank of the optimal solution, SBM achieves local linear convergence~\cite{ding2023siopt-revisit-spectral-bundle,liao2023arxiv-overview-spectral-bundle-sdp}. Similar to ALM, its linear convergence guarantees rely on quadratic growth. Recent work also incorporates SBM into ALM~\cite{liao2025arxiv-bundle-alm}.

\rebuttal{
\paragraph{First-order proximal methods and (local) linear convergence of ADMM.}
As a broad class of first-order methods derived from the monotone operator theory~\cite{ryu2023book-largescale-convexopt}, primal--dual proximal methods are widely used for SDP. In addition to ADMM~\cite{wen2010mp-admmsdp}, symmetric Gauss-Seidel (sGS)-ADMM~\cite{chen2017mp-sgsadmm} has gained traction for solving general SDPs to medium accuracy, and its connection to proximal ALM is discussed in~\cite{chen2021mp-alm-admm-equivalence}. Other proximal methods, such as the primal--dual hybrid gradient (PDHG) method~\cite{jiang2022coa-bregman-sparsesdp}, have also been explored. On the theoretical side, linear convergence of ADMM was first established for structured problems such as linear and quadratic programs~\cite{boley13siopt-linearconv-admm-lp-qp}, and later extended to strongly convex settings under additional smoothness and rank conditions~\cite{deng2016global,nishihara2015pmlr-general-analysis-admm}. Subsequent local and variational-analytic analyses further identified sufficient conditions based on geometric regularity, error bounds, calmness, and metric-subregularity~\cite{liang2017jota-local-convergence-admm,davis2017fasterconvergence,hong2017mp-linear-convergence-admm,liu2018partial,yuan2020discerning,han2018mor-linear-rate-admm}. Nonetheless, establishing local linear convergence remains much harder for SDP than for ALM or SBM. In particular, the known sufficient conditions for linear convergence of ADMM, \eg strong convexity~\cite{nishihara2015pmlr-general-analysis-admm}, local polyhedrality~\cite{liang2017jota-local-convergence-admm} and other growth conditions \cite{deng2016global,hong2017mp-linear-convergence-admm,davis2017fasterconvergence,liu2018partial,yuan2020discerning,zamani2024exact}, either fail or remain unclear for SDP. See~\cref{sec:metric} for more discussion.
}


\section{A Refined Error Bound for PSD Cone Projection}
\label{sec:error-bound-intro}

We see from \eqref{eq:intro:admm-one-step} that the only nonlinear operation in ADMM is the projection onto the PSD cone. So, to better understand the convergence of ADMM for SDP, we need to study the local behavior of the PSD cone projector $\Pi_{\psd{n}}$. A classic result on the perturbation theory of $\Pi_{\psd{n}}$ is \cite[Theorem 4.6]{sun2002mor-semismooth-matrix-valued}, which is restated below for self-containment.
\begin{lemma}[\text{\cite[Theorem 4.6]{sun2002mor-semismooth-matrix-valued}}] \label{lem:eb-intro-SS02}
    Given an $n \times n$ symmetric nonsingular matrix $Z \in \Sn$, denote its eigenvalue decomposition by
    \[
        Z = Q \LamXS Q\tran, \;\; \text{where} \ \lam{1} \geq \cdots \geq \lam{r} > 0 > \lam{r+1} \geq \cdots \geq \lam{n}
    \]
    and $Q \in \R^{n \times n}$ is an orthogonal matrix. Then, the function $\Pi_{\psd{n}} \colon \Sn \to \Sn$ is Fr\'echet differentiable and its Fr\'echet differential at $\Z$ for $\H \in \Sn$ is given by
    \[
        (\PiSnp{\Z})^\prime (\H) = Q \big(\Omega \circ (Q\tran \H Q) \big) Q\tran,
    \]
    where the $n \times n$ symmetric matrix $\Omega$ is defined as
    \begin{equation} \label{eq:eb-intro-Omega}
        \Omega = \mymat{
        1 & \cdots & 1 & \frac{\lam{1}}{\lam{1} - \lam{r+1}} & \cdots & \frac{\lam{1}}{\lam{1} - \lam{n}} \\
        \vdots & \ddots & \vdots & \vdots & \ddots & \vdots \\
        1 & \cdots & 1 & \frac{\lam{r}}{\lam{r} - \lam{r+1}} & \cdots & \frac{\lam{r}}{\lam{r} - \lam{n}} \\
        \frac{\lam{1}}{\lam{1} - \lam{r+1}} & \cdots & \frac{\lam{r}}{\lam{r} - \lam{r+1}} & 0 & \cdots & 0 \\
        \vdots & \ddots & \vdots & \vdots & \ddots & \vdots \\
        \frac{\lam{1}}{\lam{1} - \lam{n}} & \cdots & \frac{\lam{r}}{\lam{r} - \lam{n}} & 0 & \cdots & 0 
    } := \mymat{\E[r] & \Theta\tran \\ \Theta & 0}.
    \end{equation}
    Here, $E_r$ is the all-ones matrix of size $r \times r$ and $\Theta \in \Real{(n-r) \times r}$ captures the off-block-diagonal part in $\Omega$:
    \begin{equation} \label{eq:eb-intro-Theta}
        \Theta_{ij} = \frac{\lam{j}}{\lam{j} - \lam{i+r}} \in (0,1), \quad \text{for} \ i \in [n-r], \, j \in [r].
    \end{equation}
    Moreover, for any sufficiently small perturbation $\H \in \Sn$, it holds that
    \[
        \normtwo{\PiSnp{\Z + \H} - \PiSnp{\Z} - Q \big(\Omega \circ (Q\tran \H Q) \big) Q\tran} = \calO(\normtwo{\H}^2).
    \]
\end{lemma}
However, the following simple example illustrates that the above result may not be tight and motivates our refined error bound. To see this, assume for brevity that $Q = I$ in \cref{lem:eb-intro-SS02} and partition the perturbation $H \in \Sn$ as
\begin{equation} \label{eq:eb-intro-partition}
    H = \mymat{\HX & \HO\tran \\ \HO & \HS}, \quad \text{where} \ \HX \in \Sym{r}, \HS \in \Sym{n-r}, \ \text{and} \ \HO \in \Real{(n-r) \times r}.
\end{equation}
We set $\HO = 0$ and $\normtwo{\H} \le \sigmin{\Z} := \min \{\lam{r}, -\lam{r+1}\}$. Then we have
\[
    \Z + \H = \mymat{\LamX + \HX & 0 \\ 0 & \LamS + \HS},
\]
where $\LamX = \diag{\lam{1},\ldots,\lam{r}}$ and $\LamS = \diag{\lam{r+1},\ldots,\lam{n}}$. We then obtain from Weyl's inequality that
\begin{align*}
    \lammin{\LamX + \HX} &\ge \lam{r} + \lammin{\HX} \ge \lam{r} - \normtwo{\HX} \ge 0, \\
    \lammax{\LamS + \HS} &\le \lam{r+1} + \lammax{\HS} \le \lam{r+1} + \normtwo{\HS} \le 0,
\end{align*}
where we also use the facts that $\normtwo{\HX} \leq \normtwo{\H}$ and $\normtwo{\HS} \leq \normtwo{\H}$. Therefore,
\begin{align*}
    \PiSnp{\Z + \H} - \PiSnp{\Z} 
    = \mymat{
        \LamX + \HX & 0 \\
        0 & 0
    } - \mymat{\LamX & 0 \\ 0 & 0}
    = \Omega \circ \H,
\end{align*}
\ie the residual term is exactly zero while $\normtwo{\H}$ is nonzero. This motivates the following refined error bound for the PSD projection $\Pi_{\psd{n}}$, which involves the ``off-block-diagonal'' part $H_O$ in the residual.
\begin{theorem} \label{thm:eb-intro-thm}
    Given an $n \times n$ symmetric nonsingular matrix $Z \in \Sn$, denote its eigenvalue decomposition by
    \[
        Z = Q \LamXS Q\tran, \;\; \text{where} \ \lam{1} \geq \cdots \geq \lam{r} > 0 > \lam{r+1} \geq \cdots \geq \lam{n}
    \]
    and $Q \in \Real{n \times n}$ is an orthogonal matrix. Then, there exist two positive constants $C_{\mathrm{EB}}$ and $\alpha_{\mathrm{EB}}$ such that for all $\H \in \Sn$ with $\normtwo{\H} \leq C_{\mathrm{EB}}$, it holds that
    \begin{equation} \label{eq:eb-intro-bnd}
        \normtwo{\PiSnp{\Z + \H} - \PiSnp{\Z} - Q (\Omega \circ \tH) Q\tran} \le \alpha_{\mathrm{EB}} \cdot \normtwo{\tHO} \cdot \normtwo{\H}, 
    \end{equation}
    where $\tH := Q\tran H Q$ is partitioned as
    \[
        \tH = \mymat{\tHX & \tHO\tran \\ \tHO & \tHS} \quad \text{with} \ \tHX \in \Sym{r}, \ \tHS \in \Sym{n-r}, \ \text{and} \ \tHO \in \Real{(n-r) \times r}.
    \]
\end{theorem}

\begin{remark}
When $Q = I$, the bound \eqref{eq:eb-intro-bnd} reduces to
\[
    \normtwo{\PiSnp{\Z + \H} - \PiSnp{\Z} - \Omega \circ \H} \le \alpha_{\mathrm{EB}} \cdot \normtwo{\HO} \cdot \normtwo{\H}.
\]
This aligns with our observation in the motivating example: when $\HO=0$, both sides of the above inequality becomes zero. One shall also note that, in general, $\tHO$ is not the off-block-diagonal part of the perturbation $\H$. In fact, without using the notation~$\tH$, the bound~\eqref{eq:eb-intro-bnd} can be written as
\[
    \normtwo{\PiSnp{\Z + \H} - \PiSnp{\Z} - Q \big(\Omega \circ (Q\tran \H Q) \big) Q\tran} \le \alpha_{\mathrm{EB}} \cdot \normtwo{Q_S\tran \H Q_X} \cdot \normtwo{\H},
\]
where we partition the eigenvalue decomposition of $Z$ as
\[
    Z = \mymat{Q_X & Q_S} \mymat{\LamX & 0 \\ 0 & \LamS} \mymat{Q_X\tran \\ Q_S\tran}.
\]
\end{remark}

\begin{remark}
\label{rem:psdproj-eb}
As all norms are equivalent, \eqref{eq:eb-intro-bnd} implies that there exists a positive constant~$\alpha^\prime_{\mathrm{EB}}$ such that
\[
    \normF{\PiSnp{\Z + \H} - \PiSnp{\Z} - Q (\Omega \circ \tH) Q\tran} \le \alpha^\prime_{\mathrm{EB}} \cdot \normF{\tHO} \cdot \normF{\H}.
\]
In the convergence analysis (\cref{sec:linearization,sec:conv-with-nd,sec:conv-without-nd}), we mainly use the Frobenius norm of matrices, consistent with most literature.
\end{remark}

\begin{remark}
In~\cite[Proposition 3.4]{cui2016arxiv-superlinear-alm-sdp}, the authors establish another perturbation property for the PSD cone projector. Their result has two key distinctions from \cref{thm:eb-intro-thm}: (1) their results cover cases where $\Z$ is singular, while ours only focuses on the nonsingular case; (2) under the nonsingularity assumption, our results can directly lead to theirs; \ie \cref{thm:eb-intro-thm} is stronger than~\cite[Proposition 3.4]{cui2016arxiv-superlinear-alm-sdp}. See \cref{app:sec:details-cui-prop} for detailed discussion.
\end{remark}


\section{Local Linearization of ADMM}
\label{sec:linearization}

With a better understanding of the PSD cone projection, we now study the local behavior of ADMM. Our analysis is different from the standard approaches for ADMM and starts by locally linearizing the iteration~\eqref{eq:intro:admm-one-step}. \rebuttal{In particular, we show that when near optimum, the residual $\Hkpo := \Zk - \Zs$ is almost a linear transformation of the previous residual $\Hk$, plus a correction term in the order of $\calO(\normF{\HO} \normF{\H})$.}

The rest of the section is devoted to study this linearization and is organized as follows. \Cref{sec:lin-ass} lists all the assumptions made throughout the paper, \cref{sec:lin-linearization} presents the local linearization of ADMM, and \cref{sec:lin-property} describes the properties of such a linearization.

\subsection{Assumptions} \label{sec:lin-ass}

We make the following assumption on the pair of primal--dual SDPs \eqref{eq:intro-sdp}.
\begin{assumption} \label{ass:lin-sol}
\begin{enumerate}[label=(\alph*)]
    \item The linear operator $\Asdp : \Sn \to \Real{m}$ is surjective.

    \item The pair of primal--dual SDPs \eqref{eq:intro-sdp} has a nonempty set of KKT points.
\end{enumerate}
\end{assumption}

From convex duality theory, any pair of primal--dual solutions $(\Xs, \ys, \Ss)$ of \eqref{eq:intro-sdp} satisfies complementary slackness; \rebuttal{\ie $\Xs$ and $\Ss$ admit the decompositions in \eqref{eq:intro:Xs-sigSs}.} Moreover, $\inprod{\Xs}{\Ss} = 0$ and $\rank{\Xs} + \rank{\Ss} \leq n$. When the above inequality holds with equality, the solution pair $(\Xs,\ys,\Ss)$ is called strictly complementary.

It is known that under \cref{ass:lin-sol}, three-step ADMM \eqref{eq:intro:admm-three-step} converges to a KKT point $(\Xs, \ys, \Ss)$, or equivalently, one-step ADMM \eqref{eq:intro:admm-one-step} converges to the point $\Zs := \Xs - \sigma \Zs$. Our analysis assumes additionally that the convergent point of ADMM is strictly complementary.
\begin{assumption} \label{ass:lin-sc}
    Three-step ADMM \eqref{eq:intro:admm-three-step} converges to a KKT point $(\Xs, \ys, \Ss)$ satisfying strict complementarity; \ie, $\rank{\Xs} + \rank{\Ss} = n$. 
\end{assumption}
\Cref{ass:lin-sc} is equivalent to the condition that the convergent point $\Zs$ of one-step ADMM \eqref{eq:intro:admm-one-step} is nonsingular. \Cref{ass:lin-sc} is a mild assumption in the sense that it holds for generic SDPs \cite{alizadeh1997mp-complementarity-nondegeneracy-sdp} \rebuttal{(\ie the set of SDP data $(\Asdp, b, C)$ that fails strict complementarity is of Lebesque measure zero)}. Numerical experiments in \cref{sec:exp} further demonstrate that even for degenerate SDPs with multiple solutions, one-step ADMM often converges to a nonsingular $\Zs$ (if one exists) when initialized with a random (standard Gaussian) guess $\Z^{(0)}$. 

For ease of presentation, we assume without loss of generality that the convergent points $\Xs$ and $\Ss$ are diagonal, \ie $\Qs = \I[n]$ in~\eqref{eq:intro:Xs-sigSs}. This assumption does not limit the scope of our conclusions because we can readily construct a pair of SDPs equivalent to \eqref{eq:intro-sdp} and generate ADMM iterates $(\widetilde X^{(k)},\widetilde y^{(k)}, \widetilde S^{(k)})$ orthogonally similar to the iterates $(X^{(k)},y^{(k)}, S^{(k)})$ generated by \eqref{eq:intro:admm-three-step}. To see this, suppose $(X^{(k)},y^{(k)}, S^{(k)})$ converges to $(\Xs,\ys,\Ss)$ with $\Qs \neq \I[n]$. We construct another pair of SDPs 
\begin{equation} \label{eq:lin-sdp}
\begin{array}{llllll}
    \text{Primal:} \;\; & \text{minimize} & \langle{\widetilde C}, {\widetilde X} \rangle & \qquad \qquad \quad \text{Dual:} \;\; & \text{maximize} & b\tran \tilde y \\[2pt]
    & \text{subject to} & \widetilde \calA \widetilde X = b & & \text{subject to} & \widetilde \calA^\ast \tilde y + \widetilde S = \widetilde C \\[2pt]
    & & \widetilde X \in \psd{n} & & & \widetilde S \in \psd{n}
\end{array}
\end{equation}
with primal variable $\widetilde X \in \Sn$ and dual variables $(\tilde y, \widetilde S) \in \Real{m} \times \Sn$. The coefficients are $\widetilde C := \Qs\tran C \Qs$ and $\widetilde A_i := \Qs\tran A_i \Qs$ for all $i \in [m]$. If we apply three-step ADMM \eqref{eq:intro:admm-three-step} to the modified SDP \eqref{eq:lin-sdp}, starting at $\widetilde X^{(0)} := \Qs\tran X^{(0)} \Qs$ and $\widetilde S^{(0)} := \Qs\tran S^{(0)} \Qs$, straightforward calculations show that the generated sequence $(\widetilde X^{(k)},\widetilde y^{(k)}, \widetilde S^{(k)})$ is related to $(X^{(k)},y^{(k)}, S^{(k)})$ as follows:
\[
    \widetilde X^{(k)} = \Qs\tran \Xk \Qs, \qquad \tilde y^{(k)} = \yk, \qquad \widetilde S^{(k)} = \Qs\tran \Sk \Qs, \quad \text{for all $k \in \bbN$}.
\]
Moreover, the sequence $(\widetilde X^{(k)},\widetilde y^{(k)}, \widetilde S^{(k)})$ converges to $(\Qs\tran \Xs \Qs, \ys, \Qs\tran \Ss \Qs)$, a KKT point of the modified SDP \eqref{eq:lin-sdp}.

\rebuttal{
\begin{remark}
    One may notice that Assumption~\ref{ass:lin-sc}, like several other regularity conditions used to guarantee ADMM's local linear convergence (\eg metric subregularity~\cite{liu2018partial} and two-sided nondegeneracy~\cite{han2018mor-linear-rate-admm}), is \emph{post hoc}: it can generally be verified only after the algorithm has converged. This is largely unavoidable, because local linear convergence is inherently a statement about the geometry in a neighborhood of the limit point. For example, there exist simple SDP instances~\cite[Example 1]{cui2016arxiv-superlinear-alm-sdp} in which metric subregularity of the KKT operator fails at a particular KKT point, while holding at the remaining points in the optimal set.
    Compared with other regularity conditions, Assumption~\ref{ass:lin-sc} has two main advantages:
    \begin{itemize}
        \item It can be checked efficiently in practice. Given an ADMM iterate $(\Xk, \yk, \Sk)$ near a KKT point, one needs only $\calO(n^3)$ time to determine whether strict complementarity holds, via a numerical full-rankness check on $\Xk - \sigma \Sk$. By contrast, checking two-sided nondegeneracy requires $\calO(n^6)$ time because one must form a basis for $\calT_{\Xk}$ and $\calT_{\Sk}$. General metric subregularity, on the other hand, is typically not directly verifiable in computation.
        \item In a wide range of real-world SDPs, two-sided nondegeneracy is known to fail~\cite{yang2019quaternion,kang2024wafr-strom}. By contrast, strictly complementary solutions still exist for a substantial subset of such degenerate SDPs~\cite{ding2023siopt-revisit-spectral-bundle}.
    \end{itemize}
    In Section~\ref{sec:metric}, we will also see connections between strict complementarity and several forms of metric subregularity/local error bounds.
\end{remark}
}

\subsection{Local Linearization} \label{sec:lin-linearization}

Now we study the local behavior of one-step ADMM \eqref{eq:intro:admm-one-step}. In particular, we linearize the ADMM iteration when near optimum. For ease of representation, we define
\[
    \PA := \Pi_{\range(\AsdpT)} = \AsdpT (\Asdp \AsdpT)^{-1} \Asdp, \;\; \PAp := \Id - \PA, \;\;\Omep := E_n - \Omega, \;\; \Thep := E_{(n-r) \times r} - \Theta,
\]
where recall $E_r$ (resp., $E_{(n-r) \times r}$) is the all-ones matrix of size $n \times n$ (resp., $(n-r) \times r$). With the above abbreviations, we rewrite the iteration \eqref{eq:intro:admm-one-step} as
\begin{equation} \label{eq:lin-linearization}
    \Zkpo - \Zs = \Madmm (\Zk - \Zs) + \psik,
\end{equation}
where
\begin{align}
    \Madmm (\H) &:= \PA (\Omep \circ \H) + \PAp (\Omega \circ \H), \label{eq:lin-M} \\
    \psik &:= (\Id - 2\PA) \big(\PiSnp{\Zk} - \PiSnp{\Zs} - \Omega \circ (\Zk - \Zs) \big). \label{eq:lin-psik}
\end{align}
This reformulation \eqref{eq:lin-linearization} says that when near optimum, the residual $\Hkpo := \Zkpo - \Zs$ is almost a linear transformation of the previous residual $\Hk := \Zk - \Zs$, with a quadratic correction term $\psik = \calO(\normtwo{\Hk}^2)$ (see \cref{lem:eb-intro-SS02}).
\begin{proof}[Proof of \eqref{eq:lin-linearization}]
    Note that $\Zs$ is a fixed point of \eqref{eq:intro:admm-one-step}, \ie
    \[
        \Zs = \PA (-2\PiSnp{\Zs} + \Zs) + \PiSnp{\Zs} + \AsdpT (\Asdp \AsdpT)^{-1} b - \sigma (\Id - \PA) C. 
    \]
    Substituting back into \eqref{eq:intro:admm-one-step} yields
    \begin{align*}
        \Zkpo - \Zs &= (\Id - 2\PA) \big(\PiSnp{\Zk} - \PiSnp{\Zs} \big) + \PA(\Zk - \Zs) \\
        &= (\Id - 2\PA) \Omega \circ (\Zk - \Zs) + \PA(\Zk - \Zs) \\
        &\phantom{=} \mbox{} + (\Id - 2\PA) \big(\PiSnp{\Zk} - \PiSnp{\Zs} - \Omega \circ (\Zk - \Zs) \big) \\
        &= (\Id - 2\PA) \Omega \circ (\Zk - \Zs) + \PA(\Zk - \Zs) + \psik,
    \end{align*}
    by the definition of $\psik$. Moreover, for any $\H \in \Sn$, we have
    \rebuttal{
    \begin{align*}
        (\Id - 2\PA) \Omega \circ \H + \PA \H &= \PAp (\Omega \circ \H) - \PA (\Omega \circ \H) + \PA \H \\
        &= \PAp (\Omega \circ \H) + \PA(\Omep \circ \H) \\
        &= \Madmm (H),
    \end{align*}
    }
    where the first line uses $\PAp = \Id - \PA$ and the second line uses $\Omep = E_r - \Omega$.
\end{proof}

\subsection{Properties of \texorpdfstring{$\Madmm$}{M} and \texorpdfstring{$\psik$}{M}}
\label{sec:lin-property}

Now we represent some properties of the linear operator $\Madmm$ and the residual $\psik$ that will be used to establish the linear convergence of ADMM. In particular, we characterize the nonempty set of fixed points $\Fix(\Madmm) := \{H \mid \Madmm(H) = H\}$; see \cref{prop:lin-M}. Throughout this subsection, we partition the matrix $\H$ (or $\Hk$) as in~\eqref{eq:eb-intro-partition} with $r := \rank{\Xs}$.

\begin{lemma} \label{lem:lin-auxi}
    Suppose $\Omega \in \Sn$ is defined as in \eqref{eq:eb-intro-Omega} with $r := \rank{\Xs}$. For any matrix $H \in \Sn$ as partitioned in~\eqref{eq:eb-intro-partition}, it holds that
    \begin{equation} \label{eq:lin-inprod}
        \inprod{\Omega \circ \H}{\Omep \circ \H} = 2\inprod{\Theta \circ \HO}{\Thep \circ \HO} \ge 0
    \end{equation}
    with equality only if $\HO = 0$, and that
    \begin{equation} \label{eq:lin-minus}
        \normF{\H}^2 - \normF{\Madmm(\H)}^2 = \normF{\PA (\Omega \circ \H)}^2 + \normF{\PAp (\Omep \circ \H)}^2 + 4 \inprod{\Theta \circ \HO}{\Thep \circ \HO}.
    \end{equation}
\end{lemma}
\begin{proof}
    From the definition of $\Omega$ \eqref{eq:eb-intro-Omega} and the partition of $\H$ \eqref{eq:eb-intro-partition}, we see that
    \begin{align}
        \inprod{\Omega \circ \H}{\Omep \circ \H} &= \inprod{
            \mymat{E_r & \Theta\tran \\ \Theta & 0} \circ \mymat{\HX & \HO\tran \\ \HO & \HS}
        }{
            \mymat{0 & (\Thep)\tran \\ \Thep & E_{n-r}} \circ \mymat{\HX & \HO\tran \\ \HO & \HS}
        } \nonumber \\
        &= \inprod{
            \mymat{\HX & \Theta\tran \circ \HO\tran \\ \Theta \circ \HO & 0}
        }{
            \mymat{0 & (\Thep)\tran \circ \HO\tran \\ \Thep \circ \HO & \HS}
        } \nonumber \\
        &= 2 \inprod{\Theta \circ \HO}{\Thep \circ \HO} \ge 0. \label{eq:lin-inprod-prf}
    \end{align}
    Since all the entries in $\Theta$ and $\Thep$ are strictly positive, the inner product \eqref{eq:lin-inprod-prf} is zero if and only if $\HO=0$.

    To show the second conclusion, we first decompose $\H$ as
    \begin{equation} \label{eq:lin-H-decompose}
        \H = \PA (\Omega \circ \H) + \PA (\Omep \circ \H) + \PAp (\Omega \circ \H) + \PAp (\Omep \circ \H).
    \end{equation}
    Then we have
    \begin{align*}
        \normF{\H}^2 &= \normF{\PA (\Omega \circ \H)}^2 + \normF{\PA (\Omep \circ \H)}^2 + \normF{\PAp (\Omega \circ \H)}^2 + \normF{\PAp (\Omep \circ \H)}^2 \\
        &\phantom{=} \mbox{} + 2\inprod{\PA (\Omega \circ \H)}{\PA (\Omep \circ \H)} + 2\inprod{\PAp (\Omega \circ \H)}{\PAp (\Omep \circ \H)},
    \end{align*}
    and
    \[
        \normF{\Madmm(\H)}^2 = \normF{\PAp (\Omega \circ \H) + \PA (\Omep \circ \H)}^2 
        = \normF{\PAp (\Omega \circ \H)}^2 + \normF{\PA (\Omep \circ \H)}^2.
    \]
    Combining both expressions with \eqref{eq:lin-inprod} gives the desirable result.
\end{proof}

\begin{proposition} \label{prop:lin-M}
    The linear operator $\Madmm : \Sn \to \Sn$ has the following properties.
    \begin{enumerate}[label=(\alph*)]
        \item $\Madmm$ is firmly nonexpansive under the Frobenius norm.
        
        \item The sequence $\{\Madmm^k\}_{k=1}^\infty$ converges to $\pfm$.
        
        \item $H \in \Fix(\Madmm)$ if and only if the following three conditions holds
        \begin{equation} \label{eq:lin-fix}
            H_O = 0, \qquad \mymat{H_X & 0 \\ 0 & 0} \in \nullspace(\Asdp), \qquad \mymat{0 & 0 \\ 0 & H_S} \in \range(\AsdpT).
        \end{equation}

        \item $\normop{\Madmm - \pfm} < 1$.
    \end{enumerate}
\end{proposition}
\begin{proof}
    \textit{Part (a):} We verify the firm nonexpansiveness of $\Madmm$ via its definition:
    \begin{subequations}
    \begin{align}
        &\ \inprod{\calM(\H)}{\H} \nonumber \\
        =&\ \inprod{
            \PA (\Omep \circ \H) + \PAp (\Omega \circ \H)
        }{
            \PA (\Omega \circ \H) + \PA (\Omep \circ \H) + \PAp (\Omega \circ \H) + \PAp (\Omep \circ \H)
        } \nonumber \\
        =&\ \normF{\Madmm (\H)}^2 + \inprod{\PA (\Omep \circ \H) + \PAp (\Omega \circ \H)}{\PA (\Omega \circ \H) + \PAp (\Omep \circ \H)} \nonumber \\
        =&\ \normF{\Madmm (\H)}^2 + \inprod{\PA (\Omep \circ \H)}{\PA (\Omega \circ \H)} + \inprod{\PAp (\Omega \circ \H)}{\PAp (\Omep \circ \H)} \nonumber \\
        =&\ \normF{\Madmm (\H)}^2 + \inprod{\Omep \circ \H}{\Omega \circ \H} \label{eq:lin-fn-1} \\
        \ge &\ \normF{\Madmm (\H)}^2, \label{eq:lin-fn-2}
    \end{align}
    \end{subequations}
    where~\eqref{eq:lin-fn-1} uses the fact that $\PA \PAp = 0$ and~\eqref{eq:lin-fn-2} uses \eqref{eq:lin-inprod}.

    \textit{Part (b)} follows readily from part (a) and monotone operator theory; see \cite[Proposition 5.16 (ii)]{bauschke2017springer-convex-analysis-hilbert-spaces} and \cite[Corollary 2.7 (ii)]{bauschke2016springer-drs-two-subspaces}.

    \textit{Part (c):} From the decomposition of $\H$~\eqref{eq:lin-H-decompose} and the definition of $\Madmm$~\eqref{eq:lin-M}, we see that
    \[
        \H \in \Fix(\Madmm) \qquad \Longleftrightarrow \qquad \PAp (\Omep \circ \H) = 0 \text{ and } \PA (\Omega \circ \H) = 0.
    \]
    On one hand, if $\H \in \Fix(\Madmm)$, we conclude from \cref{lem:lin-auxi} that $\HO$ has to be zero. Then expanding $\PA(\Omega \circ \H) = 0$ gives
    \[
        \PA \left(\mymat{\HX & 0 \\ 0 & 0} \right) = 0,
    \]
    which is equivalent to the second condition in \eqref{eq:lin-fix} (since $\PA := \Pi_{\nullspace(\Asdp)^\perp}$). Similarly, expanding $\PAp(\Omep \circ \H) = 0$ gives the last condition in \eqref{eq:lin-fix}.

    On the other hand, the first two conditions in \eqref{eq:lin-fix} imply that $\PA (\Omega \circ \H) = 0$ (since all the entries in $\Theta$ zero strictly positive). Similarly, $\HO=0$ and the last condition in \eqref{eq:lin-fix} imply $\PAp (\Omep \circ \H) = 0$. Combining the two results yields $\H \in \Fix(\Madmm)$. 

    \textit{Part (d):} is equivalent to show that $\normF{(\Madmm - \pfm) \H} < \normF{\H}$ for any nonzero $\H$. If $\H \in \Fix(\Madmm)$ (and $\H \neq 0$), then $\normF{(\Madmm - \pfm) \H} = 0 < \normF{\H}$. Otherwise, $\H \notin \Fix(\Madmm)$, and at least one of the three conditions in \eqref{eq:lin-fix} is not satisfied. So, at least one of the three terms on the right-hand side of~\eqref{eq:lin-minus} is positive, which implies the desirable result.
\end{proof}

\begin{proposition} \label{prop:lin-psi}
    There exist two constants $\bar k_\Psi \in \bbN$ and $\alpha_\Psi > 0$ such that for any integer $k \ge \bar k_\Psi$, it holds that
    \[
        \normF{\psik} \leq \alpha_\Psi \cdot \normF{\HOk} \cdot \normF{\Hk},
    \]
    where $\Hk := \dZk$ is partitioned as in \eqref{eq:eb-intro-partition}.
\end{proposition}
\begin{proof}
    The linear operator $2\PA - \Id$ is the reflection operator, and thus preserves the Frobenius norm. Thus, we have from the definition of $\psik$ \eqref{eq:lin-psik} that
    \begin{align*}
        \normF{\psik} &= \normF{\PiSnp{\Zk} - \PiSnp{\Zs} - \Omega \circ (\dZk)} \\
        &= \normF{\PiSnp{\Zs + \Hk} - \PiSnp{\Zs} - \Omega \circ \Hk}.
    \end{align*}
    The desirable result then follows from \cref{thm:eb-intro-thm} and the fact that $\Hk \to 0$ as $k \to \infty$.
\end{proof}


\section{Local Linear Convergence with Nondegeneracy} \label{sec:conv-with-nd}

In this section, we establish local linear convergence of ADMM when primal and dual nondegeneracy holds at optimum. In this case, the pair of SDPs \eqref{eq:intro-sdp} has a unique KKT point and \cref{ass:lin-sc} is equivalent to merely existence of a strictly complementary solution, which is a common regularity condition for SDP in the literature.

We start with the simple characteristic of $\Fix(\Madmm)$ when nondegeneracy holds.
\begin{lemma} \label{lem:conv-nd-fix}
    Suppose \cref{ass:lin-sol,ass:lin-sc}, primal nondegeneracy~\eqref{eq:intro:primal-nondegeneracy} and dual nondegeneracy~\eqref{eq:intro:dual-nondegeneracy} hold. Then, it holds that $\Fix(\Madmm) = \{0\}$.
\end{lemma}
\begin{proof}
    For any $\H \in \Fix(\Madmm)$, we see from \cref{prop:lin-M} (c) and the definition of $\calT^\perp_\Xs$ that
    \[
        \mymat{0 & 0 \\ 0 & \HS} \in \calT^\perp_\Xs \ \text{and} \ \mymat{0 & 0 \\ 0 & \HS} \in \range(\AsdpT).
    \]
    Yet primal nondegeneracy suggests $\calT^\perp_\Xs \cap \range(\AsdpT) = \{0\}$. Thus, $\HS = 0$.

    Similarly, from \cref{prop:lin-M} (c) and the definition of $\calT^\perp_\Ss$, we conclude that
    \[
        \mymat{\HX & 0 \\ 0 & 0} \in \calT^\perp_\Ss \ \text{and} \ \mymat{\HX & 0 \\ 0 & 0} \in \nullspace(\Asdp).
    \]
    Yet dual nondegeneracy suggests that $\calT^\perp_\Ss \cap \nullspace(\Asdp) = \{0\}$. Thus, $\HX=0$.

    Therefore, any point $\H \in \Fix(\Madmm)$ must satisfy $\H=0$; \ie $\Fix(\Madmm) = \{0\}$.
\end{proof}

\begin{theorem} \label{thm:conv-nd}
    Suppose \cref{ass:lin-sol,ass:lin-sc}, primal nondegeneracy~\eqref{eq:intro:primal-nondegeneracy} and dual nondegeneracy~\eqref{eq:intro:dual-nondegeneracy} hold. For any $\rho \in (\normop{\Madmm}, 1)$, there exists $\bar k_{\mathrm{ND}} \in \bbN$ such that for any integer $k \ge \bar k_{\mathrm{ND}}$, it holds that
    \[
        \normF{\Zkpo - \Zs} \leq \rho \normF{\Zk - \Zs}.
    \]
\end{theorem}
\begin{proof}
    Since $\Fix(\Madmm) = \{0\}$, we have $\pfm = 0$ and $\normop{\Madmm} < 1$ from \cref{prop:lin-M}. Then \cref{prop:lin-psi} implies that there exists $\bar k_\Psi \in \bbN$ such that $\normF{\psik} \le \alpha_\Psi \normF{\HOk} \normF{\Hk}$ for any integer $k \geq \bar k_\Psi$. Then, convergence of ADMM suggests that for any $\rho \in (\normop{\Madmm}, 1)$, there exists $\bar k_O$ such that for any integer $k \geq \max\{\bar k_\Psi, \bar k_O\} =: \bar k_{\mathrm{ND}}$, we have $\alpha_\Psi \normF{\HOk} \leq \rho - \normop{\Madmm}$ and
    \[
        \normF{\psik} \le (\rho - \normop{\Madmm}) \cdot \normF{\Hk} = (\rho - \normop{\Madmm}) \cdot \normF{\dZk}.
    \]
    Finally, 
    \begin{align*}
        \normF{\Zkpo - \Zs} &= \normF{\Madmm (\dZk) + \psik} \\
        &\le \normop{\Madmm} \cdot \normF{\dZk} + \normF{\psik} \\
        &\le (\normop{\Madmm} + \rho - \normop{\Madmm}) \cdot \normF{\dZk} \\
        &= \rho \normF{\dZk}.
    \end{align*}
\end{proof}

\begin{remark}
    Our proof framework can also cover the case where ND holds and SC fails. To stay consistent with our SC assumption, the detailed proof of this case is deferred to \cref{app:sec:conv-only-nd}. So, combining \cref{thm:conv-nd} and the results in \cref{app:sec:conv-only-nd}, we establish local linear convergence of ADMM for SDP under only the nondegeneracy conditions. Though this conclusion can be drawn from \cite{han2018mor-linear-rate-admm}, our proof techniques are completely different from theirs and do not involve the metric subregularity of the KKT operator. Moreover, numerical evidence is provided in \cref{sec:exp} to support the theoretical findings in \cref{app:sec:conv-only-nd}.
\end{remark}

\section{Local R-linear Convergence without Nondegeneracy}
\label{sec:conv-without-nd}

Without two-side nondegeneracy, $\Fix(\Madmm)$ is not $\left\{ 0 \right\}$ and the proof technique in \cref{sec:conv-with-nd} does not apply anymore. Another nice property of $\Fix(\Madmm)$ in \cref{prop:lin-M} turns out to be useful: $\normop{\Madmm - \pfm} < 1$. Specifically, this property motivates us to study the ``projected sequence'':
\begin{align*}
    (\Id - \pfm) \Hkpo &= (\Id-\pfm) \Madmm \Hk + (\Id - \pfm) \psik \\
    &= (\Madmm - \pfm) (\Id - \pfm) \Hk + (\Id - \pfm) \psik,
\end{align*}
where the last equality follows from 
\[
    (\Madmm - \pfm) (\Id - \pfm) = \Madmm - \pfm = (\Id - \pfm) \Madmm.
\]
So, combining with the structure of $\Fix(\Madmm)$ and our refined error bound in \cref{thm:eb-intro-thm}, we are able to establish the (R-)linear convergence of several ``partial'' sequences
\[
    (\Id - \pfm) \Hk, \qquad \HOk, \qquad \Proj_{\calT_{\Ss}} (\Xk), \qquad \Proj_{\calT_{\Xs}} (\Sk),
\]
where the last two terms correspond to the part of $\Xk$ (resp., $\Sk$) that lies outside the minimal face of $\Xs$ (resp., $\Sk$); see \cref{lem:conv-nnd-blk,lem:conv-nnd-T}. (One may already notice that in the nondegenerate case where $\Fix(\Madmm) = \{0\}$, the sequence $(\Id - \pfm) \Hk$ is exactly $\Hk$, and the proof is done at this step.)

So, what is missing in the more general, possibly degenerate case? It turns out that an error bound for~$\Proj_{\psd{n}}$ alone is insufficient; a growth condition is needed that accounts for both the PSD cone and an affine set. More specifically, consider the spectrahedron $\calV \cap \psd{n}$, where $\calV$ is an affine space in $\Sn$. Following the convention in \cite{sturm2000siopt-error-bounds-linear-matrix-inequalities}, we call $\dist(\X, \calV \cap \psd{n})$ the \textit{forward error} and $\dist(\X, \calV) + [-\lammin{\X}]_+$ the \textit{backward error}. \rebuttal{In the polyhedral case} (\ie $\psd{n}$ reduces to the nonnegative orthant), the backward error and the forward error are in the same order \cite{hoffman2003ws-approximate-sol-linear-inequalities}, which leads to the \textit{sharpness} condition and linear convergence of first-order methods in linear programming \cite{applegate2023mp-faster-lp-sharpness}. In the spectrahedron case, however, it is shown in \cite{sturm2000siopt-error-bounds-linear-matrix-inequalities} that
\[
    \text{forward error} = \calO \big( (\text{backward error})^{1/2} \big)
\]
under mild conditions.
So SDPs are not sharp in general and linear convergence does not follow in a straightforward manner.

Fortunately, by exploiting the geometry of the PSD cone, it is shown in \cite[Lemma~2.3]{sturm2000siopt-error-bounds-linear-matrix-inequalities} that the forward error is in the same order as the backward error with respect to the regularized system
\[
    \calV \cap \minface{\Xs}{\Symp{n}}, \quad \text{where} \ \minface{\Xs}{\Symp{n}} := \left\{\mymat{\Gamma & 0 \\ 0 & 0} \Bigm\vert \Gamma\in \Symp{r} \right\} = \Symp{n} \cap \calT_{\Ss}^\perp.
\]
(The simple characteristic of the minimal face needs the assumption, made without loss of generality, that~$\Xs$ is diagonal.) Extending the conclusion in \cite[Lemma~2.3]{sturm2000siopt-error-bounds-linear-matrix-inequalities}, we obtain a linear growth condition on the distance to optimality. More specifically, we upper bound the distance from $\Zk$ to the optimal set by the sum of the following three terms:
\[
    \normF{\Zkpo - \Zk}, \qquad \normF{\Proj_{\calT_{\Ss}} (\Xk)}, \qquad \normF{\Proj_{\calT_{\Xs}} (\Sk)},
\]
where the last two terms correspond to the part of $\Xk$ (resp., $\Sk$) that lies outside the minimal face of~$\Xs$ (resp., $\Ss$); see \cref{lem:conv-nnd-kkt,lem:conv-nnd-Z}. Finally, combining the two ingredients (convergence of some partial sequences and the new growth condition) yields the desirable linear convergence guarantees of ADMM without nondegeneracy conditions.

Below we dive into the details, we remind that without two-side nondegeneracy, the primal and dual solutions may not be unique. So we denote by $\Xopt$ the optimal set for primal SDP in \eqref{eq:intro-sdp}, by $\Sopt$ the set of dual optimal $S$, and by $\Zopt$ the set of fixed points for the one-step ADMM \eqref{eq:intro:admm-one-step}.

We begin our analysis with some basic results on ADMM, of which the proof mainly uses the ADMM update rule \eqref{eq:intro:admm-one-step}. In fact, the inequality \eqref{eq:conv-nnd-Zdiff-ineq} is a special case of \cite[Proposition 3.1]{zhang2018arxiv-linearly-convergent-admm}.
\begin{lemma} \label{lem:conv-nnd-Zdiff}
    The sequence $\{\Zk\}$ generated by one-step ADMM \eqref{eq:intro:admm-one-step} satisfies
    \begin{equation} \label{eq:conv-nnd-Zdiff-eq}
        \normF{\Zkpo-\Zk}^2 = \normF{\PA (\Xk - \widetilde X)}^2 + \sigma^2 \normF{\PAp (\Sk - C)}^2,
    \end{equation}
    where $\widetilde X$ is an arbitrary matrix satisfying $\Asdp \widetilde X = b$.
    And
    \begin{equation} \label{eq:conv-nnd-Zdiff-ineq}
        \normF{\Zkpo-\Zk}^2 \leq \dist^2 (\Zk, \Zopt) - \dist^2 (\Zkpo, \Zopt),
    \end{equation}
    for all $k \in \bbN$.
\end{lemma}
\begin{proof}
    See \cref{app:sec:conv-nnd-Zdiff-prf}.
\end{proof}

\subsection{R-linear Decay outside Minimal Faces}

\begin{lemma} \label{lem:conv-nnd-blk}
    Suppose \cref{ass:lin-sol,ass:lin-sc} hold. Let $\Hk := \dZk$, $k \in \bbN$, be partitioned as in \eqref{eq:eb-intro-partition}. Then, for any $\rho \in (\normop{\Madmm - \pfm}, 1)$, there exists $\bar k_\rho \in \bbN$ such that for any integer $k \geq \bar k_\rho$, it holds that
    \[
        \normF{(\Id - \pfm) \Hkpo} \leq \rho \normF{(\Id - \pfm) \Hk}.
    \]
    Moreover, $\normF{\HOk}$ converges R-linearly.
\end{lemma}
\begin{proof}
    We first show that
    \[
        \frac{\normF{(\Id - \pfm) \psik}}{\normF{(\Id - \pfm) \Hk}} \to 0 \quad \text{as} \ k \to \infty.
    \]
    To see this, we first note from \cref{prop:lin-M} (c) that the off-block-diagonal part of $\pfm(\Hk)$ is zero, which implies that $(\Id - \pfm) \Hk$ and $\Hk$ have the same off-block-diagonal part. So,
    \[
        \normF{(\Id - \pfm) \Hk} \geq \sqrt 2 \normF{\HOk}.
    \]
    Then, we conclude from \cref{prop:lin-psi} that there exist $\bar k_\Psi \in \bbN$ and $\alpha_\Psi > 0$ such that for any integer $k \geq \bar k_\Psi$, it holds that
    \[
        \frac{\normF{(\Id - \pfm) \psik}}{\normF{(\Id - \pfm) \Hk}} \leq 
        \frac{\normF{\psik}}{\normF{(\Id - \pfm) \Hk}} \leq
        \frac{\alpha_\Psi \normF{\HOk} \normF{\Hk}}{\normF{(\Id - \pfm) \Hk}} \leq
        \frac{\alpha_\Psi}{\sqrt 2} {\normF{\Hk}},
    \]
    which goes to $0$ as $k \to \infty$.

    Finally, the R-linear convergence of $\normF{\HOk}$ follows naturally from the fact that $\normF{\HOk} \leq \normF{(\Id - \pfm) \Hk}$.
\end{proof}

\Cref{thm:eb-intro-thm} plays a vital role in the proof of \cref{lem:conv-nnd-blk}. If we used \cref{lem:eb-intro-SS02}, we could only upper bound $\normF{(\Id - \pfm) \psik}$ by $\normF{\Hk}^2$. This could only imply
\begin{equation} \label{eq:eq:conv-nnd-liang}
    \frac{\normF{(\Id - \pfm) \psik}}{\normF{(\Id - \pfm) \Hk}} = \calO\left( \frac{\normF{\Hk}^2}{\normF{\HOk}} \right).
\end{equation}
In \cite{liang2017jota-local-convergence-admm}, the authors investigate general ADMM for convex problems with partially smooth objectives and attempt to establish the linear convergence of the projected sequence ${(\Id - \pfm) \Hk}$. In \cite[pp. 911, line 5]{liang2017jota-local-convergence-admm}, the authors assert (without providing a justification)  that the left-hand side of \eqref{eq:eq:conv-nnd-liang} vanishes as $k \to \infty$. With the refined error bound in \cref{thm:eb-intro-thm}, our analysis confirms this claim in the context of SDP. However, its validity in the more general convex setting remains unclear to us.

\begin{remark}
    The R-linear convergence of $\normF{\HO}$ requires the assumption, made without loss of generality, that $\Xs$ and $\Ss$ are diagonal. Otherwise, when $\Qs \neq I$, the R-linearly convergent sequence is $\normF{Q_{\star,S}\tran \H Q_{\star, X}}$, where $\Qs = \mymat{Q_{\star, X} & Q_{\star, S}}$.
\end{remark}

\begin{lemma} \label{lem:conv-nnd-T}
    Suppose \cref{ass:lin-sol,ass:lin-sc} hold. for any $\rho \in (\normop{\Madmm - \pfm}, 1)$, there exists $\bar k_\calT \in \bbN$ such that for any integer $k \geq \bar k_\calT$, the two norms
    \[
        \normF{\Proj_{\calT_{\Ss}} (\Xk)} \quad \text{and} \quad \normF{\Proj_{\calT_{\Xs}} (\Sk)}
    \]
    converge R-linearly.
\end{lemma}
\begin{proof}
    For any $k \in \bbN$, we have
    \begin{subequations}
    \begin{align}
        \normF{\psik} &= \normF{\PiSnp{\Zs + \Hk} - \PiSnp{\Zs} - \Omega \circ \Hk} \nonumber \\
        &= \normF{\Xk - \Xs - \Omega \circ \Hk} \label{eq:conv-nnd-face-1a} \\
        &\ge \normF{\Proj_{\calT_{\Ss}} (\Xk - \Xs - \Omega \circ \Hk)}. \label{eq:conv-nnd-face-1}
    \end{align}
    \end{subequations}
    where \eqref{eq:conv-nnd-face-1a} uses $\Xk = \PiSnp{\Zk}$.
     Then, from the definition of $\calT_{\Ss}$ \eqref{eq:intro:orthogonal-complement}, we have
    \rebuttal{
    \begin{equation} \label{eq:conv-nnd-face-2}
        \Proj_{\calT_{\Ss}} (\Omega \circ \Hk) = \mymat{0 & \Theta\tran \circ (\HOk)\tran \\ \Theta \circ \HOk & 0}.
    \end{equation}
    }
    Combining \eqref{eq:conv-nnd-face-1}, and \eqref{eq:conv-nnd-face-2} and \cref{prop:lin-psi}, we conclude that there exists $\bar k_\Psi$ such that for any integer $k \geq \bar k_\Psi$, we have
    \begin{align*}
        \normF{\Proj_{\calT_{\Ss}} (\Xk)} &= \normF{\Proj_{\calT_{\Ss}} (\dXk)} \\
        &\leq \normF{\Proj_{\calT_{\Ss}} (\Xk - \Xs - \Omega \circ \Hk)} + \normF{\Proj_{\calT_{\Ss}} (\Omega \circ \Hk)} \\
        &\leq \normF{\psik} + \sqrt 2 \normF{\HOk} \\
        &\leq (\alpha_\Psi \normF{\Hk} + \sqrt 2) \normF{\HOk},
    \end{align*}
    The convergence of ADMM suggests that for sufficiently large $k \in \bbN$, the residual $\Hk$ is bounded, and thus $\normF{\Proj_{\calT_{\Ss}} (\dXk)}$ is bounded above by a multiple of $\normF{\HOk}$, a R-linearly convergent sequence (\cref{lem:conv-nnd-blk}).

    The second part of the lemma follows similarly since $\Sk = (1/\sigma) \PiSnp{-\Zk}$.
\end{proof}

\subsection{Linear Growth of Distance to Optimality}

In this section, we present the one-iteration analysis for our convergence measure $\dist(\Zk, \Zopt)$. The following lemma is inspired by \cite[Lemma~2.3]{sturm2000siopt-error-bounds-linear-matrix-inequalities} and gives the regularized backward error for the (scaled) KKT system.
\begin{lemma} \label{lem:conv-nnd-kkt}
    Let $(\Xs, \ys, \Ss)$ be the convergent point of ADMM \eqref{eq:intro:admm-three-step} satisfying strict complementarity~\eqref{eq:intro:strict-complementarity}. Then, there exist three positive constants $(\delta_X, \delta_S, \kappa)$ such that for all $(X, S) \in \Sn \times \Sn$ with $\normF{X} \leq \delta_X$ and $\normF{\sigS} \leq \delta_S$, it holds that
    \begin{align*}
        &\ \kappa \cdot \dist \big((X, \sigS), \Xopt \times (\sigma \Sopt) \big) \\
        \leq&\ \normF{\PA (X - \widetilde X)} + \normF{\PAp (\sigS - \sigC)} + |\inprod{X}{\sigC} + \langle {\widetilde X}, {\sigS} \rangle - \langle {\widetilde X}, {\sigC} \rangle| \\
        &\ + [-\lammin{X}]_+ + [-\lammin{\sigS}]_+ + \normF{\Proj_{\calT_\Ss} (X)} + \normF{\Proj_{\calT_\Xs} (\sigS)},
    \end{align*}
    where $\widetilde X$ is an arbitrary matrix with $\Asdp \widetilde X = b$ and $\sigma>0$ is the parameter in ADMM.
\end{lemma}
\begin{proof}
    See \cref{app:sec:conv-nnd-kkt-prf}.
\end{proof}

\begin{lemma} \label{lem:conv-nnd-Z}
    Suppose \cref{ass:lin-sol,ass:lin-sc} hold. Then, there exists $\bar k_Z \in \bbN$ and $\alpha_Z > 0$ such that for any integer $k \geq \bar k_Z$, it holds that
    \[
        \dist (\Zk, \Zopt) \leq \alpha_Z (\normF{\Zkpo - \Zk} + \normF{\Proj_{\calT_{\Ss}} (\Xk)} + \sigma \normF{\Proj_{\calT_\Xs} (\Sk)}),
    \]
\end{lemma}
\begin{proof}
    We first bound the distance $\dist(\Zk, \Zopt)$ by the distance from $(\Xk, \sigma \Sk)$ to the set $\Xopt \times (\sigma \Sopt)$ as follows:
    \begin{align*}
        \dist(\Zk, \Zopt) &= \inf \{\normF{\Zk - \Z} \mid \Z \in \Zopt\} \\
        &= \inf \{\normF{\Zk - (X - \sigma S)} \mid X \in \Xopt, S \in \Sopt\} \\
        &= \inf \{\normF{(\widetilde X - \sigma \widetilde S) - (X - \sigma S)} \mid \widetilde X - \sigma \widetilde S = \Zk, X \in \Xopt, S \in \Sopt\} \\
        &\leq \sqrt 2 \cdot \inf \Big\{\sqrt{\normF{\widetilde X - X}^2 + \sigma^2 \normF{\widetilde S - S}^2} \Bigm\vert \widetilde X - \sigma \widetilde S = \Zk, X \in \Xopt, S \in \Sopt \Big\} \\
        &\leq \sqrt 2 \cdot \dist\big( (\Xk, \sigma \Sk), \Xopt \times (\sigma \Sopt) \big).
    \end{align*}
    Then, we bound $\dist ((\Xk, \sigma \Sk), \Xopt \times (\sigma \Sopt))$ using \cref{lem:conv-nnd-kkt} and the facts that $\Xk \in \psd{n}$ and $\Sk \in \psd{n}$ for all $k \in \bbN$. More specifically, there exist $(\bar k_Z, \delta_X, \delta_S) \in \bbN \times \Real{}_{++} \times \Real{}_{++}$ such that for any integer $k \geq \bar k_Z$, we have $\normF{\Xk} \leq \delta_X$, $\normF{\Sk} \leq \delta_S$ and
    \begin{align}
        &\ \kappa \cdot \dist\big( (\Xk, \sigma \Sk), \Xopt \times (\sigma \Sopt) \big) \nonumber \\
        \leq&\ \normF{\PA (\Xk - \widetilde X)} + \sigma \normF{\PAp (\Sk - C)} + \sigma |\langle {\Xk}, {C} \rangle + \langle {\widetilde X}, {\Sk} \rangle - \langle {\widetilde X}, {C} \rangle| \nonumber \\
        &\ + \normF{\Proj_{\calT_{\Ss}} (\Xk)} + \sigma \normF{\Proj_{\calT_{\Xs}} (\Sk)}, \label{eq:conv-nnd-Z-1}
    \end{align}
    where $\widetilde X$ is an arbitrary matrix satisfying $\Asdp \widetilde X = b$. The inner product on the right-hand side of \eqref{eq:conv-nnd-Z-1} can be further bounded by
    \rebuttal{
    \begin{subequations}
    \begin{align}
        &\ \sigma |\langle {\Xk}, {C} \rangle + \langle {\widetilde X}, {\Sk} \rangle - \langle {\widetilde X}, {C} \rangle | \nonumber \nonumber \\
        =&\ \sigma |\langle {\Xk}, {C} \rangle + \langle {\widetilde X}, {\Sk} \rangle - \langle {\widetilde X}, {C} \rangle| \nonumber \\
        =&\ \sigma |\langle {\Xk-\widetilde X}, {\Sk-C} \rangle| \label{eq:conv-nnd-Z-2a} \\
        =&\ \sigma |\langle {\PA(\Xk-\widetilde X)}, {\PA(\Sk-C)} \rangle + \langle {\PAp(\Xk-\widetilde X)}, {\PAp(\Sk-C)} \rangle| \nonumber \\
        \leq&\ \sigma \normF{\PA(\Xk-\widetilde X)} \normF{\PA(\Sk-C)} + \sigma \normF{\PAp(\Xk - \widetilde X)} \normF{\PAp(\Sk-C)} \nonumber \\
        \leq&\ \alpha_Z^\prime \big(\normF{\PA (\Xk - \widetilde X)} + \sigma \normF{\PAp(\Sk-C)} \big), \label{eq:conv-nnd-Z-2b}
    \end{align} 
    \end{subequations}
    }
    where \eqref{eq:conv-nnd-Z-2a} uses the fact $\langle \Xk, \Sk \rangle = 0$ and in \eqref{eq:conv-nnd-Z-2b} we define
    \[
        \alpha_Z^\prime := \max\{\sigma(\delta_S + \normF{C}), \, \delta_X + \normF{\Asdp^\dagger b}\} 
    \]
    (recall $\Asdp^\dagger b$ is a valid choice for $\widetilde X$). Finally, combining \eqref{eq:conv-nnd-Z-1}, \eqref{eq:conv-nnd-Z-2b} with \eqref{eq:conv-nnd-Zdiff-eq} in \cref{lem:conv-nnd-Zdiff} and denoting $\alpha_Z := {\sqrt 2 (1+\alpha_Z^\prime)}/{\kappa}$ give the desirable result.
\end{proof}

\subsection{Main Theorem}

Now we are ready to present our main theorem.
\begin{theorem} \label{thm:conv-nnd}
    Suppose \cref{ass:lin-sol,ass:lin-sc} hold. Then, for any $\rho_0 \in (\normop{\Madmm - \pfm}, 1)$, there exists $\bar k \in \bbN$ such that for any integer $k \geq \bar k$, the distance to optimality $\dist(\Zk, \Zopt)$ converges R-linearly; \ie there exists $(\alpha, \rho) \in \Real{}_{++} \times (0,1)$ such that
    \[
        \dist(\Zk, \Zopt) \leq \alpha \rho^k.
    \]
\end{theorem}
\begin{proof}
    Define $\ak := \dist(\Zk, \Zopt)$ for brevity. From \cref{lem:conv-nnd-T}, we deduce that there exist $(\bar k_\calT, \alpha_X, \alpha_S) \in \bbN \times \Real{}_{++} \times \Real{}_{++}$ such that for any integer $k \geq \bar k_\calT$, we have
    \begin{equation} \label{eq:conv-nnd-thm-1}
        \normF{\Proj_{\calT_{\Ss}} (\Xk)} \leq \alpha_X \rho_0^k, \qquad \normF{\Proj_{\calT_{\Ss}} (\Sk)} \leq \alpha_S \rho_0^k.
    \end{equation}
    Then, with $\bar k_Z \in \bbN$ as defined in \cref{lem:conv-nnd-Z}, we have for any integer $k \geq \bar k := \max\{\bar k_\calT, \bar k_Z\} + 1$ that
    \begin{subequations}
    \begin{align}
        \akpo &\leq \ak \label{eq:conv-nnd-thm-2a} \\
        &\leq \alpha_Z (\normF{\Zkpo - \Zk} + \normF{\Proj_{\calT_{\Ss}} (\Xk)} + \sigma \normF{\Proj_{\calT_\Xs} (\Sk)}) \label{eq:conv-nnd-thm-2b} \\
        &\leq \alpha_Z \Big(\sqrt{(\ak)^2 - (\akpo)^2} + \normF{\Proj_{\calT_{\Ss}} (\Xk)} + \sigma \normF{\Proj_{\calT_\Xs} (\Sk)} \Big) \label{eq:conv-nnd-thm-2c} \\
        &\leq \alpha_Z \sqrt{(\ak)^2 - (\akpo)^2} + \alpha_Z (\alpha_X + \alpha_S) \rho_0^k. \label{eq:conv-nnd-thm-2d}
    \end{align}
    \end{subequations}
    In \eqref{eq:conv-nnd-thm-2a} we use \eqref{eq:conv-nnd-Zdiff-ineq} in \cref{lem:conv-nnd-Zdiff}, \eqref{eq:conv-nnd-thm-2b} uses \cref{lem:conv-nnd-Z}, \eqref{eq:conv-nnd-thm-2c} uses \eqref{eq:conv-nnd-Zdiff-ineq} again, and finally \eqref{eq:conv-nnd-thm-2d} uses \eqref{eq:conv-nnd-thm-1}.

    Then, we partition the index set $\{k \in \bbN \mid k \geq \bar k\}$ into
    \[
        \calI := \{k \in \bbN \mid k \geq \bar k, \, \akpo \geq 2\alpha_Z (\alpha_X + \alpha_S) \rho_0^k\}
    \]
    and its complement $\calI^{\mathrm c} := \{k \in \bbN \mid k \geq \bar k\} \setminus \calI$.
    \begin{enumerate}
        \item If $k \in \calI$, then from \eqref{eq:conv-nnd-thm-2d} we have
        \rebuttal{
        \[
            (\ak)^2 - (\akpo)^2 \geq (\akpo - \alpha_Z (\alpha_X + \alpha_S) \rho_0^k)^2 \geq \tfrac{1}{4\alpha_Z^2} (\akpo)^2,
        \] 
        }
        which implies that
        \[
            \akpo \leq \sqrt{\tfrac{4\alpha_Z^2}{1+4\alpha_Z^2}} \ak. 
        \]

        \item If $k \in \calI^{\mathrm c}$, then readily we have
        \rebuttal{
        \[
            a^{(k+1)} < 2\alpha_Z (\alpha_X + \alpha_S) \rho_0^k.
        \]
        }
    \end{enumerate}
    Combining the two cases yields the desirable local R-linear convergence of $\ak$. More specifically, define 
    \[
        \bk := \max\{\ak, 2\alpha_Z(\alpha_X + \alpha_S) \rho_0^{k-1}\}, \qquad \rho := \max \left\{\sqrt{\tfrac{4\alpha_Z^2}{1+4\alpha_Z^2}}, \rho_0 \right\} \in (0,1),
    \]
    and consider any pair $(k-1,k)$.
    \begin{enumerate}
        \item If $(k-1,k) \in \calI \times \calI$, then $\bkpo = \akpo \leq \rho \ak = \rho \bk$.
        
        \rebuttal{
        \item If $(k-1,k) \in \calI \times \calI^{\mathrm c}$, then $\bkpo = 2\alpha_Z(\alpha_X + \alpha_S) \rho_0^k \leq \rho_0 \ak \leq \rho \bk$.
        }

        \rebuttal{
        \item If $(k-1,k) \in \calI^{\mathrm c} \times \calI$, then $\bkpo = \akpo \leq \rho \ak < \rho \bk$.
        }
        
        \rebuttal{
        \item If $(k-1,k) \in \calI^{\mathrm c} \times \calI^{\mathrm c}$, then $\bkpo = 2\alpha_Z (\alpha_X + \alpha_S) \rho_0^k = \rho_0 \bk \le \rho \bk$.
        }
    \end{enumerate}
    To conclude, $\{\bk\}$ is a linearly convergent sequence with rate $\rho \in (0,1)$ and an upper bound for $\{\ak\}$. So, $\ak := \dist(\Zk, \Zopt)$ converges R-linearly for sufficiently large $k \in \bbN$.
\end{proof}
\rebuttal{
\begin{remark}
    We note that our current analysis framework is intentionally restricted to a fixed penalty parameter $\sigma$, since the linearization is built around a stationary fixed-point map. If $\sigma$ varies with the iteration index, then the underlying operator becomes time-varying, and the present contraction and $\Fix(\Madmm)$ arguments no longer apply directly. Extending our framework to the adaptive-$\sigma$ setting would therefore be an interesting direction for future work.
\end{remark}
}


\section{Proof of the Refined Error Bound} \label{sec:error-bound}

In this section, we detail the proof of \cref{thm:eb-intro-thm}, which builds an error bound for the PSD cone projection:
\[
    \normtwo{\PiSnp{\Z + \H} - \PiSnp{\Z} - \Omega \circ \H} \le \alpha_{\mathrm{EB}} \cdot \normtwo{\HO} \cdot \normtwo{\H},
\]
where we have assumed without loss of generality that $\Z$ is diagonal. A traditional way to compute the orthogonal projection onto the PSD cone is via eigenvalue decomposition. Though conceptually simple, this method destroys the block structure of $\Z + \H$ (as $\H$ is not diagonal). So, to prove \cref{thm:eb-intro-thm}, we advocate for a seemingly much more complex procedure inspired by iterative methods for eigenvalue decomposition (see, \eg~\cite{parlett1987book-symmetric-eig}). We detail the iterative algorithm in \blue{Algorithm \ref{alg:error-bound:iterative-elimination}} and briefly discuss the high-level intuition here.
\rebuttal{

\begin{figure}[h]
    \centering 
    \includegraphics[width=0.95\textwidth]{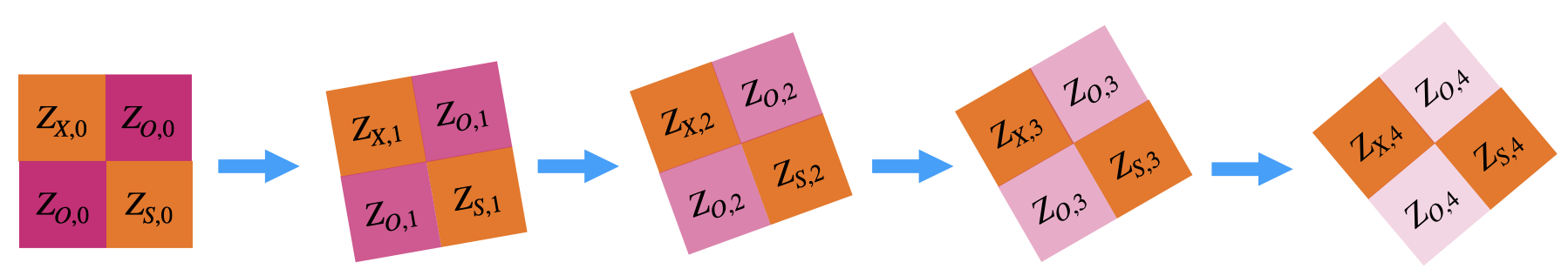}
    \caption{
        \label{fig:psdproj-eb}
        Visualization of the elimination procedure (Algorithm \ref{alg:error-bound:iterative-elimination}).
    }
\end{figure}

\noindent
\textbf{A roadmap to the proof of \cref{thm:eb-intro-thm}.} A standard way to project onto the PSD cone is via eigenvalue decomposition. Although conceptually simple, this destroys the block structure of $\Z + \H$ because $\H$ is not diagonal. To prove \cref{thm:eb-intro-thm}, we instead use a more structured procedure inspired by iterative methods for eigenvalue decomposition; see, \eg~\cite{parlett1987book-symmetric-eig}. As detailed in Algorithm~\ref{alg:error-bound:iterative-elimination}, in the eigenbasis of $Z$, the perturbation $H$ is nearly block diagonal, and its off-block-diagonal part is the only term coupling the positive and negative spectral subspaces. Each Sylvester equation then produces a small skew-symmetric rotation that cancels this coupling to first order, leaving only higher-order terms. Iterating this step yields a rapidly convergent block diagonalization; see \cref{fig:psdproj-eb} for an illustration. Overall, this construction locally mimics PSD cone projection while isolating the first- and second-order terms needed in the proof.
}
\begin{equation} \label{eq:error-bound:intro}
    \Z + \H = \LamXS + \mymat{\HX & \HO\tran \\ \HO & \HS} =: \mymat{\ZX[0] & \ZO[0]\tran \\ \ZO[0] & \ZS[0]}.
\end{equation}
When $\normtwo{H}$ is sufficiently small, we have $\ZX[0] = \LamX + \HX \in \Sympp{r}$ and $\ZS[0] = \LamS + \HS \in \Symnn{n-r}$. \blue{Algorithm~\ref{alg:error-bound:iterative-elimination}} explicitly constructs an orthogonal matrix $\Y[\infty]$ that is close to $\I[n]$ and satisfies
\begin{equation} \label{eq:error-bound:Qinfty}
    \Y[\infty]\tran \mymat{\ZX[0] & \ZO[0]\tran \\ \ZO[0] & \ZS[0]} \Y[\infty] = \mymat{\ZX[\infty] & 0 \\ 0 & \ZS[\infty]}.
\end{equation}
(So roughly speaking, $\ZX[\infty] \in \Sympp{r}$ (resp., $\ZS[\infty] \in \Symnn{n-r}$) is also close to $\ZX[0]$ (resp., $\ZS[0]$).) To compute the orthogonal matrix $\Y[\infty]$ in \eqref{eq:error-bound:Qinfty}, we solve a series of Sylvester equations~\eqref{eq:error-bound:sylvester-eq} for $\W[\ell]$ and show that the recursively defined matrix $\Y[\ell+1] \leftarrow \Y[\ell] \exp(\W[\ell])$ converges to $\Y[\infty]$. As we will see later, each Sylvester equation helps build a skew-symmetric matrix $\W[\ell]$ such that the off-block-diagonal part of $(\I[n] + \W[\ell])\tran \Z[\ell] (\I[n] + \W[\ell])$ is gradually removed at each iteration; see \eqref{eq:error-bound:main-iter}.
Then, with $\Y[\infty]$ computed (or approximated), we can derive a fine-grained error bound for
\[
    \PiSnp{\Z + \H} - \PiSnp{\Z} - \Omega \circ \H = \Y[\infty] \mymat{\ZX[\infty] & 0 \\ 0 & 0} \Y[\infty]\tran - \PiSnp{\Z} - \Omega \circ \H,
\]
which further leads to the conclusion in \cref{thm:eb-intro-thm}.

\begin{algorithm}[tbp]
    \SetAlgoLined
    \caption{An iterative elimination procedure for PSD cone projection}
    \label{alg:error-bound:iterative-elimination}
    \KwIn{A nonsingular matrix $\Z \in \Sn$ with $r$ positive eigenvalues, and a perturbation $\H \in \Sn$.}
    \KwOut{$\PiSnp{\Z + \H} \leftarrow \V[\infty].$}
    Initialization: $\Z[0] := \mymat{\ZX[0] & \ZO[0]\tran \\ \ZO[0] & \ZS[0]} \leftarrow \Z + \H$ and $\Y[0] \leftarrow \I[n]$.

    \For{$\ell = 0$ to $\infty$}{
        (1) Solve the following Sylvester equation for $\WO \in \Real{(n-r) \times r}$
        \begin{align}
            \label{eq:error-bound:sylvester-eq}
            \WO \ZX[\ell] + (-\ZS[\ell]) \WO = \ZO[\ell]
        \end{align}
        and obtain $\WO[\ell] \leftarrow \WO$.

        (2) Compute 
        \begin{align}
            \W[\ell] &\leftarrow \mymat{0 & -\WO[\ell]\tran \\ \WO[\ell] & 0} \nonumber \\
            \Z[\ell+1] := \mymat{\ZX[\ell+1] & \ZO[\ell+1]\tran \\ \ZO[\ell+1] & \ZS[\ell+1]} &\leftarrow \exp(\W[\ell])\tran \mymat{\ZX[\ell] & \ZO[\ell]\tran \\ \ZO[\ell] & \ZS[\ell]} \exp(\W[\ell]). \label{eq:error-bound:main-iter}
        \end{align}\

        (3) Compute $\Y[\ell+1] \leftarrow \Y[\ell] \exp(\W[\ell])$ and
        \begin{align}
            \label{eq:error-bound:V-iter}
            \V[\ell+1] \leftarrow \Y[\ell+1] \mymat{\ZX[\ell+1] & 0 \\ 0 & 0} \Y[\ell+1]\tran.
        \end{align}
    }
\end{algorithm}

\paragraph{Proof outline.}
\rebuttal{
In summary, the proof consists of three major steps:
}
\begin{enumerate}
    \item We show that at each iteration, the Sylvester equation~\eqref{eq:error-bound:sylvester-eq} is well-defined and has a unique solution. We ensure this by showing $\ZX[\ell] \in \Sympp{r}, \ZS[\ell] \in \Symnn{n-r}$ for all $\ell \in \bbN$. 

    \item We show that the limit of the sequence $\left\{ \V[\ell] \right\}_{\ell=0}^{\infty}$ exists and is exactly $\PiSnp{\Z + \H}$. This is achieved by showing the exponential decay of the three sequences
    \[
        \normtwo{\ZX[\ell+1] - \ZX[\ell]}, \qquad \normtwo{\ZS[\ell+1] - \ZS[\ell]}, \qquad \normtwo{\Y[\ell+1] - \Y[\ell]}.
    \]

    \item Last, we show that
    \begin{align}
        &\ \normtwo{\PiSnp{\Z + \H} - \PiSnp{\Z} - \Omega \circ \H} \nonumber \\
        =&\ \left\|{\Y[\infty] \mymat{\ZX[\infty] & 0 \\ 0 & 0} \Y[\infty]\tran - \PiSnp{\Z} - \Omega \circ \H}\right\|_2 \nonumber \\
        \le &\ \left\| (\I[n] + \W[0]) \mymat{\ZX[\infty] & 0 \\ 0 & 0} (\I[n] + \W[0])\tran - \PiSnp{\Z} - \Omega \circ \H \right\|_2 \nonumber \\
        &\ + \left\|{\Y[\infty] \mymat{\ZX[\infty] & 0 \\ 0 & 0} \Y[\infty]\tran - (\I[n] + \W[0]) \mymat{\ZX[\infty] & 0 \\ 0 & 0} (\I[n] + \W[0])\tran} \right\|_2 \label{eq:error-bound:two-terms}  
    \end{align}
    Then, we bound the growth of the first term on the right-hand side of \eqref{eq:error-bound:two-terms} by $\calO\left( \normtwo{\HO} \cdot \normtwo{\H} \right)$ and that of the second term by $\calO\left( \normtwo{\HO}^2 \right)$. 
\end{enumerate}

\begin{remark}
    We reiterate that in~\eqref{eq:error-bound:intro}, we have assumed without loss of generality that $\Z$ is diagonal. Extension of the presented proof to the non-diagonal case is straightforward and detailed in \cref{sec:error-bound:ndiag}.
\end{remark}

\subsection{Step 1: Error Bound for Sylvester Equations} 

\Cref{lem:error-bound:error-control-one-step} shows that $\ZX[\ell] \in \Sympp{r}$ and $\ZS[\ell] \in \Symnn{n-r}$ imply the well-posedness of the Sylvester equation \eqref{eq:error-bound:sylvester-eq}. \rebuttal{Then, \cref{lem:error-bound:induction} proves that the definiteness of} $\ZX[\ell]$ and $\ZS[\ell]$ holds as long as $\ZX[0] \in \Sympp{r}$, $\ZS[0] \in \Symnn{n-r}$ and $\normtwo{\ZO[0]}$ is sufficiently small. With these two lemmas, we complete Step 1 in the proof outline. For ease of notation, we define $d := \sqrt{\min\{r, n-r\}}$ and
\begin{align}
    \label{eq:error-bound:eta}
    \etanew[\ell] := \frac{d}{\deltalam{\ZX[\ell]}{\ZS[\ell]}} \quad \text{for} \ \ell \in \bbN.
\end{align}

\begin{lemma}
    \label{lem:error-bound:error-control-one-step}
    At iteration $\ell$ in \blue{Algorithm~\ref{alg:error-bound:iterative-elimination}}, suppose $\ZX[\ell] \in \Sympp{r}$, $\ZS[\ell] \in \Symnn{n-r}$, and $\normtwo{\ZO[\ell]} \le \frac{3}{4\etanew[\ell]}$. Then, the Sylvester equation~\eqref{eq:error-bound:sylvester-eq} has a unique solution $\WO[\ell]$ satisfying $\normtwo{\WO[\ell]} \le \etanew[\ell] \cdot \normtwo{\ZO[\ell]}$. Moreover, it holds that
    \begin{align*}
        &\ \max\{\norm{\ZX[\ell+1] - \ZX[\ell]}_2, \, \norm{\ZS[\ell+1] - \ZS[\ell]}_2, \, \norm{\ZO[\ell+1]}_2\} \\
        \le&\ \left( 
            \frac{4}{9}\etanew[\ell]^4 \cdot \normtwo{\Z[0]}^3 + \frac{4}{3} \etanew[\ell]^3 \cdot \normtwo{\Z[0]}^2 + \frac{13}{3} \etanew[\ell]^2 \cdot \normtwo{\Z[0]} + 4\etanew[\ell] 
         \right) \cdot \norm{\ZO[\ell]}_2^2
    \end{align*}
    for all $\ell \in \bbN$.
\end{lemma}
\begin{proof}
    See \cref{app:sec:error-bound:error-control-one-step}.
\end{proof}

The proof of \cref{lem:error-bound:induction} needs two auxiliary functions. Define $f(x): [0, \infty) \mapsto \Real{}$ as any fixed continuous and monotonically increasing function satisfying: (1) $f(0)=0$; (2) $f(x) \ge \frac{9}{4} x^4 + \frac{4}{3} x^3 + \frac{13}{3} x^2 + 4 x$ for all $x \geq 0$. Then, define $g(y): [0, \infty) \mapsto \Real{}$ as:
\begin{align}
    \label{eq:error-bound:g-func}
    g(y) := y \cdot f\left( 2\etanew[0] \cdot (\normtwo{\ZX[0]} + \normtwo{\ZS[0]} + y) \right) \cdot f\left( \etanew[0] \cdot (\normtwo{\ZX[0]} + \normtwo{\ZS[0]} + y) \right)
\end{align}
So, $g(y)$ is also monotonically increasing on $[0,\infty)$ and $g(0) = 0$.

\begin{lemma}
    \label{lem:error-bound:induction}
    Suppose $\ZX[0] \in \Sympp{r}$ and $\ZS[0] \in \Symnn{n-r}$. Define two positive constants $\alpha_K$ and $C_K$:
    \begin{align}
        \label{eq:error-bound:alphaK}
        \alpha_K = \frac{f\left( 
            \etanew[0] \cdot (\normtwo{\ZX[0]} + \normtwo{\ZS[0]} + \frac{1}{2})
         \right)}{\normtwo{\ZX[0]}}, \qquad C_K := \min\left\{ C_1, C_2, C_3, C_4, C_5 \right\},
    \end{align} 
    where 
    \begin{gather*}
        C_1 = \frac{1}{2}, \quad C_2 = g^{-1}(\normtwo{\ZX[0]}^2), \quad C_3 = \frac{3}{8\etanew[0]}, \\
        C_4 = \frac{1}{\sqrt{4 \etanew[0] \alpha_K}}, \quad C_5 = \sqrt{\frac{1}{4 \alpha_K} \min \left\{ \lammin{\ZX[0]}, -\lammax{\ZS[0]} \right\}}.
    \end{gather*}
    
    For any $\normtwo{\ZO[0]} \le C_K$ and for any integer $\ell \ge 1$, it holds that
    \begin{subequations} \label{eq:error-bound:induction}
        \begin{align}
            \normtwo{\ZO[\ell]} &\le \alpha_K \cdot \normtwo{\ZO[0]}^{\ell+1} \label{eq:error-bound:induction-1} \\
            \normtwo{\ZX[\ell] - \ZX[0]} &\le \alpha_K \cdot \sum_{i=0}^{\ell-1} \normtwo{\ZO[0]}^{i+2} \label{eq:error-bound:induction-2} \\
            \normtwo{\ZS[\ell] - \ZS[0]} &\le \alpha_K \cdot \sum_{i=0}^{\ell-1} \normtwo{\ZO[0]}^{i+2}. \label{eq:error-bound:induction-3} 
        \end{align}
    \end{subequations}
    Moreover, for any integer $\ell \ge 1$, it holds that
    \[
        \frac{2}{3} \etanew[0] \le \etanew[\ell] \le 2 \etanew[0], \quad \lammin{\ZX[\ell]} \ge \frac{1}{2} \lammin{\ZX[0]} > 0 \quad \lammax{\ZS[\ell]} \le \frac{1}{2} \lammax{\ZS[0]} < 0.
    \]
    Thus, $\ZX[\ell] \in \Sympp{n}$ and $\ZS[\ell] \in \Symnn{n}$ for all $\ell \in \bbN$.
\end{lemma}
\begin{proof}
    See \cref{app:sec:error-bound:induction}.
\end{proof}

\subsection{Step 2: Convergence of \texorpdfstring{$\{\V[\ell]\}_{\ell=0}^{\infty}$}{Vl}}

Now we show that the sequence $\left\{ \V[\ell] \right\}_{\ell=0}^{\infty}$ converges to $\PiSnp{\Z + \H}$; see \cref{lem:error-bound:V-limit}. This is achieved by bounding the distance between $\Y[\ell+1]$ and $\Y[\ell]$ and that between $\Y[\ell]$ and $\I[n] + \W[0]$; see \cref{lem:error-bound:distance-Yk}.

\begin{lemma}
    \label{lem:error-bound:distance-Yk}
    For any integer $\ell \ge 1$, it holds that
    \begin{align}
        \label{eq:error-bound:distance-Yk-1}
        \normtwo{\Y[\ell+1] - \Y[\ell]} \le \frac{8}{3} \etanew[0] \alpha_K \cdot \normtwo{\ZO[0]}^{\ell+1}
    \end{align}
    and
    \begin{align}
        \label{eq:error-bound:distance-Yk-2}
        \normtwo{\Y[\ell] - (\I[n] + \W[0])} \le \frac{2}{3}\etanew[0]^2 \cdot \normtwo{\ZO[0]}^2 + \frac{8}{3} \etanew[0] \alpha_K \cdot \sum_{i=1}^{\ell-1} \normtwo{\ZO[0]}^{i+1}
    \end{align}
    where $\etanew[0]$ is defined in~\eqref{eq:error-bound:eta} and $\alpha_K$ in~\eqref{eq:error-bound:alphaK}.
\end{lemma}
\begin{proof}
    See \cref{app:sec:error-bound:distance-Yk}.
\end{proof}

\begin{lemma}
    \label{lem:error-bound:V-limit}
    The sequence $\left\{ \V[\ell] \right\}_{\ell=1}^\infty$ generated in \blue{Algorithm~\ref{alg:error-bound:iterative-elimination}} converges to $\PiSnp{\Z + \H}$.
\end{lemma}
\begin{proof}
    See \cref{app:sec:error-bound:V-limit}.
\end{proof}

\subsection{Step 3: Proof of \texorpdfstring{\cref{thm:eb-intro-thm}}{Theorem 2}}

Before we execute the last step of our proof, two more details are needed. First, observe that all the three constants, $\etanew[0]$ in~\eqref{eq:error-bound:eta}, $\alpha_K$ and $C_K$ in~\eqref{eq:error-bound:alphaK} implicitly rely on $\ZX[0] = \LamX + \HX$ and $\ZS[0] = \LamS + \HS$ (though independent of $\ZO[0] = \HO$). Yet, the constants $\alpha_{\mathrm{EB}}$ and $C_{\mathrm{EB}}$ in \cref{thm:eb-intro-thm} should be independent of the perturbation $\H$. The uniform bounds of $\etanew[0]$, $\alpha_K$ and $C_K$ is achieved in \cref{lem:error-bound:lower-bound-C-eta-alpha}.
\begin{lemma}
    \label{lem:error-bound:lower-bound-C-eta-alpha}
    Suppose $\normtwo{\H} \le \frac{1}{2}\min\{\lam{r}, -\lam{r+1}\}$. Then, it holds that
    \[
        \lammin{\ZX[0]} \ge \frac{1}{2} \lam{r} > 0 \qquad \lammax{\ZS[0]} \le \frac{1}{2} \lam{r+1} < 0.
    \]
    Moreover, there exist three positive constants $\alpha_{K,f}$, $\etanew[0,f]$ and $C_{K,f}$, only depending on $n$, $r$ and the eigenvalues $\{\lam{i}\}_{i=1}^n$ of $Z$, such that
    \[
        \etanew[0] \leq \etanew[0,f], \qquad \alpha_K \leq \alpha_{K,f}, \qquad C_K \geq C_{K,f} > 0. 
    \] 
\end{lemma}
\begin{proof}
    See \cref{app:sec:error-bound:lower-bound-C-eta-alpha}.
\end{proof}

With \cref{lem:error-bound:lower-bound-C-eta-alpha}, as long as $\normtwo{\H} \le \min\{C_{K,f}, \frac{1}{2} \min\{\lam{r}, -\lam{r+1}\} \}$, we can safely replace $\alpha_K$ and $\etanew[0]$ in \cref{lem:error-bound:error-control-one-step,lem:error-bound:induction,lem:error-bound:distance-Yk,lem:error-bound:V-limit} with $\alpha_{K,f}$ and $\etanew[0,f]$.

As the last ingredient, \cref{lem:error-bound:first-sylvester-eq} is needed to control the error in the first Sylvester equation ($\ell=0$), which only relies on $n$, $r$, and the eigenvalues $\{\lam{i}\}_{i=1}^n$ of $Z$.
\begin{lemma}
    \label{lem:error-bound:first-sylvester-eq}
    Suppose that $\HX$ and $\HS$ satisfy $\normtwo{\HX} + \normtwo{\HS} \le \frac{\lam{r} - \lam{r+1}}{2nd}$, and that $\WO[0]$ is the solution for 
    \begin{align*}
        & \WO \ZX[0] + (-\ZS[0]) \WO = \ZO[0] \\
        \Longleftrightarrow \quad & \WO (\LamX + \HX) - (\LamS + \HS) \WO= \HO.
    \end{align*}
    Then, it holds that
    \[
        \normtwo{\WO[0] - \Theta_0 \circ \HO} \le \frac{2nd}{(\lam{r} - \lam{r+1})^2}\cdot \normtwo{\HO} \cdot (\normtwo{\HX} + \normtwo{\HS}),
    \]
    where
    \[
        \Theta_0 = \mymat{
            \frac{1}{\lam{1} - \lam{r+1}} & \cdots & \frac{1}{\lam{r} - \lam{r+1}} \\
            \vdots & \ddots & \vdots \\
            \frac{1}{\lam{1} - \lam{n}} & \cdots & \frac{1}{\lam{r} - \lam{n}}
        } \in \Real{(n-r) \times r}.
    \]
\end{lemma}
\begin{proof}
    See \cref{app:sec:error-bound:first-sylvester-eq}.
\end{proof}

To prove \cref{thm:eb-intro-thm}, it only remains to upper bound the two terms on the right-hand side of~\eqref{eq:error-bound:two-terms} one-by-one. Define $C_{\mathrm{EB}}$ as
\begin{align}
    \label{eq:error-bound:Cf}
    C_{\mathrm{EB}} := \min\left\{ C_{K, f}, \, \frac{1}{2}\min\left\{ \lam{r}, -\lam{r+1} \right\}, \, \frac{\lam{r} - \lam{r+1}}{4nd} \right\}.
\end{align}
Note that in the following proof, we have already replaced $\alpha_K$ and $\etanew[0]$ in \cref{lem:error-bound:error-control-one-step,lem:error-bound:induction,lem:error-bound:distance-Yk,lem:error-bound:V-limit} with $\alpha_{K,f}$ and $\etanew[0,f]$.
\begin{enumerate}
    \item The first term on the right-hand side of \eqref{eq:error-bound:two-terms} is bounded by
    \begin{align}
        &\ \left\|(\I[n] + \W[0]) \mymat{\ZX[\infty] & 0 \\ 0 & 0} (\I[n] + \W[0])\tran - \PiSnp{\Z} - \Omega \circ \H \right\|_2 \nonumber \\
        =&\ \left\|(\I[n] + \W[0]) \mymat{\ZX[\infty] & 0 \\ 0 & 0} (\I[n] + \W[0])\tran - \mymat{\LamX & 0 \\ 0 & 0} - \Omega \circ \H \right\|_2 \nonumber \\
        =&\ \left\|\mymat{\ZX[\infty] & \ZX[\infty] \WO[0]\tran \\ \WO[0] \ZX[\infty] & \WO[0] \ZX[\infty] \WO[0]\tran} - \mymat{\LamX + \HX & \Theta\tran \circ \HO\tran \\ \Theta \circ \HO & 0} \right\|_2 \nonumber \\
        \le&\ \normtwo{\ZX[\infty] - (\LamX + \HX)} + \normtwo{\WO[0] \ZX[\infty] \WO[0]\tran} + \normtwo{\WO[0] \ZX[\infty] - \Theta \circ \HO} \nonumber \\
        \leq&\ \normtwo{\ZX[\infty] - (\LamX + \HX)} + \normtwo{\WO[0] \ZX[\infty] \WO[0]\tran} + \normtwo{\WO[0] \ZX[0] - \Theta \circ \HO} \nonumber \\
        &\ + \normtwo{\WO[0] (\ZX[\infty] - \ZX[0])}. \label{eq:error-bound:thm-1}
    \end{align}
    Again, we bound the right-hand side of \eqref{eq:error-bound:thm-1} one-by-one.
    \begin{enumerate}
        \item For the term $\normtwo{\ZX[\infty] - (\LamX + \HX)}$, we have from $\ZO[0] = \HO$ that
        \begin{subequations}
        \begin{align}
            & \normtwo{\ZX[\infty] - (\LamX + \HX)} = \normtwo{\ZX[\infty] - \ZX[0]} \nonumber \\
            = & \left\|{\sum_{i=0}^\infty (\ZX[i+1] - \ZX[0])} \right\|_2 \le \sum_{i=0}^\infty \normtwo{\ZX[i+1] - \ZX[0]} \nonumber \\
            \le & \alpha_{K,f} \cdot \sum_{i=0}^\infty \normtwo{\ZO[0]}^{i+2}  \label{eq:error-bound:thm-1a-1} \\
            = & \alpha_{K,f} \cdot \sum_{i=0}^\infty \normtwo{\HO}^{i+2} = \alpha_{K,f} \cdot \frac{\normtwo{\HO}^2}{1 - \normtwo{\HO}} \nonumber \\
            \le & 2\alpha_{K,f} \cdot \normtwo{\HO}^2, \label{eq:error-bound:thm-1a-2}
        \end{align}
    \end{subequations}
        where~\eqref{eq:error-bound:thm-1a-1} uses~\eqref{eq:error-bound:induction-1} and~\eqref{eq:error-bound:thm-1a-2} uses $\normtwo{\HO} \le \normtwo{\H} \le C_{K,f} \le \frac{1}{2}$.
        
        \item For the term $\normtwo{\WO[0] \ZX[\infty] \WO[0]\tran}$, we have
        \[
            \normtwo{\WO[0] \ZX[\infty] \WO[0]\tran} \le \normtwo{\WO[0]}^2 \cdot \normtwo{\ZX[\infty]} \le \etanew[0,f]^2 \cdot \normtwo{\HO}^2 \cdot \normtwo{\ZX[\infty]}.
        \]
        Since $\normtwo{\ZX[\infty]} \le \normtwo{\Z[\infty]} = \normtwo{\Z + \H}$ is bounded, there exists a positive constant $\alpha_1$ such that $\normtwo{\WO[0] \ZX[\infty] \WO[0]\tran} \le \alpha_1 \cdot \normtwo{\HO}^2$.
        
        \item For the term $\normtwo{\WO[0] \ZX[0] - \Theta \circ \HO}$, we have
        \begin{subequations}
        \begin{align}
            &\ \normtwo{\WO[0] \ZX[0] - \Theta \circ \HO} \nonumber \\
            =&\ \normtwo{\WO[0] (\LamX + \HX) - \Theta \circ \HO} \nonumber \\
            \le&\ \normtwo{\WO[0] \LamX - \Theta \circ \HO} + \normtwo{\WO[0] \HX} \nonumber \\
            =&\ \normtwo{\WO[0] \LamX - (\HO \circ \Theta_0) \LamX} + \normtwo{\WO[0] \HX} \label{eq:error-bound:thm-1c-1} \\ 
            \le&\ \lam{1} \cdot \normtwo{\WO[0]  - (\HO \circ \Theta_0) } + \normtwo{\WO[0] \HX} \nonumber \\
            \le&\ \frac{2nd\lam{1}}{(\lam{r} - \lam{r+1})^2}\cdot \normtwo{\HO} \cdot (\normtwo{\HX} + \normtwo{\HS}) + \normtwo{\WO[0] \HX} \label{eq:error-bound:thm-1c-2} \\
            \le&\ \frac{2nd\lam{1}}{(\lam{r} - \lam{r+1})^2}\cdot \normtwo{\HO} \cdot (\normtwo{\HX} + \normtwo{\HS}) + \etanew[0,f] \cdot \normtwo{\HO} \cdot \normtwo{\HX} \label{eq:error-bound:thm-1c-3} \\
            \leq&\ \alpha_2 \normtwo{\HO} \cdot (\normtwo{\HX} + \normtwo{\HS}) \label{eq:error-bound:thm-1c-4}
        \end{align}
        \end{subequations}
        for some positive constant $\alpha_2$. Here,~\eqref{eq:error-bound:thm-1c-1} holds since for a diagonal matrix~$D$:
        \[
            (AD) \circ B = B \circ (AD) = (B \circ A) D,
        \]
        \eqref{eq:error-bound:thm-1c-2} comes from \cref{lem:error-bound:first-sylvester-eq} since
        \[
            \normtwo{\HX} + \normtwo{\HS} \le 2 \normtwo{\H} \le 2 C_{K,f} \le 2 \cdot \frac{\lam{r} - \lam{r+1}}{4nd} = \frac{\lam{r} - \lam{r+1}}{2nd},
        \]
        and~\eqref{eq:error-bound:thm-1c-3} follows from \cref{lem:error-bound:error-control-one-step} and $\normtwo{\WO[0]} \le \etanew[0,f] \cdot \normtwo{\ZO[0]} = \etanew[0,f] \cdot \normtwo{\HO}$. 

        \item For the term $\normtwo{\WO[0] (\ZX[\infty] - \ZX[0])}$, we have
        \begin{subequations}
        \begin{align}
            \normtwo{\WO[0] (\ZX[\infty] - \ZX[0])} \le&\ \normtwo{\WO[0]} \cdot \normtwo{\ZX[\infty] - \ZX[0]} \nonumber \\
            \le&\ \etanew[0,f] \cdot \normtwo{\HO} \cdot \normtwo{\ZX[\infty] - \ZX[0]} \nonumber \\
            \le&\ \etanew[0,f] \cdot \normtwo{\HO} \cdot \alpha_{K,f} \cdot \sum_{i=0}^\infty \normtwo{\HO}^{i+2} \label{eq:error-bound:thm-1d-1} \\
            \le&\ \etanew[0,f] \alpha_{K,f} \cdot \normtwo{\HO} \cdot \frac{\normtwo{\HO}^2}{1 - \normtwo{\HO}} \nonumber \\
            \le&\ 2 \etanew[0,f] \alpha_{K,f} \cdot \normtwo{\HO}^3, \label{eq:error-bound:thm-1d-2}
        \end{align}
        \end{subequations}
        where~\eqref{eq:error-bound:thm-1d-1} follows from~\eqref{eq:error-bound:induction-2}.
    \end{enumerate}

    \item For the second term on the right-hand side of \eqref{eq:error-bound:two-terms}, we see from \cref{lem:error-bound:distance-Yk} that
    \begin{align*}
        \normtwo{\Y[\infty] - (\I[n] + \W[0])} &\le \frac{2}{3}\etanew[0,f]^2 \cdot \normtwo{\HO}^2 + \frac{8}{3} \etanew[0,f] \alpha_{K,f} \cdot \sum_{i=1}^{\infty} \normtwo{\HO}^{i+1} \\
        &= \frac{2}{3}\etanew[0,f]^2 \cdot \normtwo{\HO}^2 + \frac{8}{3} \etanew[0,f] \alpha_{K,f} \cdot \frac{\normtwo{\HO}^2}{1 - \normtwo{\HO}} \\
        &\le \frac{2}{3}\etanew[0,f]^2 \cdot \normtwo{\HO}^2 + \frac{16}{3} \etanew[0,f] \alpha_{K,f} \cdot \normtwo{\HO}^2.
    \end{align*}
    Together with the boundedness of $\normtwo{\I[n] + \W[0]}$ and $\ZX[\infty]$, we conclude that there exists a positive constant $\alpha_3$ such that
    \begin{align}
        &\ \left\|\Y[\infty] \mymat{\ZX[\infty] & 0 \\ 0 & 0} \Y[\infty]\tran - (\I[n] + \W[0]) \mymat{\ZX[\infty] & 0 \\ 0 & 0} (\I[n] + \W[0])\tran \right\|_2 \nonumber \\
        \le&\ 2 \left\|{\left( \Y[\infty] - (\I[n] + \W[0]) \right) \mymat{\ZX[\infty] & 0 \\ 0 & 0} (\I[n] + \W[0])\tran} \right\|_2 \nonumber \\
        &\ + \left\|{\left( \Y[\infty] - (\I[n] + \W[0]) \right) \mymat{\ZX[\infty] & 0 \\ 0 & 0} \left( \Y[\infty] - (\I[n] + \W[0]) \right)\tran}\right\|_2 \nonumber \\
        \le&\ \alpha_3 \cdot \normtwo{\HO}^2. \label{eq:error-bound:thm-1e}
    \end{align}
\end{enumerate}
Therefore, combining \eqref{eq:error-bound:two-terms}, \eqref{eq:error-bound:thm-1}, \eqref{eq:error-bound:thm-1a-2}, \eqref{eq:error-bound:thm-1c-4}, \eqref{eq:error-bound:thm-1d-2} and \eqref{eq:error-bound:thm-1e} yields
\begin{align*}
    &\ \normtwo{\PiSnp{\Z + \H} - \PiSnp{\Z} - \Omega \circ \H} \\
    \le&\ 2\alpha_{K,f} \cdot \normtwo{\HO}^2 + \alpha_2 \cdot \normtwo{\HO}^2 + \alpha_2 \cdot \normtwo{\HO} \cdot (\normtwo{\HX} + \normtwo{\HS}) \\
    &\ + 2\etanew[0,f] \alpha_{K,f} \cdot \normtwo{\HO}^2 + \alpha_3 \cdot \normtwo{\HO}^2 \\
    \le&\ \alpha_{\mathrm{EB}} \cdot \normtwo{\HO} \cdot \normtwo{\H}
\end{align*}
for some positive constant $\alpha_{\mathrm{EB}}$. This concludes the proof.

\subsection{Generalization to the Non-diagonal Case} \label{sec:error-bound:ndiag}

Though the previous analysis is performed under the assumption that $Z$ is a diagonal matrix, straightforward computation generalizes our result to the more general, non-diagonal case. When $Q \neq \I[n]$, we have
\begin{subequations}
\begin{align}
    &\ \normtwo{\PiSnp{\Z + \H} - \PiSnp{\Z} - \Y (\Y\tran \H \Y) \Y\tran} \nonumber \\
    =&\ \normtwo{\Y \big( \Y\tran \PiSnp{\Z + \H} \Y - \Y\tran \PiSnp{\Z} \Y - \Y\tran \H \Y \big) \Y\tran} \nonumber \\
    =&\ \normtwo{\Y\tran \PiSnp{\Z + \H} \Y - \Y\tran \PiSnp{\Z} \Y - \Y\tran \H \Y} \label{eq:error-bound:generalization-1} \\
    =&\ \normtwo{\PiSnp{\Y\tran \Z \Y + \Y\tran \H \Y} -  \PiSnp{\Y\tran \Z \Y} - \Y\tran \H \Y} \label{eq:error-bound:generalization-2} \\
    =&\ \normtwo{\PiSnp{\Y\tran \Z \Y + \tH} -  \PiSnp{\Y\tran \Z \Y} - \tH} \nonumber \\
    \le&\ \alpha_{\mathrm{EB}} \cdot \normtwo{\tHO} \cdot \normtwo{\tH}, \label{eq:error-bound:generalization-3}
\end{align}
\end{subequations}
where~\eqref{eq:error-bound:generalization-1} follows from the fact that $\normtwo{\Y A} = \normtwo{A}$, for any matrix $A$,~\eqref{eq:error-bound:generalization-2} uses $\PiSnp{\Y\tran X \Y} = \Y\tran \PiSnp{X} \Y$, and~\eqref{eq:error-bound:generalization-3} holds since $\Y\tran \Z \Y$ is diagonal.


\section{Numerical Experiments}
\label{sec:exp}

In this section, numerical evidence is reported to support our theoretical findings. In particular, numerical experiments are conducted to demonstrate the following.
\begin{enumerate}
    \item Local (R-)linear convergence is observed, regardless of the (non)degeneracy of the SDP.

    \item The established (R-)linear rate of convergence (\eg in \cref{thm:conv-nd} and \cref{lem:conv-nnd-blk}) is numerically tight.

    \item When SC is close to failure, ADMM for SDP may be extremely slow and no clear linear convergence can be observed within the stated computational budget.
\end{enumerate}

Experiments are performed on a high-performance workstation equipped with a 2.7 GHz AMD 64-Core sWRX8 Processor and 1 TB of RAM. For the standard SDP~\eqref{eq:intro-sdp}, we denote primal infeasibility $r_{\mathrm p}$, dual infeasibility $r_{\mathrm d}$, and relative gap $r_{\mathrm{gap}}$ as:
\[
    r_{\mathrm p} := \frac{\normtwo{\Asdp X - b}}{1 + \normtwo{b}}, \quad r_{\mathrm d} := \frac{\normF{\AsdpT y + S - C}}{1 + \normF{C}}, \quad r_{\mathrm{gap}} := \frac{\abs{\inprod{C}{X} - b\tran y}}{1 + \abs{\inprod{C}{X}} + \abs{b\tran y}},
\]
and define the maximum KKT residual $\rmax := \max \{r_{\mathrm p}, r_{\mathrm d}, r_{\mathrm{gap}} \}$. Unless specified, the stopping criteria are $\rmax \leq 10^{-10}$, or the maximum iteration number goes beyond $10^6$, or the CPU time exceeds $100$ hours. \Cref{tab:instance} presents the data for all the tested SDP instances. The strict complementarity condition is checked by computing $\lammin{\abs{\Zs}}$, the smallest eigenvalues of $\Zs$ in absolute values.

\begin{table}
    \begin{tabular}{cc}
        \centering
        \noindent
        \begin{minipage}{0.49\textwidth}
            \centering
            \begin{tabular}{|c|c|c|c|}
                \hline
                & $n$ & $m$ & $\sigma$ \\
                \hline 
                \texttt{MAXCUT-G*} & $800$ & $800$ & $1$ \\
                \hline 
                \texttt{hamming-10-2} & $1024$ & $23041$ & $0.01$ \\
                \hline 
                \texttt{hamming-7-5-6} & $128$ & $1793$ & $0.01$ \\
                \hline 
                \texttt{hamming-9-5-6} & $512$ & $53761$ & $0.01$ \\
                \hline 
                \texttt{theta-102} & $500$ & $37467$ & $1$ \\
                \hline 
                \texttt{XM-48} & $144$ & $241$ & $100$ \\
                \hline 
                \texttt{XM-149} & $447$ & $746$ & $100$ \\
                \hline
                \texttt{BQP-r*-30-*} & $496$ & $91326$ & $100$ \\
                \hline 
                \texttt{QS-20} & $231$ & $16402$ & $100$ \\
                \hline 
                \texttt{QS-40} & $861$ & $236202$ & $100$ \\
                \hline
                \texttt{Quasar-200} & $804$ & $122601$ & $100$ \\
                \hline
                \texttt{swissroll} & $800$ & $3380$ & $1$ \\
                \hline
                \texttt{1dc-1024} & $1024$ & $24064$ & $100$ \\
                \hline
                \texttt{neosfbr25} & $577$ & $14376$ & $1$ \\
                \hline
            \end{tabular}
        \end{minipage}

        \begin{minipage}{0.49\textwidth}
            \centering
            \begin{tabular}{|c|c|c|c|}
                \hline
                & $n$ & $m$ & $\sigma$ \\
                \hline 
                \texttt{hamming-9-8} & $512$ & $2305$ & $0.01$ \\
                \hline 
                \texttt{hamming-11-2} & $2048$ & $56321$ & $0.01$ \\
                \hline 
                \texttt{hamming-8-3-4} & $256$ & $16129$ & $0.01$ \\
                \hline 
                \texttt{theta-12} & $600$ & $17979$ & $1$ \\
                \hline 
                \texttt{theta-123} & $600$ & $90020$ & $1$ \\
                \hline 
                \texttt{XM-93} & $279$ & $466$ & $100$ \\
                \hline 
                \texttt{BQP-r*-20-*} & $231$ & $20601$ & $100$ \\
                \hline 
                \texttt{BQP-r*-40-*} & $861$ & $296001$ & $100$ \\
                \hline 
                \texttt{QS-30} & $496$ & $77377$ & $100$ \\
                \hline
                \texttt{Quasar-100} & $404$ & $31301$ & $100$ \\
                \hline
                \texttt{Quasar-500} & $2004$ & $756501$ & $100$ \\
                \hline
                \texttt{cnhil10} & $220$ & $5005$ & $0.01$ \\
                \hline
                \texttt{rose13} & $105$ & $2379$ & $1$ \\
                \hline
                & & & \\
                \hline
            \end{tabular}
        \end{minipage}
    \end{tabular}
    \caption{Details about all the tested SDP instances. $n$ is the size of the matrix, $m$ is the number of equality constraints, and $\sigma$ is the fixed penalty paramater in ADMM.}
    \label{tab:instance}
\end{table}

\subsection{Demonstration of Local Linear Convergence}

In this section, we solve a considerable number of SDPs arising from various applications. In all the experiments, local linear convergence of ADMM is clearly observed, regardless of the (non)degeneracy of the SDPs.
\begin{itemize}
    \item \textbf{MAXCUT~\cite{davis2011acm-florida-sparse-matrix-collection}.}
    \Cref{fig:maxcut} reports three representative examples. In all three cases, strict complementarity holds numerically and ADMM enters the linear convergence region rather quickly.

\begin{figure}[tbp]
    \centering

    \begin{minipage}{\textwidth}
        \centering
        \hspace{5mm} \includegraphics[width=0.35\columnwidth]{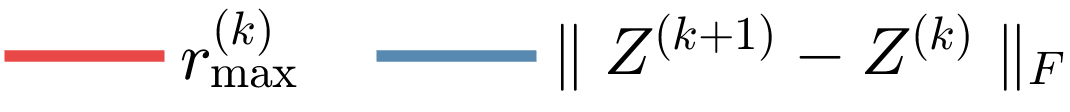}
    \end{minipage}

    \begin{minipage}{\textwidth}
        \centering
        \begin{tabular}{ccc}
            \begin{minipage}{0.30\textwidth}
                \centering
                \includegraphics[width=\columnwidth]{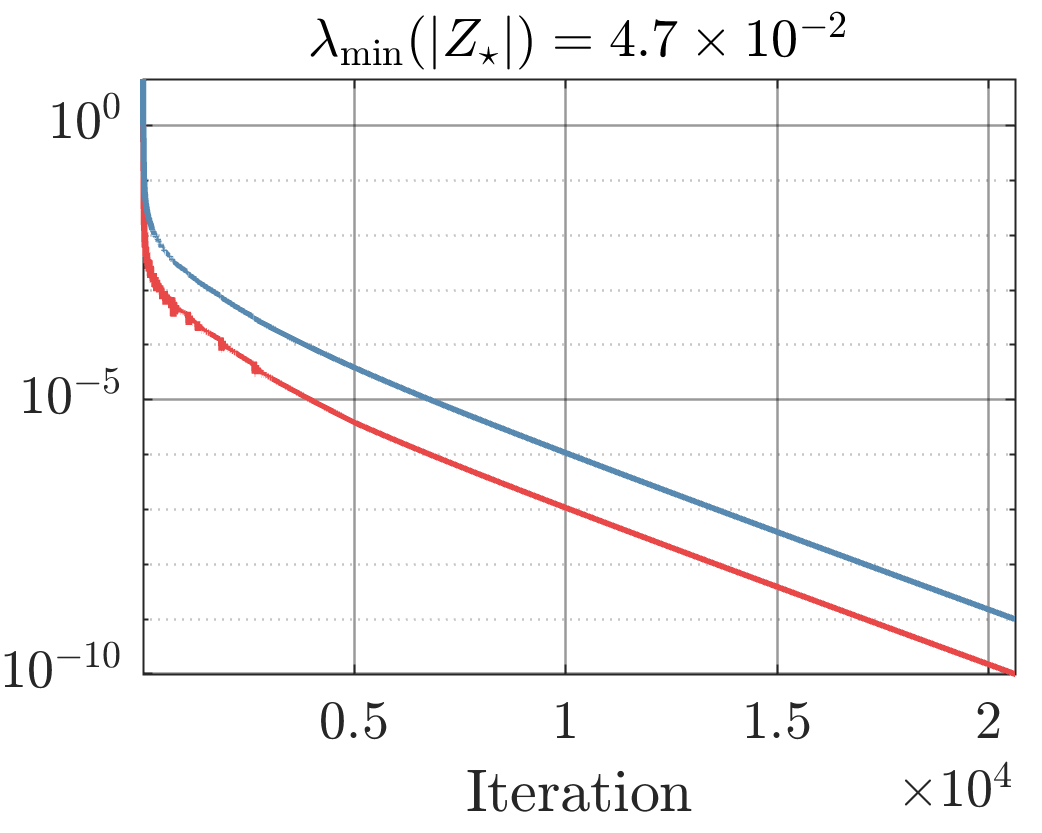}
                \texttt{MAXCUT-G1}
            \end{minipage}

            \begin{minipage}{0.30\textwidth}
                \centering
                \includegraphics[width=\columnwidth]{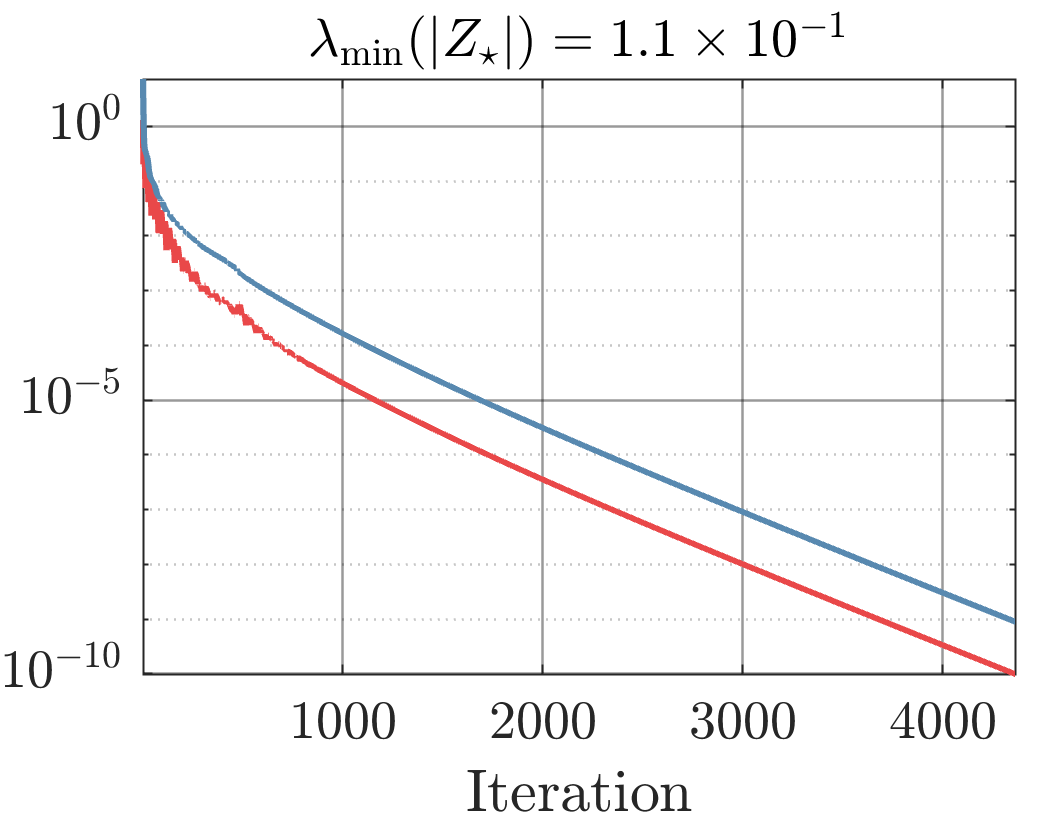}
                \texttt{MAXCUT-G9}
            \end{minipage}

            \begin{minipage}{0.30\textwidth}
                \centering
                \includegraphics[width=\columnwidth]{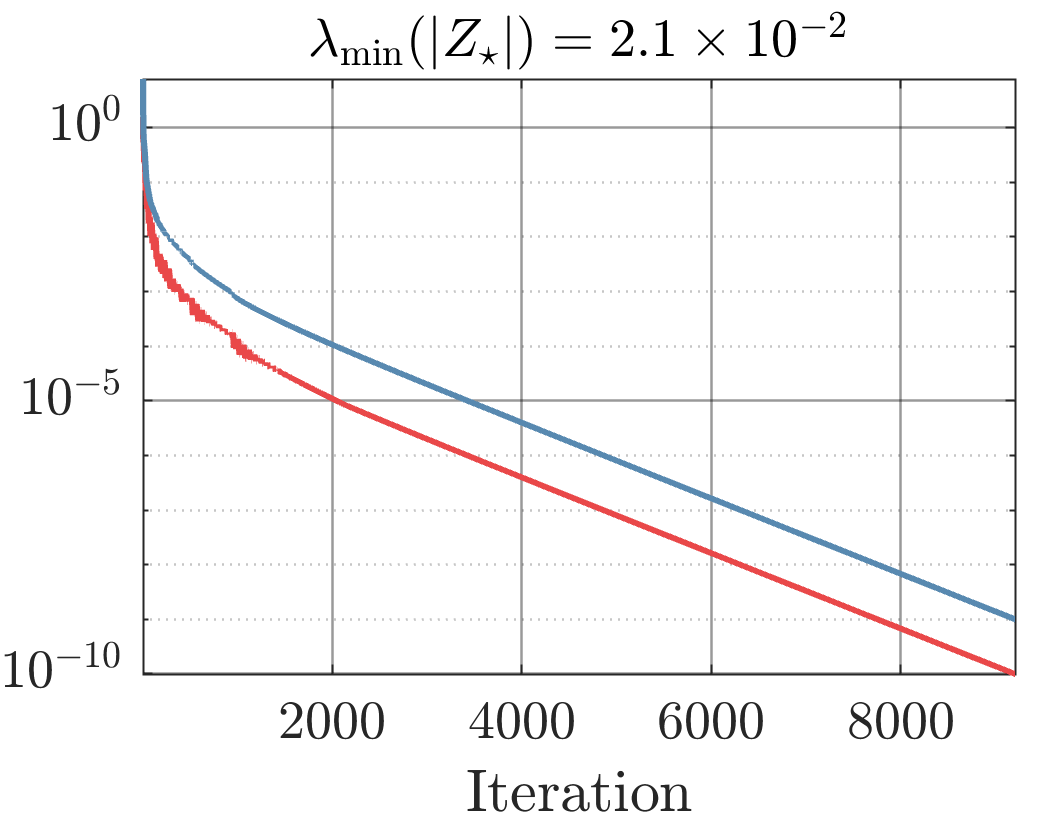}
                \texttt{MAXCUT-G18}
            \end{minipage}
        \end{tabular}
    \end{minipage}
    \vspace{1mm}

    \caption{MAXCUT problems with with random (standard Gaussian) initial guess. In all cases, the converging $\Zs$ is nonsingular. \label{fig:maxcut}}
\end{figure}

    \item \textbf{Hamming set problems~\cite{abor1999dataset-dimacs}.}
    \Cref{fig:hamming} reports three representative examples. In all three cases, strict complementarity holds numerically.

\begin{figure}[tbp]
    \centering

    \begin{minipage}{\textwidth}
        \centering
        \hspace{5mm} \includegraphics[width=0.35\columnwidth]{figs/legends/legend_rmax_dZ.png}
    \end{minipage}

    \begin{minipage}{\textwidth}
        \centering
        \begin{tabular}{ccc}
            \begin{minipage}{0.30\textwidth}
                \centering
                \includegraphics[width=\columnwidth]{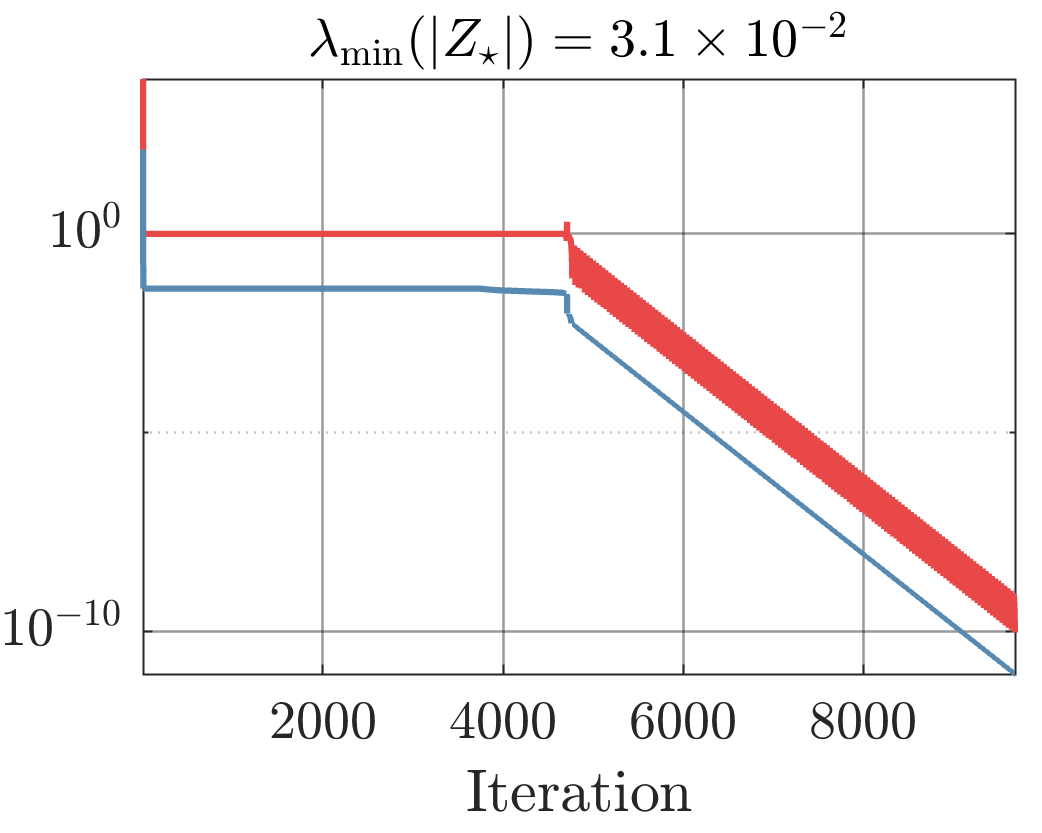}
                \texttt{hamming-9-8}
            \end{minipage}

            \begin{minipage}{0.30\textwidth}
                \centering
                \includegraphics[width=\columnwidth]{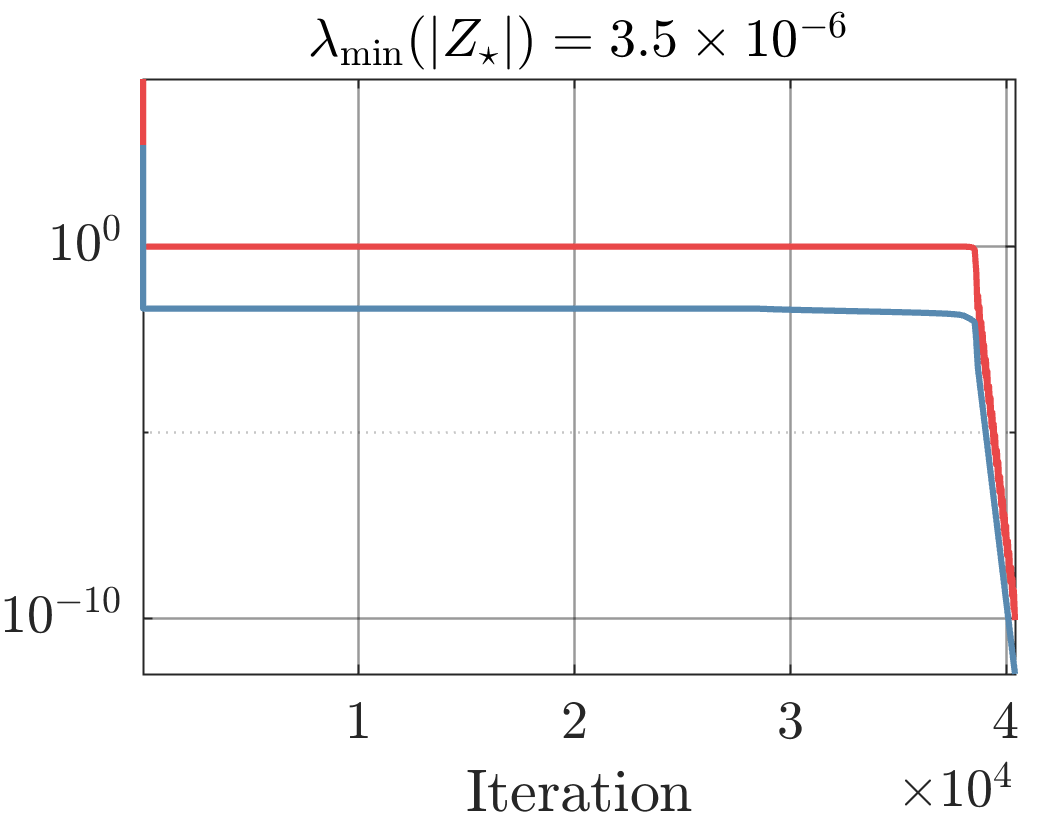}
                \texttt{hamming-11-2}
            \end{minipage}

            \begin{minipage}{0.30\textwidth}
                \centering
                \includegraphics[width=\columnwidth]{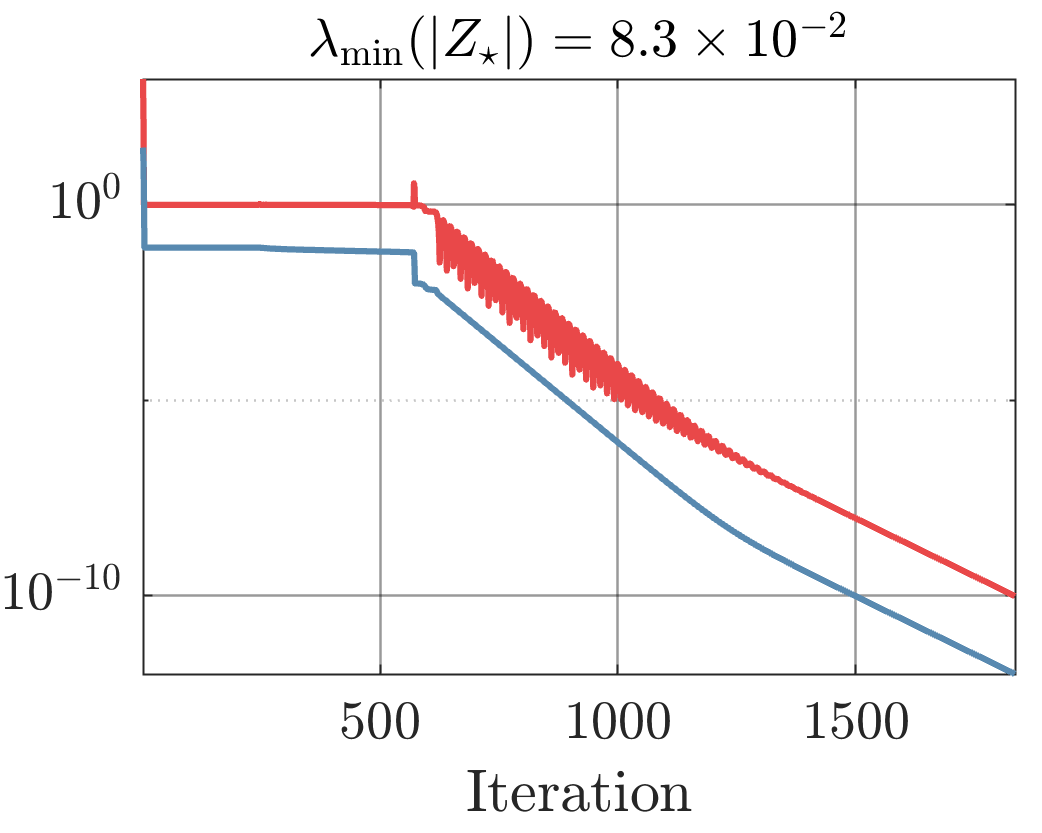}
                \texttt{hamming-7-5-6}
            \end{minipage}
        \end{tabular}
    \end{minipage}

    \caption{Additional Hamming graph problems with with random (standard Gaussian) initial guess. In all cases, the converging $\Zs$ is nonsingular.}
    \label{fig:hamming}
\end{figure}

    \item \textbf{Maximum stable set problems~\cite{mittelmann2006dataset-sparse-sdp-problems}.}
    \Cref{fig:theta} reports three representative examples. In all three cases, strict complementarity holds numerically.

\begin{figure}[htbp]
    \centering

    \begin{minipage}{\textwidth}
        \centering
        \hspace{5mm} \includegraphics[width=0.35\columnwidth]{figs/legends/legend_rmax_dZ.png}
    \end{minipage}

    \begin{minipage}{\textwidth}
        \centering
        \begin{tabular}{ccc}
            \begin{minipage}{0.30\textwidth}
                \centering
                \includegraphics[width=\columnwidth]{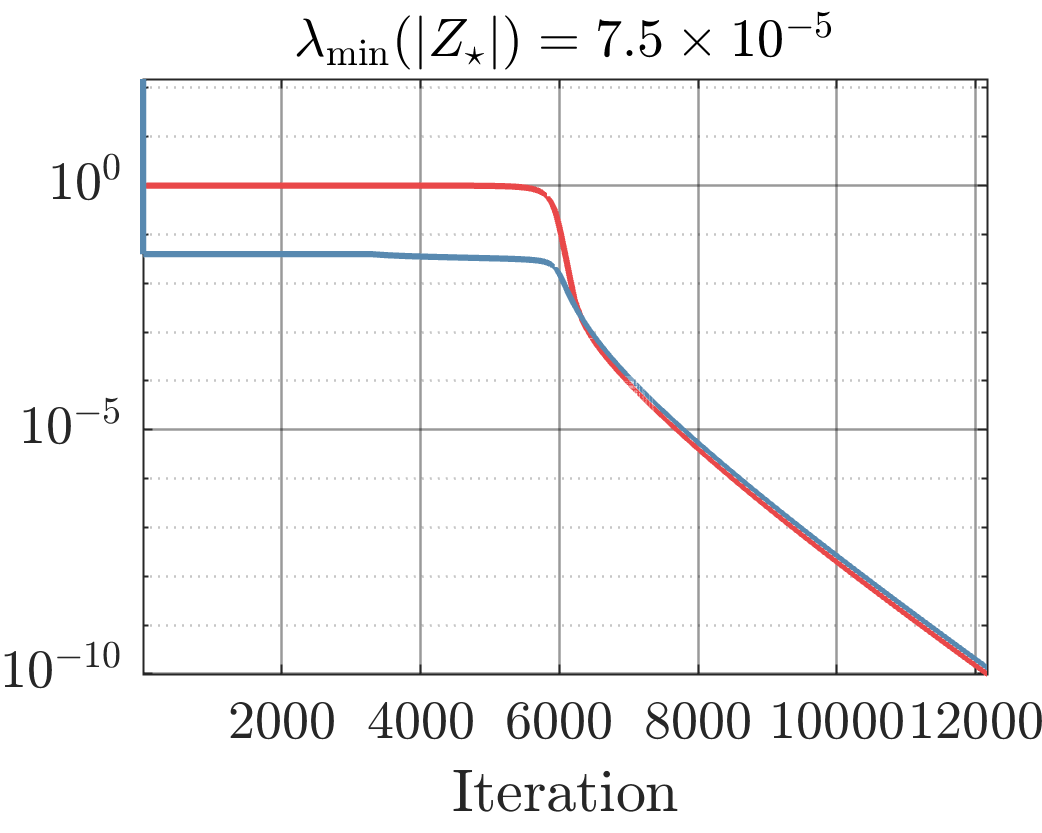}
                \texttt{theta-12}
            \end{minipage}

            \begin{minipage}{0.30\textwidth}
                \centering
                \includegraphics[width=\columnwidth]{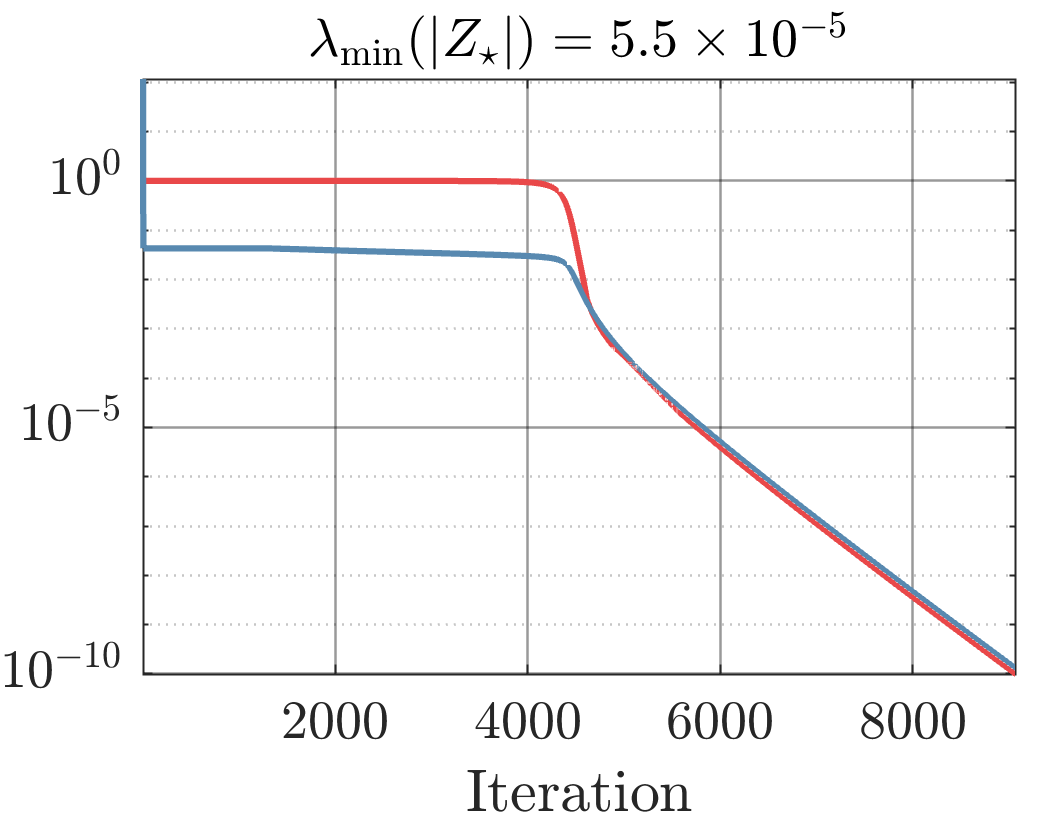}
                \texttt{theta-102}
            \end{minipage}

            \begin{minipage}{0.30\textwidth}
                \centering
                \includegraphics[width=\columnwidth]{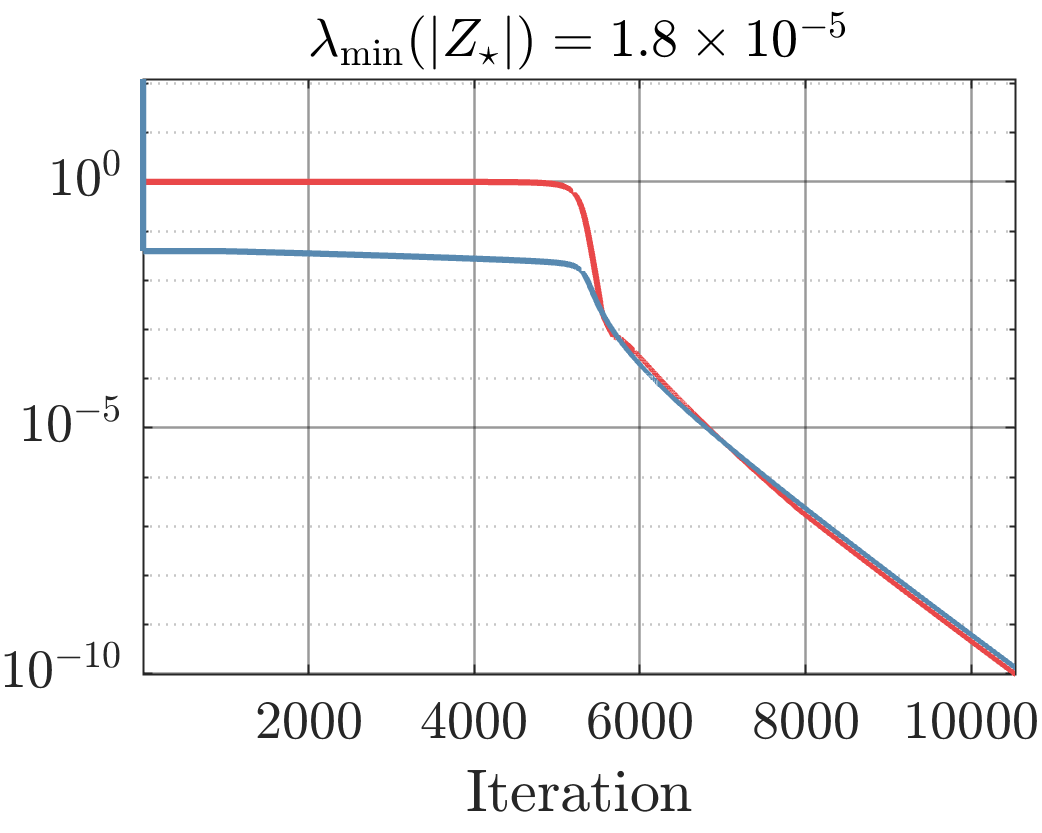}
                \texttt{theta-123}
            \end{minipage}
        \end{tabular}
    \end{minipage}

    \caption{Maximum stable set problems with with random (standard Gaussian) initial guess. In all cases, the converging $\Zs$ is nonsingular. \label{fig:theta}}
\end{figure}

    \item \textbf{Structure from motion problems~\cite{han2025arxiv-xm}.}
    \Cref{fig:XM} reports three representative examples. In all three cases, strict complementarity holds numerically. For \texttt{XM-48} and \texttt{XM-93}, the maximum KKT residual $\rmax$ can only reach $10^{-8}$ due to numerical errors from eigenvalue decomposition and the unbalance between infeasibility and relative gap.

\begin{figure}[tbp]
    \centering

    \begin{minipage}{\textwidth}
        \centering
        \hspace{5mm} \includegraphics[width=0.35\columnwidth]{figs/legends/legend_rmax_dZ.png}
    \end{minipage}

    \begin{minipage}{\textwidth}
        \centering
        \begin{tabular}{ccc}
            \begin{minipage}{0.30\textwidth}
                \centering
                \includegraphics[width=\columnwidth]{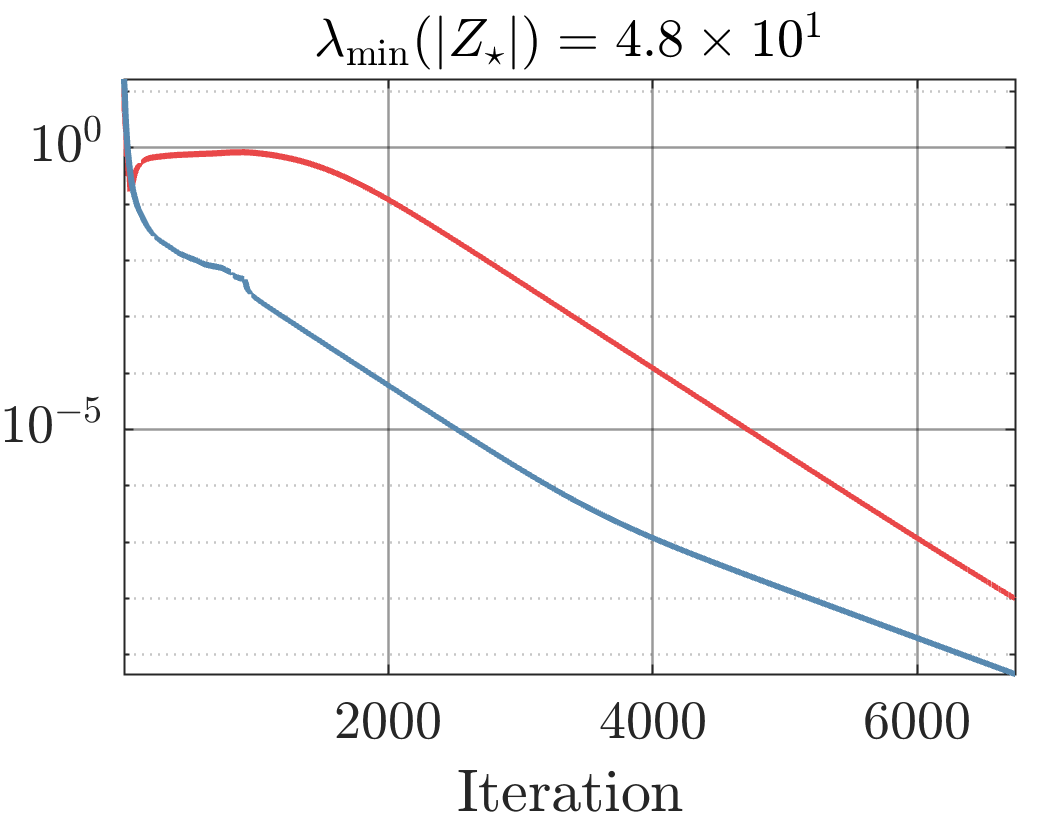}
                \texttt{XM-48}
            \end{minipage}

            \begin{minipage}{0.30\textwidth}
                \centering
                \includegraphics[width=\columnwidth]{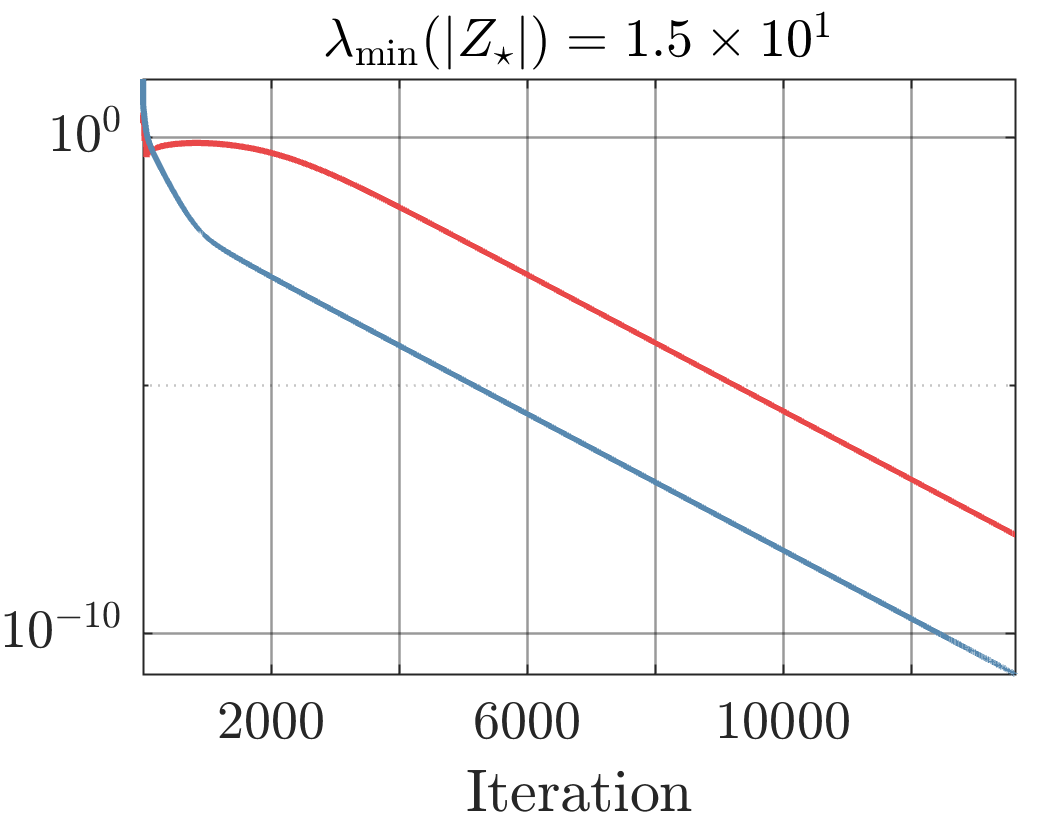}
                \texttt{XM-93}
            \end{minipage}

            \begin{minipage}{0.30\textwidth}
                \centering
                \includegraphics[width=\columnwidth]{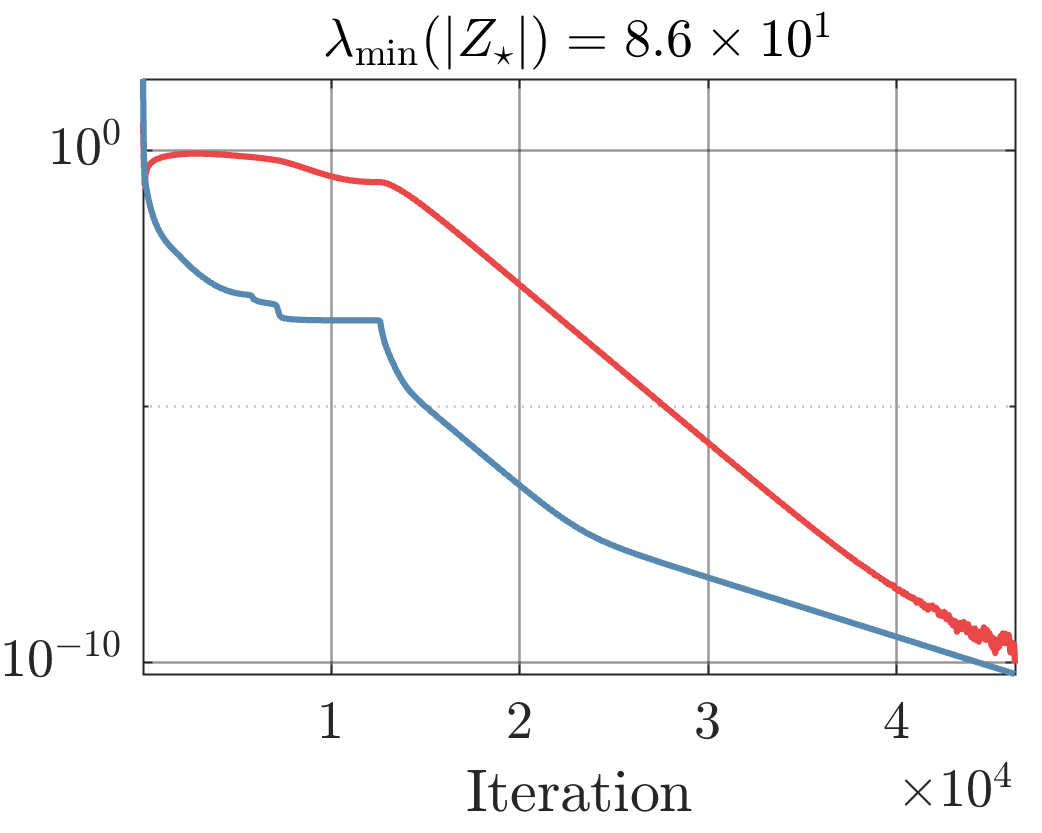}
                \texttt{XM-149}
            \end{minipage}
        \end{tabular}
    \end{minipage}

    \caption{Structure-from-motion problems with with random (standard Gaussian) initial guess. In all cases, the converging $\Zs$ is nonsingular. \label{fig:XM}}
\end{figure}

    \item \textbf{Binary quadratic programming (BQP).}
    We consider the second-order moment-sum-of-squares (moment-SOS) relaxation \cite{lasserre2001siopt-global} of the following polynomial optimization problem
    \[
        \begin{array}{ll}
            \text{minimize} & \frac{1}{2} x\tran Q x + c\tran x \\
            \text{subject to} & 1 - x_i^2 = 0, \;\; i \in [n],
        \end{array}
    \]
    where the optimization variable is $x \in \Real{n}$, and the data are $Q \in \Sn$ and $c \in \Real{n}$. Depending on the data, nondegeneracy (ND) and strict complementarity (SC) conditions may or may not hold.
    \begin{itemize}
        \item \textbf{Case 1: primal ND fails and SC holds.}
        When $c \sim \calN(0,\I[n])$, the SDP relaxation is empirically tight and the primal optimal solution has rank one \cite{yang2023mp-stride,wang2023arxiv-manisdp}. In this case, primal nondegeneracy fails. \Cref{fig:BQP-r1} reports three representative examples with random (standard Gaussian) initial guess.

\begin{figure}[tbp]
    \centering

    \begin{minipage}{\textwidth}
        \centering
        \hspace{5mm} \includegraphics[width=0.35\columnwidth]{figs/legends/legend_rmax_dZ.png}
    \end{minipage}

    \begin{minipage}{\textwidth}
        \centering
        \begin{tabular}{ccc}
            \begin{minipage}{0.30\textwidth}
                \centering
                \includegraphics[width=\columnwidth]{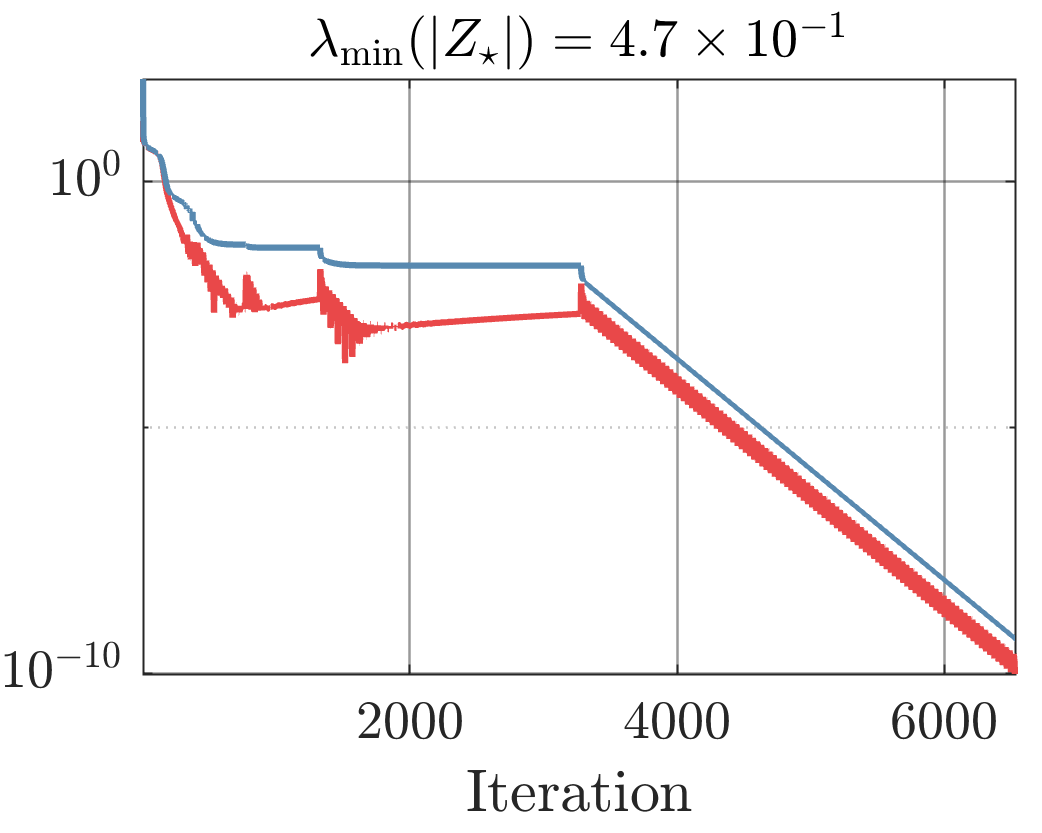}
                \texttt{BQP-r1-20-1}
            \end{minipage}

            \begin{minipage}{0.30\textwidth}
                \centering
                \includegraphics[width=\columnwidth]{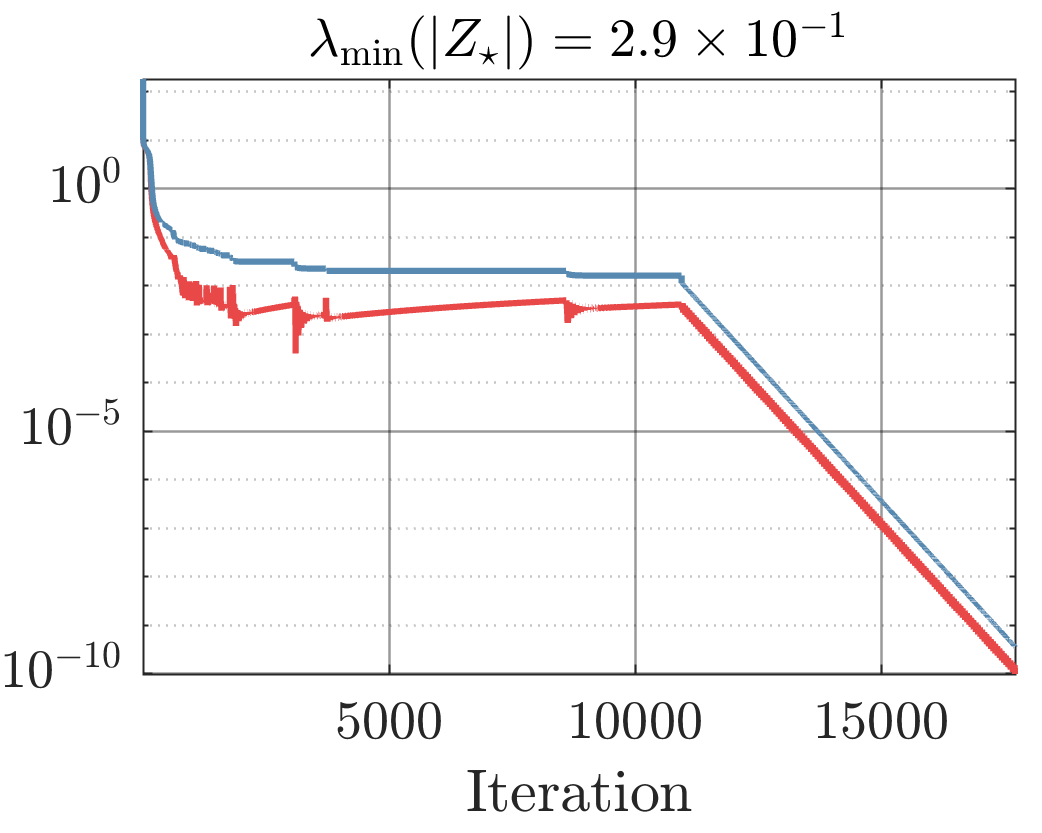}
                \texttt{BQP-r1-30-1}
            \end{minipage}

            \begin{minipage}{0.30\textwidth}
                \centering
                \includegraphics[width=\columnwidth]{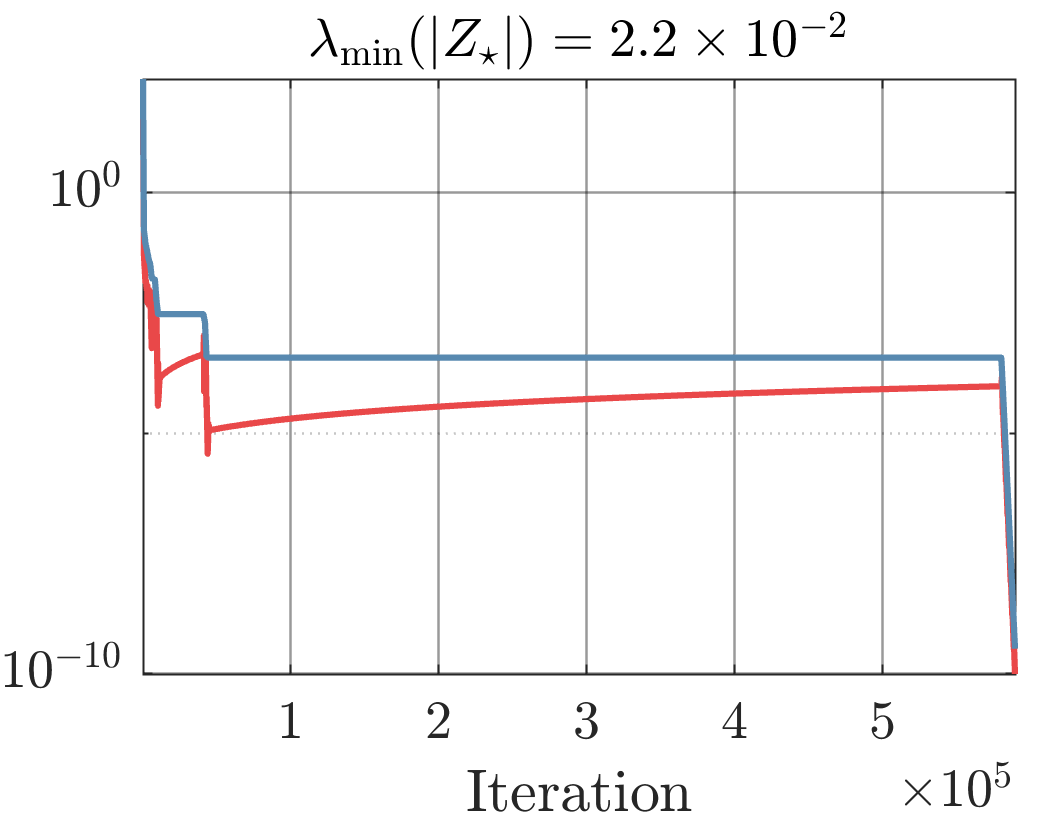}
                \texttt{BQP-r1-40-1}
            \end{minipage}
        \end{tabular}
    \end{minipage}

    \caption{Random BQP problems with $c \sim \calN(0, I_n)$ with random (standard Gaussian) initial guess. In all cases, the converging $\Zs$ is nonsingular. \label{fig:BQP-r1}}
\end{figure}

        \item \textbf{Case 2: primal ND fails and SC fails.}
        We test the same BQP instances with all-zeros initialization. As shown in \cref{fig:BQP-r1-zero}, strict complementarity seems to fail. Nonetheless, the failure of strict complementarity does not deteriorate the linear convergence rate.

\begin{figure}
    \centering

    \begin{minipage}{\textwidth}
        \centering
        \hspace{5mm} \includegraphics[width=0.35\columnwidth]{figs/legends/legend_rmax_dZ.png}
    \end{minipage}

    \begin{minipage}{\textwidth}
        \centering
        \begin{tabular}{ccc}
            \begin{minipage}{0.30\textwidth}
                \centering
                \includegraphics[width=\columnwidth]{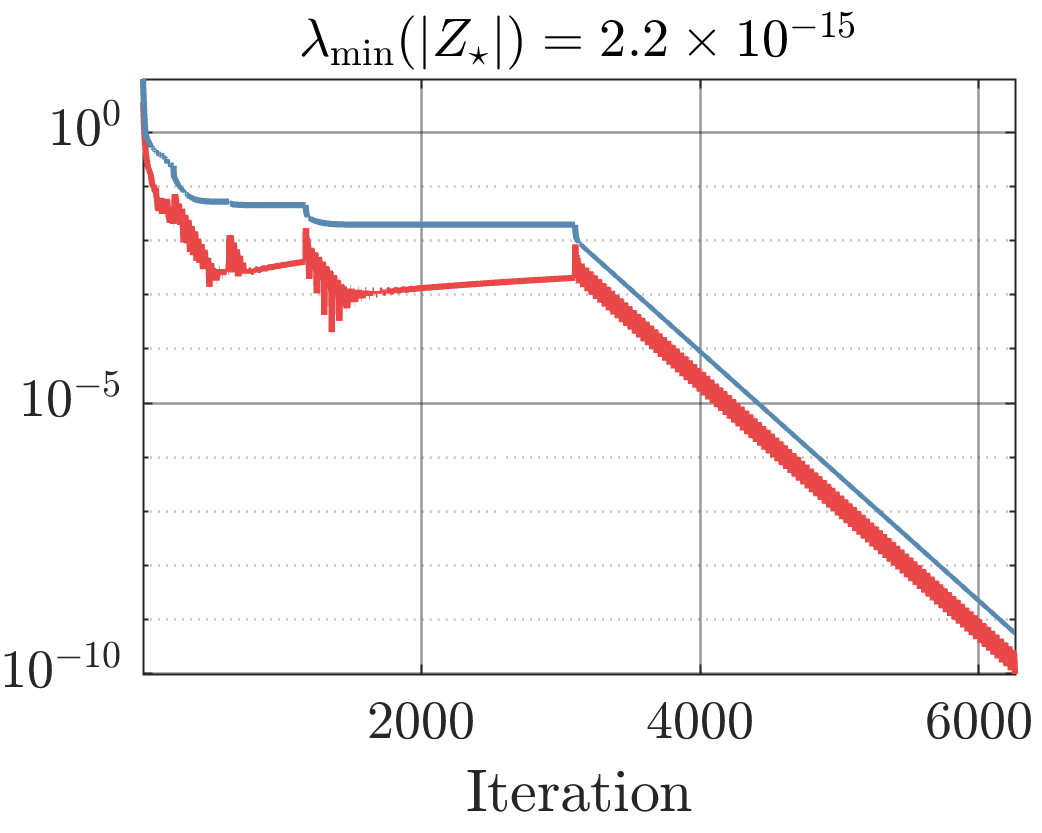}
                \texttt{BQP-r1-20-1}
            \end{minipage}

            \begin{minipage}{0.30\textwidth}
                \centering
                \includegraphics[width=\columnwidth]{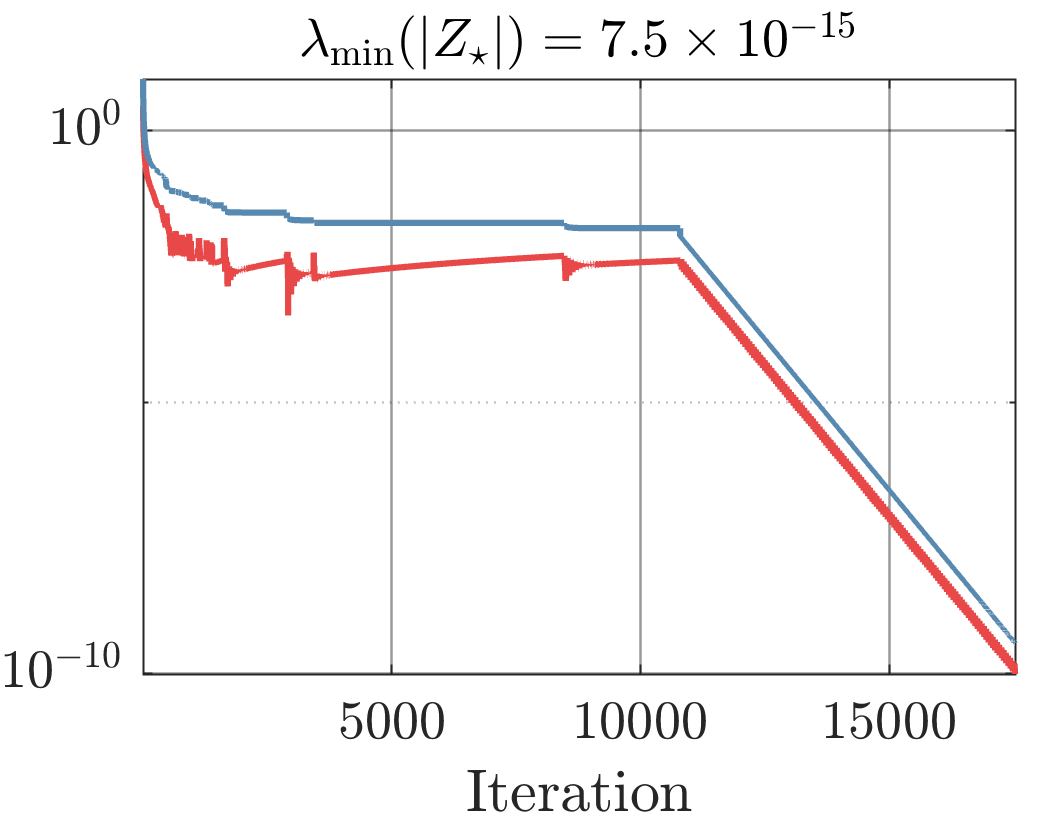}
                \texttt{BQP-r1-30-1}
            \end{minipage}

            \begin{minipage}{0.30\textwidth}
                \centering
                \includegraphics[width=\columnwidth]{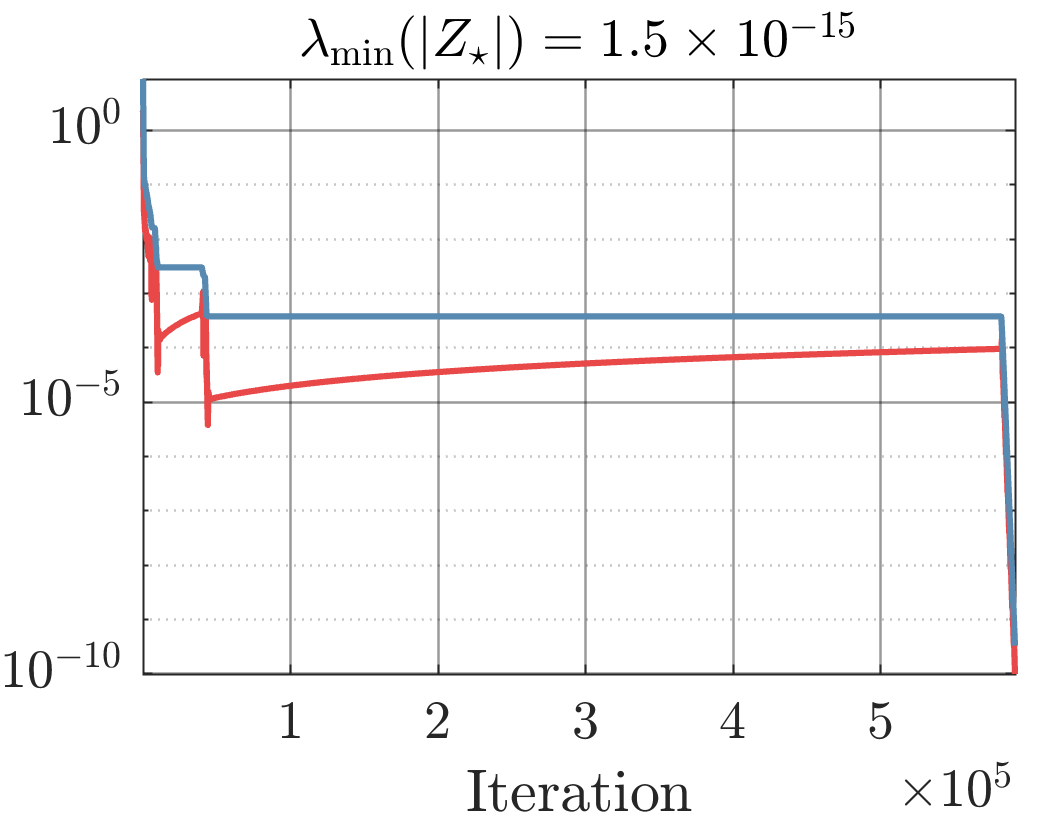}
                \texttt{BQP-r1-40-1}
            \end{minipage}
        \end{tabular}
    \end{minipage}

    \caption{Random BQP problems with $c \sim \calN(0, I_n)$ with all-zero initial guess. In all cases, the converging~$\Zs$ is singular.}
    \label{fig:BQP-r1-zero}
\end{figure}

        \item \textbf{Case 3: both primal and dual ND fail and SC holds.}
        When $c=0$, the SDP relaxation is still empirically tight \cite{yang2023mp-stride,wang2023arxiv-manisdp}. However, the primal optimal solution is no longer unique (due to sign symmetry). In this case, both primal and dual nondegeneracy fail and linear convergence is still observed; see \Cref{fig:BQP-r2}.

\begin{figure}[htbp]
    \centering

    \begin{minipage}{\textwidth}
        \centering
        \hspace{5mm} \includegraphics[width=0.35\columnwidth]{figs/legends/legend_rmax_dZ.png}
    \end{minipage}

    \begin{minipage}{\textwidth}
        \centering
        \begin{tabular}{ccc}
            \begin{minipage}{0.30\textwidth}
                \centering
                \includegraphics[width=\columnwidth]{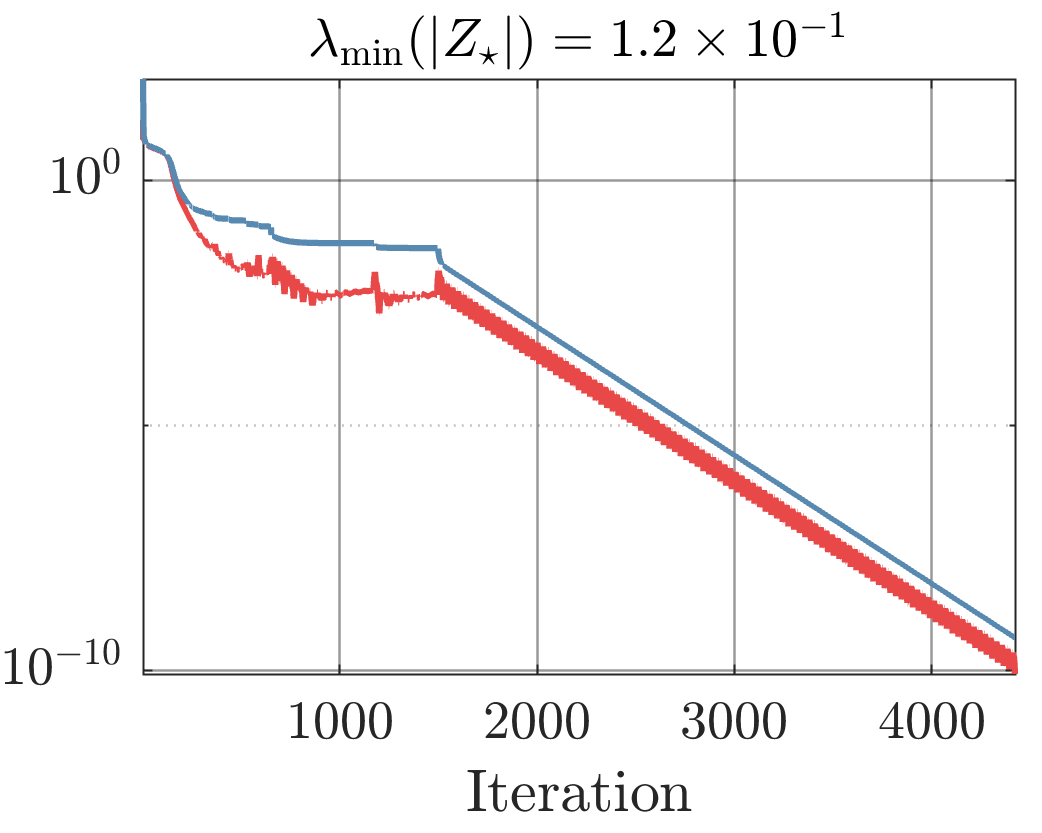}
                \texttt{BQP-r2-20-1}
            \end{minipage}

            \begin{minipage}{0.30\textwidth}
                \centering
                \includegraphics[width=\columnwidth]{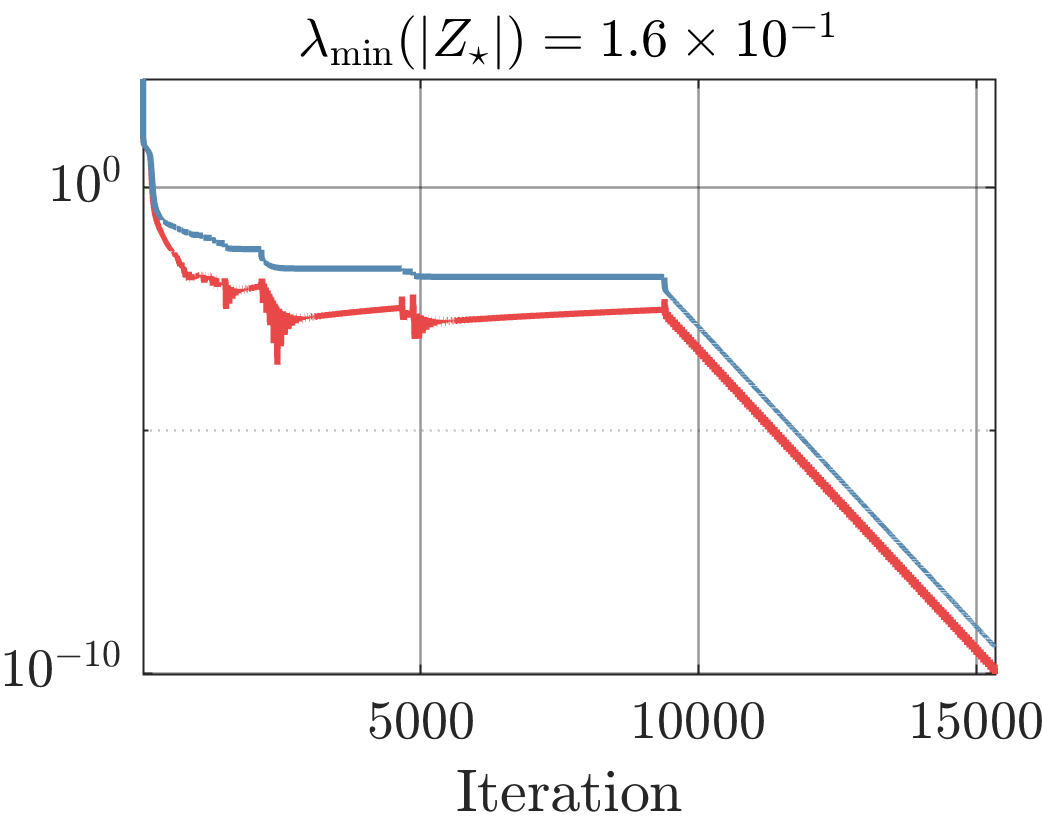}
                \texttt{BQP-r2-30-1}
            \end{minipage}

            \begin{minipage}{0.30\textwidth}
                \centering
                \includegraphics[width=\columnwidth]{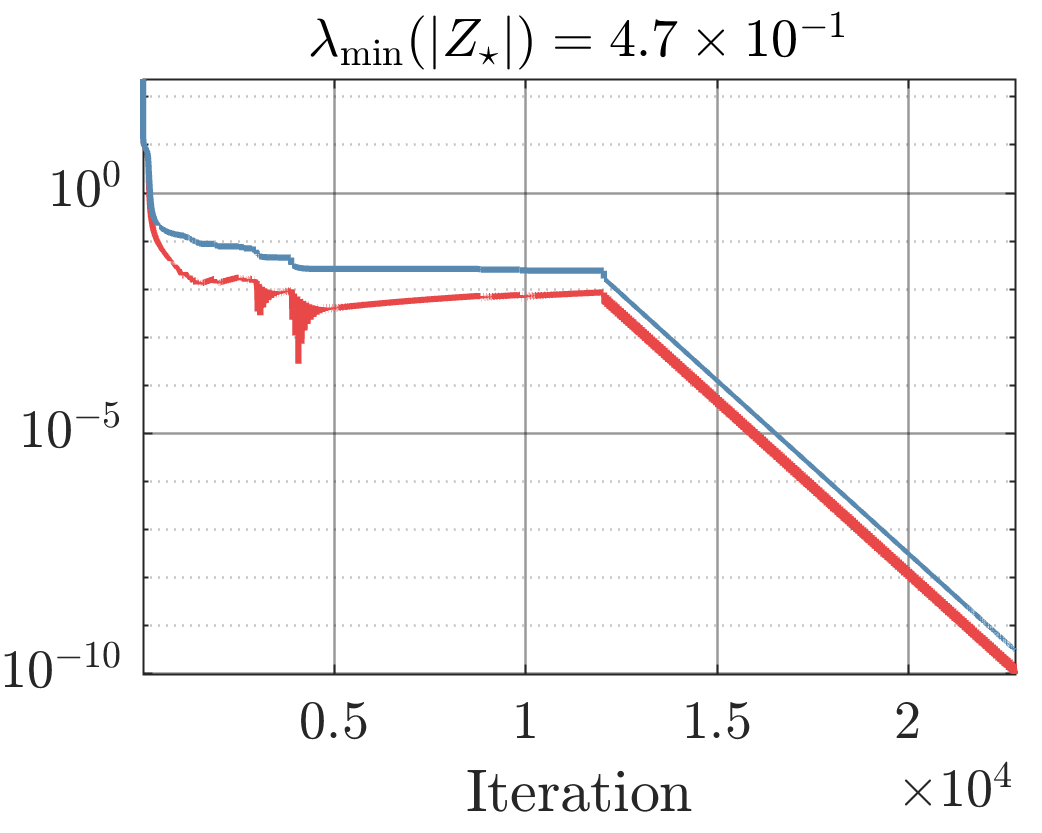}
                \texttt{BQP-r2-40-1}
            \end{minipage}
        \end{tabular}
    \end{minipage}

    \caption{random BQP problems with $c = 0$ with random (standard Gaussian) initial guess, under which both primal and dual nondegeneracy fail. In all cases, the converging $\Zs$ is nonsingular. \label{fig:BQP-r2}}
\end{figure}
    \end{itemize}

    \item \textbf{Quasar problems~\cite{yang2019quaternion}.} 
    In Quasar problems, the primal solution is unique and has rank one. Similar to BQP, primal nondegeneracy always fails in Quasar problems~\cite{yang2019quaternion}. \Cref{fig:quasar} reports three examples, in which strict complementarity holds numerically.

\begin{figure}[tbp]
    \centering

    \begin{minipage}{\textwidth}
        \centering
        \hspace{5mm} \includegraphics[width=0.35\columnwidth]{figs/legends/legend_rmax_dZ.png}
    \end{minipage}

    \begin{minipage}{\textwidth}
        \centering
        \begin{tabular}{ccc}
            \begin{minipage}{0.30\textwidth}
                \centering
                \includegraphics[width=\columnwidth]{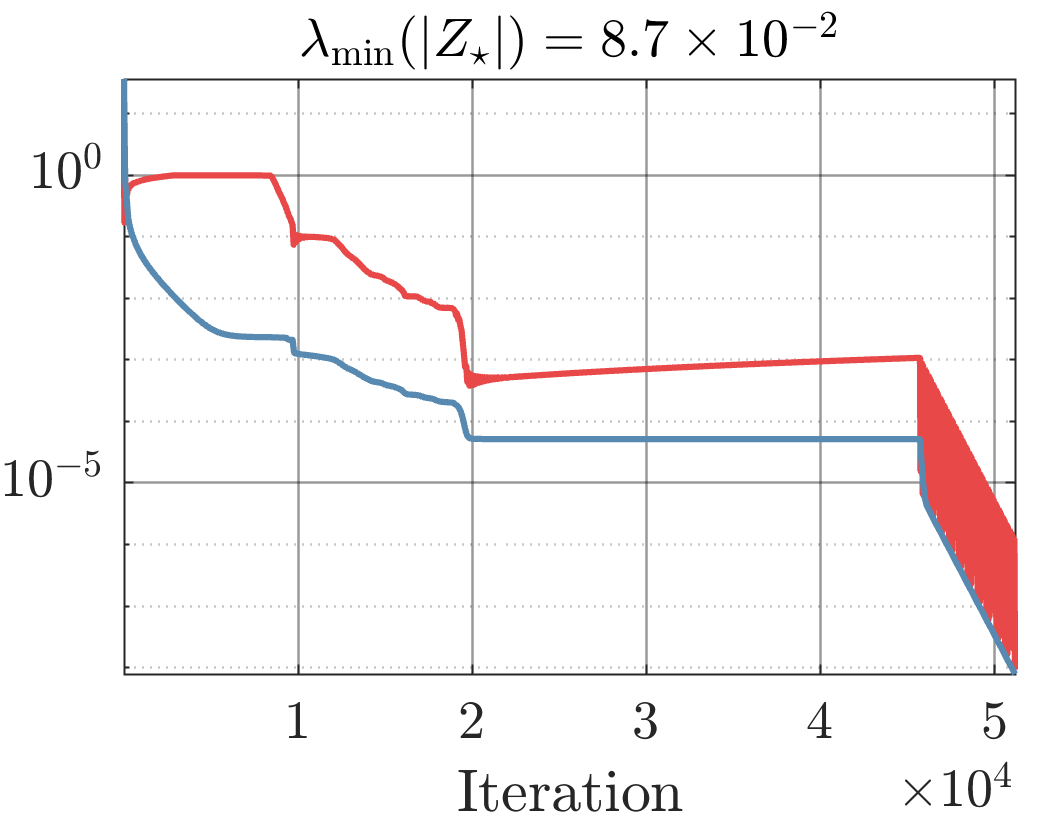}
                \texttt{Quasar-100}
            \end{minipage}

            \begin{minipage}{0.30\textwidth}
                \centering
                \includegraphics[width=\columnwidth]{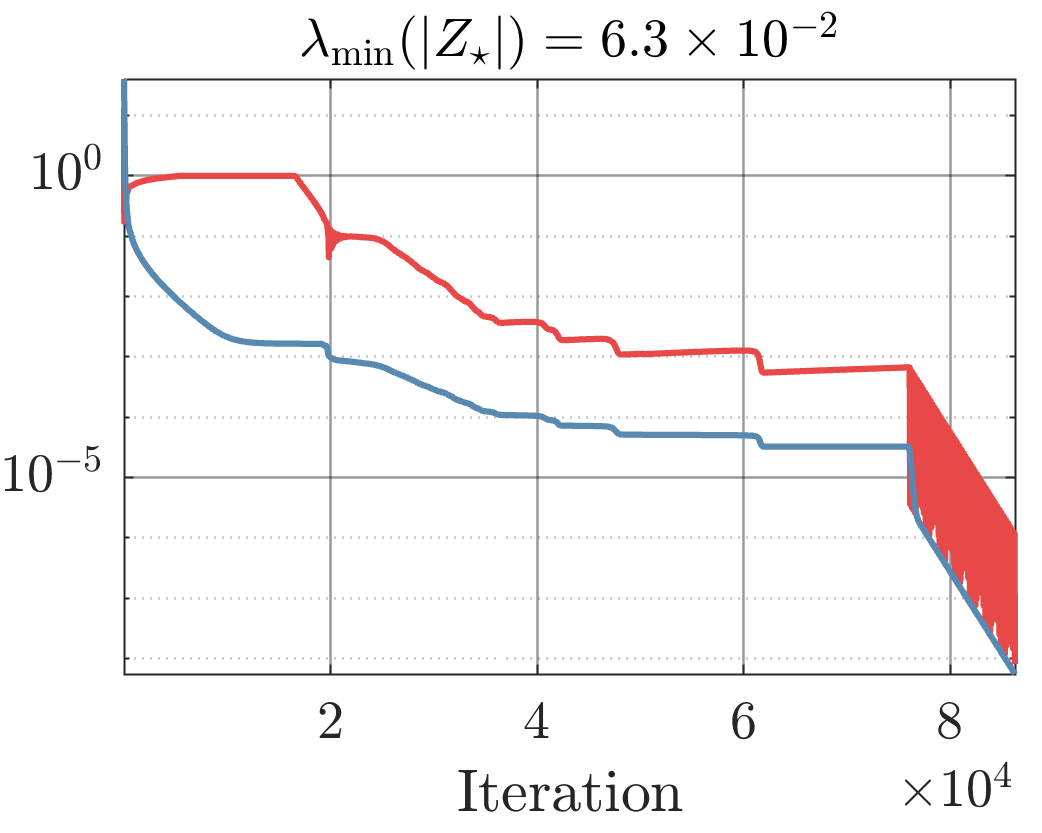}
                \texttt{Quasar-200}
            \end{minipage}

            \begin{minipage}{0.30\textwidth}
                \centering
                \includegraphics[width=\columnwidth]{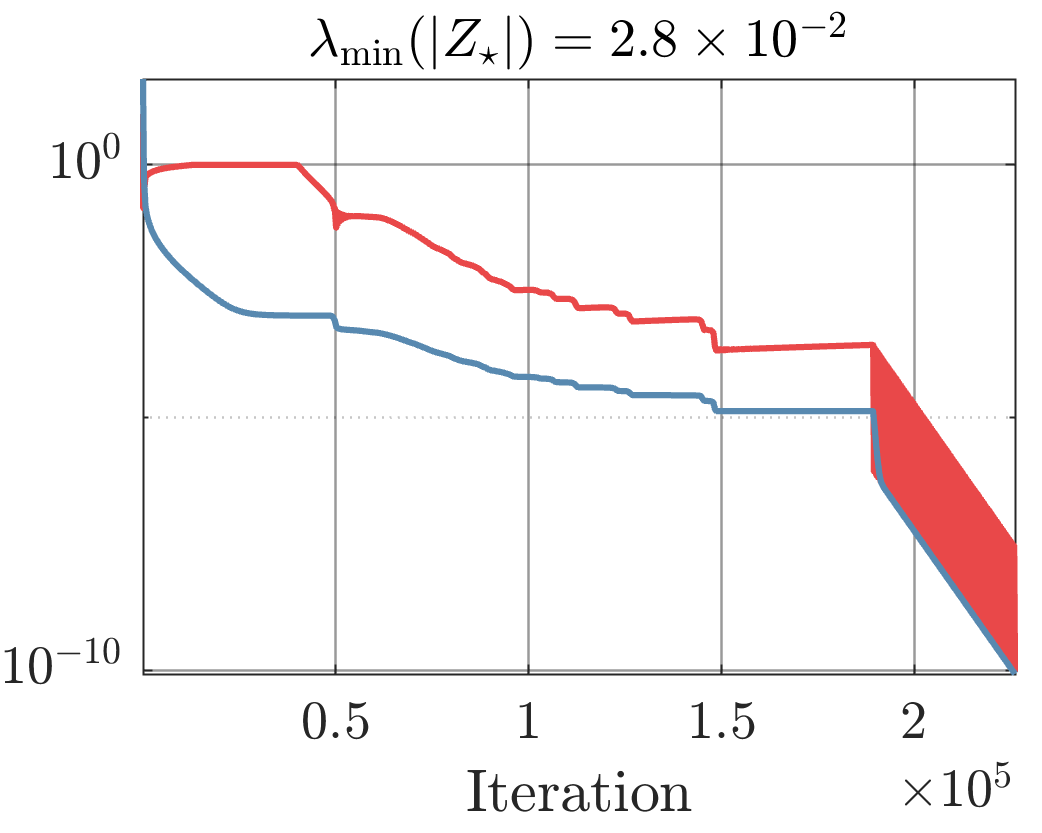}
                \texttt{Quasar-500}
            \end{minipage}
        \end{tabular}
    \end{minipage}

    \caption{Quasar problems with random (standard Gaussian) initial guess. In all cases, the converging $\Zs$ is nonsingular. \label{fig:quasar}}
\end{figure}

    \item \textbf{Quartic function over sphere (QS).}
    Another classical polynomial optimization problem~\cite{yang2023mp-stride,wang2023arxiv-manisdp}. In its second-order relaxation, the primal solution is unique and has rank one. Similar to BQP, primal nondegeneracy of QS always fails. In comparison, \cref{fig:QS} reports three representative examples. In these cases, strict complementarity seems to fail numerically, but linear convergence is still observed.

\begin{figure}[tbp]
    \centering

    \begin{minipage}{\textwidth}
        \centering
        \hspace{5mm} \includegraphics[width=0.35\columnwidth]{figs/legends/legend_rmax_dZ.png}
    \end{minipage}

    \begin{minipage}{\textwidth}
        \centering
        \begin{tabular}{ccc}
            \begin{minipage}{0.30\textwidth}
                \centering
                \includegraphics[width=\columnwidth]{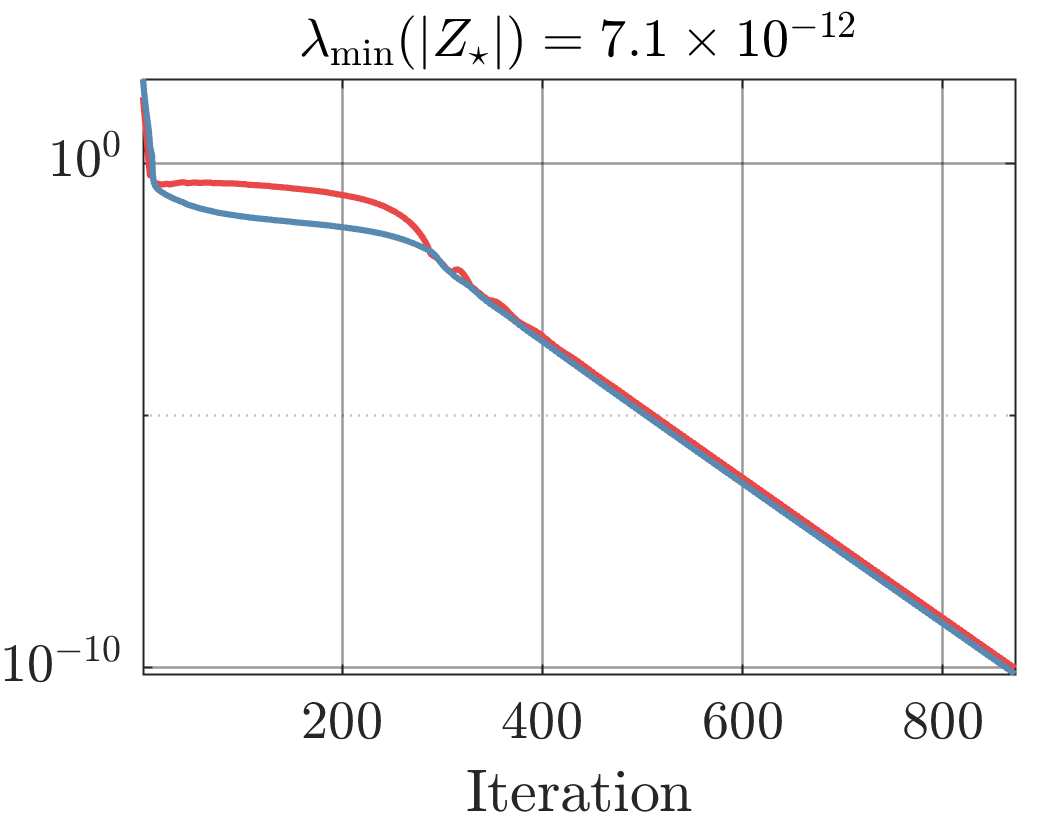}
                \texttt{QS-20}
            \end{minipage}

            \begin{minipage}{0.30\textwidth}
                \centering
                \includegraphics[width=\columnwidth]{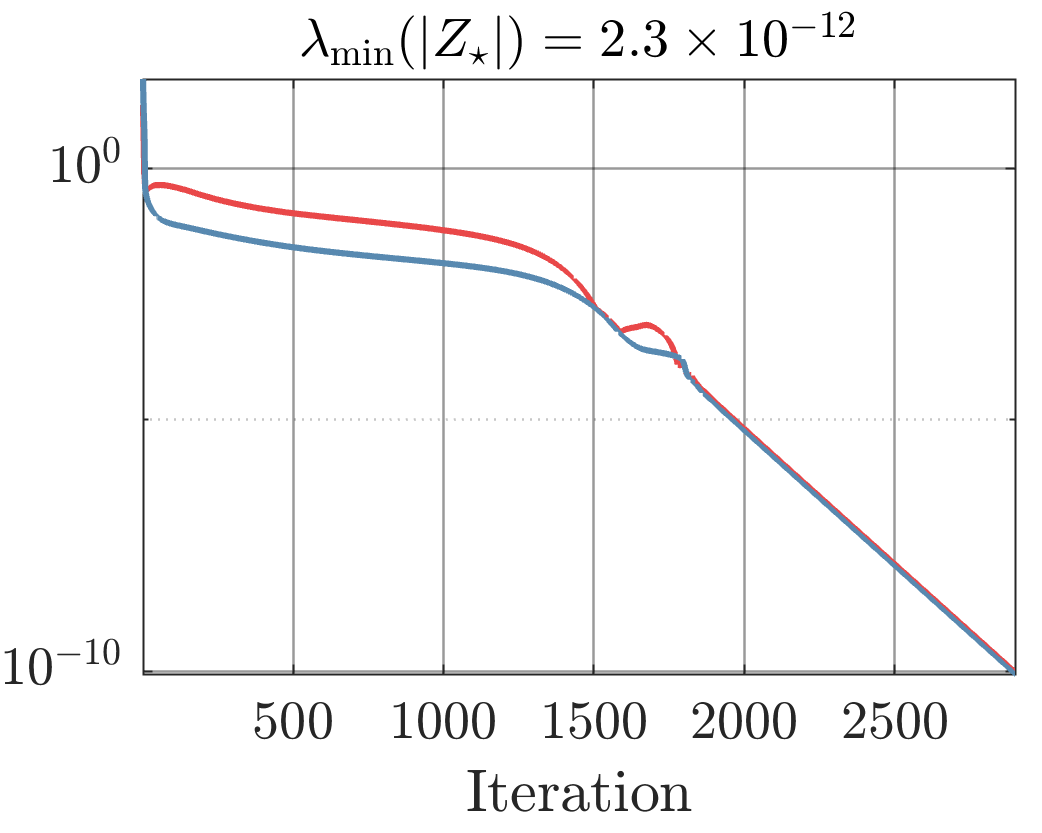}
                \texttt{QS-30}
            \end{minipage}

            \begin{minipage}{0.30\textwidth}
                \centering
                \includegraphics[width=\columnwidth]{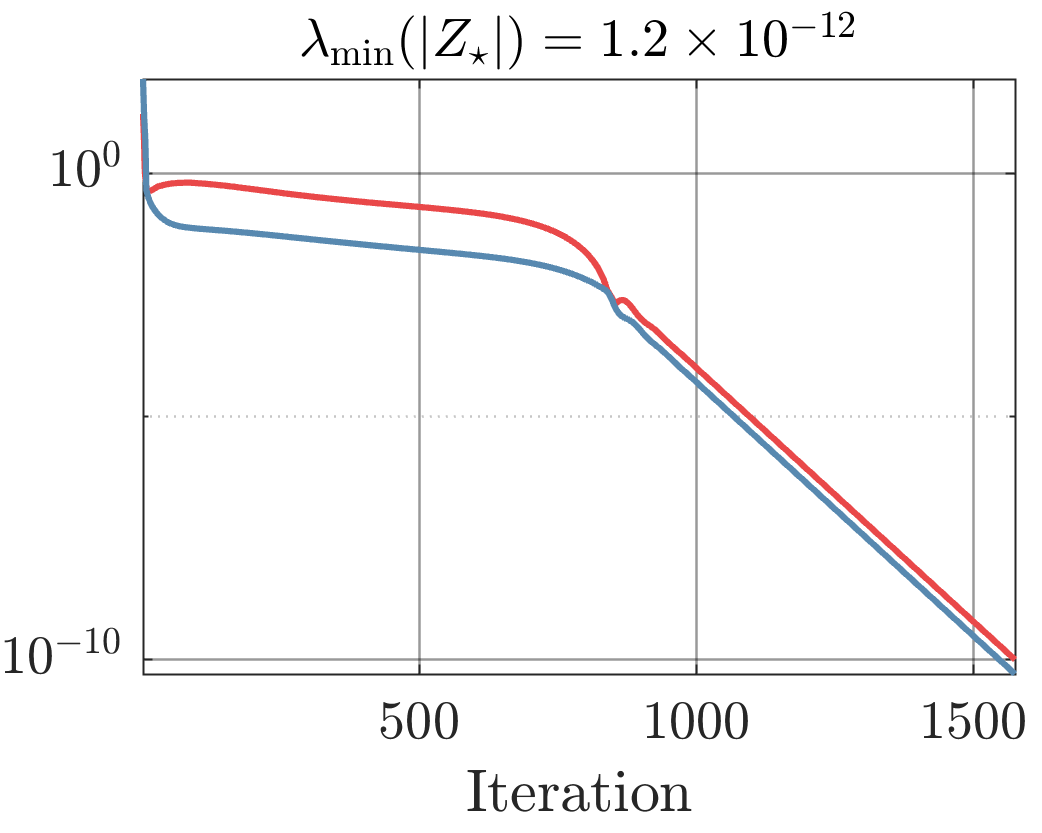}
                \texttt{QS-40}
            \end{minipage}
        \end{tabular}
    \end{minipage}

    \caption{Random QS problems with random (standard Gaussian) initial guess. In all cases, the converging~$\Zs$ is singular. \label{fig:QS}}
\end{figure}
\end{itemize}

Additional numerical results can be found in \cref{app:sec:exp}. 
\rebuttal{
One may also notice two interesting empirical phenomena across a broad class of degenerate SDPs (\eg BQP and Quasar):
\begin{itemize}
    \item ADMM often converges to a strictly complementary primal--dual solution pair when initialized from a random Gaussian point. At present, this is a \emph{purely empirical} observation, and we view this apparent solution-selection tendency as an interesting direction for future study.
    \item Even when ADMM converges to a primal--dual solution pair that is not strictly complementary, linear convergence is still often observed. This suggests that \cref{thm:conv-nnd} provides only a \emph{sufficient} condition for local linear convergence, rather than a necessary one. Further relaxing the strict complementarity assumption would therefore be an intriguing direction for future work.
\end{itemize}
}

\subsection{Demonstration of Numerical Rates}

In this section, we numerically verify that the tightness of the derived (R-)linear rate of convergence. In the following two experiments, primal and dual nondegeneracy are checked numerically as follows. We compute 
\begin{align*}
    W_1 &:= \mymat{\svec[A_1] & \svec[A_2] & \cdots & \svec[A_m]} \\
    W_2 &:= \mymat{\cdots & \svec[\Qs E_{i,j} \Qs\tran] & \cdots}, \;\; \text{for} i = r+1, \dots, n, \ \text{and} \ j = i, \dots, n,
\end{align*}
where $E_{i,j} \in \Real{n \times n}$ is the $(i,j)$th elementary matrix; \ie all the elements are zero except the $(i,j)$th entry is one. It is clear that the columns of $W_1 \in \Real{t(n) \times m}$ form a basis of $\range(\AsdpT)$ and those of $W_2 \in \Real{t(n) \times t(n-r)}$ form a basis of $\calT^\perp_{\Xs}$. To check primal nondegeneracy $\range(\AsdpT) \cap \calT^\perp_{\Xs} = \{0\}$, it suffices to check the following rank condition:
\begin{equation} \label{eq:exp-rank}
    \rank{W_1} + \rank{W_2} = \rank{\mymat{W_1 & W_2}}.
\end{equation}
Dual nondegeneracy can be checked in a similar manner.
\begin{enumerate}[label=(\alph*)]
    \item \textbf{ND holds and SC holds.}
    In \cref{fig:exp:rate} (a), we consider a toy problem from structure-from-motion dataset with $15$ frames. In this case, the matrix size is $n=15$ and the numerical ranks are
    \[
        \rank{W_1} = 903, \qquad \rank{W_2} = 76, \qquad \rank{\mymat{W_1 & W_2}} = 979, 
    \]
    which satisfies the condition \eqref{eq:exp-rank}, and thus primal nondegeneracy holds. Dual nondegeneracy is verified similarly. In this nondegenerate case, we see from \cref{thm:conv-nd} that the sequence $\normF{\Hk}$ converges linearly with rate $\normop{\Madmm} = 0.998$, which matches quite well with the numerical rate $0.996$ from \cref{fig:exp:rate} (a).

    \item \textbf{ND fails and SC holds.}
    In \cref{fig:exp:rate} (b), we consider a toy BQP problem with 10 binary variables. So, the matrix size is $n=66$ and the numerical ranks are
    \[
        \rank{W_1} = 1826, \qquad \rank{W_2} = 2145, \qquad \rank{\mymat{W_1 & W_2}} = 2211, 
    \]
    which implies the failure of primal nondegeneracy. Similarly, dual nondegeneracy holds numerically. From \cref{lem:conv-nnd-blk}, the sequence $\normF{\HOk}$ converges R-linearly with rate $\normop{\Madmm - \pfm} = 0.984$. As expected, the numerical rate from \cref{fig:exp:rate} is also $0.984$, which suggests the tightness of our theory. More interestingly, the numerical rate of $\normF{\HOk}$ and that of $\normF{\Hk}$ are exactly the same in this example.
\end{enumerate}


\begin{figure}[tbp]
    \centering
    \begin{minipage}{\textwidth}
        \centering
        \begin{tabular}{cc}
            \begin{minipage}{0.43\textwidth}
                \centering
                \includegraphics[width=\columnwidth]{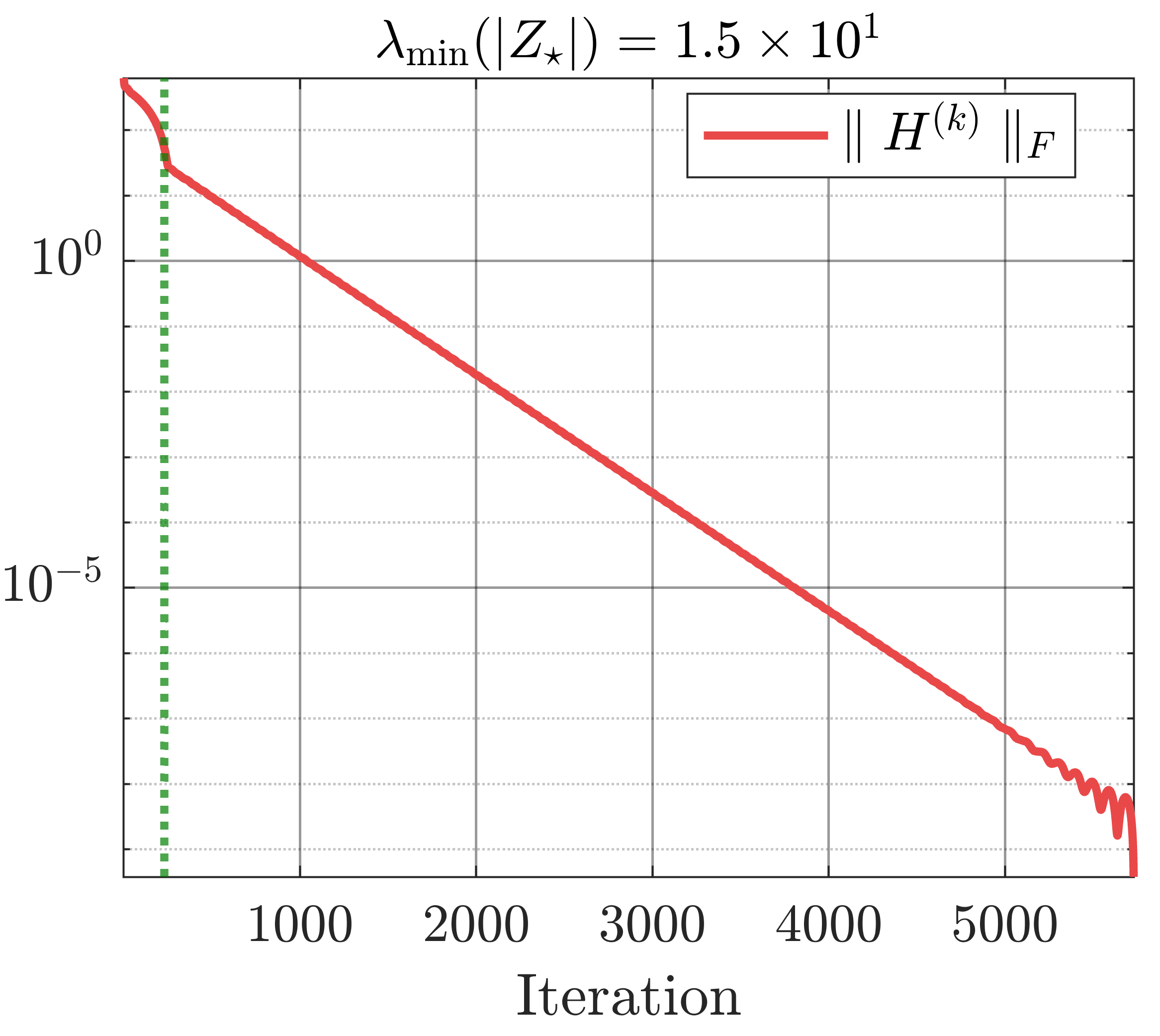}
                (a) ND holds and SC holds
            \end{minipage}

            \begin{minipage}{0.43\textwidth}
                \centering
                \includegraphics[width=\columnwidth]{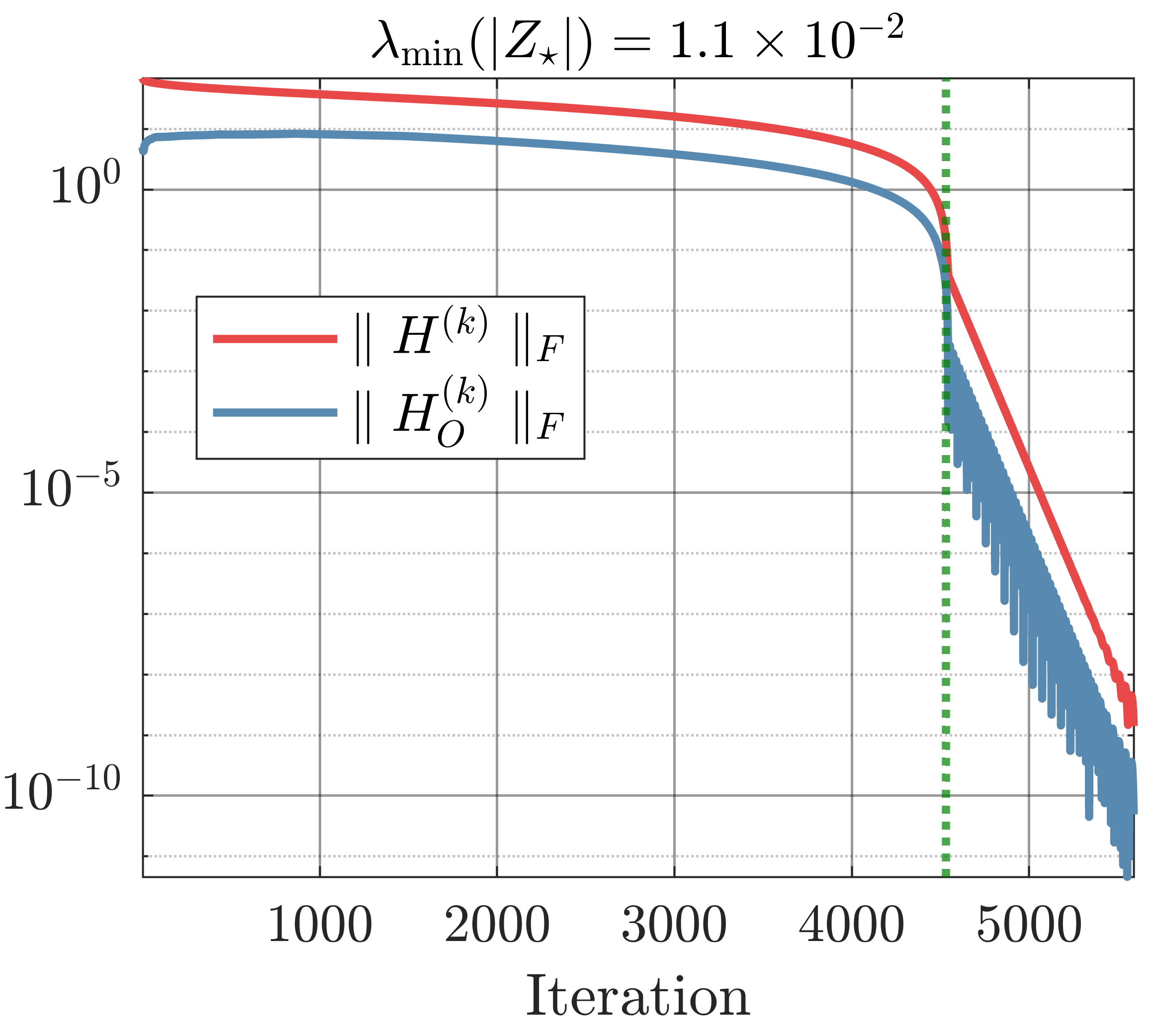}
                (b) ND fails and SC holds
            \end{minipage}
        \end{tabular}
    \end{minipage}
    \caption{Demonstration of numerical rates. (a) Plot of $\normF{\HOk}$ from a toy structure-from-motion problem. (b) Plot of both $\normF{\Hk}$ and $\normF{\HOk}$ from a toy BQP problem. In both cases, the numerical rates match quite well with the theory; see \cref{thm:conv-nd} and \cref{lem:conv-nnd-blk}.}
    \label{fig:exp:rate} 
\end{figure}

\subsection{\rebuttal{``Failure''} Cases}

In \cref{fig:failure}, we report some SDP instances for which ADMM fails to achieve $\rmax \leq 10^{-10}$ within the stated budget and no clear linear convergence is observed. A common feature in these instances is that the values $\lammin{\abs{\Zs}}$ tend to be small (\eg $10^{-4} \sim 10^{-9}$), yet not exactly zero (compared to QS and BQP cases where $\lammin{\abs{\Zs}} < 10^{-14}$). 
\rebuttal{
There are two possible explanations for the observed slow convergence: (\romannumeral1) ADMM has already entered the linear-convergence regime, but the rate is close to $1$, so the method essentially stalls; or (\romannumeral2) the local linear-convergence regime of ADMM, if it exists, has not yet emerged. We conjecture that much of the observed slow convergence falls into case (\romannumeral2). Indeed, \cref{thm:eb-intro-thm} shows that the refined error bound holds only for sufficiently small perturbations satisfying $\|\H\|_2 \leq C_{\mathrm{EB}}$, where $C_{\mathrm{EB}}$ is proportional to $\min\{\lam{r}, -\lam{r+1}\}$; see \eqref{eq:error-bound:Cf}. Hence, when $\lammin{\abs{\Zs}} = \min\{\lam{r}, -\lam{r+1}\}$ is small, the admissible perturbation radius $C_{\mathrm{EB}}$ is also small, so any eventual linear convergence of ADMM may appear only at a very late stage. This interpretation is also consistent with recent observations for first-order methods for LP~\cite{lu2024mp-geometry-pdhg-lp}. At present, however, we do not have a rigorous theoretical proof of this conjecture.
}

\begin{figure}[htbp!]
    \centering

    \begin{minipage}{\textwidth}
        \centering
        \hspace{5mm} \includegraphics[width=0.35\columnwidth]{figs/legends/legend_rmax_dZ.png}
    \end{minipage}

    \begin{minipage}{\textwidth}
        \centering
        \begin{tabular}{ccc}
            \begin{minipage}{0.30\textwidth}
                \centering
                \includegraphics[width=\columnwidth]{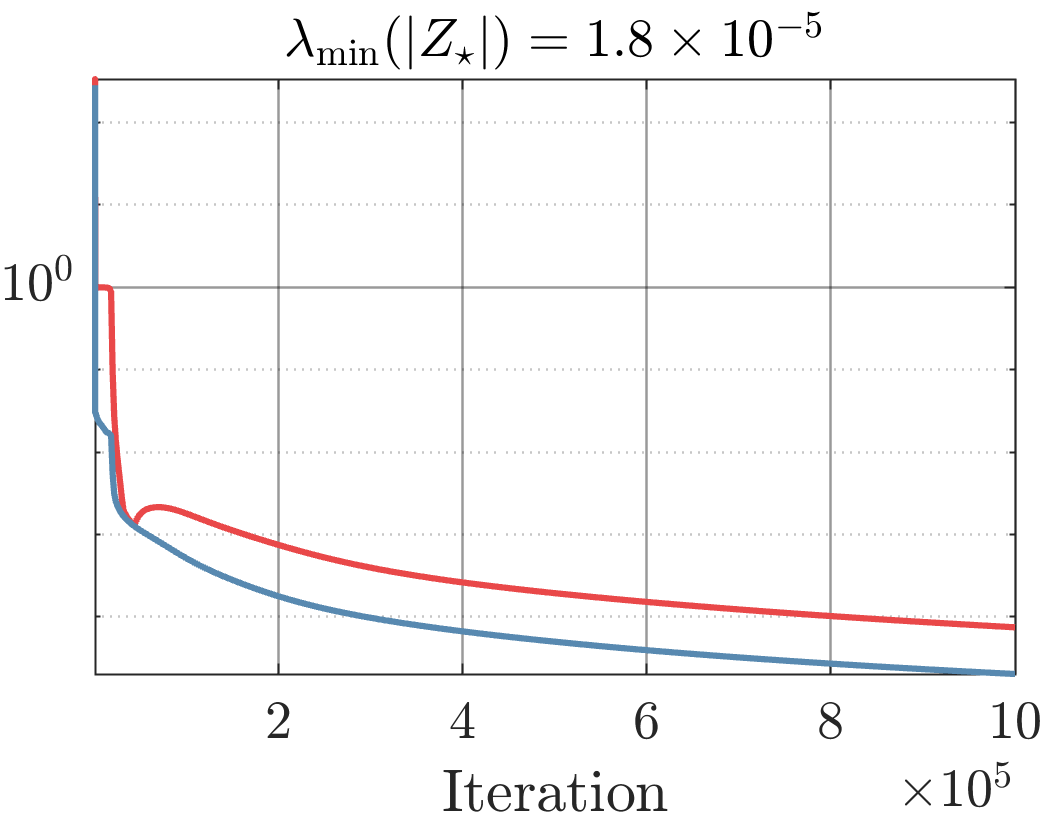}
                \texttt{1dc-1024}
            \end{minipage}

            \begin{minipage}{0.30\textwidth}
                \centering
                \includegraphics[width=\columnwidth]{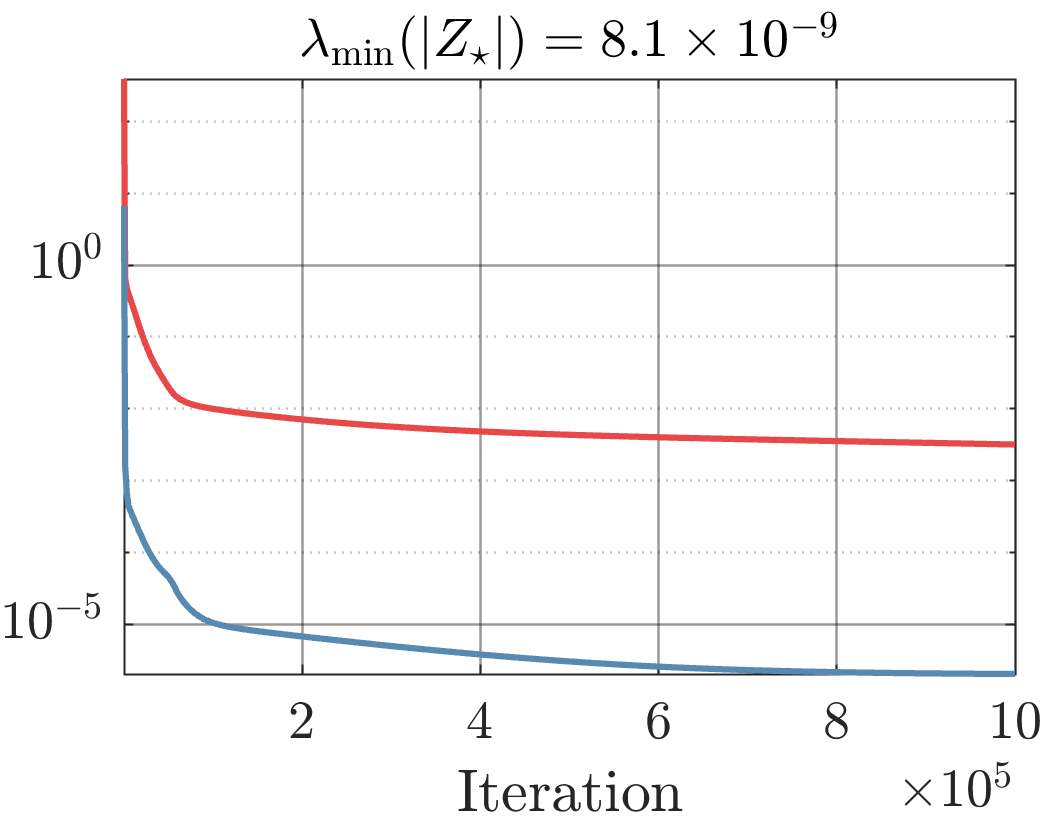}
                \texttt{cnhil10}
            \end{minipage}

            \begin{minipage}{0.30\textwidth}
                \centering
                \includegraphics[width=\columnwidth]{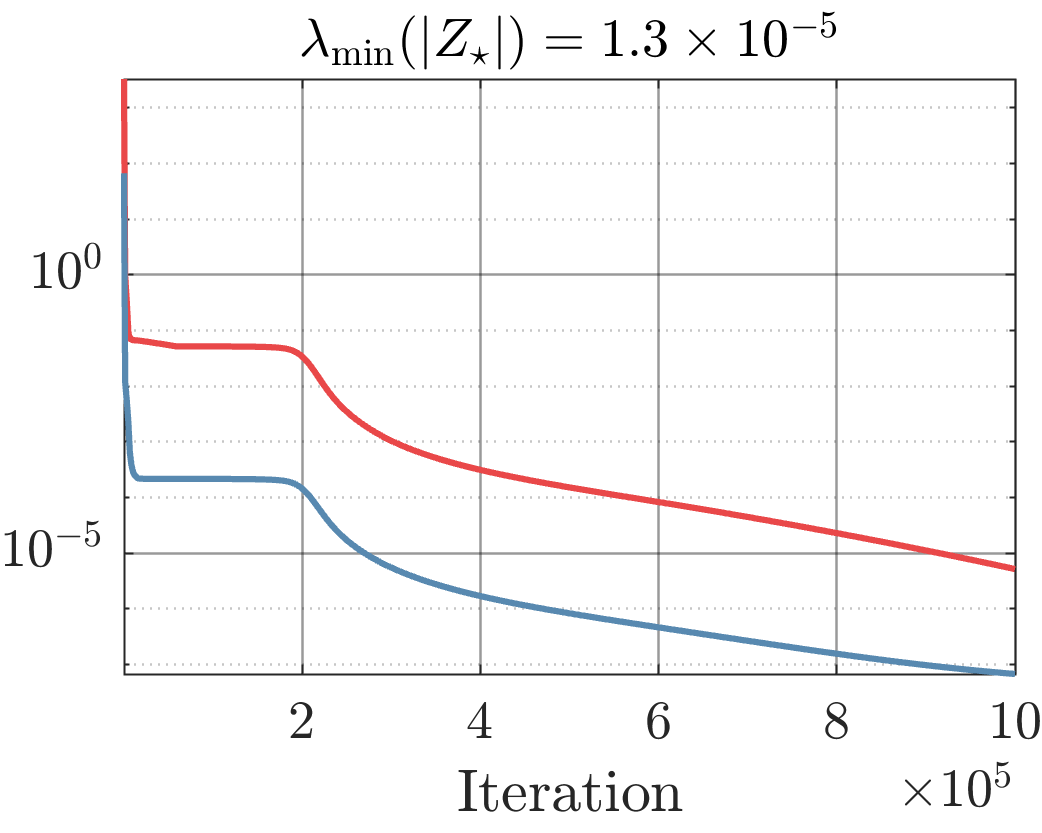}
                \texttt{neosfbr25}
            \end{minipage}
        \end{tabular}
    \vspace{2pt}
    \end{minipage}
    \begin{minipage}{\textwidth}
        \centering
        \begin{tabular}{ccc}
            \begin{minipage}{0.30\textwidth}
                \centering
                \includegraphics[width=\columnwidth]{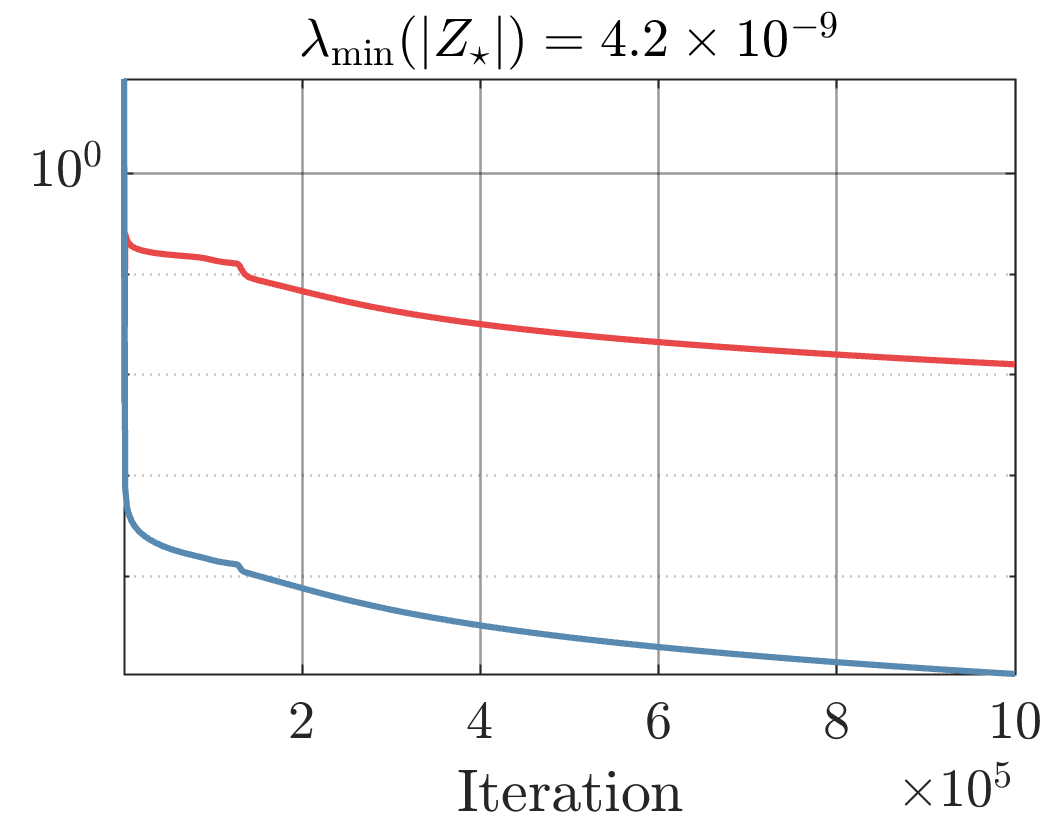}
                \texttt{rose13}
            \end{minipage}

            \begin{minipage}{0.30\textwidth}
                \centering
                \includegraphics[width=\columnwidth]{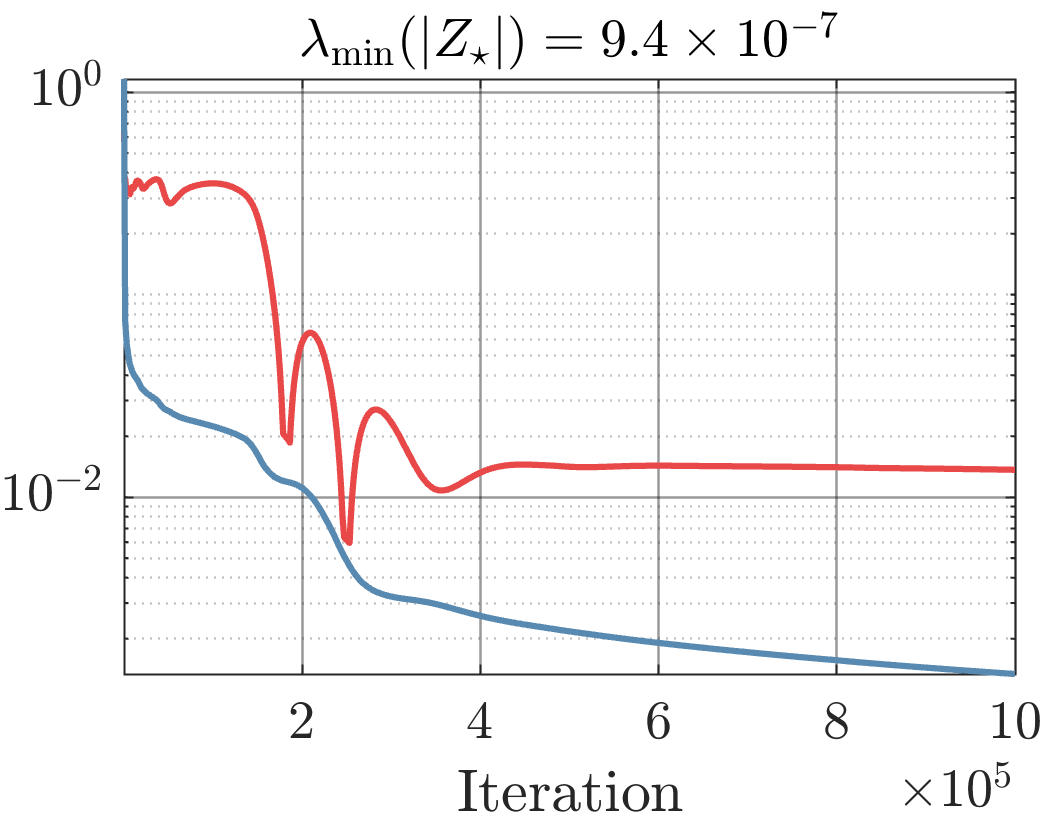}
                \texttt{swissroll}
            \end{minipage}

            \begin{minipage}{0.30\textwidth}
                \centering
                \includegraphics[width=\columnwidth]{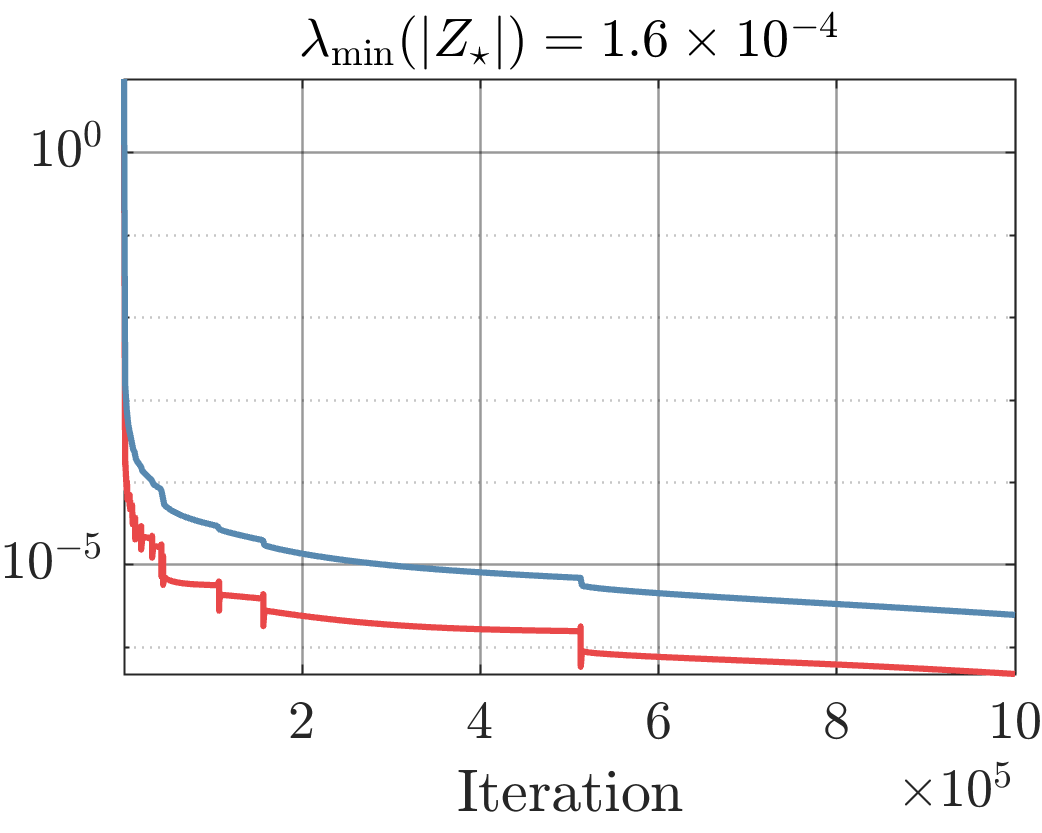}
                \texttt{MAXCUT-G11}
            \end{minipage}
        \end{tabular}
    \end{minipage}
    \caption{Cases from datasets~\cite{mittelmann2006dataset-sparse-sdp-problems,davis2011acm-florida-sparse-matrix-collection}. For these SDPs, ADMM fails to achieve $\rmax \leq 10^{-10}$ within the stated budget, and no clear linear convergence is observed.}
    \label{fig:failure}
\end{figure}


\section{Discussion: Rank Identification and Linear Convergence}
\label{sec:rank-id}

First-order methods (\eg PDHG) for LP (a special case of SDP) are known to have an intriguing two-stage phenomenon \cite{lu2024mp-geometry-pdhg-lp}. The first stage identifies the basis and finishes in a finite number of iterations, with a sublinear rate. Then, the second stage of the algorithm converges linearly, with a rate related to the local sharpness constant. In view of the equivalence between ADMM and PDHG \cite{oconnor2020equivalence}, as well as the similarity between the numerical results in \cref{sec:exp} and those in \cite{lu2024mp-geometry-pdhg-lp}, it is natural to ask whether ADMM for SDP has a similar two-stage performance and whether it could identify the solution \textit{rank} (\textit{c.f.}, basis in LP) within a finite number of iterations. This section aims to provide a partial answer to the above questions, both theoretically and empirically.

\paragraph{Finite-time rank identification.}
In the context of ADMM for SDP, the fact that rank identification occurs within a finite number of iterations is readily guaranteed by the well-known partial smoothness theory \cite{wright1993identifiable,lewis2002siopt-activesets-nonsmoothness-sensitivity,drusvyatskiy2014optimality}. More precisely, rank identification means that ADMM identifies the rank of the solution it converges to and all the subsequent ADMM iterates have the same rank. Here, we provide a more direct proof in the context of SDP, without invoking the more general partial smoothness theory.
\begin{proposition} \label{prop:rank-id}
    Suppose that \cref{ass:lin-sol,ass:lin-sc} hold and that ADMM \eqref{eq:intro:admm-three-step} converges to $(\Xs,\ys,\Ss)$. Then, there exists $\bar k_{\mathrm{ID}} \in \bbN$ such that for any integer $k \geq \bar k_{\mathrm{ID}}$, it holds that
    \[
        \rank{\Xk} = \rank{\Xs}, \qquad \rank{\Sk} = \rank{\Ss}.
    \]
\end{proposition}
\begin{proof}
    First, we show that
    \[
        \rank{\PiSnp{\Zs + \Hk}} = \rank{\PiSnp{\Zs}} \quad \text{if} \ \normtwo{\Hk} < \min\{\lam{r}, -\lam{r+1}\}. 
    \]
    To see this, denote by $\gamma_r$ and $\gamma_{r+1}$ the $r$th and $(r+1)$st largest eigenvalue of $\Zs + \Hk$, respectively. Then, by Weyl's inequality, we have
    \begin{align*}
        \gamma_r &\ge \lam{r} - \normtwo{\Hk} > \lam{r} - \min\{\lam{r}, -\lam{r+1}\} \geq 0, \\
        \gamma_{r+1} &\le \lam{r+1} + \normtwo{\Hk} < \lam{r+1} + \min\{\lam{r}, -\lam{r+1}\} \le 0,
    \end{align*}
    where recall $\lam{r}$ and $\lam{r+1}$ are the $r$th and $(r+1)$st largest eigenvalue of $\Zs$, respectively. Thus, we have $\gamma_r > 0 > \gamma_{r+1}$ and
    \[
        \rank{\Xk} = \rank{\PiSnp{\Zs + \Hk}} = r = \rank{\PiSnp{\Zs}}.  
    \]
    The dual part follows in a symmetric manner:
    \[
        \rank{\Sk} = \rank{\PiSnp{-\Zs - \Hk}} = n-r = \rank{\PiSnp{-\Zs}}.  
    \]

    Second, since $\normtwo{\Hk} \to 0$ as $k \to \infty$, there exists $\bar k_{\mathrm{ID}} \in \bbN$ such that $\normtwo{\Hk} < \min\{\lam{r}, \lam{r+1}\}$.  This concludes the proof.
\end{proof}

\paragraph{On the relation between rank identification and linear convergence.}
Considering both rank identification and local linear convergence, it is natural to investigate the relationship of these two phenomena: which one occurs first? Unlike the case of PDHG for LP, it remains unclear whether rank identification is the trigger for linear convergence.

\rebuttal{
Here, we provide a simple example to discuss the sutble role of rank identification in our proof procedure.
}
Recall from our analysis (specifically \cref{lem:conv-nnd-T}) that the R-linear convergence of $\normF{\Hk}$ is built upon that of the two sequences
\begin{equation} \label{eq:rank-id:minimal-face}
    \normF{\Proj_{\calT_\Ss} (\Xk)} = \calO (\normF{\HOk}), \qquad \normF{\Proj_{\calT_\Xs} (\Sk)} = \calO (\normF{\HOk}).
\end{equation}
In view of this, we build an SDP instance in which rank identification does not occur and \eqref{eq:rank-id:minimal-face} fails to hold. So, in the worst case, \eqref{eq:rank-id:minimal-face} needs rank identification. \rebuttal{
We conjecture that, if \eqref{eq:rank-id:minimal-face} were also necessary for linear convergence (which we are currently unable to prove), then rank identification would occur no later than the onset of the final (R-)linear convergence regime.
}
\begin{example} \label{ex:rank-id:example}
    Consider the SDP \eqref{eq:intro-sdp} with $n=3$. Suppose \cref{ass:lin-sol}, primal nondegeneracy \eqref{eq:intro:primal-nondegeneracy} and dual nondegeneracy \eqref{eq:intro:dual-nondegeneracy} hold. Suppose $\rank{\Xs} = 1$ and $\rank{\Ss} = 2$. (So \cref{ass:lin-sc} holds.) Suppose $\Zs = \diag{1,-\delta,-\delta}$, where $\delta > 0$ can be arbitrarily small. Then, \cref{prop:rank-id} implies that rank identification must occur if $\normtwo{\Hk} < \delta$.

    Assume, without loss of generality, that ADMM starts at the following points (with $\epsilon > 0$)
    \begin{align*}
        X^{(0)} & = \mymat{
            1 & 0 & 0 \\
            0 & \frac{\epsilon}{2} & \frac{\epsilon}{2} \\
            0 & \frac{\epsilon}{2} & \frac{\epsilon}{2}
        } = \mymat{
            1 & 0 \\
            0 & \sqrt{\frac{\epsilon}{2}} \\
            0 & \sqrt{\frac{\epsilon}{2}}
        }\mymat{
            1 & 0 \\
            0 & \sqrt{\frac{\epsilon}{2}} \\
            0 & \sqrt{\frac{\epsilon}{2}}
        }\tran, \\
        \sigma S^{(0)} & = \mymat{
            0 & 0 & 0 \\
            0 & \delta + \frac{\epsilon}{2} & -\delta - \frac{\epsilon}{2} \\
            0 & -\delta - \frac{\epsilon}{2} & \delta + \frac{\epsilon}{2}
        } = \mymat{
            0 \\ 
            \sqrt{\delta + \frac{\epsilon}{2}} \\
            -\sqrt{\delta + \frac{\epsilon}{2}}
        } \mymat{
            0 \\ 
            \sqrt{\delta + \frac{\epsilon}{2}} \\
            -\sqrt{\delta + \frac{\epsilon}{2}}
        }\tran.
    \end{align*}
    It is clear that $\rank{X^{(0)}} = 2$, $\rank{S^{(0)}} = 1$, $\langle {X^{(0)}}, {S^{(0)}} \rangle = 0$, and
    \[
        H^{(0)} = X^{(0)} - \sigma S^{(0)} - \Zs = \mymat{
            0 & 0 & 0 \\ 
            0 & 0 & \delta + \epsilon \\
            0 & \delta + \epsilon & 0
        }.
    \]
    Moreover, $\normtwo{H^{(0)}} = \delta + \epsilon$ and $\HO^{(0)} = 0$. On the other hand,
    \[
        \Proj_{\calT_\Ss} (X^{(0)}) = \mymat{
            0 & 0 & 0 \\
            0 & \frac{\epsilon}{2} & \frac{\epsilon}{2} \\
            0 & \frac{\epsilon}{2} & \frac{\epsilon}{2}
        }.
    \]
    To conclude, rank identification does not occur and \eqref{eq:rank-id:minimal-face} fails to hold.
\end{example}

\paragraph{Numerical evidence.}
As shown in \cref{fig:rank}, for many tested SDP instances, as soon as $\Xk$ identifies the solution rank, the ADMM iterates simultaneously steps into the final region of linear convergence.

\begin{figure}[tbp]
    \centering

    \begin{minipage}{\textwidth}
        \centering
        \hspace{5mm} \includegraphics[width=0.35\columnwidth]{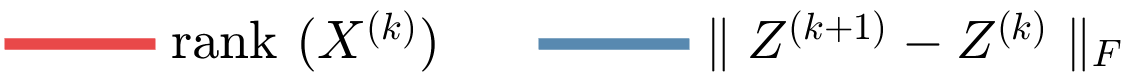}
    \end{minipage}

    \begin{minipage}{\textwidth}
        \centering
        \begin{tabular}{ccc}
            \begin{minipage}{0.30\textwidth}
                \centering
                \includegraphics[width=\columnwidth]{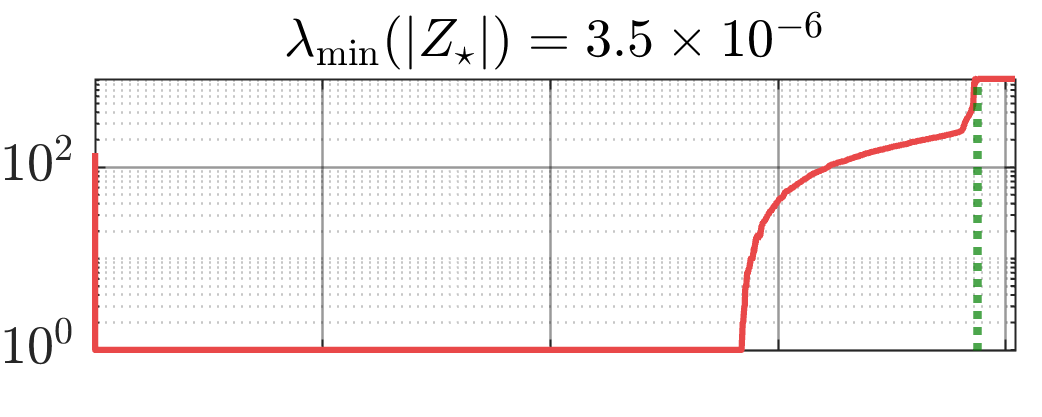}
                \includegraphics[width=\columnwidth]{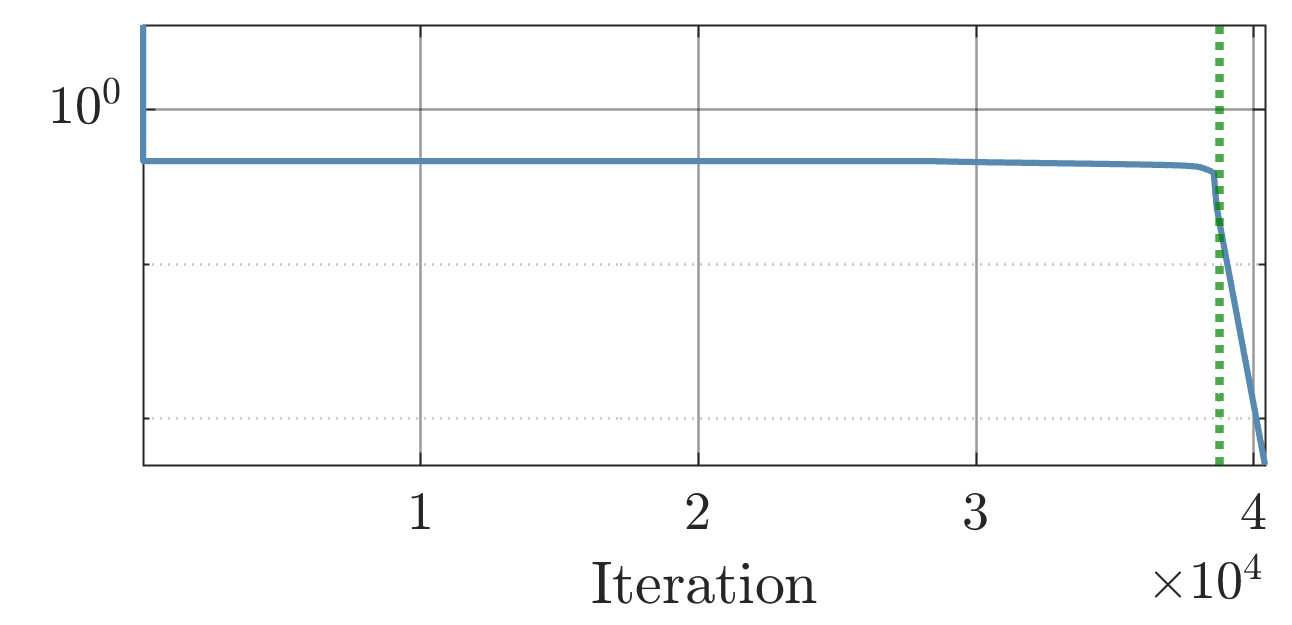}
                \texttt{hamming-11-2}
            \end{minipage}

            \begin{minipage}{0.30\textwidth}
                \centering
                \includegraphics[width=\columnwidth]{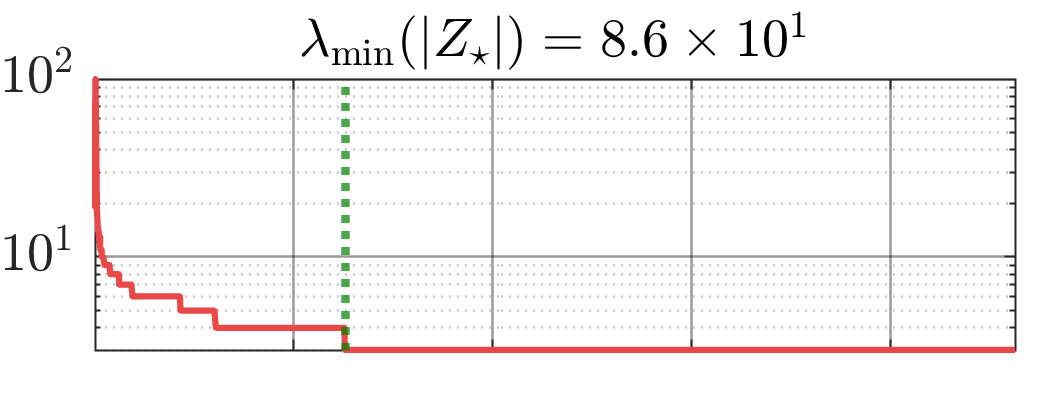}
                \includegraphics[width=\columnwidth]{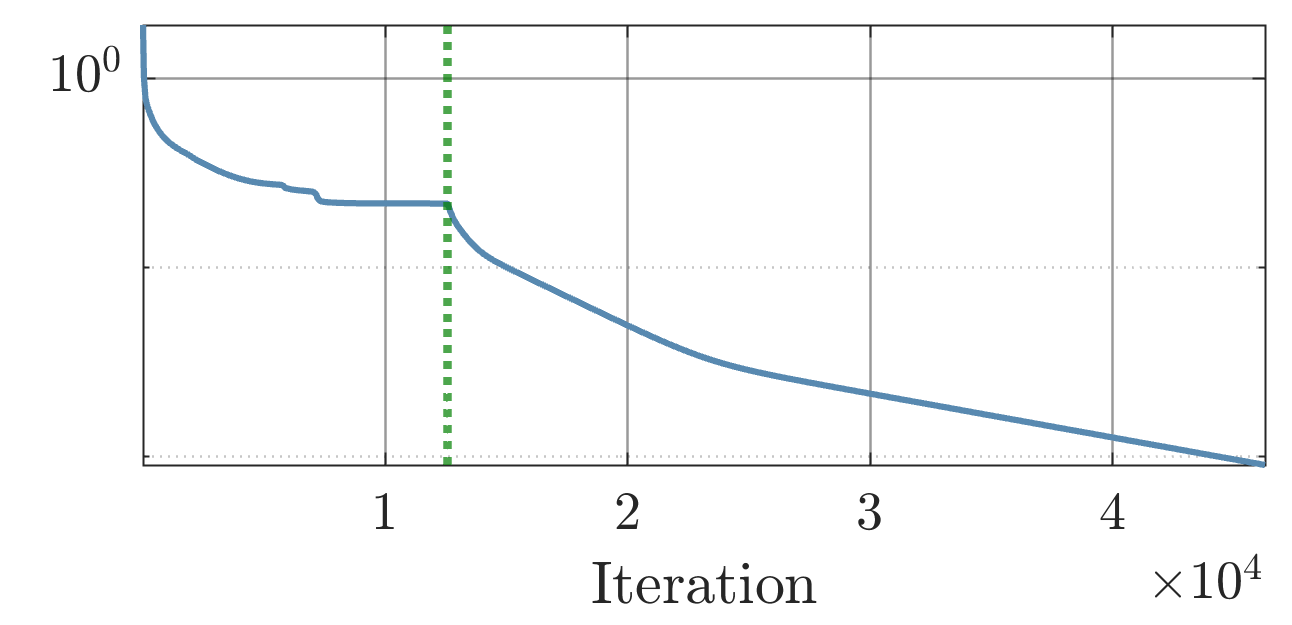}
                \texttt{XM-149}
            \end{minipage}

            \begin{minipage}{0.30\textwidth}
                \centering
                \includegraphics[width=\columnwidth]{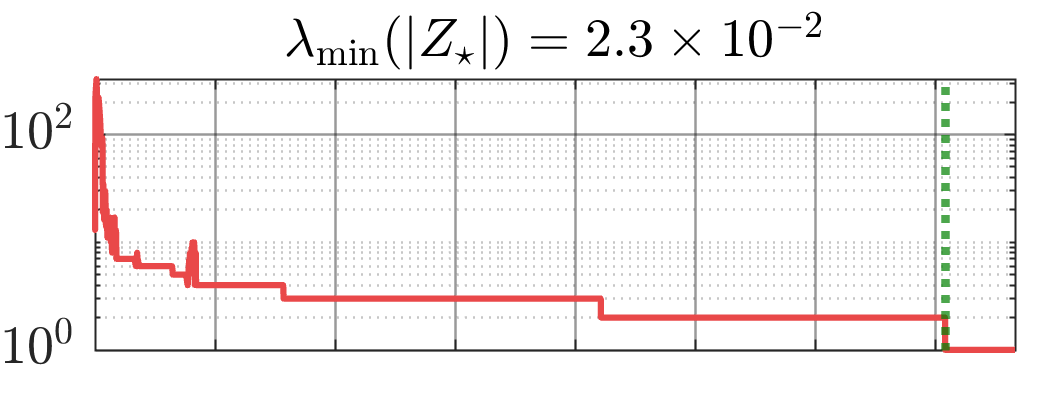}
                \includegraphics[width=\columnwidth]{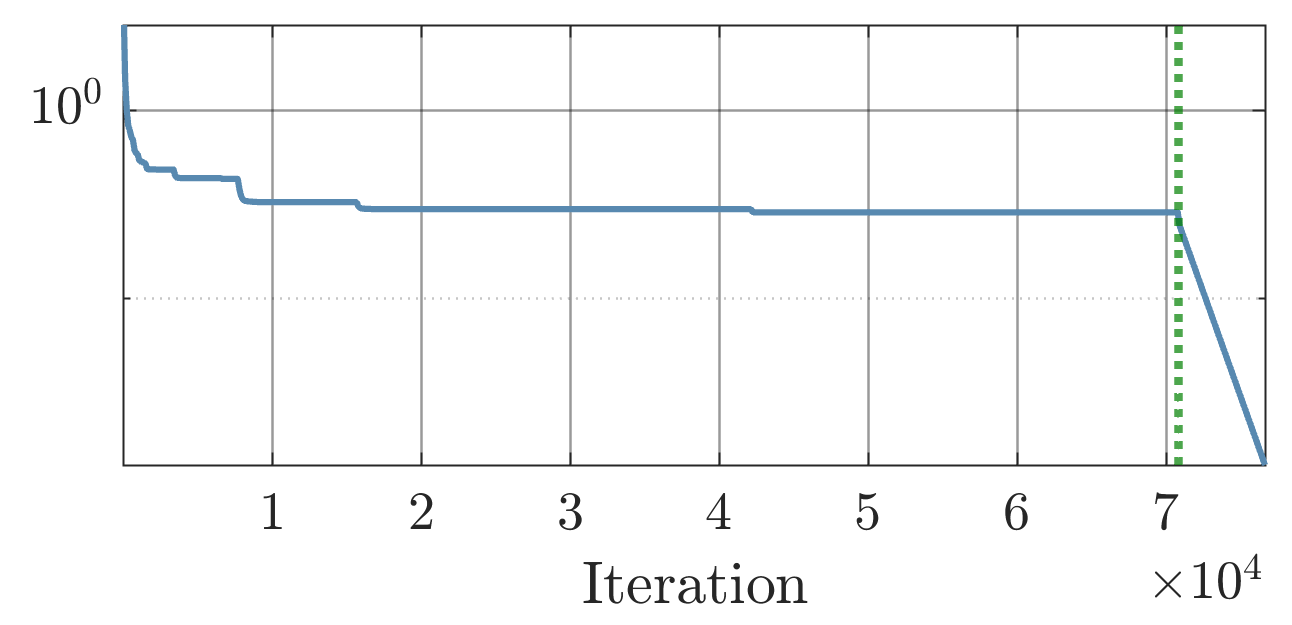}
                \texttt{BQP-r1-30-3}
            \end{minipage}
        \end{tabular}
    \end{minipage}

    \begin{minipage}{\textwidth}
        \centering
        \begin{tabular}{ccc}
            \begin{minipage}{0.30\textwidth}
                \centering
                \includegraphics[width=\columnwidth]{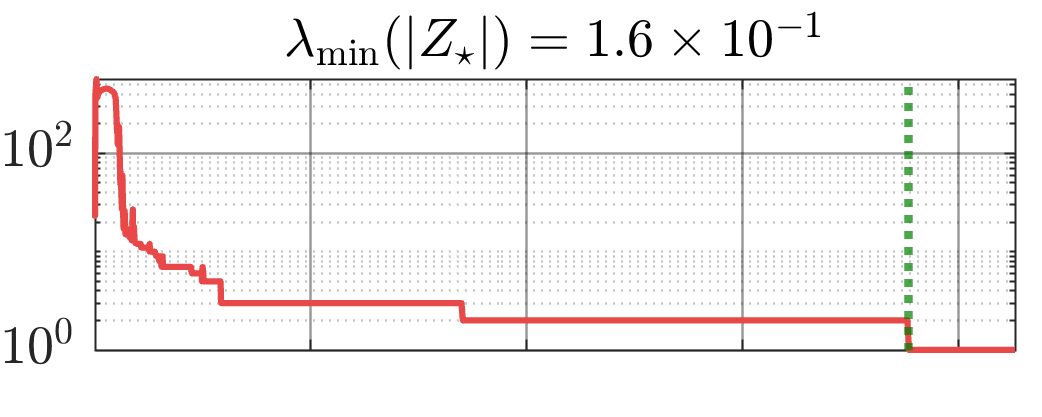}
                \includegraphics[width=\columnwidth]{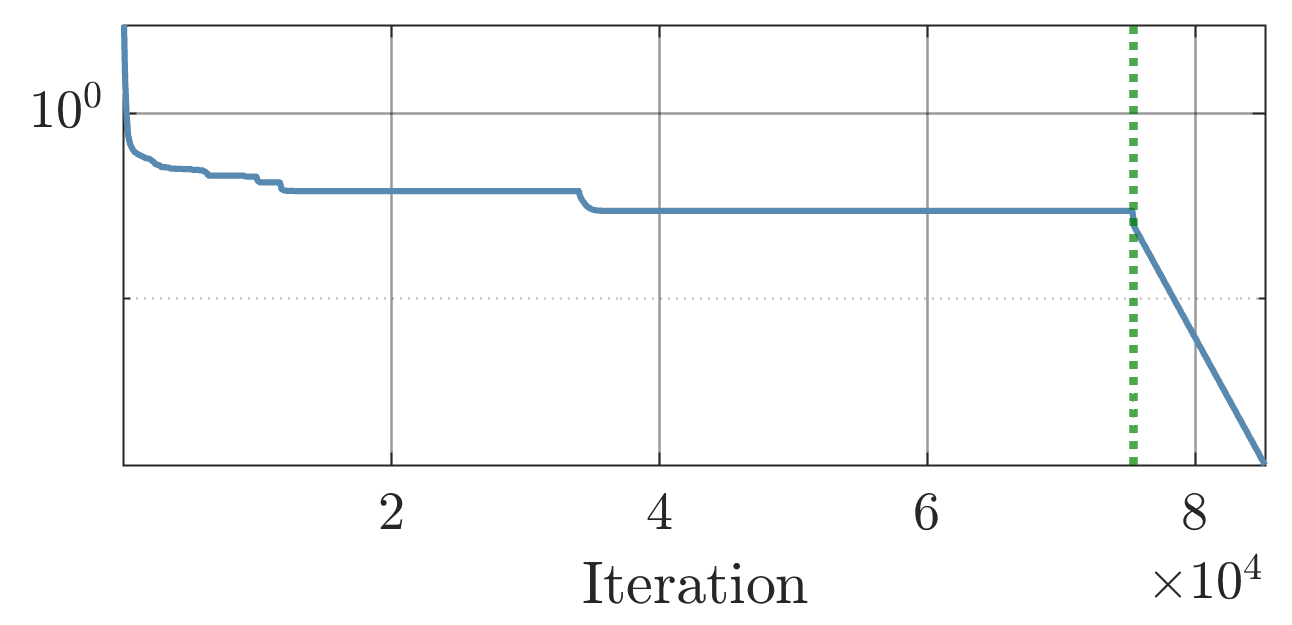}
                \texttt{BQP-r1-40-3}
            \end{minipage}

            \begin{minipage}{0.30\textwidth}
                \centering
                \includegraphics[width=\columnwidth]{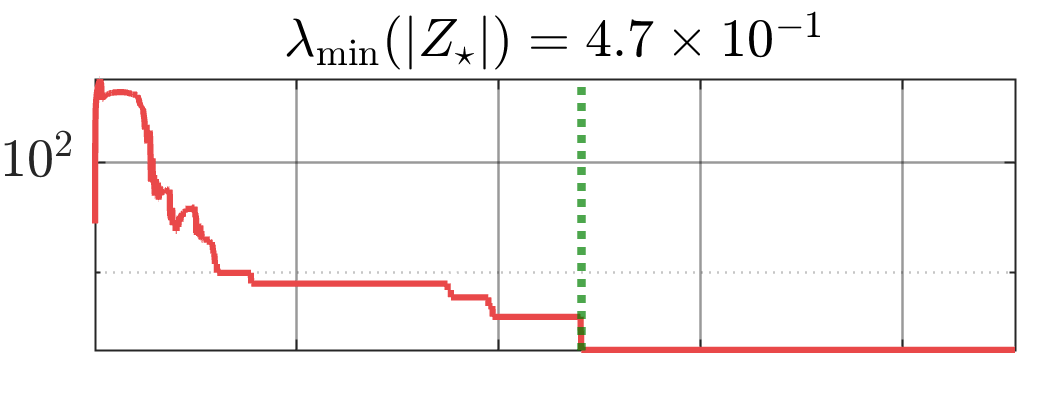}
                \includegraphics[width=\columnwidth]{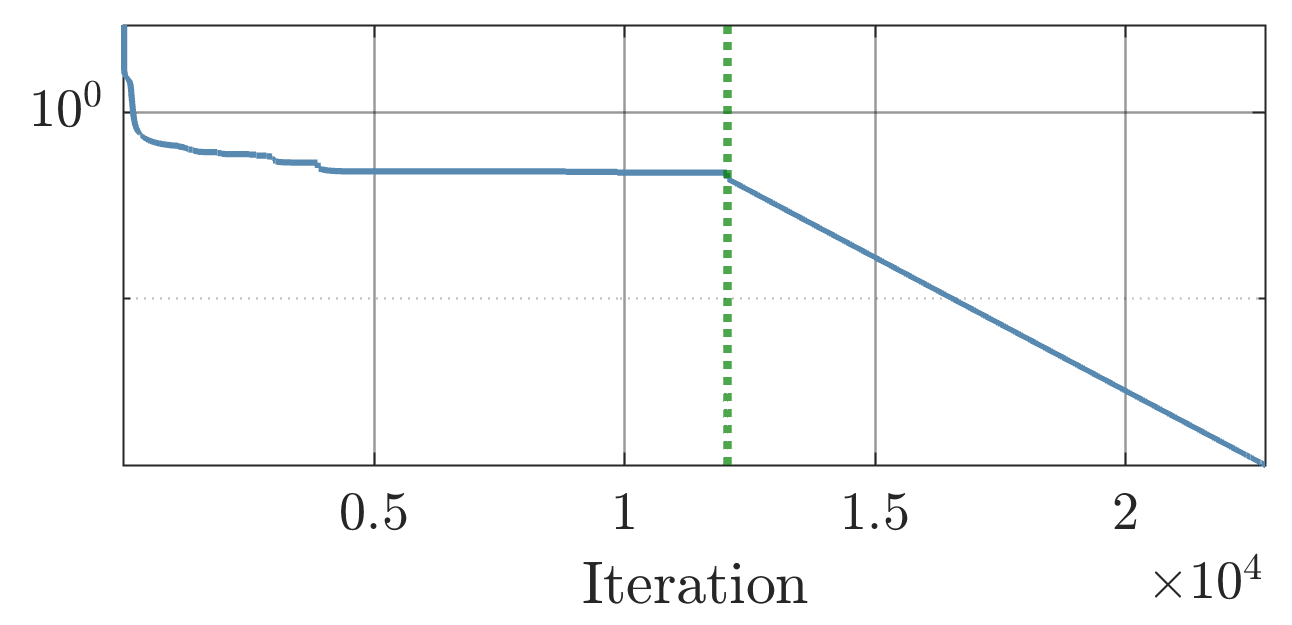}
                \texttt{BQP-r2-40-2}
            \end{minipage}

            \begin{minipage}{0.30\textwidth}
                \centering
                \includegraphics[width=\columnwidth]{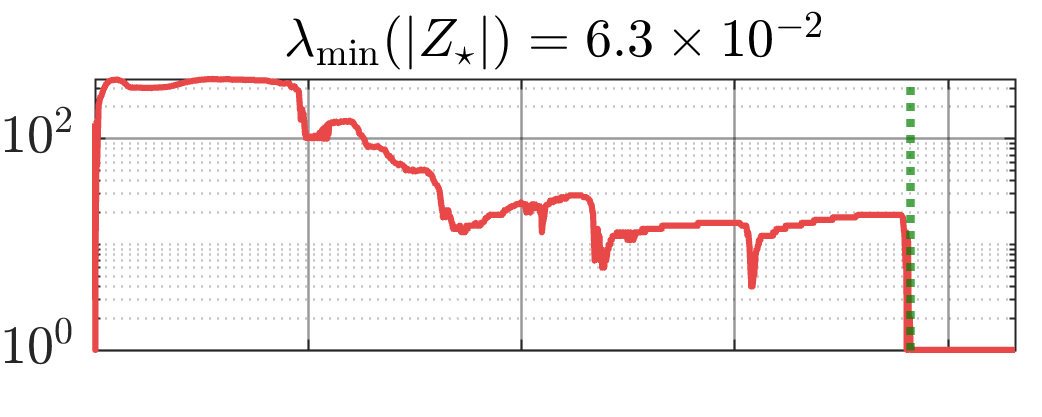}
                \includegraphics[width=\columnwidth]{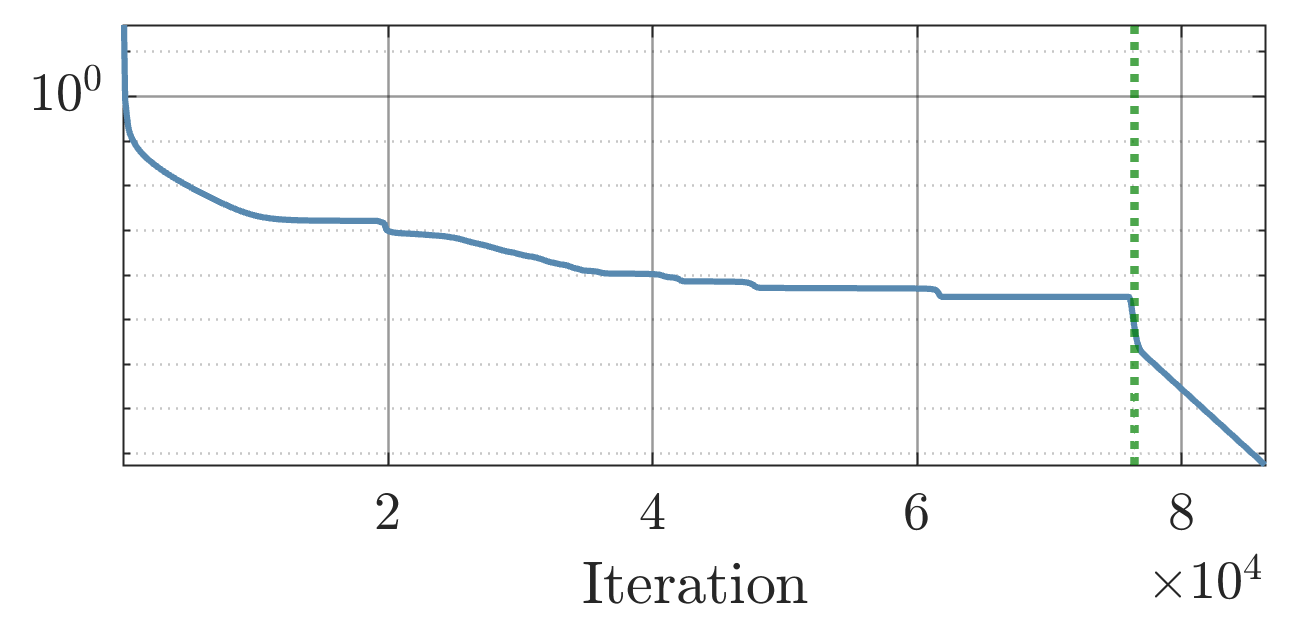}
                \texttt{Quasar-200}
            \end{minipage}
        \end{tabular}
    \end{minipage}

    \caption{Six representative SDP instances illustrating rank identification: almost at the same time when~$\Xk$ identifies the solution rank, ADMM steps into the final linear convergence region.}
    \label{fig:rank}
\end{figure}

In view of \cref{prop:rank-id}, \cref{ex:rank-id:example} and \cref{fig:rank}, one may already identify a gap between theory and practice.
\begin{quotation}
    \noindent
    \textbf{Open problems}: In what type of SDPs is rank identification a necessary condition for (R-)linear convergence? Under which conditions will rank identification and (R-)linear convergence occur simultaneously?
\end{quotation}


\rebuttal{

\section{Discussion: Connections to Metric Subregularity and Local Error Bounds}
\label{sec:metric}

Our proof framework is closely related to several notions of metric subregularity and local error bounds studied in the literature~\cite{yuan2020discerning,liu2018partial,liang2017jota-local-convergence-admm,cui2016arxiv-superlinear-alm-sdp,han2018mor-linear-rate-admm,abbaszadehpeivasti25arxiv-convrate-drs}. Recall that a set-valued mapping $\phi: \calX \rightrightarrows \calW$ is metric subregular at $\bar{x} \in \calX$ for $\bar{w} \in \calW$ with modulus $\kappa$ if $(\bar{x}, \bar{w}) \in \mathrm{gph}(\phi)$ and $\exists \epsilon > 0$, s.t.
\begin{align*}
    \dist(x, \phi^{-1}(\bar{w})) \le \kappa \dist(\bar{w}, \phi(x)), \quad \forall x \in \mathbb{B}_\epsilon(\bar{x}),
\end{align*}
where $\mathbb{B}_\epsilon(\bar{x})$ denotes the ball of radius $\epsilon$ centered at $\bar{x}$.
For first-order splitting methods, $\phi$ is often taken to be a (variant of the) KKT solution mapping defined in~\eqref{eq:intro:kkt}. Let $\calK$ denote the set of KKT points. Then $\bar{x}$ is typically a KKT point $(\bar{X}, \bar{y}, \bar{S}) \in \calK$, with $\bar{w} = 0$. In this case, $\dist(x, \phi^{-1}(0))$ is the distance from $(X, y, S)$ to $\calK$ (the forward error~\cite{sturm2000siopt-error-bounds-linear-matrix-inequalities}), while $\dist(0, \phi(x))$ corresponds to a (variant of the) KKT residual (the backward error~\cite{sturm2000siopt-error-bounds-linear-matrix-inequalities}). Thus, metric subregularity at a KKT point for $0$ naturally yields a local error bound (EB).

\vspace{2pt}
\noindent
\textbf{Variants of local error bounds.}~\cite{liu2018partial} introduces four types of error bounds, which we now specialize to the ADMM-for-SDP setting.
The first three are equivalent~\cite[Appendix A]{liu2018partial}: full error bound (FEB)~\cite[Eq. (1.7)]{liu2018partial}, Proximal EB-I~\cite[Eq. (1.12)]{liu2018partial}, and Proximal EB-II~\cite[Eq. (1.14)]{liu2018partial}. We only state Proximal EB-I, which also coincides with the metric subregularity of the KKT mapping $\calR$ in~\cite[Eq. (7)]{han2018mor-linear-rate-admm}. It is derived from the projection form of the KKT mapping:
\begin{align}
    \label{eq:metric:proximal-eb}
    & \dist((X, y, S), \calK) \le \kappa (\normtwo{\Asdp X - b} + \normF{\AsdpT y + S - C} + \normF{S - \Pi_{\psd{n}}(S - X)} ), \\
    & \forall (X, y, S) \in \mathbb{B}_\epsilon((\bar{X}, \bar{y}, \bar{S})). \nonumber
\end{align}
The fourth is the partial error bound (PEB)~\cite[Definition 5.2]{liu2018partial}:
\begin{align}
    \label{eq:metric:peb}
    & \dist((X, y, S), \calK) \le \kappa (\normtwo{\Asdp X - b} + \normF{\AsdpT y + S - C}  ), \\
    & \forall (X, y, S) \in \mathbb{B}_\epsilon((\bar{X}, \bar{y}, \bar{S})), \ X \succeq 0, S \succeq 0, \inprod{X}{S} = 0. \nonumber
\end{align}
Note that the term $\normF{S - \Pi_{\psd{n}}(S - X)}$ in~\eqref{eq:metric:proximal-eb} is always $0$ whenever $X \succeq 0, S \succeq 0, \inprod{X}{S} = 0$, a condition naturally satisfied by ADMM for SDP~\eqref{eq:intro:admm-three-step}. Therefore, PEB is weaker than Proximal EB-I. Finally, we state an error bound for DRS~\cite{abbaszadehpeivasti25arxiv-convrate-drs,van25arxiv-linconv-eb-admm}, or equivalently, for one-step ADMM~\eqref{eq:intro:admm-one-step}:
\begin{align}
    \label{eq:metric:drs-eb}
    \dist(Z, \Zopt) \le \kappa \normF{
        \underbrace{-\PA (\Pi_{\psd{n}}(Z) - \widetilde{X}) + \PAp ( \Pi_{\psd{n}}(-Z) - \sigma C)}_{=: \delta(Z)}
    }, \ \forall Z \in \mathbb{B}_\epsilon(\bar{Z}),
\end{align}
where $\bar{Z} \in \Zopt$.
We abbreviate~\eqref{eq:metric:drs-eb} as DRS-EB. The update in~\eqref{eq:intro:admm-one-step} can be written equivalently as $\Zkpo - \Zk = \delta(\Zk)$. DRS-EB gives a necessary and sufficient condition for the local Q-linear convergence of $\dist(\Zk, \Zopt)$~\cite[Theorem 3 and Proposition 1]{abbaszadehpeivasti25arxiv-convrate-drs}. 

\begin{proposition}[PEB $\Rightarrow$ DRS-EB]
    \label{prop:metric:peb-imply-drs-eb}
    In the ADMM for SDP setting, PEB~\eqref{eq:metric:peb} implies DRS-EB~\eqref{eq:metric:drs-eb} (up to a different modulus). 
\end{proposition}
\begin{proof}
    Fix $\bar{Z} \in \Zopt$ and $Z \in \mathbb{B}_\epsilon(\bar{Z})$. Let 
    \begin{align*}
        & X := \Pi_{\psd{n}}(Z), \ S := -\frac{1}{\sigma}(\Pi_{\psd{n}}(-Z)), \ y := (\Asdp \AsdpT)^{-1}\Asdp (C-S), \\ 
        & \bar{X} := \Pi_{\psd{n}}(\bar{Z}), \ \bar{S} := -\frac{1}{\sigma}(\Pi_{\psd{n}}(-\bar{Z})), \ \bar{y} := (\Asdp \AsdpT)^{-1}\Asdp(C - \bar{S}).
    \end{align*} 
    Thus, $(\bar{X}, \bar{y}, \bar{S}) \in \calK$ and $(X, y, S)$ satisfies conditions in~\eqref{eq:metric:peb} (up to a different ball radius). Moreover, $\delta(Z) = -\PA (X - \widetilde{X}) + \sigma \PAp (S - C)$. On the other hand,
    \begin{align*}
        & \normtwo{\Asdp X - b} = \normtwo{\Asdp \PA (X - \widetilde{X})} \le \normop{\Asdp} \normF{\PA(X - \widetilde{X})} \le \normop{\Asdp} \normF{\delta(Z)}, \\
        & \normF{\AsdpT y + S - C} = \normF{\PAp (S - C)} \le \frac{1}{\sigma} \normF{\delta(Z)}.
    \end{align*}
    Thus, invoking PEB~\eqref{eq:metric:peb}, $
        \dist((X, y, S), \calK) \le \kappa (\normop{\Asdp} + \frac{1}{\sigma}) \normF{\delta(Z)}.
    $ Now, define a linear map $T: \Sn \times \Real{m} \times \Sn \mapsto \Sn$ as $T(X, y, S) = X - \sigma S$. Then, $Z = T(X, y, S)$, $\bar{Z} = T(\bar{X}, \bar{y}, \bar{S})$, and $T(\calK) = \Zopt$. Since $T$ is $\sqrt{1 + \sigma^2}$-Lipschitz, 
    \begin{align*}
        \dist(Z, \Zopt) = \dist(T(X, y, S), T(\calK)) \le \sqrt{1 + \sigma^2} \dist((X, y, S), \calK) \le \kappa' \normF{\delta(Z)},
    \end{align*}
    where $\kappa^\prime = \kappa \sqrt{1 + \sigma^2} (\normop{\Asdp} + \frac{1}{\sigma})$. 
\end{proof}

\vspace{2pt}
\noindent
\textbf{The role of strict complementarity.} Using the ingredients developed in the proof of the R-linear convergence of $\dist(\Zk, \Zopt)$, we now show that DRS-EB~\eqref{eq:metric:drs-eb} holds at a nonsingular point $\Zs \in \Zopt$.

\begin{theorem}[SC $\Rightarrow$ DRS-ER]
    \label{thm:metric:sc-imply-drs-er} 
    Given a KKT point $(\Xs, \ys, \Ss)$ satisfying Assumption~\ref{ass:lin-sc}. DRS-EB~\eqref{eq:metric:drs-eb} holds at $\Zs = \Xs - \sigma \Ss$. 
\end{theorem}
\begin{proof}
    Let $C_{\mathrm{EB}}$ be the constant radius defined in Theorem~\ref{thm:eb-intro-thm}. Define $\epsilon_0 = \frac{1}{\sqrt{n}} C_{\mathrm{EB}}$. For any $Z \in \mathbb{B}_{\epsilon_0}(\Zs) := \{Z \mid \normF{Z - \Zs}\le \epsilon_0\}$, denote $\H := Z - \Zs$ following the block partition in~\eqref{eq:eb-intro-partition}, $X := \Pi_{\psd{n}}(Z)$, and $S = -\frac{1}{\sigma} \Pi_{\psd{n}}(-Z)$. By dropping all superscripts $(k)$ in Lemma~\ref{lem:conv-nnd-T}'s proof procedure, we get
    \begin{align}
        \label{eq:metric:sc-proof-1}
        \normF{\Pi_{\calT_{\Ss}}(X)} + \sigma \normF{\Pi_{\calT_{\Xs}}(S)} \le \alpha_T \normF{\HO},
    \end{align}
    with some constant $\alpha_T > 0$.
    Similarly, for Lemma~\ref{lem:conv-nnd-Z}, we get 
    \begin{align}
        \label{eq:metric:sc-proof-2}
        \dist (Z, \Zopt) \leq \alpha_Z (\normF{\delta(Z)} + \normF{\Proj_{\calT_{\Ss}} (X)} + \sigma \normF{\Proj_{\calT_\Xs} (S)}),
    \end{align}
    with $\alpha_Z$ defined in Lemma~\ref{lem:conv-nnd-Z}. Combining~\eqref{eq:metric:sc-proof-1} and~\eqref{eq:metric:sc-proof-2}, all we need to proof is 
    \begin{align}
        \label{eq:metric:sc-proof-3}
        \normF{\HO} \le \alpha_O \normF{\delta(Z)},
    \end{align}
    with some constant $\alpha_O > 0$. To see this, we rewrite $\delta(Z)$ as the summation of first-order terms and residuals: 
    \begin{align*}
        \delta(Z) = \underbrace{-\PA (\Omega \circ \H) - \PAp (\Omega^\perp \circ H)}_{=: F(H)} + (\PAp - \PA) \Psi,
    \end{align*}
    where $\Omega$ is defined in~\eqref{eq:eb-intro-Omega}, $\Omega^\perp := E_n - \Omega$, $\Psi$ is defined in~\eqref{eq:lin-psik}. Expand $F(H)$ as:
    \begin{align*}
        F(H) = \underbrace{-\PA \mymat{\HX & 0 \\ 0 & 0} - \PAp \mymat{0 & 0 \\ 0 & \HS}}_{=: G_1 (\HX, \HS)} 
            \underbrace{-\PA \mymat{0 & \TheHOT \\ \TheHO & 0} - \PAp \mymat{0 & \ThepHOT \\ \ThepHO & 0}}_{=:G_2(\HO)},
    \end{align*}
    where $\Theta$ is defined in~\eqref{eq:eb-intro-Omega}.
    Define $f(\HO)$ as 
    \begin{align*}
        f(H_O) := \min_{\HX, \HS} \normF{F(H)}^2 = \min_{\HX, \HS} \normF{
            G_1(\HX, \HS) + G_2(\HO)
        }^2.
    \end{align*}
    Since $G_1: \Sn \times \Sn \mapsto \Sn$ is a linear map, $\calG := \{G_1(\HX, HS) \mid \HX \in \Sym{r}, \HS \in \Sym{n-r}\}$ is a subspace. Therefore, the minimum is attained and 
    \begin{align*}
        f(H_O) = \min_{Y \in \calG} \normF{Y + G_2(\HO)}^2 = \dist(-G_2(\HO), \calG)^2.
    \end{align*}
    On the other hand, $G_2: \Real{(n-r) \times r} \mapsto \Sn$ is also a linear map. Therefore, $f(H_O)$ is continuous, non-negative, and $2$-homogeneous. Denote $c$ as $\min_{\normF{\HO} = 1} f(\HO)$. Since $\normF{\HO} = 1$ is compact and $f(\HO)$ is continuous, the minimum is attained and $c$ is well defined. We shall see $c > 0$. If not, then there exists $\HX, \HS$ and $\HO$ with $\normF{\HO} = 1$, such that $\normF{F(H)}^2 = 0$. However, $\normF{F(H)}^2 = \normF{\PA (\Omega \circ \H)}^2 + \normF{\PAp (\Omega^\perp \circ H)}^2$, which further implies $\PA (\Omega \circ \H) = 0$ and $\PAp (\Omega^\perp \circ H) = 0$. Thus, $\inprod{\Omega \circ \H}{\Omega^\perp \circ H} = 0$. 
    On the other hand, $\inprod{\Omega \circ \H}{\Omega^\perp \circ H} = 2 \inprod{\TheHO}{\ThepHO}$ from~\eqref{eq:lin-inprod-prf}.
    Since each element of $\Theta$ lies in $(0, 1)$, this forces $\HO = 0$, which leads to a contradiction. In addition with $f$'s $2$-homogeneity, we get 
    \begin{align*}
        \normF{F(H)}^2 \ge f(\HO) \ge c \normF{\HO}^2, \quad \forall H \in \Sn. 
    \end{align*}
    Finally, by Remark~\ref{rem:psdproj-eb}, $\normF{\Psi} \le \alpha^\prime_{\mathrm{EB}} \normF{H} \normF{HO}$. Thus, 
    \begin{align}
        \label{eq:metric:sc-proof-4}
        \normF{\delta(Z)} \ge \normF{F(H)} - \normF{(\PAp - \PA) \Psi} \ge \sqrt{c} \normF{\HO} - \alpha^\prime_{\mathrm{EB}} \normF{H} \normF{\HO}. 
    \end{align}
    Now set $\epsilon := \min\{\epsilon_0, \frac{\sqrt{c}}{2 \alpha^\prime_{\mathrm{EB}}}\}$. As long as $Z \in \mathbb{B}_{\epsilon}(\Zs)$, $\normF{\delta(Z)} \ge \frac{\sqrt{c}}{2} \normF{\HO}$, which proves~\eqref{eq:metric:sc-proof-3}.
\end{proof}

The proof of Theorem~\ref{thm:metric:sc-imply-drs-er} again highlights the importance of Theorem~\ref{thm:eb-intro-thm}: without it, neither~\eqref{eq:metric:sc-proof-1} nor~\eqref{eq:metric:sc-proof-4} can be established. In summary, the relationships among the five local error bounds in ADMM for SDP and SC are as follows:
}
\[
    \begin{tikzcd}[column sep=3em]
    \mathrm{FEB}
    \arrow[r, Leftrightarrow, shorten <=2pt, shorten >=2pt]
    &
    \text{Proximal EB-I}
    \arrow[r, Leftrightarrow, shorten <=2pt, shorten >=2pt]
    &
    \text{Proximal EB-II}
    \arrow[r, Rightarrow, shorten <=2pt, shorten >=2pt]
    &
    \mathrm{PEB}
    \arrow[r, Rightarrow, shorten <=2pt, shorten >=2pt]
    &
    \mathrm{DRS\text{-}EB}
    \\
    &&&&
    \mathrm{SC} \arrow[u, Rightarrow, shorten <=2pt, shorten >=2pt]
    \end{tikzcd}
\]

\rebuttal{%
\vspace{2pt}
\noindent
\textbf{Connections to hidden polyhedrality.} In the literature~\cite{liang2017jota-local-convergence-admm,yuan2020discerning}, various forms of hidden polyhedrality have been exploited as special cases of metric subregularity and local error bounds. Specifically,~\cite{liang2017jota-local-convergence-admm} defines local polyhedrality under the partly smoothness setting. Their framework further requires strict complementarity at the KKT point, together with vanishing higher-order terms~\cite[Proposition 5.2]{liang2017jota-local-convergence-admm}, \ie $\Psi^{(k)} = 0$ in~\eqref{eq:lin-psik}. On the other hand,~\cite[Assumption 1.3]{yuan2020discerning} introduces structured polyhedricity, a model-structural variational assumption. This assumption requires one of ADMM's block functions to be piecewise linear-quadratic and the other to be locally strongly convex. Although these conditions have successfully established local linear convergence in piecewise linear-quadratic optimization settings, they do not apply directly to ADMM for SDP due to the non-polyhedrality of $\delta_{\psd{n}}(\cdot)$.

It is natural to ask whether further connections hold between the other error bounds and the SC condition:
\begin{quotation}
    \textbf{Open problems}: In ADMM for SDP, does the SC condition (Assumption~\ref{ass:lin-sc}) imply PEB~\eqref{eq:metric:peb}? Does the SC condition (Assumption~\ref{ass:lin-sc}) imply Proximal EB-I~\eqref{eq:metric:proximal-eb} (and hence FEB and Proximal EB-II)?
\end{quotation}

}


\section{Conclusion}
\label{sec:conclusion}

We established a new sufficient condition for the local linear convergence of the Alternating Direction Method of Multipliers (ADMM) in solving semidefinite programming (SDP) problems. Contrary to the conventional belief that ADMM is inherently slow for SDPs, we demonstrated that when the converged primal--dual optimal solutions satisfy strict complementarity, ADMM exhibits local linear convergence, regardless of nondegeneracy conditions. Our theoretical analysis is grounded in a direct local linearization of the ADMM operator and a refined error bound for the projection onto the positive semidefinite cone, revealing the anisotropic nature of projection residuals and improving previous bounds.

Extensive numerical experiments validated our theoretical findings, showing that ADMM achieves local linear convergence across a variety of SDP instances, including those where nondegeneracy fails. Furthermore, we identified cases where ADMM struggles to reach high accuracy, linking these difficulties to near violations of strict complementarity. This observation aligns with recent results in linear programming. 

Our numerical results also revealed intriguing connections between rank identification and linear convergence. While we provided a qualitative analysis, a complete understanding remains open. Future work could further investigate this relationship, examine whether linear convergence can occur in the absence of both nondegeneracy and strict complementarity, and develop new algorithms to accelerate ADMM and other first-order methods.

\subsection*{Acknowledgements}
We sincerely thank Ying Cui, Jingwei Liang, Ling Liang, Feng-Yi Liao, Haihao Lu, Defeng Sun, and Tianyun Tang for their valuable discussions. Xin Jiang extends special gratitude to Adrian S.\ Lewis for his insightful perspectives and unwavering support.

\clearpage

\appendix
\begin{appendices}


\section{Discussion on \texorpdfstring{\cite[Proposition 3.4]{cui2016arxiv-superlinear-alm-sdp}}{Prop}}
\label{app:sec:details-cui-prop}

In this section, we show that \cite[Proposition 3.4]{cui2016arxiv-superlinear-alm-sdp}, under an additional nonsingularity assumption, can be readily derived from \cref{thm:eb-intro-thm}. We first restate~\cite[Proposition 3.4]{cui2016arxiv-superlinear-alm-sdp} (with the nonsingularity assumption) below.
\begin{corollary}
    \label{app:cor:cui-prop}
    Let $Z \in \Sn$ being nonsingular and $X,S \in \Symp{n}$ be defined as
    \[
        Z := \mymat{\LamX & 0 \\ 0 & \LamS}, \qquad
        X := \mymat{\LamX & 0 \\ 0 & 0}, \qquad
        S := \mymat{0 & 0 \\ 0 & -\LamS}, 
    \]
    where, without loss of generality, $\LamX \in \R^{r \times r}_{++}$ and $\LamS \in \R^{s \times s}_{++}$ are diagonal matrices of the form in~\eqref{eq:intro:Xs-sigSs}, and $r+s=n$. For a sufficiently small perturbation $\H \in \Sn$, define
    \[
        X^\prime := \PiSnp{Z+H}, \qquad S^\prime := X^\prime - (Z+H), \qquad \dX := X^\prime - X, \qquad \dS := S^\prime - S. 
    \]
    Denote $\alpha := \{1,2,\ldots,r\}$ and $\gamma := \{r+1,\ldots,n\}$. Then, it holds that
    \begin{gather*}
        \begin{array}{llll}
            X^\prime_{\alpha\alpha} &= \LamX + \calO (\norm{\dX}), \qquad & S^\prime_{\gamma\gamma} &= \LamS + \calO (\norm{\dS}), \\
            X^\prime_{\gamma\alpha} &= \calO (\min \left\{ \norm{\dX}, \norm{\dS} \right\}), \qquad \qquad &S^\prime_{\gamma\alpha} &= \calO (\min \left\{ \norm{\dX}, \norm{\dS} \right\}), \\
            X^\prime_{\gamma\gamma} &= \calO (\norm{\dX} \cdot \norm{\dS}), \qquad &S^\prime_{\alpha\alpha} &= \calO (\norm{\dX} \cdot \norm{\dS}), \\
        \end{array} \\
        S^\prime_{\gamma\alpha} \LamX - \LamS X^\prime_{\gamma\alpha} = \calO (\norm{\dX} \cdot \norm{\dS}),
    \end{gather*}
    where $\norm{\cdot}$ is an arbitrary matrix norm.
\end{corollary}
\begin{proof}
    When $\norm{\H}$ is sufficiently small, we conclude from \cref{thm:eb-intro-thm} that there exists two positive constants $\kappa_X$ and $\kappa_S$ (depending on the norm type) such that
    \begin{align*}
        \norm{\dX - \Omega \circ \H} &= \left\|{\mymat{X^\prime_{\alpha\alpha} - \LamX - \HX & X^\prime_{\alpha\gamma} - \Theta\tran \circ \HO\tran \\ X^\prime_{\gamma\alpha} - \Theta \circ \HO & X^\prime_{\gamma\gamma}}}\right\|\le \kappa_X \cdot \norm{\HO} \cdot \norm{\H}, \\
        \norm{\dS - (E_n - \Omega) \circ \H} &= \left\|{\mymat{S^\prime_{\alpha\alpha} & S^\prime_{\alpha\gamma} - (E_{(n-r) \times r} - \Theta)\tran \circ \HO\tran \\ S^\prime_{\gamma\alpha} - (E_{(n-r) \times r} - \Theta) \circ \HO & S^\prime_{\gamma\gamma} - \LamS - \HS}}\right\| \\
        &\le \kappa_S \cdot \norm{\HO} \cdot \norm{\H}.
    \end{align*}
    Thus, there exist four positive constraints $\kappa_1, \kappa_2, \kappa_3, \kappa_4$ such that
    \begin{subequations}
        \label{app:eq:details-cui-prop:dX-dS}
        \begin{gather}
            \kappa_1 \cdot (\norm{\HX} + \norm{\HO}) \le \norm{\dX} \le \kappa_2 \cdot (\norm{\HX} + \norm{\HO}) \\
            \kappa_3 \cdot (\norm{\HS} + \norm{\HO}) \le \norm{\dS} \le \kappa_4 \cdot (\norm{\HS} + \norm{\HO}).
        \end{gather}
    \end{subequations}
    We only prove the $X$ part; the $S$ part follows directly by symmetry.
    \begin{enumerate}
        \item Since $X^\prime_{\alpha\alpha} - \LamX = (\dX)_{\alpha\alpha}$, the first conclusion $X^\prime_{\alpha\alpha} = \LamX + \calO (\norm{\dX})$ naturally holds.

        \item For $X^\prime_{\gamma\alpha}$, we have
        \begin{equation}
            \label{app:eq:details-cui-prop:Xnew-gamma-alpha}
            \norm{X^\prime_{\gamma\alpha} - \Theta \circ \HO} \le \kappa_X \cdot \norm{\HO} \cdot \norm{\H},
        \end{equation}
        \ie there exist $\kappa_5, \kappa_6 > 0$ such that $\kappa_5 \norm{\HO} \le \norm{X^\prime_{\gamma\alpha}} \le \kappa_6 \norm{\HO}$. From~\eqref{app:eq:details-cui-prop:dX-dS}, we conclude $X^\prime_{\gamma\alpha} = \calO (\min \{ \norm{\dX}, \norm{\dS} \})$.

        \item The norm of $X^\prime_{\gamma\gamma}$ is upper bounded by $\norm{X^\prime_{\gamma\gamma}} \leq \kappa_X \norm{\HO} \norm{\H}$. On the other hand, we have
        \[
            \norm{\dX} \cdot \norm{\dS} \ge \kappa_1 \kappa_3 \cdot (\norm{\HX} + \norm{\HO}) \cdot (\norm{\HS} + \norm{\HO}) \ge \kappa_7 \cdot \norm{\HO} \cdot \norm{\H}
        \]
        for some positive constant $\kappa_7$.

        \item Note that
        \begin{align*}
            \LamS (\Theta \circ \HO) = (\LamS \Theta) \circ \HO &= \mymat{ 
                \frac{-\lam{1}\lam{r+1}}{\lam{1} - \lam{r+1}} & \cdots & \frac{-\lam{r}\lam{r+1}}{\lam{r} - \lam{r+1}} \\
            \vdots & \ddots & \vdots \\
            \frac{-\lam{1}\lam{n}}{\lam{1} - \lam{n}} & \cdots & \frac{-\lam{r}\lam{n}}{\lam{r} - \lam{n}}  
            } \circ \HO \\
            &= \left( (E_{(n-r) \times r} - \Theta) \LamX \right) \circ \HO \\
            &= \left( (E_{(n-r) \times r} - \Theta) \circ \HO \right) \LamX,
        \end{align*}
        where we use the fact that $(AD) \circ B = B \circ (AD) = (B \circ A) D$ for diagonal $D$. Then,
        \begin{align*}
            &\ \norm{S^\prime_{\gamma\alpha} \LamX - \LamS X^\prime_{\gamma\alpha}} \\
            =&\ \norm{
                S^\prime_{\gamma\alpha} \LamX - \big((E_{(n-r) \times r} - \Theta) \circ \HO \big) \LamX - \big(\LamS X^\prime_{\gamma\alpha} - \LamS (\Theta \circ \HO)\big)
            } \\
            \le&\ \norm{S^\prime_{\gamma\alpha} \LamX - \big( (E_{(n-r) \times r} - \Theta) \circ \HO \big) \LamX} + \norm{\LamS X^\prime_{\gamma\alpha} - \LamS (\Theta \circ \HO)} \\
            \le&\ (\kappa_X + \kappa_S) \kappa_8 \cdot \norm{\HO} \cdot \norm{\H} 
        \end{align*}
        for some positive constant $\kappa_8$. We conclude $S^\prime_{\gamma\alpha} \LamX - \LamS X^\prime_{\gamma\alpha} = \calO (\norm{\dX} \norm{\dS})$, following the same argument as in item 3.
    \end{enumerate}
\end{proof}
From the above proof procedure, we see that \cref{thm:eb-intro-thm} provides a subtle and accurate control of the linearization residual, especially the $\norm{\HO}$ term. Otherwise, using the classic result in \cref{lem:eb-intro-SS02} only, inequalities like~\eqref{app:eq:details-cui-prop:dX-dS} and~\eqref{app:eq:details-cui-prop:Xnew-gamma-alpha} may not be derived in a straightforward manner.


\section{Local Linear Convergence with Nondegeneracy but without SC}
\label{app:sec:conv-only-nd}

In this section, we establish the local linear convergence of ADMM applied to SDPs in which primal and dual nondegeneracy hold but strict complementarity fails. That is to say, we consider the case where $\Zs$ is singular, \ie $s + r < n$. Again, we assume without loss of generality that $\Qs = \I[n]$ in~\eqref{eq:intro:Xs-sigSs}. We also need the following index sets
\[
    \alpha := \left\{ 1, \dots , r \right\}, \quad \beta := \left\{ r+1, \dots, n-s \right\}, \quad \gamma := \left\{ n-s+1, \dots, n \right\},
\]
and then any matrix $\H \in \Sn$ can be partitioned as
\begin{align}
    \label{app:eq:conv-only-nd:H-partition}
    \H = \mymat{
        \H[\alpha\alpha] & \H[\beta\alpha]\tran & \H[\gamma\alpha]\tran \\
        \H[\beta\alpha] & \H[\beta\beta] & \H[\gamma\beta]\tran \\
        \H[\gamma\alpha] & \H[\gamma\beta] & \H[\gamma\gamma]
    }.
\end{align}
When $\Zs$ is singular, the projector $\PiSnp{\cdot}$ is no longer Fr\'echet differentiable around~$\Zs$ \cite[Theorem 4.8]{sun2002mor-semismooth-matrix-valued}. However, its directional derivative always exists \cite[Theorem 4.7]{sun2002mor-semismooth-matrix-valued}.
\begin{lemma}[\text{\cite[Theorem 4.7]{sun2002mor-semismooth-matrix-valued}}] 
    \label{app:lem:conv-only-nd:dd}
    The PSD cone projection $\PiSnp{\cdot}$ is directionally differentiable at $\Zs$ and, for any $\H \in \Sn$ partitioned as in~\eqref{app:eq:conv-only-nd:H-partition}, its directional derivative at $\H$ is
    \[
        \Dd{H} := \mymat{
            \H[\alpha \alpha] & \H[\beta \alpha]\tran & \widetilde \Theta\tran \circ \H[\gamma \alpha]\tran \\
            \H[\beta \alpha] & \Pi_{\Symp{\abs{\beta}}} (\H[\beta \beta]) & 0 \\
            \widetilde \Theta \circ \H[\gamma \alpha] & 0 & 0 
        },
    \]
    where $\abs{\beta} = n - r - s$ is the cardinality of the index set $\beta$ and the matrix $\widetilde \Theta \in \Real{s \times r}$ is defined as
    \begin{equation}
        \label{app:eq:conv-only-nd:Theta}
        \widetilde \Theta_{i,j} := \frac{\lam{j}}{\lam{j} - \lam{n-s+i}}, \quad \text{for} \ i \in [s], \, j \in [r].
    \end{equation}
\end{lemma}
From the definition of the directional derivative, we have for sufficiently small $H \in \Sn$ that
\begin{equation}
    \label{app:eq:conv-only-nd:psd-proj-res}
    \PiSnp{\Zs + \H} = \PiSnp{\Zs} + \Dd{\H} + o(\normF{\H}).
\end{equation}
Recall our definition $\Thep := E_{s \times r} - \Theta$; similarly, we denote $\Ddp{\H}$ as
\begin{align*}
    \Ddp{\H} := \H - \Dd{\H} = \mymat{
        0 & 0 & (\Thep)\tran \circ \H[\gamma \alpha]\tran \\
        0 & -\Pi_{\Symp{\abs{\beta}}} (-\H[\beta \beta]) & \H[\gamma \beta]\tran \\
        \Thep \circ \H[\gamma \alpha] & \H[\gamma \beta] & \H[\gamma \gamma]
    },
\end{align*}
where we use the fact that $\H[\beta \beta] = \Pi_{\Symp{\abs{\beta}}} (\H[\beta \beta]) - \Pi_{\Symp{\abs{\beta}}} (-\H[\beta \beta])$. Similar to~\eqref{eq:lin-linearization}, we split one-step ADMM operator into two parts:
\begin{align*}
    \Zkpo - \Zs = \widetilde \Madmm (\Zk - \Zs) + \widetilde \Psi^{(k)},
\end{align*} 
where 
\begin{align}
    \widetilde \Madmm (\H) & := \PA \Ddp{\H} + \PAp \Dd{\H}, \label{app:eq:conv-only-nd:lin-M} \\
    \widetilde \Psi^{(k)} & := (\Id - 2\PA) \big(\PiSnp{\Zk} - \PiSnp{\Zs} - \Dd{\Zk - \Zs} \big) \nonumber \\
    &\phantom{:}= o(\normF{\Zk - \Zs}). \label{app:eq:conv-only-nd:lin-psik}
\end{align}
Although $\widetilde \Madmm$ is no longer a linear operator (because of $\Pi_{\Symp{\abs{\beta}}}$ in $\widetilde \Omega$), it is still positively homogenous. This is because $\Pi_{\Symp{\abs{\beta}}}$ is positively homogenous and other parts of $\widetilde \Madmm$ are linear. So we can still obtain the following result similar to \cref{lem:lin-auxi}.
\begin{lemma}
    \label{app:lem:conv-only-nd:lin-auxi}
    For any matrix $H \in \Sn$ partitioned as in~\eqref{app:eq:conv-only-nd:H-partition}, it holds that
    \begin{equation} 
        \label{app:eq:conv-only-nd:lin-inprod}
        \inprod{\Dd{\H}}{\Ddp{\H}} = 2\inprod{\widetilde \Theta \circ \H_{\gamma\alpha}}{\widetilde \Theta^\perp \circ \H_{\gamma\alpha}} \ge 0,
    \end{equation}
    with equality if and only if $\H[\gamma\alpha] = 0$, and that
    \begin{equation}
        \label{eq:conv-only-nd:lin-minus}
        \normF{\H}^2 - \normF{\widetilde \Madmm(\H)}^2 
        = \normF{\PA \Dd{\H}}^2 + \normF{\PAp \Ddp{\H}}^2 + 4 \inprod{\widetilde \Theta \circ \H_{\gamma\alpha}}{\widetilde \Theta^\perp \circ \H_{\gamma\alpha}}.
    \end{equation}
\end{lemma}
\begin{proof}
    From the definition of $\widetilde \Theta$~\eqref{app:eq:conv-only-nd:Theta} and the partition of $\H$ \eqref{app:eq:conv-only-nd:H-partition}, we see that
    \begin{align}
        &\ \inprod{\Dd{\H}}{\Ddp{\H}} \nonumber \\
        =&\ \inprod{
            \mymat{
                \H[\alpha \alpha] & \H[\beta \alpha]\tran & \widetilde \Theta\tran \circ \H[\gamma \alpha]\tran \\
                \H[\beta \alpha] & \Pi_{\Symp{\abs{\beta}}} (\H[\beta \beta]) & 0 \\
                \widetilde \Theta \circ \H[\gamma \alpha] & 0 & 0 
            }
        }{
            \mymat{
                0 & 0 & (\widetilde \Theta^\perp)\tran \circ \H[\gamma \alpha]\tran \\
                0 & -\Pi_{\Symp{\abs{\beta}}} (-\H[\beta \beta]) & \H[\gamma \beta]\tran \\
                \widetilde \Theta^\perp \circ \H[\gamma \alpha] & \H[\gamma \beta] & \H[\gamma \gamma]
            }
        } \nonumber \\
        =&\ 2 \inprod{\widetilde \Theta \circ \H[\gamma \alpha]}{\widetilde \Theta^\perp \circ \H[\gamma \alpha]} \ge 0. \label{app:eq:conv-only-nd:lin-inprod-prf}
    \end{align}
    where we already use the fact that $\langle \Pi_{\Symp{\abs{\beta}}}(\H[\beta\beta]), \Pi_{\Symp{\abs{\beta}}} (-\H[\beta \beta]) \rangle = 0$.
    Since all the entries in $\widetilde \Theta$ and $\widetilde \Theta^\perp$ are strictly positive, the inner product \eqref{app:eq:conv-only-nd:lin-inprod-prf} is zero if and only if $\H[\gamma \alpha]=0$.

    To show the second conclusion, we first decompose $\H$ as
    \[
        \H = \PA (\Dd{\H}) + \PA (\Ddp{\H}) + \PAp (\Dd{\H}) + \PAp (\Ddp{\H}).
    \]
    Then, we have
    \begin{align*}
        \normF{\H}^2 &= \normF{\PA (\Dd{\H})}^2 + \normF{\PA (\Ddp{\H})}^2 + \normF{\PAp (\Dd{\H})}^2 + \normF{\PAp (\Ddp{\H})}^2 \\
        &\phantom{=} \ + 2\inprod{\PA (\Dd{\H})}{\PA (\Ddp{\H})} + 2\inprod{\PAp (\Dd{\H})}{\PAp (\Ddp{\H})} \\
        &= \normF{\PA (\Dd{\H})}^2 + \normF{\PA (\Ddp{\H})}^2 + \normF{\PAp (\Dd{\H})}^2 + \normF{\PAp (\Ddp{\H})}^2 \\
        &\phantom{=} \ + 2 \inprod{\Dd{\H}}{\Ddp{\H}},
    \end{align*}
    and
    \begin{align*}
        \normF{\widetilde \Madmm(\H)}^2 = \normF{\PAp (\Dd{\H}) + \PA (\Ddp{\H})}^2 
        = \normF{\PAp (\Dd{\H})}^2 + \normF{\PA (\Ddp{\H})}^2.
    \end{align*}
    Combining both expressions with \eqref{app:eq:conv-only-nd:lin-inprod-prf} gives the desirable result.
\end{proof}

\begin{theorem}
    \label{app:thm:conv-only-nd}
    Suppose \cref{ass:lin-sol}, primal nondegeneracy~\eqref{eq:intro:primal-nondegeneracy} and dual nondegeneracy~\eqref{eq:intro:dual-nondegeneracy} hold. Define
    \begin{align*}
        \rho_{\mathrm{ND}} := \sup_{\normF{\H} = 1} \normF{\widetilde \Madmm(\H)} < 1
    \end{align*}
    For any $\rho \in (\rho_{\mathrm{ND}}, 1)$, there exists $\bar k_{\mathrm{ND}} \in \bbN$ such that for any integer $k \ge \bar k_{\mathrm{ND}}$, it holds that
    \[
        \normF{\Zkpo - \Zs} \leq \rho \normF{\Zk - \Zs}.
    \]
\end{theorem}
\begin{proof}
    First, we show that for any $\H \ne 0$, we have $\normF{\widetilde \Madmm(\H)} < \normF{\H}$. To see this, suppose there exists a matrix $\H \in \Sn$ partitioned as in~\eqref{app:eq:conv-only-nd:H-partition} such that $\normF{\widetilde \Madmm(\H)} \geq \normF{\H}$. Then, from \cref{app:lem:conv-only-nd:lin-auxi}, we see that
    \begin{align*}
        \PA \Dd{\H} = 0, \quad \PAp \Ddp{\H} = 0, \quad \H[\gamma\alpha] = 0.
    \end{align*} 
    From $\H[\gamma\alpha] = 0$ condition, we have
    \begin{align*}
        \PA \Dd{\H} = 0 \quad \Longleftrightarrow \quad & \PA \mymat{
            \H[\alpha \alpha] & \H[\beta \alpha]\tran & 0 \\
            \H[\beta \alpha] & \Pi_{\Symp{\abs{\beta}}} (\H[\beta \beta]) & 0 \\
            0 & 0 & 0 
        } = 0 \\
        \quad \Longleftrightarrow \quad & \mymat{
            \H[\alpha \alpha] & \H[\beta \alpha]\tran & 0 \\
            \H[\beta \alpha] & \Pi_{\Symp{\abs{\beta}}} (\H[\beta \beta]) & 0 \\
            0 & 0 & 0 
        } \in \calN(\Asdp) \cap \calN_{\Ss}.
    \end{align*}
    On the other hand, dual nondegeneracy~\eqref{eq:intro:dual-nondegeneracy} implies $\calN(\Asdp) \cap \calN_{\Ss} = \left\{ 0 \right\}$, and thus
    \begin{align}
        \label{app:eq:conv-only-nd:dual-side}
        \H[\alpha \alpha] = 0, \qquad \H[\beta \alpha] = 0, \qquad \Pi_{\Symp{\abs{\beta}}} (\H[\beta \beta]) = 0.
    \end{align}
    Symmetrically, we have
    \begin{align*}
        \PAp \Ddp{\H} \quad \Longleftrightarrow \quad & \PAp \mymat{
            0 & 0 & 0 \\
            0 & -\Pi_{\Symp{\abs{\beta}}} (-\H[\beta \beta]) & \H[\gamma \beta]\tran \\
            0 & \H[\gamma \beta] & \H[\gamma \gamma]
        } = 0 \\
        \Longleftrightarrow \quad & \mymat{
            0 & 0 & 0 \\
            0 & -\Pi_{\Symp{\abs{\beta}}} (-\H[\beta \beta]) & \H[\gamma \beta]\tran \\
            0 & \H[\gamma \beta] & \H[\gamma \gamma]
        } \in \calR(\AsdpT) \cap \calN_{\Xs}.
    \end{align*}
    On the other hand, primal nondegeneracy~\eqref{eq:intro:primal-nondegeneracy} implies $\calR(\AsdpT) \cap \calN_{\Xs} = \left\{ 0 \right\}$, and thus
    \begin{equation}
        \label{app:eq:conv-only-nd:primal-side}
        \H[\gamma \gamma] = 0, \quad \H[\gamma \beta] = 0, \quad \Pi_{\Symp{\abs{\beta}}} (-\H[\beta \beta]) = 0.
    \end{equation}
    Finally, combining~\eqref{app:eq:conv-only-nd:dual-side},~\eqref{app:eq:conv-only-nd:primal-side} together with $\H[\gamma\alpha] = 0$ induces $\H = 0$. 

    Second, we show that ${\normF{\widetilde \Madmm(\H)}}/{\normF{\H}} \le \rho_{\mathrm{ND}} < 1$ for all $\H \ne 0$. Since $\normF{\widetilde \Madmm(\cdot)}$ is continuous and the set $\left\{ \H \mymid \normF{\H} = 1 \right\}$ is compact, we draw from Extreme Value Theorem that 
    \[
        \rho_{\mathrm{ND}} := \sup_{\normF{\H} = 1} \normF{\widetilde \Madmm(\H)} < 1.
    \]
    On the other hand, because $\widetilde \Madmm$ is positively homogenous, we have for any $\H \ne 0$ that
    \[
        \frac{\normF{\widetilde \Madmm(\H)}}{\normF{\H}} = \normF{\widetilde \Madmm(\H / \normF{\H})} \le \sup_{\normF{\H'} = 1} \normF{\widetilde \Madmm(\H')} < 1.
    \]

    Third, we prove the locally linear decay of $\normF{\Zk - \Zs}$. It follows from~\eqref{app:eq:conv-only-nd:lin-psik} that for any $\rho \in (\rho_{\mathrm{ND}}, 1)$, these exists $\bar k_{\mathrm{ND}} \in \bbN$ such that for any $k \geq \bar k_{\mathrm{ND}}$, $\normF{\widetilde \Psi^{(k)}} \le (\rho - \rho_{\mathrm{ND}}) \normF{\Zk - \Zs}$. Finally, 
    \begin{align*}
        \normF{\Zkpo - \Zs} &= \normF{\widetilde \Madmm (\dZk) + \widetilde \Psi^{(k)}} \\
        &\le \rho_{\mathrm{ND}} \cdot \normF{\dZk} + \normF{\widetilde \Psi^{(k)}} \\
        &\le (\rho_{\mathrm{ND}} + \rho - \rho_{\mathrm{ND}}) \cdot \normF{\dZk} \\
        &= \rho \normF{\dZk}.
    \end{align*}
\end{proof}


\section{Missing Materials in \texorpdfstring{\cref{sec:conv-without-nd}}{Section 6}}
\label{app:sec:conv-without-nd}

\subsection{Proof of \texorpdfstring{\cref{lem:conv-nnd-Zdiff}}{Lemma}}
\label{app:sec:conv-nnd-Zdiff-prf}

From one-step ADMM \eqref{eq:intro:admm-one-step}, we have
\begin{align*}
    \Zkpo - \Zk &= -2\PA \PiSnp{\Zk} + \PA \Zk + \PiSnp{\Zk} + \Asdp^\dagger b - \sigma \PAp C - \Zk \\
    &= -2\PA \Xk + \Xk - \PAp \Zk + \Asdp^\dagger b - \sigma \PAp C \\
    &= -\PA \Xk + \PAp \Xk - \PAp (\Xk - \sigma \Sk) + \Asdp^\dagger b - \sigma \PAp C \\
    &= -\PA \Xk + \Asdp^\dagger b + \sigma \PAp (\Sk - C).
\end{align*}
Since for any $\widetilde{X}$ such that $\Asdp \widetilde{X} = b$,
\begin{align*}
    \PA \widetilde{X} = \AsdpT (\Asdp \AsdpT)^{-1} \Asdp \widetilde{X} = \AsdpT (\Asdp \AsdpT)^{-1} b = \Asdp^\dagger b,
\end{align*} 
the equality \eqref{eq:conv-nnd-Zdiff-eq} follows from the orthogonality between $\PA$ and $\PAp$.

Now we prove the inequality \eqref{eq:conv-nnd-Zdiff-ineq}. Let $\overline Z$ be an arbitrary point in $\Zopt$; \ie $\overline Z$ may not be the convergent point $\Zs$ of ADMM. Define $\overline X := \PiSnp{\overline Z}$ and $\overline S := (1/\sigma) \PiSnp{-\overline Z}$. So,
\[
    \PA (\PiSnp{\overline Z} - \widetilde X) = 0, \qquad \PAp (\PiSnp{\overline Z} - C) = 0
\]
for any matrix $\widetilde X \in \Sn$ satisfying $\Asdp \widetilde X = b$. Then, we have
\begin{align}
    &\ \normF{\Zk-\overline Z}^2 - \normF{\Zkpo-\overline Z}^2 \nonumber \\
    =&\ \normF{\Zk-\overline Z}^2 - \normF{\Zkpo - \Zk + \Zk - \overline Z}^2 \nonumber \\
    =&\ -2\langle {\Zkpo-\Zk}, {\Zk-\overline Z} \rangle - \normF{\Zkpo-\Zk}^2 \nonumber \\
    =&\ 2\langle {\PA (\Xk - \overline X) - \sigma \PAp (\Sk-\overline S)}, {\Zk-\Zs} \rangle - \normF{\Zkpo-\Zk}^2. \label{eq:conv-nnd-Zdiff-prf-1}
\end{align}
We further decompose $\Zk-\overline Z$ as 
\begin{align*}
    &\ \Zk-\overline Z \\
    =&\ \Xk - \sigSk - (\overline X - \sigma \overline S) \\
    =&\ \PA (\Xk-\overline X) + \PAp (\Xk-\overline X) - \sigma \PAp (\Sk-C) - \sigma \PAp (\Sk-C).
\end{align*}
Then the inner product term on the right-hand side of \eqref{eq:conv-nnd-Zdiff-prf-1} becomes
\begin{align*}
    &\ \langle {\PA (\Xk - \overline X) - \sigma \PAp (\Sk-\overline S)}, {\Zk-\Zs} \rangle \\
    =&\ \langle {\PA (\Xk - \overline X) - \sigma \PAp (\Sk-\overline S)}, {\PA (\Xk-\overline X) - \sigma \PAp (\Sk-C)} \rangle \\
    &\ - \langle {\PA (\Xk - \overline X) - \sigma \PAp (\Sk-\overline S)}, {\PAp (\Xk-\overline X) - \sigma \PA (\Sk-C)} \rangle \\
    =&\ \normF{\Zkpo-\Zk}^2 - \sigma \langle{\PA (\Xk-\overline X)}, {\PAp (\PA (\Sk-\overline S))} \\
    &\ - \sigma \langle {\PAp (\Xk-\overline X)}, {\PAp (\Sk-\overline S)} \rangle \\
    =&\ \normF{\Zkpo-\Zk}^2 - \sigma \langle {\Xk-\overline X}, {\Sk-\overline S} \rangle \\
    =&\ \normF{\Zkpo-\Zk}^2 - \sigma \langle {\Xk}, {\overline S} \rangle + \sigma \langle {\overline X}, {\Sk} \rangle \\
    \ge &\ \normF{\Zkpo-\Zk}^2,
\end{align*}
where the last equality uses the fact that $\inprod{\overline X}{\overline S} = 0$ and $\langle \Xk, \Sk \rangle = 0$. Combining with \eqref{eq:conv-nnd-Zdiff-prf-1} yields
\[
    \normF{\Zk-\overline Z}^2 - \normF{\Zkpo-\overline Z}^2 \geq \normF{\Zkpo-\Zk}^2.
\]
Now we choose $\overline Z$ as the closest point in $\Zopt$ to $\Zk$. Then, we have
\begin{align*}
    \dist^2 (\Zk,\Zopt) &= \normF{\Zk - \overline Z}^2 \\
    &\geq \normF{\Zkpo-\Zk}^2 + \normF{\Zkpo - \overline Z}^2 \\
    &\geq \normF{\Zkpo-\Zk}^2 + \dist^2 (\Zkpo,\Zopt).
\end{align*}

\subsection{Proof of \texorpdfstring{\cref{lem:conv-nnd-kkt}}{Lemma}}
\label{app:sec:conv-nnd-kkt-prf} 

(1) First, we show that $(\Xs,\Ss)$ is a KKT point for \eqref{eq:intro-sdp} if and only if
\[
    \PA (\Xs - \widetilde X) = 0, \;\; \PAp (\Ss - C) = 0, \;\; \inprod{\Xs}{C} + \langle{\widetilde X}, {\Ss} \rangle - \langle{\widetilde X}, {C} \rangle= 0, \;\; X \in \Symp{n}, \;\; S \in \Symp{n},
\]
where $\widetilde X \in \Sn$ is an arbitrary matrix satisfying $\Asdp \widetilde X = b$.
\begin{itemize}
\item If $\Asdp \Xs=b$, then $\PA (\Xs-\widetilde X) = \Asdp^\dagger \Asdp (\Xs-\widetilde X) = \Asdp^\dagger (b-b) = 0$. The converse holds because $\Asdp$ is surjective and thus $\Asdp \Asdp^\dagger = \Id$. Together with $\Xs \in \Symp{n}$, it gives primal feasibility.

\item Similarly, we note that $\PAp (\Ss - C) = 0$ is equivalent to $\Ss - C \in \range(\Asdp^\ast)$, which is further equivalent to $\Asdp^\ast y + S = C$ for some $y \in \Real{m}$. Together with $\Ss \in \Symp{n}$, this gives dual feasibility.

\item The third condition implies zero duality gap:
\begin{align*}
    \inprod{\Xs}{C} - b\tran \ys = 0 \quad \Longleftrightarrow \quad &\inprod{\Xs}{C} + \inprod{b}{(\Asdp^\ast)^\dagger (\Ss-C)} = 0 \\
    \Longleftrightarrow \quad &\inprod{\Xs}{C} + \langle{\Asdp^\dagger b}, {\Ss} \rangle - \langle{\Asdp^\dagger b}, {C} \rangle = 0 \\
    \Longleftrightarrow \quad &\inprod{\Xs}{C} + \langle{\widetilde X}, {\Ss} \rangle - \langle{\widetilde X}, {C} \rangle = 0.
\end{align*}
\end{itemize}
For notational convenience, we define
\begin{gather*}
    \begin{array}{llll}
    \calF_X &:= \{(X,\sigS) \mid \PA (\X - \widetilde X) = 0\}, \qquad \qquad
    &\calF_Z &:= \{(X,\sigS) \mid \PAp (S - C) = 0\}, \\
    \calR_X &:= \{(X,\sigS) \mid \Proj_{\calT_{\Ss}} (X) = 0\}, \qquad
    &\calR_S &:= \{(X,\sigS) \mid \Proj_{\calT_{\Xs}} (\sigS) = 0\}, \qquad \\
    \end{array} \\
    \calF_{\mathrm{gap}} := \{(X,\sigS) \mid \inprod{\Xs}{\sigC} + \langle{\widetilde X}, {\sigSs} \rangle - \langle{\widetilde X}, {\sigC} \rangle = 0\}, \\
    \calF := \calF_X \cap \calF_S \cap \calF_{\mathrm{gap}}, \qquad \calR := \calR_X \cap \calR_S.
\end{gather*}
(Note that the choice of $\widetilde X$ does not matter in $\calF_X$, as long as $\Asdp \widetilde X = b$.)

(2) Second, we show that $\Xopt \cap (\sigma \Sopt) = \calF \cap (\Symp{n} \times \Symp{n}) = \calF \cap (\Symp{n} \times \Symp{n}) \times \calR$. To see this, we choose any pair of optimal solutions $(\X, \sigS) \in \Xopt \times (\sigma \Sopt)$ and write down the complementary slackness condition: $\inprod{\X}{\sigma \Ss} = 0$ and $\inprod{\Xs}{\sigS} = 0$. Combined with the facts that $\X \in \Symp{n}$ and $\sigS \in \Symp{n}$, we have
\[
    \X \in \minface{\Xs}{\Symp{n}} = \Symp{n} \times \calN_{\Ss}, \qquad
    \sigS \in \minface{\sigSs}{\Symp{n}} = \Symp{n} \cap \calN_{\Xs}.
\]
Thus, $\Proj_{\calT_{\Ss}} (X) = 0$ and $\Proj_{\calT_{\Xs}} (\sigS) = 0$, which directly implies $(\X, \sigS) \in \calR_X \cap \calR_S = \calR$.

(3) Third, we project an arbitrary $(X, \sigS) \in \Sn \times \Sn$ to the regularized linear system $\calF \cap \calR$. By Hoffman's error bound~\cite{hoffman2003ws-approximate-sol-linear-inequalities}, there exists $\kappa_0 > 0$ such that
\begin{align}
    &\ \kappa_0 \cdot \dist ((X, \sigS), \calF \cap \calR) \nonumber \\
    \le&\ \dist ((X, \sigS), \calF_X) + \dist((X, \sigS), \calF_S) + \dist((X, \sigS), \calF_{\mathrm{gap}}) \nonumber \\
    &\ + \dist ((X, \sigS), \calR_X) + \dist ((X, \sigS), \calR_S), \label{app:eq:regularized-backforward-kkt-0}
\end{align}
where
\begin{gather*}
    \begin{array}{llll}
        \dist((X, \sigS), \calF_X) &= \normF{\PA (\X - \widetilde X)}, \qquad \quad
        &\dist((X, \sigS), \calF_Z) &= \sigma \normF{\PAp (S-C)} \\
        \dist((X, \sigS), \calR_X) &= \normF{\Proj_{\calT_{\Ss}} (X)}, \qquad
        &\dist((X, \sigS), \calR_S) &= \sigma \normF{\Proj_{\calT_{\Xs}} (S)},
    \end{array} \\
    \dist((X, \sigS), \calF_{\mathrm{gap}}) = \tfrac{\sigma}{\sqrt{\sigma^2 \normF{C}^2 + \normF{\Asdp^\dagger b}^2}} |\inprod{\Xs}{C} + \langle{\widetilde X}, {\Ss} \rangle - \langle{\widetilde X}, {C} \rangle|.
\end{gather*}

(4) Fourth, we explicitly construct a point belonging to $\Xopt \times (\sigma \Sopt)$. Take 
\[
    \beta := \max \left\{ 
        \frac{1}{\lam{r}} \cdot \left( [-\lammin{X}]_+ + \normtwo{Z_X} \right), \ 
        \frac{1}{-\lam{r+1}} \cdot \left( [-\lammin{\sigS}]_+ + \normtwo{Z_S} \right)
    \right\}
\]
and define $(Z_X,Z_S) := (\X, \sigS) - \Proj_{\calF \cap \calR} (\X, \sigS)$. Thus, by definition
\[
    \sqrt{\normF{Z_X}^2 + \normF{Z_S}^2} = \dist((X, \sigS), \calF \cap \calR).
\]
Then, consider the point $(X, \sigS) - (Z_X, Z_S) + \beta \cdot (\Xs, \sigSs)$. For the primal part:
\begin{subequations}
\begin{align}
    &\ \lammin{(X - Z_X + \beta \cdot \Xs)_{1:r, 1:r}} \nonumber \\
    \ge&\ \beta \cdot \lammin{(\Xs)_{1:r, 1:r}} + \lammin{(X - Z_X)_{1:r, 1:r}} \label{app:eq:regularized-backforward-kkt-1a} \\
    =&\ \beta \lam{r} + \min \left\{ \lammin{X - Z_X}, 0 \right\} \label{app:eq:regularized-backforward-kkt-1b} \\
    \ge&\ \beta \lam{r} + \min \left\{ \lammin{X} - \normtwo{Z_X}, 0 \right\} \label{app:eq:regularized-backforward-kkt-1c}  \\
    \ge&\ \frac{1}{\lam{r}} \big([-\lammin{X}]_+ + \normtwo{Z_X} \big) \cdot \lam{r} + \min \left\{ \lammin{X} - \normtwo{Z_X}, 0 \right\} \nonumber \\
    \ge&\ 0, \nonumber
\end{align}
\end{subequations}
where~\eqref{app:eq:regularized-backforward-kkt-1a} and~\eqref{app:eq:regularized-backforward-kkt-1c} come form Weyl's inequality,~\eqref{app:eq:regularized-backforward-kkt-1b} holds since $X - Z_X \in \calN_\Ss$.  \rebuttal{Combining the above inequality with the facts that} $\Xs \in \calN_\Ss$ and $X - Z_X \in \calN_\Ss$ gives
\begin{align*}
    & \lammin{X - Z_X + \beta \Xs} = \min \left\{ \lammin{(X - Z_X + \beta \Xs)_{1:r, 1:r}}, 0 \right\} \ge 0.
\end{align*}
Symmetrically, $\lammin{\sigS - Z_S + \beta \sigSs} \ge 0$. Therefore, $(X - Z_X + \beta \Xs, \sigS - Z_S + \beta \sigSs) \in \Symp{n} \times \Symp{n}$. Combining the fact that both $(X - Z_X, \sigS - Z_S)$ and $(\Xs, \sigSs)$ belong to $\calF \cap \calR$, we conclude that
\[
    \frac{1}{1 + \beta} \cdot (X - Z_X + \beta \Xs, \sigS - Z_S + \beta \sigSs) \in \calF \cap \calR \cap (\Symp{n} \times \Symp{n}) = \Xopt \cap (\sigma \Sopt).
\]

(5) Finally, we upper bound the distance from $(X, \sigS)$ to $\Xopt \cap (\sigma \Sopt)$ by
\begin{align}
    &\ \dist((X, \sigS), \Xopt \times \sigma \Sopt) \nonumber \\
    \le&\ \left\|{(X, \sigS) - \frac{1}{1 + \beta} \cdot (X - Z_X + \beta \Xs, \sigS - Z_S + \beta \sigSs)}\right\|_{\mathsf{F} \times \mathsf{F}} \nonumber \\
    =&\ \sqrt{
        \left\|{X - \frac{1}{1 + \beta} (X - Z_X + \beta \Xs)}\right\|_{\mathsf{F}}^2 
        + \left\|{\sigS - \frac{1}{1 + \beta} (\sigS - Z_S + \beta \sigSs)}\right\|_{\mathsf{F}}^2
    } \nonumber \\
    =&\ \sqrt{2} \left\|{X - \frac{1}{1 + \beta} (X - Z_X + \beta \Xs)}\right\|_{\mathsf{F}} + \sqrt{2} \left\|{\sigS - \frac{1}{1 + \beta} (\sigS - Z_S + \beta \sigSs)}\right\|_{\mathsf{F}}. \label{app:eq:regularized-backforward-kkt-2a}
\end{align}
To bound the first term on the right-hand side, we have
\begin{align}
    \left\|{X - \frac{1}{1 + \beta} \cdot (X - Z_X + \beta \Xs)}\right\|_{\mathsf{F}} 
    =&\ \left\|{\frac{1}{1 + \beta} \cdot (\beta \cdot X - Z_X + \beta \Xs)}\right\|_{\mathsf{F}} \nonumber \\
    \le&\ \frac{1}{1 + \beta} \cdot \big(\normF{\beta X} + \normF{Z_X} + \normF{\beta \Xs} \big) \nonumber \\
    \le&\ \beta \normF{X} + \normF{Z_X} + \beta \normF{\Xs} \nonumber \\
    \le&\ (\delta_X + \lam{1}) \beta  + \normF{Z_X}. \label{app:eq:regularized-backforward-kkt-2b}
\end{align}
Similarly, we have 
\begin{equation} \label{app:eq:regularized-backforward-kkt-2c}
    \left\|{\sigS - \frac{1}{1 + \beta} \cdot (\sigS - Z_S + \beta \sigSs)}\right\|_{\mathsf{F}} \le (\delta_S - \lam{n}) \beta  + \normF{Z_S}.
\end{equation}
Combining \eqref{app:eq:regularized-backforward-kkt-2a}--\eqref{app:eq:regularized-backforward-kkt-2c} gives
\begin{align}
    &\ \dist ((X, \sigS), \Xopt \times (\sigma \Sopt)) \nonumber \\
    \le&\ \normF{Z_X} + \normF{Z_S} + (\delta_X + \delta_S + \lam{1} - \lam{n}) \beta \nonumber \\
    \le&\ \sqrt{2 (\normF{Z_X}^2 + \normF{Z_S}^2)} + (\delta_X + \delta_S + \lam{1} - \lam{n}) \nonumber \\
    &\ \cdot \max \left\{ 
        \frac{1}{\lam{r}} \cdot \left( [-\lammin{X}]_+ + \normtwo{Z_X} \right), \ 
        \frac{1}{-\lam{r+1}} \cdot \left( [-\lammin{\sigS}]_+ \normtwo{Z_S} \right)
        \right\} \nonumber \\
    \le&\ \sqrt{2 (\normF{Z_X}^2 + \normF{Z_S}^2)} + (\delta_X + \delta_S + \lam{1} - \lam{n}) \cdot \max \left\{ \frac{1}{\lam{r}}, \frac{1}{-\lam{r+1}} \right\} \nonumber \\
    &\ \cdot \left( [-\lammin{X}]_+ + \normtwo{Z_X} + [-\lammin{\sigS}]_+ + \normtwo{Z_S} \right) \nonumber \\
    \le&\ \kappa_1 \cdot \left( [-\lammin{X}]_+ + [-\lammin{\sigS}]_+ \right) + \kappa_2 \cdot \sqrt{\normF{Z_X}^2 + \normF{Z_S}^2} \nonumber \\
    \leq&\ \kappa_1 \cdot \left( [-\lammin{X}]_+ + [-\lammin{\sigS}]_+ \right) + \kappa_2 \cdot \dist ((X, \sigS), \calF \cap \calR), \label{app:eq:regularized-backforward-kkt-3}
\end{align}
where
\[
    \kappa_1 := (\delta_X + \delta_S + \lam{1} - \lam{n}) \cdot \max \left\{ \frac{1}{\lam{r}}, \, \frac{1}{-\lam{r+1}} \right\}, \quad \kappa_2 := \kappa_1 + \sqrt{2}.
\]
Therefore, combining \eqref{app:eq:regularized-backforward-kkt-0} and \eqref{app:eq:regularized-backforward-kkt-3} yields the desirable result with
\[
    \kappa = \left(\kappa_1 + \frac{\kappa_2}{\kappa_0} \cdot \left( 1 + \frac{1}{\sqrt{\sigma^2 \normF{C}^2 + \normF{\Asdp^\dagger b}^2}} \right) \right)^{-1}.
\]


\section{Missing Materials in \texorpdfstring{\cref{sec:error-bound}}{Section 7}}
\label{app:sec:error-bound}


\subsection{Proof of \texorpdfstring{\cref{lem:error-bound:error-control-one-step}}{Lemma 9}}
\label{app:sec:error-bound:error-control-one-step}

The proof of \cref{lem:error-bound:error-control-one-step} needs the following two auxiliary lemmas.
\begin{lemma}
    \label{lem:error-bound:exp-bound}
    Let $S \in \Sym{n}$ satisfy $\normtwo{S} \le \frac{3}{4}$ and define $\psi(S) := \exp(S) - I - S$. Then, it holds that
    \[
        \normtwo{\psi(S)} \le \frac{2}{3} \normtwo{S}^2.
    \]
\end{lemma}
\begin{proof}
    From the definition of $\psi$, we have
    \[
        \psi(S) = \exp(S) - I - S = \sum_{k=2}^\infty \frac{1}{k!} S^k,
    \]
    and then
    \begin{align*}
        \normtwo{\psi(S)} &\le \sum_{k=2}^\infty \frac{1}{k!} \normtwo{S}^k 
        = \frac{1}{2} \normtwo{S}^2 \cdot \sum_{k=0}^\infty \frac{2}{(k+2)!} \normtwo{S}^k
        \le \frac{1}{2} \normtwo{S}^2 \cdot \sum_{k=0}^\infty \left( \frac{1}{3} \right)^k \normtwo{S}^k \\
        &= \frac{1}{2} \normtwo{S}^2 \cdot \frac{1}{1 - \frac{1}{3} \normtwo{S}},
    \end{align*}
    where the last inequality uses the fact that $\frac{2}{n!} \le \left( \frac{1}{3} \right)^{n-2}$ for all $n \ge 3$. Finally, we conclude from the assumption that $\normtwo{S} \le \frac{3}{4}$:
    \[
        \normtwo{\psi(S)} \le \frac{1}{2} \normtwo{S}^2 \cdot \frac{4}{3} = \frac{2}{3} \normtwo{S}^2.
    \]
\end{proof}

\begin{lemma}
    \label{lem:error-bound:twonorm}
    Let $X \in \Sn$ be partitioned as
    \[
        X = \mymat{A & B\tran \\ B & C} \in \Sn.
    \]
    Then, it holds that
    \[
        \max\{\normtwo{A}, \normtwo{B}, \normtwo{C}\} \le \normtwo{X} \le \normtwo{A} + \normtwo{B} + \normtwo{C}.
    \]
\end{lemma}
\begin{proof}
    On one hand, we have
    \begin{align*}
        \normtwo{X} = \sup_{\normtwo{x}^2 + \normtwo{y}^2 = 1} \left\|{\mymat{A & B\tran \\ B & C} \mymat{x \\ y}}\right\|_2 \ge \sup_{\normtwo{x}^2 = 1} \left\|{\mymat{A & B\tran \\ B & C} \mymat{x \\ 0}}\right\|_2 = \sup_{\normtwo{x}^2 = 1} \left\|{\mymat{Ax \\ Bx}}\right\|_2,
    \end{align*}
    and then 
    \[
        \normtwo{X} \ge \sup_{\normtwo{x}^2 = 1}  \normtwo{Ax} = \normtwo{A}, \qquad
        \normtwo{X} \ge \sup_{\normtwo{x}^2 = 1}  \normtwo{Bx} = \normtwo{B}.
    \]
    Similarly, $\normtwo{X} \le \normtwo{C}$.

    On the other hand, we see that
    \begin{align*}
        \normtwo{X} &\le \left\|{\mymat{A & 0 \\ 0 & 0}}\right\|_2 + \left\|{\mymat{0 & 0 \\ 0 & C}}\right\|_2 + \left\|{\mymat{0 & B\tran \\ B & 0}}\right\|_2 \\
        &= \normtwo{A} + \normtwo{C} + \left\|{\mymat{0 & B\tran \\ B & 0}}\right\|_2 \\
        &= \normtwo{A} + \normtwo{C} + \normtwo{B}.
    \end{align*}
\end{proof}

Now we are ready to prove \cref{lem:error-bound:error-control-one-step}.
\begin{proof}[Proof of \cref{lem:error-bound:error-control-one-step}]
Since $\ZX[\ell] \in \Sympp{n}$ and $\ZS[\ell] \in \Symnn{n}$, the Sylvester equation~\eqref{eq:error-bound:sylvester-eq} has a unique solution with 
\begin{align*}
    \text{vec} (\WO[\ell]) & = \left( I_{r} \otimes (-\ZS[\ell]) + \ZX[\ell] \otimes I_{n-r}  \right)^{-1} \text{vec} (\ZO[\ell]) \\
    & = \left( (-\ZS[\ell]) \oplus \ZX[\ell] \right)^{-1} \text{vec} (\ZO[\ell]) 
\end{align*}
From~\cite[Theorem 2.5]{schacke2004thesis-kronecker-product}, we see that the eigenvalues of $(-\ZS[\ell]) \oplus \ZX[\ell]$ equal to the sum of the eigenvalues of $-\ZS[\ell]$ and $\ZX[\ell]$. Thus, $(-\ZS[\ell]) \oplus \ZX[\ell]$ is positive definite and 
\begin{align*}
    \normtwo{\text{vec} (\WO[\ell])} &\le \frac{1}{\lammin{(-\ZS[\ell]) \oplus \ZX[\ell]}} \cdot \normtwo{\text{vec} (\ZO[\ell])} \\
    &= \frac{1}{\deltalam{\ZX[\ell]}{\ZS[\ell]}} \cdot \normtwo{\text{vec} (\ZO[\ell])}.
\end{align*}
Therefore, we can upper bound $\normtwo{\WO[\ell]}$ and $\normtwo{\W[\ell]}$ by $\normtwo{\ZO[\ell]}$:
\begin{align}
    \label{app:eq:error-control-one-step:WO}
    & \normtwo{\WO[\ell]} \le \normF{\WO[\ell]} = \normtwo{\text{vec} (\WO[\ell])} 
    \le \frac{d}{\deltalam{\ZX[\ell]}{\ZS[\ell]}} \cdot \normtwo{\ZO[\ell]} = \etanew[\ell] \cdot \normtwo{\ZO[\ell]}
\end{align}
and 
\begin{align}
    \label{app:eq:error-control-one-step:W-bound}
    \normtwo{\W[\ell]} = \left\|{\mymat{0 & -\WO[\ell]\tran \\ \WO[\ell] & 0}}\right\|_2 = \normtwo{\WO[\ell]} \le \etanew[\ell] \cdot \normtwo{\ZO[\ell]}
\end{align}

In addition, we have 
\begin{align*}
    \mymat{\I[r] & \WO[\ell]\tran \\ -\WO[\ell] & I_{n-r}} \mymat{\ZX[\ell] & \ZO[\ell]\tran \\ \ZO[\ell] & \ZS[\ell]} \mymat{\I[r] & -\WO[\ell]\tran \\ \WO[\ell] & I_{n-r}} = \mymat{\QX & \QO\tran \\ \QO & \QS}, \quad \text{where} \ \W[\ell] = \mymat{0 & -\WO[\ell]\tran \\ \WO[\ell] & 0}
\end{align*}
and
\begin{align*}
    & \QX = \ZX[\ell] + \WO[\ell]\tran \ZO[\ell] + \ZO[\ell]\tran \WO[\ell] + \WO[\ell]\tran \ZS[\ell] \WO[\ell] \\
    & \QS = \WO[\ell] \ZX[\ell] \WO[\ell]\tran - \ZO[\ell] \WO[\ell]\tran - \WO[\ell] \ZO[\ell]\tran + \ZS[\ell] \\
    & \QO = -\WO[\ell] \ZX[\ell] + \ZO[\ell] - \WO[\ell] \ZO[\ell]\tran \WO[\ell] + \ZS[\ell] \WO[\ell].
\end{align*}
From the definition of $\WO[\ell]$~\eqref{eq:error-bound:sylvester-eq}, we further have $\QO =  - \WO[\ell] \ZO[\ell]\tran \WO[\ell]$.

Now, we are ready to bound the following spectral norm:
\begin{align}
    &\ \left\|{\mymat{\ZX[\ell+1] - \ZX[\ell] & \ZO[\ell+1]\tran \\ \ZO[\ell+1] & \ZS[\ell+1] - \ZS[\ell]}}\right\|_2 \nonumber \\
    =&\ \left\|{\exp(\W[\ell])\tran \mymat{\ZX[\ell] & \ZO[\ell]\tran \\ \ZO[\ell] & \ZS[\ell]} \exp(\W[\ell]) - \mymat{\ZX[\ell] & 0 \\ 0 & \ZS[\ell]}}\right\|_2 \nonumber \\
    \le&\ \left\|{(I+\W[\ell])\tran \mymat{\ZX[\ell] & \ZO[\ell]\tran \\ \ZO[\ell] & \ZS[\ell]}(I+\W[\ell]) - \mymat{\ZX[\ell] & 0 \\ 0 & \ZS[\ell]}} \right\|_2 \nonumber \\
    &\ + \left\|\exp(\W[\ell])\tran \mymat{\ZX[\ell] & \ZO[\ell]\tran \\ \ZO[\ell] & \ZS[\ell]} \exp(\W[\ell]) - (I+\W[\ell])\tran \mymat{\ZX[\ell] & \ZO[\ell]\tran \\ \ZO[\ell] & \ZS[\ell]}(I+\W[\ell])\right\|_2. \label{app:eq:error-control-one-step:0}
\end{align}
We then bound the two terms on right-hand side of \eqref{app:eq:error-control-one-step:0} one-by-one.
\begin{enumerate}
    \item By definition, we have
    \begin{align}
        &\ \left\|{(I+\W[\ell])\tran \mymat{\ZX[\ell] & \ZO[\ell]\tran \\ \ZO[\ell] & \ZS[\ell]}(I+\W[\ell]) - \mymat{\ZX[\ell] & 0 \\ 0 & \ZS[\ell]}} \right\|_2 \nonumber \\
        =&\ \left\|{\mymat{\QX - \ZX[\ell] & \QO\tran \\ \QO & \QS - \QS[\ell]}}\right\|_2 \nonumber \\
        \le&\ \normtwo{\QS - \ZX[\ell]} + \normtwo{\QO} + \normtwo{\QS - \ZS[\ell]} \nonumber \\
        =&\ \normtwo{\WO[\ell]\tran \ZO[\ell] + \ZO[\ell]\tran \WO[\ell] + \WO[\ell]\tran \ZS[\ell] \WO[\ell]} + \normtwo{\WO[\ell] \ZO[\ell]\tran \WO[\ell]} \nonumber \\
        &\ + \normtwo{\WO[\ell] \ZX[\ell] \WO[\ell]\tran - \ZO[\ell] \WO[\ell]\tran - \WO[\ell] \ZO[\ell]\tran}, \label{app:eq:error-control-one-step:1}
    \end{align}
    where the inequality follows from \cref{lem:error-bound:twonorm}. We can further bound the three terms on the right-hand side of \eqref{app:eq:error-control-one-step:1} one-by-one.
    \begin{enumerate}
        \item For the first term, we have from \eqref{app:eq:error-control-one-step:WO} that
        \begin{align}
            &\ \normtwo{ \WO[\ell]\tran \ZO[\ell] + \ZO[\ell]\tran \WO[\ell] + \WO[\ell]\tran \ZS[\ell] \WO[\ell]} \nonumber \\
            \le&\ 2 \normtwo{\WO[\ell]} \cdot \normtwo{\ZO[\ell]} + \normtwo{\WO[\ell]}^2 \cdot \normtwo{\ZS[\ell]} \nonumber \\
            \le&\ 2\etanew[\ell] \cdot \normtwo{\ZO[\ell]}^2 + \etanew[\ell]^2 \cdot \normtwo{\ZS[\ell]} \cdot \normtwo{\ZO[\ell]}^2 \nonumber \\
            \le&\ 2\etanew[\ell] \cdot \normtwo{\ZO[\ell]}^2 + \etanew[\ell]^2 \cdot \normtwo{\Z[\ell]} \cdot \normtwo{\ZO[\ell]}^2. \label{app:eq:error-control-one-step:1a}
        \end{align}

        \item For the second term, we obtain from \eqref{app:eq:error-control-one-step:WO} that
        \begin{equation} \label{app:eq:error-control-one-step:1b}
            \normtwo{\WO[\ell] \ZO[\ell]\tran \WO[\ell]} \le \etanew[\ell]^2 \cdot \normtwo{\ZO[\ell]}^3 \le \etanew[\ell]^2 \cdot \normtwo{\Z[\ell]} \cdot \normtwo{\ZO[\ell]}^2.
        \end{equation}

        \item Again from \eqref{app:eq:error-control-one-step:WO}, we have 
        \begin{equation} \label{app:eq:error-control-one-step:1c}
            \normtwo{\WO[\ell] \ZX[\ell] \WO[\ell]\tran - \ZO[\ell] \WO[\ell]\tran - \WO[\ell] \ZO[\ell]\tran} \le 2\etanew[\ell] \cdot \normtwo{\ZO[\ell]}^2 + \etanew[\ell]^2 \cdot \normtwo{\Z[\ell]} \cdot \normtwo{\ZO[\ell]}^2.
        \end{equation}
    \end{enumerate}

    \item Similarly, the second term on the right-hand side of \eqref{app:eq:error-control-one-step:0} can be bounded as
    \begin{align}
        &\ \normtwo{\exp(\W[\ell])\tran \Z[\ell] \exp(\W[\ell]) - (I+\W[\ell])\tran \Z[\ell] (I+\W[\ell])} \nonumber \\
        \le&\ 2 \normtwo{\exp(\W[\ell])- \I[n] - \W[\ell]} \cdot \normtwo{\Z[\ell]} \cdot \normtwo{\I[n] + \W[\ell]} + \normtwo{\exp(\W[\ell])- \I[n] - \W[\ell]}^2 \cdot \normtwo{\Z[\ell]}. \label{app:eq:error-control-one-step:2}
    \end{align}
    Again, we bound the two terms on the right-hand side one-by-one.
    \begin{enumerate}
        \item We see from \cref{lem:error-bound:exp-bound} and \eqref{app:eq:error-control-one-step:W-bound} that
        \begin{align}
            &\ \normtwo{\exp(\W[\ell])- \I[n] - \W[\ell]} \cdot \normtwo{\Z[\ell]} \cdot \normtwo{\I[n] + \W[\ell]} \nonumber \\
            \le&\ \frac{2}{3} \cdot \normtwo{\W[\ell]}^2 \cdot \normtwo{\Z[\ell]} \cdot \normtwo{\I[n] + \W[\ell]} \nonumber \\
            \le&\ \frac{2}{3} \cdot \normtwo{\W[\ell]}^2 \cdot \normtwo{\Z[\ell]} \cdot (1 + \normtwo{\W[\ell]}) \nonumber \\
            \le&\ \frac{2}{3}\etanew[\ell]^2 \cdot \normtwo{\ZO[\ell]}^2 \cdot \normtwo{\Z[\ell]} \cdot (1 + \etanew[\ell] \normtwo{\ZO[\ell]}) \nonumber \\
            \le&\ \frac{2}{3}\etanew[\ell]^2 \cdot \normtwo{\Z[\ell]} \cdot \normtwo{\ZO[\ell]}^2  + \frac{2}{3}\etanew[\ell]^3 \cdot \normtwo{\Z[\ell]}^2 \cdot \normtwo{\ZO[\ell]}^2. \label{app:eq:error-control-one-step:2a}
        \end{align}

        \item The second term on the right-hand side of \eqref{app:eq:error-control-one-step:2} can be readily bounded by
        \begin{equation} \label{app:eq:error-control-one-step:2b}
            \normtwo{\exp(\W[\ell])- \I[n] - \W[\ell]}^2 \cdot \normtwo{\Z[\ell]} \le \frac{4}{9} \cdot \normtwo{\W[\ell]}^4 \cdot \normtwo{\Z[\ell]} \le \frac{4}{9} \etanew[\ell]^4 \cdot \normtwo{\Z[\ell]}^3 \cdot \normtwo{\ZO[\ell]}^2.
        \end{equation}
    \end{enumerate}
\end{enumerate}
Then, combining \cref{lem:error-bound:twonorm} and \eqref{app:eq:error-control-one-step:0}--\eqref{app:eq:error-control-one-step:2b} yields
\begin{align}
    &\ \max\{\normtwo{\ZX[\ell+1] - \ZX[\ell]}, \normtwo{\ZS[\ell+1] - \ZS[\ell]}, \normtwo{\ZO[\ell+1]}\} \nonumber \\
    \le&\ \left\|{\mymat{\ZX[\ell+1] - \ZX[\ell] & \ZO[\ell+1]\tran \\ \ZO[\ell+1] & \ZS[\ell+1] - \ZS[\ell]}}\right\|_2 \nonumber \\
    \le&\ \left( 
        4 \etanew[\ell] + 2\etanew[\ell]^2 \cdot \normtwo{\Z[\ell]} + \etanew[\ell]^2 \cdot \normtwo{\Z[\ell]} + \frac{4}{3} \etanew[\ell]^2 \cdot \normtwo{\Z[\ell]} + \frac{4}{3} \etanew[\ell]^3 \cdot \normtwo{\Z[\ell]}^2 + \frac{4}{9} \etanew[\ell]^4 \cdot \normtwo{\Z[\ell]}^3
        \right) \cdot \normtwo{\ZO[\ell]} \nonumber \\
    =&\ \left( 
        \frac{4}{9}\etanew[\ell]^4 \cdot \normtwo{\Z[\ell]}^3 + \frac{4}{3} \etanew[\ell]^3 \cdot \normtwo{\Z[\ell]}^2 + \frac{13}{3} \etanew[\ell]^2 \cdot \normtwo{\Z[\ell]} + 4\etanew[\ell] 
        \right) \cdot \normtwo{\ZO[\ell]}^2. \label{app:eq:error-control-one-step:3}
\end{align}
Finally, notice that $\W[\ell]$ is skew-symmetric and $\exp(\W[\ell])$ is orthogonal. Thus, the eigenvalues of $\Z[\ell]$ remain the same for all $\ell \in \bbN$, so $\normtwo{\Z[\ell]} = \normtwo{\Z[0]}$ for all $\ell \in \bbN$. Therefore, replacing $\normtwo{\Z[\ell]}$ with $\normtwo{\Z[0]}$ in \eqref{app:eq:error-control-one-step:3} gives the desirable result.
\end{proof}

\subsection{Proof of \texorpdfstring{\cref{lem:error-bound:induction}}{Lemma 10}}
\label{app:sec:error-bound:induction}

\paragraph{A strengthened version.}
Here we prove a strenthened version of \eqref{eq:error-bound:induction}:
\begin{subequations} \label{app:eq:induction:strengthen}
    \begin{align}
        \normtwo{\ZO[\ell]} &\le \frac{1}{\normtwo{\Z[0]}} \cdot f(\etanew[0] \cdot \normtwo{\Z[0]}) \cdot \normtwo{\ZO[0]}^{\ell+1}  \label{app:eq:induction:strengthen-1} \\
        \normtwo{\ZX[\ell] - \ZX[0]} &\le \frac{1}{\normtwo{\Z[0]}} \cdot f(\etanew[0] \cdot \normtwo{\Z[0]}) \cdot \sum_{i=0}^{\ell-1} \normtwo{\ZO[0]}^{i+2} \label{app:eq:induction:strengthen-2} \\
        \normtwo{\ZS[\ell] - \ZS[0]} &\le \frac{1}{\normtwo{\Z[0]}} \cdot f(\etanew[0] \cdot \normtwo{\Z[0]}) \cdot \sum_{i=0}^{\ell-1} \normtwo{\ZO[0]}^{i+2}. \label{app:eq:induction:strengthen-3} 
    \end{align}
\end{subequations}
More specifically, \eqref{app:eq:induction:strengthen} implies \eqref{eq:error-bound:induction} because
\begin{align*}
    \frac{1}{\normtwo{\Z[0]}} \cdot f(\etanew[0] \cdot \normtwo{\Z[0]}) &\le \frac{1}{\normtwo{\ZX[0]}} \cdot f(\etanew[0] \cdot (\normtwo{\ZX[0]} + \normtwo{\ZS[0]} + C_K)) \\
    &\le \frac{1}{\normtwo{\ZX[0]}} \cdot f(\etanew[0] \cdot (\normtwo{\ZX[0]} + \normtwo{\ZS[0]} + \frac{1}{2})) = \alpha_K,
\end{align*}
following from the monotonicity of $f$, \cref{lem:error-bound:twonorm}, and the definition of $C_K$ \eqref{eq:error-bound:alphaK}: $\normtwo{\ZO[0]} \leq C_K \leq \frac{1}{2}$.

\paragraph{Proof by induction.}
Now we prove \eqref{app:eq:induction:strengthen} by induction.
\begin{enumerate}
    \item \textit{Base case.}
    When $\ell=1$, we see from $\normtwo{\ZO[0]} \leq C_K \leq \frac{3}{4\etanew[0]}$ and \cref{lem:error-bound:error-control-one-step} that
    \begin{align*}
        &\ \max\{\normtwo{\ZX[1] - \ZX[0]}, \ \normtwo{\ZS[1] - \ZS[0]}, \ \normtwo{\ZO[1]}\} \\
        \le&\ \frac{1}{\normtwo{\Z[0]}} \left( 
        \frac{4}{9} (\etanew[0] \normtwo{\Z[0]})^4 + \frac{4}{3} (\etanew[0] \normtwo{\Z[0]})^3 + \frac{13}{3} (\etanew[0] \normtwo{\Z[0]})^2 + 4 (\etanew[0] \normtwo{\Z[0]}) 
        \right) \normtwo{\ZO[0]}^2 \\
        \le&\ \frac{1}{\normtwo{\Z[0]}} f(\etanew[0] \cdot \normtwo{\Z[0]}) \cdot \normtwo{\ZO[0]}^2.
    \end{align*}

    \item \textit{Induction.}
    Suppose \eqref{app:eq:induction:strengthen} holds for index $\ell$. The proof for the induction case $\ell+1$ takes the following three steps.
    \begin{enumerate}
        \item First, \eqref{app:eq:induction:strengthen-2} implies that
        \begin{subequations}
        \begin{align}
            \lammin{\ZX[\ell]} &\ge \lammin{\ZX[0]} + \lammin{\ZX[\ell] - \ZX[0]} \label{app:eq:induction:lamX-1} \\
            & \ge \lammin{\ZX[0]} - \normtwo{\ZX[\ell] - \ZX[0]} \nonumber \\
            & \ge \lammin{\ZX[0]} - \frac{f(\etanew[0] \cdot \normtwo{\Z[0]})}{\normtwo{\Z[0]}} \cdot \sum_{i=0}^{\ell-1} \normtwo{\ZO[0]}^{i+2} \label{app:eq:induction:lamX-2}  \\
            & \ge \lammin{\ZX[0]} - \frac{f(\etanew[0] \cdot \normtwo{\Z[0]})}{\normtwo{\Z[0]}} \cdot \frac{\normtwo{\ZO[0]}^2}{1 - \normtwo{\ZO[0]}} \nonumber \\
            & \ge \lammin{\ZX[0]} - \frac{f(\etanew[0] \cdot \normtwo{\Z[0]})}{\normtwo{\Z[0]}} \cdot \frac{1}{\alpha_K} \cdot \frac{\frac{1}{4} \lammin{\ZX[0]}}{1 - \frac{1}{2}} \label{app:eq:induction:lamX-3} \\
            & \ge \lammin{\ZX[0]} - \frac{1}{2} \lammin{\ZX[0]} \label{app:eq:induction:lamX-4} \\
            &= \frac{1}{2} \lammin{\ZX[0]}, \label{app:eq:induction:lamX-5}
        \end{align}
        \end{subequations}
        where~\eqref{app:eq:induction:lamX-1} comes from Weyl's inequality,~\eqref{app:eq:induction:lamX-2} holds by the induction hypothesis, \eqref{app:eq:induction:lamX-3} follows from the definition of $C_K$:
        \[
            \normtwo{\ZO[0]} \le C_K \le \frac{1}{2}, \qquad
            \normtwo{\ZO[0]} \le C_K \le \sqrt{\frac{1}{4\alpha_K} \lammin{\ZX[0]}},
        \]
        and~\eqref{app:eq:induction:lamX-4} holds since $\alpha_K \ge \frac{1}{\normtwo{\Z[0]}} \cdot f(\etanew[0] \normtwo{\Z[0]})$. Similarly, we obtain 
        \begin{align*}
            \lammin{\ZX[\ell]} &\le \lammin{\ZX[0]} + \lammax{\ZX[\ell] - \ZX[0]} \\
            &\le \lammin{\ZX[0]} + \normtwo{\ZX[\ell] - \ZX[0]} \\
            &\le \frac{3}{2} \lammin{\ZX[0]},
        \end{align*}
        and
        \[
            \frac{3}{2} \lammax{\ZS[0]} \le \lammax{\ZS[0]} \le \frac{1}{2} \lammax{\ZS[0]}.
        \]
        Rearranging to get
        \begin{align*}
            \deltalam{\ZX[\ell]}{\ZS[\ell]} &\ge \frac{1}{2} \big(\deltalam{\ZX[0]}{\ZS[0]}\big) \\
            \deltalam{\ZX[\ell]}{\ZS[\ell]} &\le \frac{3}{2} \big(\deltalam{\ZX[0]}{\ZS[0]} \big),
        \end{align*}
        or equivalently,
        \begin{equation} \label{app:eq:induction:eta-bound}
            \frac{2}{3} \etanew[0] \le \etanew[\ell] \le 2 \etanew[0].
        \end{equation}

        \item Second, we show \eqref{app:eq:induction:strengthen-1} for index $\ell+1$:
        \begin{subequations}
        \begin{align}
            \normtwo{\ZO[\ell+1]} &\le \frac{f(\etanew[\ell] \cdot \normtwo{\Z[0]})}{\normtwo{\Z[0]}} \cdot \normtwo{\ZO[\ell]}^2 \label{app:eq:induction:ZO-1} \\
            &\le \frac{f(2\etanew[0] \cdot \normtwo{\Z[0]})}{\normtwo{\Z[0]}} \cdot \normtwo{\ZO[\ell]}^2 \label{app:eq:induction:ZO-2} \\
            &\le \frac{f(2\etanew[0] \cdot \normtwo{\Z[0]})}{\normtwo{\Z[0]}} \cdot \frac{f(\etanew[0] \cdot \normtwo{\Z[0]})^2}{\normtwo{\Z[0]}^2} \cdot \normtwo{\ZO[0]}^{2\ell+2}, \label{app:eq:induction:ZO-3}
        \end{align}
        \end{subequations}
        In~\eqref{app:eq:induction:ZO-1}, we use \cref{lem:error-bound:error-control-one-step}, the fact $\ZX[\ell] \in \Sympp{r}$ and $\ZS[\ell] \in \Symnn{n-r}$ from (a), and the fact that for any integer $\ell \ge 1$, we have
        \begin{equation}
            \label{app:eq:induction:eta-times-ZO}
            \etanew[\ell] \normtwo{\ZO[\ell]} \le 2 \etanew[0] \alpha_K \normtwo{\ZO[0]}^{\ell+1} \le 2 \etanew[0] \alpha_K \normtwo{\ZO[0]}^2 \le \frac{3}{4}
        \end{equation}
        since $\normtwo{\ZO[0]} \le C_K \le \min \{\frac{1}{2}, \frac{1}{\sqrt{4\etanew[0] \alpha_K}}\}$. Then,~\eqref{app:eq:induction:ZO-2} uses~\eqref{app:eq:induction:eta-bound}, and~\eqref{app:eq:induction:ZO-3} uses the induction hypothesis.

        On the other hand, we see from the definition of $C_K$ that
        \begin{subequations}
        \begin{align}
            & \normtwo{\ZO[0]} \le C_K \le g^{-1}(\normtwo{\ZX[0]}^2) \nonumber \\
            \Longleftrightarrow \; & \normtwo{\ZO[0]} \cdot f\left( 2\etanew[0] (\normtwo{\ZX[0]} + \normtwo{\ZS[0]} + \normtwo{\ZO[0]}) \right) \nonumber \\
            & \cdot f \big( \etanew[0] (\normtwo{\ZX[0]} + \normtwo{\ZS[0]} + \normtwo{\ZO[0]}) \big) \le \normtwo{\ZX[0]}^2 \label{app:eq:induction:ginv-1} \\
            \Longrightarrow \; & \normtwo{\ZO[0]} \cdot f\left( 2\etanew[0] \normtwo{\Z[0]} \right) \cdot f\left( \etanew[0] \normtwo{\Z[0]} \right) \le \normtwo{\ZX[0]}^2 \label{app:eq:induction:ginv-2} \\
            \Longrightarrow \; & \normtwo{\ZO[0]} \cdot f\big(2\etanew[0] \normtwo{\Z[0]} \big) \cdot f\big(\etanew[0] \normtwo{\Z[0]} \big) \le \normtwo{\Z[0]}^2 \nonumber \\
            \Longleftrightarrow \; & \normtwo{\ZO[0]}^{\ell+3} \cdot \frac{f(2\etanew[0] \normtwo{\Z[0]})}{\normtwo{\Z[0]}} \cdot \frac{f(\etanew[0] \normtwo{\Z[0]})^2}{\normtwo{\Z[0]}^2} \le \frac{f(\etanew[0] \normtwo{\Z[0]})}{\normtwo{\Z[0]}} \cdot \normtwo{\ZO[0]}^{\ell+2} \nonumber \\
            \Longrightarrow \; & \normtwo{\ZO[0]}^{2\ell+2} \cdot \frac{f(2\etanew[0] \normtwo{\Z[0]})}{\normtwo{\Z[0]}} \cdot \frac{f(\etanew[0] \normtwo{\Z[0]})^2}{\normtwo{\Z[0]}^2} \le \frac{f(\etanew[0] \normtwo{\Z[0]})}{\normtwo{\Z[0]}} \cdot \normtwo{\ZO[0]}^{\ell+2}, \label{app:eq:induction:ginv-3} 
        \end{align}
        \end{subequations}
        where~\eqref{app:eq:induction:ginv-1} follows from the monotonicity of $g$,~\eqref{app:eq:induction:ginv-2} from that of~$f$, \eqref{app:eq:induction:ginv-3} uses the definition of $C_K$ (so $\normtwo{\ZO[0]} \le C_K \le \frac{1}{2} < 1$) and the fact $2\ell + 2 \ge \ell + 3$. Thus, 
        \[
            \normtwo{\ZO[\ell+1]} \le \frac{f(\etanew[0] \normtwo{\Z[0]})}{\normtwo{\Z[0]}} \cdot \normtwo{\ZO[0]}^{\ell+2}.
        \]

        \item It remains to prove \eqref{app:eq:induction:strengthen-2} and \eqref{app:eq:induction:strengthen-3} for index $\ell+1$. From \cref{lem:error-bound:error-control-one-step} and \eqref{app:eq:induction:strengthen-1}, we have
        \begin{align*}
            &\ \normtwo{\ZX[\ell+1] - \ZX[0]} \\
            \le&\ \normtwo{\ZX[\ell+1] - \ZX[\ell]} + \normtwo{\ZX[\ell] - \ZX[0]} \\
            \le&\ \frac{f(\etanew[\ell] \cdot \normtwo{\Z[0]})}{\normtwo{\Z[0]}} \cdot \normtwo{\ZO[\ell]}^2 + \frac{f(\etanew[0] \cdot \normtwo{\Z[0]})}{\normtwo{\Z[0]}} \cdot \left( \sum_{i=0}^{\ell-1} \normtwo{\ZO[0]}^{i+2} \right) \\
            \le&\ \frac{f(2\etanew[0] \cdot \normtwo{\Z[0]})}{\normtwo{\Z[0]}} \cdot \normtwo{\ZO[\ell]}^2 + \frac{f(\etanew[0] \cdot \normtwo{\Z[0]})}{\normtwo{\Z[0]}} \cdot \left( \sum_{i=0}^{\ell-1} \normtwo{\ZO[0]}^{i+2} \right) \\
            \le&\ \frac{f(2\etanew[0] \normtwo{\Z[0]})}{\normtwo{\Z[0]}} \frac{f(\etanew[0] \normtwo{\Z[0]})^2}{\normtwo{\Z[0]}^2} \cdot \normtwo{\ZO[0]}^{2\ell+2} + \frac{f(\etanew[0] \normtwo{\Z[0]})}{\normtwo{\Z[0]}} \left( \sum_{i=0}^{\ell-1} \normtwo{\ZO[0]}^{i+2} \right) \\
            \le&\ \frac{f(\etanew[0] \cdot \normtwo{\Z[0]})}{\normtwo{\Z[0]}} \cdot \normtwo{\ZO[0]}^{\ell+2} + \frac{f(\etanew[0] \cdot \normtwo{\Z[0]})}{\normtwo{\Z[0]}} \cdot \left( \sum_{i=0}^{\ell-1} \normtwo{\ZO[0]}^{i+2} \right) \\
            =&\ \frac{f(\etanew[0] \cdot \normtwo{\Z[0]})}{\normtwo{\Z[0]}} \cdot \left( \sum_{i=0}^{\ell} \normtwo{\ZO[0]}^{i+2} \right),
        \end{align*}
        where the last inequality follows from~\eqref{app:eq:induction:ginv-3}.
    \end{enumerate}
\end{enumerate}
So we finish the proof of the strengthened inequalities \eqref{app:eq:induction:strengthen}, which implies~\eqref{eq:error-bound:induction}. From the above mathematical induction process, we can also conclude that
\[
    \frac{2}{3} \etanew[0] \le \etanew[\ell] \le 2 \etanew[0], \quad \lammin{\ZX[\ell]} \ge \frac{1}{2} \lammin{\ZX[0]} > 0 \quad \lammax{\ZS[\ell]} \le \frac{1}{2} \lammax{\ZS[0]} < 0,
\]
which follows directly from~\eqref{app:eq:induction:lamX-4} and~\eqref{app:eq:induction:eta-bound}.

\subsection{Proof of \texorpdfstring{\cref{lem:error-bound:distance-Yk}}{Lemma 11}}
\label{app:sec:error-bound:distance-Yk}

First of all, we have
\begin{equation} \label{app:eq:distance-Yk:W-2}
    \normtwo{\W[\ell]} = \normtwo{\WO[\ell]} \overset{\text{(a)}}{\le} \etanew[\ell] \normtwo{\ZO[\ell]} \le 2\etanew[0] \normtwo{\ZO[\ell]} \overset{\text{(b)}}{\le} 2\etanew[0] \alpha_K \normtwo{\ZO[0]}^{\ell + 1} \overset{\text{(c)}}{\le} \frac{1}{2},
\end{equation}
where step (a) follows from~\eqref{app:eq:error-control-one-step:W-bound}, step (b) holds since 
$\ZO[0] \le \min \{ C_K, \frac{1}{\sqrt{4 \etanew[0] \alpha_K}} \}$ for any integer $\ell \ge 1$, and step (c) holds for the same reason as in~\eqref{app:eq:induction:ZO-1}.
Then, we see that
\begin{subequations}
\begin{align}
    \normtwo{\Y[\ell+1] - \Y[\ell]} &= \normtwo{\Y[\ell] \left( \exp(\W[\ell]) - \I[n] \right)} \nonumber \\
    &\le \normtwo{\Y[\ell]} \cdot \normtwo{\exp(\W[\ell]) - \I[n]} \nonumber \\
    &\le \normtwo{\exp(\W[\ell]) - \I[n]} \label{app:eq:distance-Yk:Yk-1} \\
    &\le \normtwo{\exp (\W[\ell]) - (\I[n] + \W[\ell])} + \normtwo{\W[\ell]} \nonumber \\
    &\le \frac{2}{3} \normtwo{\W[\ell]}^2 + \normtwo{\W[\ell]} \le \frac{4}{3} \normtwo{\W[\ell]} \label{app:eq:distance-Yk:Yk-2}  \\
    &\le \frac{4}{3} \etanew[\ell] \normtwo{\ZO[\ell]} \le \frac{8}{3} \etanew[0] \normtwo{\ZO[\ell]} \nonumber \\
    &\le \frac{8}{3} \etanew[0] \alpha_K \cdot \normtwo{\ZO[0]}^{\ell+1}, \nonumber
\end{align}
\end{subequations}
where~\eqref{app:eq:distance-Yk:Yk-1} holds because $\Y[\ell] = \prod_{i=1}^{\ell} \exp(\W[i])$ is orthogonal,~\eqref{app:eq:distance-Yk:Yk-2} uses \cref{lem:error-bound:exp-bound} and~\eqref{app:eq:distance-Yk:W-2}. 

For the second inequality in the lemma, we have
\begin{subequations}
\begin{align}
    \normtwo{\Y[\ell] - (\I[n] + \W[0])} &= \left\|{\Y[1] - (\I[n]+\W[0]) + \sum_{i=1}^{\ell-1} (\Y[i+1] - \Y[i])}\right\|_2 \nonumber \\
    &\le \normtwo{\Y[1] - (\I[n]+\W[0])} + \sum_{i=1}^{\ell-1} \normtwo{\Y[i+1] - \Y[i]} \nonumber \\
    &\le \frac{2}{3} \normtwo{\W[0]}^2 + \frac{8}{3} \etanew[0] \alpha_K \sum_{i=1}^{\ell-1} \normtwo{\ZO[0]}^{i+1} \label{app:eq:distance-Yk-final:1} \\
    & \le \frac{2}{3}\etanew[0]^2 \normtwo{\ZO[0]}^2 + \frac{8}{3} \etanew[0] \alpha_K \sum_{i=1}^{\ell-1} \normtwo{\ZO[0]}^{i+1}, \label{app:eq:distance-Yk-final:2}
\end{align}
\end{subequations}
where~\eqref{app:eq:distance-Yk-final:1} holds because of $\etanew[\ell] \normtwo{\ZO[\ell]} \le \frac{3}{4}$ for all $\ell \in \bbN$ from~\eqref{app:eq:induction:eta-times-ZO} and~\eqref{eq:error-bound:distance-Yk-1}, and~\eqref{app:eq:distance-Yk-final:2} follows from~\eqref{app:eq:error-control-one-step:W-bound}.

\subsection{Proof of \texorpdfstring{\cref{lem:error-bound:V-limit}}{Lemma~12}}
\label{app:sec:error-bound:V-limit}

We first prove that the limit $\V[\infty] := \lim_{\ell \rightarrow \infty} \V[\ell]$ exists. From \cref{lem:error-bound:induction}, we already know that $\ZO[\ell] \rightarrow 0$ as $\ell \rightarrow \infty$. It remains to show that $\ZX[\ell]$ and $\ZS[\ell]$ are also convergent. For any $\epsilon > 0$, we define
\[
    N := \ceil{
    \frac{1}{2} \log_{\normtwo{\ZO[0]}} \left( \frac{\epsilon}{\beta} (1 - \normtwo{\ZO[0]}) \right) - 1}, \quad \text{where} \ \beta = \alpha_K^2 f(2\eta_0 \normtwo{\Z[0]}).
\]
Then, for any integers $m, n \ge N$, we have 
\begin{subequations}
\begin{align}
    &\ \normtwo{\ZX[n] - \ZX[m]} \nonumber \\
    \le &\ \sum_{\ell=m}^{n-1} \normtwo{\ZX[\ell+1] - \ZX[\ell]} \nonumber \\
    \le &\ \sum_{\ell=m}^{n-1} \frac{1}{\normtwo{\Z[0]}} \cdot f(\etanew[\ell] \cdot \normtwo{\Z[0]}) \cdot \normtwo{\ZO[\ell]}^2 \label{app:eq:V-limit:cauchy-1} \\
    \le &\ \sum_{\ell=m}^{n-1} \frac{f(2\etanew[0] \cdot (\normtwo{\ZX[0]} + \normtwo{\ZS[0]} + 0.5))}{\normtwo{\ZX[0]}}  \cdot \normtwo{\ZO[\ell]}^2 \label{app:eq:V-limit:cauchy-2} \\
    \le &\ \sum_{\ell=m}^{n-1} \frac{f(2\etanew[0] \cdot (\normtwo{\ZX[0]} + \normtwo{\ZS[0]} + 0.5))}{\normtwo{\ZX[0]}} \cdot \alpha_K^2 \cdot \normtwo{\ZO[0]}^{2\ell+2} \label{app:eq:V-limit:cauchy-3} \\
    \le &\ \sum_{\ell=m}^{\infty} \beta \cdot \normtwo{\ZO[0]}^{2\ell+2} 
    \le \beta \cdot \frac{\normtwo{\ZO[0]}^{2m+2}}{1 - \normtwo{\ZO[0]}} \le \beta \cdot \frac{\normtwo{\ZO[0]}^{2N+2}}{1 - \normtwo{\ZO[0]}} \le \epsilon, \label{app:eq:V-limit:cauchy-4}
\end{align}
\end{subequations}
Here,~\eqref{app:eq:V-limit:cauchy-1} holds because of \cref{lem:error-bound:error-control-one-step} and the fact that $\etanew[\ell] \cdot \normtwo{\ZO[\ell]} \le \frac{3}{4}$, for any integer $\ell \ge 1$ (see~\eqref{app:eq:induction:eta-times-ZO}), ~\eqref{app:eq:V-limit:cauchy-2} uses \cref{lem:error-bound:twonorm}, the monotonicity of $f$, and the fact $\normtwo{\ZO[0]} \le C_K \le \frac{1}{2}$,~\eqref{app:eq:V-limit:cauchy-3} uses~\eqref{eq:error-bound:induction-1}, and~\eqref{app:eq:V-limit:cauchy-4} follows directly from the definition of $\beta$ and $N$.

Thus, $\left\{ \ZX[\ell] \right\}_{\ell=1}^{\infty} \subseteq \Sym{r}$ is a Cauchy sequence and the limit $\ZX[\infty] := \lim_{\ell \rightarrow \infty} \ZX[\ell]$ exists. The same argument applies to $\ZS[\infty]$. Similarly, from \cref{lem:error-bound:distance-Yk}, we conclude that $\left\{ \Y[\ell] \right\}_{\ell=1}^\infty$ is convergent to $\Y[\infty]$, and thus the limit
\begin{align*}
    \V[\infty] := \lim_{\ell \rightarrow \infty} \Y[\ell] \mymat{\ZX[\ell] & 0 \\ 0 & 0} \Y[\ell]\tran = \Y[\infty] \mymat{\ZX[\infty] & 0 \\ 0 & 0} \Y[\infty]\tran 
\end{align*}
exists.

It only remains to show that $\V[\infty] = \PiSnp{\Z + \H}$.  It readily follows from
\begin{align*}
    \Z[\ell] &= \exp(\W[\ell])\tran \mymat{\ZX[\ell] & \ZO[\ell]\tran \\ \ZO[\ell] & \ZS[\ell]} \exp(\W[\ell]) \\
    &= \left( \exp(\W[\ell-1])\exp(\W[\ell]) \right)\tran \mymat{\ZX[\ell-1] & \ZO[\ell-1]\tran \\ \ZO[\ell-1] & \ZS[\ell-1]} \exp(\W[\ell-1])\exp(\W[\ell]) \\
    &\;\; \vdots \\
    &= \left( \prod_{i=0}^{\ell} \exp(\W[i]) \right)\tran \mymat{\ZX[0] & \ZO[0]\tran \\ \ZO[0] & \ZS[0]} \left( \prod_{i=0}^{\ell} \exp(\W[i]) \right) \\
    &= \Y[\ell]\tran (\Z + \H) \Y[\ell].
\end{align*}
Thus, we conclude that
\[
    \Z[\infty] = \lim_{\ell \rightarrow \infty} \Z[\ell] = \Y[\infty]\tran (\Z + \H) \Y[\infty].
\]
Recall that for any $\ell \in \bbN$, $\exp(\W[\ell])$ is orthogonal, so $\Y[\infty] = \lim_{\ell \rightarrow \infty} \prod_{i=0}^{\ell} \exp(\W[i])$ is also orthogonal. It implies that
\[
    \PiSnp{\Z + \H} = \Y[\infty] \cdot \PiSnp{\Z[\infty]} \cdot \Y[\infty]\tran.
\]
On the other hand, recall from \cref{lem:error-bound:induction} that $\lammin{\ZX[\ell]} \ge \frac{1}{2} \lammin{\ZX[0]} > 0$ and $-\lammax{\ZS[\ell]} \ge -\frac{1}{2} \lammax{\ZS[0]} > 0$, so we have $\ZX[\infty] \in \Sympp{r}$, $-\ZS[\infty] \in \Symnn{n-r}$, and
\begin{align*}
    \PiSnp{\Z[\infty]} &= \Pi_{\Sympp{n}} \left(\mymat{\ZX[\infty] & 0 \\ 0 & \ZS[\infty]}\right) = \mymat{\ZX[\infty] & 0 \\ 0 & 0}, \\
    \PiSnp{\Z + \H} &= \Y[\infty] \mymat{\ZX[\infty] & 0 \\ 0 & 0} \Y[\infty]\tran = \V[\infty].
\end{align*}
This concludes the proof.

\subsection{Proof of \texorpdfstring{\cref{lem:error-bound:lower-bound-C-eta-alpha}}{Lemma 13}}
\label{app:sec:error-bound:lower-bound-C-eta-alpha}

We prove the four conclusions one-by-one.
\begin{enumerate}
    \item Since $\normtwo{\H} \le \frac{1}{2} \min \{\lam{r}, -\lam{r+1}\}$, we have 
    \[
        \lammin{\ZX[0]} \ge \lammin{\LamX} - \normtwo{\ZX} \ge \lam{r} - \normtwo{\H} \ge \frac{1}{2} \lam{r} > 0
    \]
    by Weyl's inequality and \cref{lem:error-bound:twonorm}. Symmetrically, we obtain $\lammax{\ZS[0]} \le \frac{1}{2} \lam{r+1} < 0$. Similarly, we can obtain an upper bound for $\lammin{\ZX[0]}$ and a lower bound for $\lammax{\ZS[0]}$:
    \[
        \lammin{\ZX[0]} \le \frac{3}{2} \lam{r}, \qquad \lammax{\ZS[0]} \le \frac{3}{2} \lam{r+1}.
    \]

    \item Now we bound $\etanew[0]$. On one hand, we have
    \begin{equation}
        \label{app:eq:lower-bound-C-eta-alpha:eta0f}
        \etanew[0] = \frac{d}{\deltalam{\ZX[0]}{\ZS[0]}} \le \frac{d}{\frac{1}{2} \lam{r} - \frac{1}{2} \lam{r+1}} = \frac{2d}{\lam{r} - \lam{r+1}} =: \etanew[0,f].
    \end{equation}
    On the other hand, we have an lower bound for $\etanew[0]$:
    \[
        \etanew[0] \ge \frac{d}{1.5 \lam{r} - 1.5 \lam{r+1}} = \frac{2d}{3(\lam{r} - \lam{r+1})}.
    \]

    \item To bound $\alpha_K$, we first notice from $\normtwo{H} \le \frac{1}{2} \min \{ \lam{r}, -\lam{r+1} \}$ that
    \begin{align*}
        \normtwo{\Z + \H} &\ge \normtwo{\Z} - \normtwo{\H} \ge \normtwo{\Z} - 0.5 \min \left\{ \lam{r}, -\lam{r+1} \right\}, \\
        \normtwo{\Z + \H} &\le \normtwo{\Z} + \normtwo{\H} \le \normtwo{\Z} + 0.5 \min \{ \lam{r}, -\lam{r+1} \}.
    \end{align*}
    Then, from the definition of $\alpha_K$~\eqref{eq:error-bound:alphaK}, we deduce that
    \begin{align}
        \alpha_K &= \frac{f\left( 
                \etanew[0] \cdot (\normtwo{\ZX[0]} + \normtwo{\ZS[0]} + 0.5)
             \right)}{\normtwo{\ZX[0]}} \nonumber \\
        &\le \frac{f\left( 
            \frac{2d}{3(\lam{r} - \lam{r+1})} \cdot (2(\lam{1} - \lam{n}) + \min \left\{ \lam{r}, -\lam{r+1} \right\} + 0.5)
         \right)}{
            \lam{1} - 0.5 \lam{r}
         } \nonumber \\
        &=: \alpha_{K,f}, \label{app:eq:lower-bound-C-eta-alpha:alphaKf}
    \end{align}
    where the inequality uses Weyl's inequality:
    \begin{align*}
        \normtwo{\Z[0]} &\ge \max\{\normtwo{\ZX[0]}, \normtwo{\ZS[0]}\}, \\
        \normtwo{\Z[0]} &\le \normtwo{\Z} + \frac{1}{2} \min \{\lam{r}, -\lam{r+1} \} \le (\lam{1} - \lam{n}) + \frac{1}{2} \min \{\lam{r}, -\lam{r+1}\} 
    \end{align*}
    and recall $\normtwo{\ZX[0]} \ge \lam{1} - \frac{1}{2} \lam{r}$.

    \item It remains to find a lower bound for $C_K$. Since $C_K$ is the minimum of five terms, we bound each of them one-by-one.
    \begin{itemize}
        \item $C_1=\frac{1}{2}$, which is trivial.
        
        \item $C_2 = g^{-1} (\normtwo{\ZX[0]}^2)$. We see from the definition of $g$ that
        \begin{align*}
            g(y) &= y \cdot f\left( 2\etanew[0] \cdot (\normtwo{\ZX[0]} + \normtwo{\ZS[0]} + y) \right) \cdot f\left( \etanew[0] \cdot (\normtwo{\ZX[0]} + \normtwo{\ZS[0]} + y) \right) \\
            &\le y \cdot f\left( 2\etanew[0,f] \cdot (2(\lam{1} - \lam{n}) + \min \left\{ \lam{r}, -\lam{r+1} \right\} + y) \right) \\
            &\phantom{\le} \ \cdot f\left( \etanew[0,f] \cdot (2(\lam{1} - \lam{n}) + \min \left\{ \lam{r}, -\lam{r+1} \right\} + y) \right) =: g_1(y).
        \end{align*}
        It is clear that $g_1$ is monotonically increasing on $[0, \infty)$ and $g_1(0)=0$. Combined with the fact that $\normtwo{\ZX[0]} \geq \lam{1} - \frac{1}{2} \lam{r}$, we conclude that
        \[
            g^{-1}(\normtwo{\ZX[0]}^2) \ge g_1^{-1}(\normtwo{\ZX[0]}^2) \ge g_1^{-1}((\lam{1} - 0.5\lam{r})^2) =: C_{2,f} > 0,
        \]
        which serves as a lower bound for $C_2$.

        \item $C_3=\frac{3}{8\etanew[0]} \geq \frac{3}{8\etanew[0,f]} =: C_{3,f}$.

        \item $C_4=\frac{1}{\sqrt{4\etanew[0] \alpha_K}} \geq \frac{1}{\sqrt{4\etanew[0,f] \alpha_{K,f}}} =: C_{4,f}$.
        
        \item Lastly, we have
        \begin{align*}
            C_5 &= \sqrt{\frac{1}{4\alpha_K} \min\{\lammin{\ZX[0]}, - \lammax{\ZS[0]}\}} \\
            &\ge \sqrt{\frac{1}{4\alpha_{K,f}} \min\{\frac{1}{2} \lammin{\ZX[0]}, - \frac{1}{2} \lammax{\ZS[0]}\}} =: C_{5,f}.
        \end{align*}
    \end{itemize}
    Therefore, $C_{K,f} := \min\{\frac{1}{2}, C_{2,f}, C_{3,f}, C_{4,f}, C_{5,f}\}$ serves as a lower bound for $C_K$.
\end{enumerate}

\subsection{Proof of \texorpdfstring{\cref{lem:error-bound:first-sylvester-eq}}{Lemma 14}}
\label{app:sec:error-bound:first-sylvester-eq}

We first show $\Theta_0 \circ \HO$ is the solution of 
\[
    \WO \LamX - \LamS \WO = \HO.
\]
To see this, expand both sides of the equity and consider the $(i,j)$th element, $i \in \{r+1, \ldots, n\}$ and $j \in \{1, \ldots, r\}$:
\[
    \WO[i,j] \cdot \lam{j} - \lam{i} \cdot \WO[i,j] = \HO[i,j] \; \Longrightarrow \; \WO[i,j] = \frac{1}{\lam{j} - \lam{i}} \HO[i,j] \; \Longrightarrow \; \WO = \Theta_0 \circ \HO.
\]

Second, we show that the perturbations $\HX$ and $\HS$ only affect the second- and higher-order terms. To see this, we explicitly write down $\WO[0]$ as:
\begin{align*}
    & \WO[0] (\LamX + \HX) - (\LamS + \HS) \WO[0] = \HO \\
    \Longrightarrow \quad & \text{vec}(\WO[0]) = (A + \dA)^{-1} \text{vec} (\HO),
\end{align*}
where $A := \I[r] \otimes (-\LamS) + \LamX \otimes \I[n-r]$ and $\dA := \I[r] \otimes (-\HS) + \HX \otimes \I[n-r]$. We bound $\normtwo{\dA}$ as
\begin{align}
    \normtwo{\dA} \le \normF{\dA} &= \normF{\I[r] \otimes (-\HS) + \HX \otimes \I[n-r]} \nonumber \\
    &\le \normF{\I[r] \otimes (-\HS)} + \normF{\HX \otimes \I[n-r]} \nonumber \\
    &= \normF{\I[r]} \cdot \normF{\HS} + \normF{\I[n-r]} \cdot \normF{\HX} \nonumber \\
    &= r \normF{\HS} + (n-r) \normF{\HX} \nonumber \\
    &\le n (\normF{\HX} + \normF{\HS}), \label{app:eq:first-sylvester-eq:dA-normtwo}
\end{align}
and bound $\normtwo{A^{-1}}$ as
\begin{equation} \label{app:eq:first-sylvester-eq:Ainv-normtwo}
    \normtwo{A^{-1}} = \big(\lammin{\I[r] \otimes (-\LamS) + \LamX \otimes \I[n-r]} \big)^{-1} = \frac{1}{\lam{r} - \lam{r+1}},
\end{equation}
which follows from~\cite[Theorem 2.5]{schacke2004thesis-kronecker-product} and the fact that $A \in \Sympp{n}$. Combining \eqref{app:eq:first-sylvester-eq:dA-normtwo} and \eqref{app:eq:first-sylvester-eq:Ainv-normtwo} gives
\begin{equation}
    \normtwo{A^{-1} \dA} \le \frac{n}{\lam{r} - \lam{r+1}} (\normF{\HX} + \normF{\HS}) \le \frac{n d}{\lam{r} - \lam{r+1}} (\normtwo{\HX} + \normtwo{\HS}) \le \frac{1}{2} \label{app:eq:first-sylvester-eq:Ainv-dA-normtwo}
\end{equation}
where the last inequliaty is from the assumption of the lemma. Therefore, we have 
\begin{subequations}
\begin{align}
    &\ \normtwo{\WO[0] - \Theta_0 \circ \HO} \nonumber \\
    \le&\ \normF{\WO[0] - \Theta_0 \circ \HO} \nonumber \\
    =&\ \normtwo{\vvec(\WO[0]) - \vvec(\Theta_0 \circ \HO)} \nonumber \\
    =&\ \normtwo{(A + \dA)^{-1} \vvec(\HO) - A^{-1} \vvec(\HO)} \nonumber \\
    \le&\ \normtwo{A^{-1} \vvec(\HO)} \cdot \normtwo{(I + A^{-1}\dA)^{-1} - I} \nonumber \\
    =&\ \normtwo{A^{-1} \vvec(\HO)} \cdot \left\|\sum_{i=0}^{\infty} (-A^{-1}\dA)^i - I\right\|_2 \label{app:eq:first-sylvester-eq:bound-1} \\
    \le&\ \normtwo{A^{-1} \vvec(\HO)} \cdot \sum_{i=1}^\infty \normtwo{A^{-1}\dA}^i \nonumber \\
    \le&\ \normtwo{A^{-1}} \cdot \normtwo{\vvec(\HO)} \cdot \frac{\normtwo{A^{-1}\dA}}{1 - \normtwo{A^{-1}\dA}} \nonumber \\
    \le&\ \frac{1}{\lam{r} - \lam{r+1}} \cdot \normtwo{\HO} \cdot 2 \cdot \frac{1}{\lam{r} - \lam{r+1}} \cdot \normtwo{\dA} \label{app:eq:first-sylvester-eq:bound-2} \\
    \le&\ \frac{2nd}{(\lam{r} - \lam{r+1})^2} \cdot \normtwo{\HO} \cdot (\normtwo{\HX} + \normtwo{\HS}). \label{app:eq:first-sylvester-eq:bound-3}
\end{align}
\end{subequations}
In~\eqref{app:eq:first-sylvester-eq:bound-1}, the expansion of the Neumann series is valid because $\normtwo{A^{-1}\dA} < \frac{1}{2}$ (see~\eqref{app:eq:first-sylvester-eq:Ainv-normtwo});~\eqref{app:eq:first-sylvester-eq:bound-2} uses~\eqref{app:eq:first-sylvester-eq:Ainv-normtwo} and~\eqref{app:eq:first-sylvester-eq:Ainv-dA-normtwo}; and finally~\eqref{app:eq:first-sylvester-eq:bound-3} uses~\eqref{app:eq:first-sylvester-eq:dA-normtwo}.


\section{Additional Numerical Results}
\label{app:sec:exp}

\Cref{app:fig:hamming} presents additional numerical results for the Hamming set problems \cite{abor1999dataset-dimacs}. \Cref{app:fig:BQP-r1} reports additional examples for the BQP problems, with $c \sim \calN(0, \I[n])$ and random (standard Gaussian) initial guess, and \cref{app:fig:BQP-r1-zero} reports for the same problems as in \cref{app:fig:BQP-r1}, but with all-zeros initialization. Last, \cref{app:fig:BQP-r2} shows numerical results for BQP problems with $c = 0$ and random (standard Gaussian) initial guess.


\begin{figure}[htbp]
    \centering

    \begin{minipage}{\textwidth}
        \centering
        \hspace{5mm} \includegraphics[width=0.35\columnwidth]{figs/legends/legend_rmax_dZ.png}
    \end{minipage}

    \begin{minipage}{\textwidth}
        \centering
        \begin{tabular}{ccc}
            \begin{minipage}{0.30\textwidth}
                \centering
                \includegraphics[width=\columnwidth]{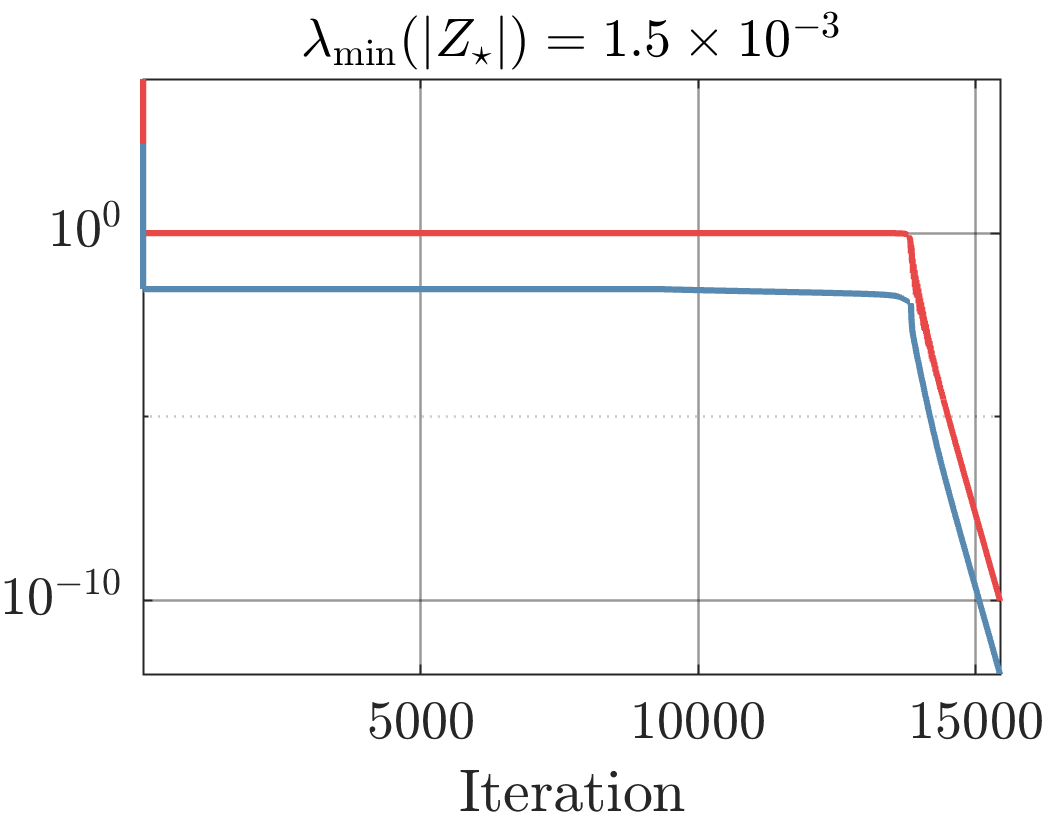}
                \texttt{hamming-10-2}
            \end{minipage}

            \begin{minipage}{0.30\textwidth}
                \centering
                \includegraphics[width=\columnwidth]{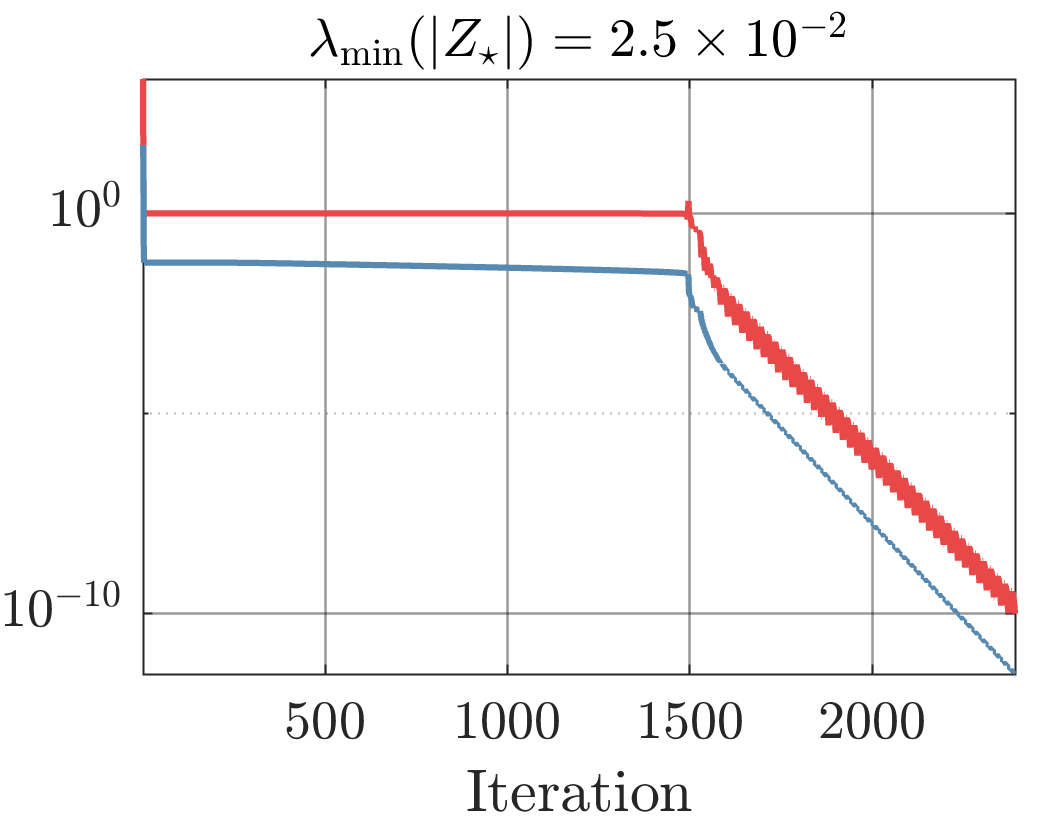}
                \texttt{hamming-8-3-4}
            \end{minipage}

            \begin{minipage}{0.30\textwidth}
                \centering
                \includegraphics[width=\columnwidth]{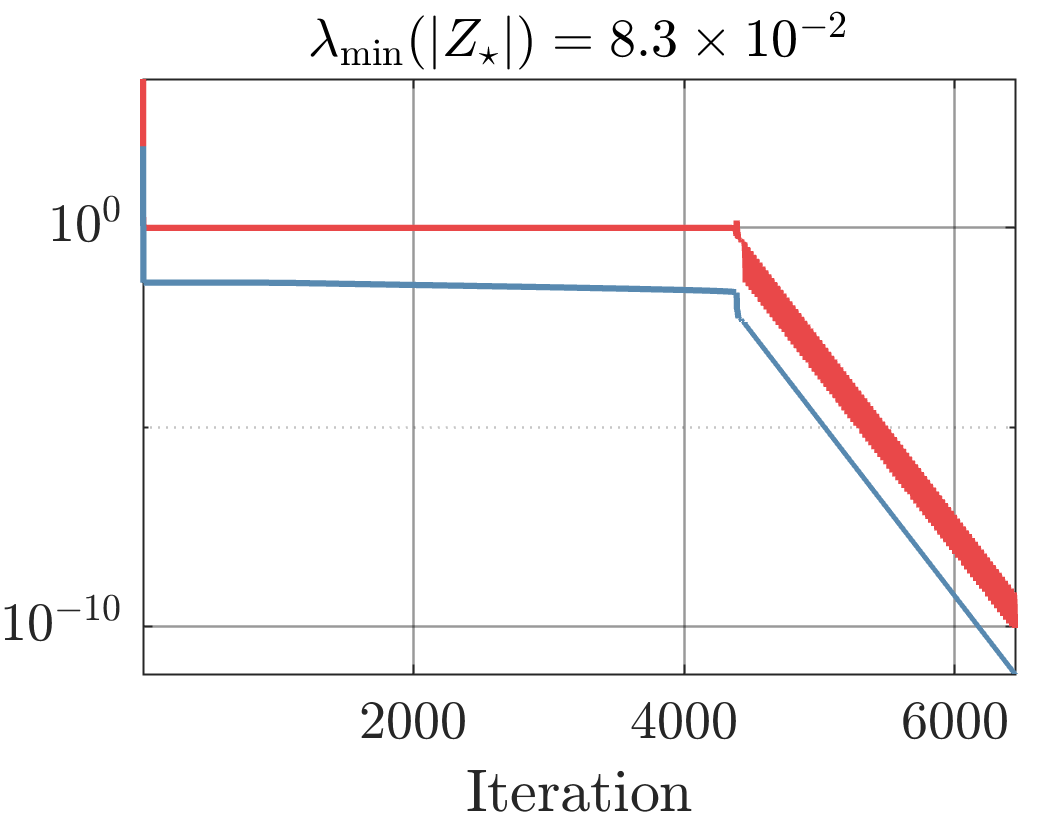}
                \texttt{hamming-9-5-6}
            \end{minipage}
        \end{tabular}
    \end{minipage}
    \caption{More Hamming graph problems with with random (standard Gaussian) initial guess. In all cases, the converging $\Zs$ is nonsingular. \label{app:fig:hamming}}
\end{figure}


\begin{figure}[tbp]
    \centering

    \begin{minipage}{\textwidth}
        \centering
        \hspace{5mm} \includegraphics[width=0.35\columnwidth]{figs/legends/legend_rmax_dZ.png}
    \end{minipage}

    \begin{minipage}{\textwidth}
        \centering
        \begin{tabular}{ccc}
            \begin{minipage}{0.30\textwidth}
                \centering
                \includegraphics[width=\columnwidth]{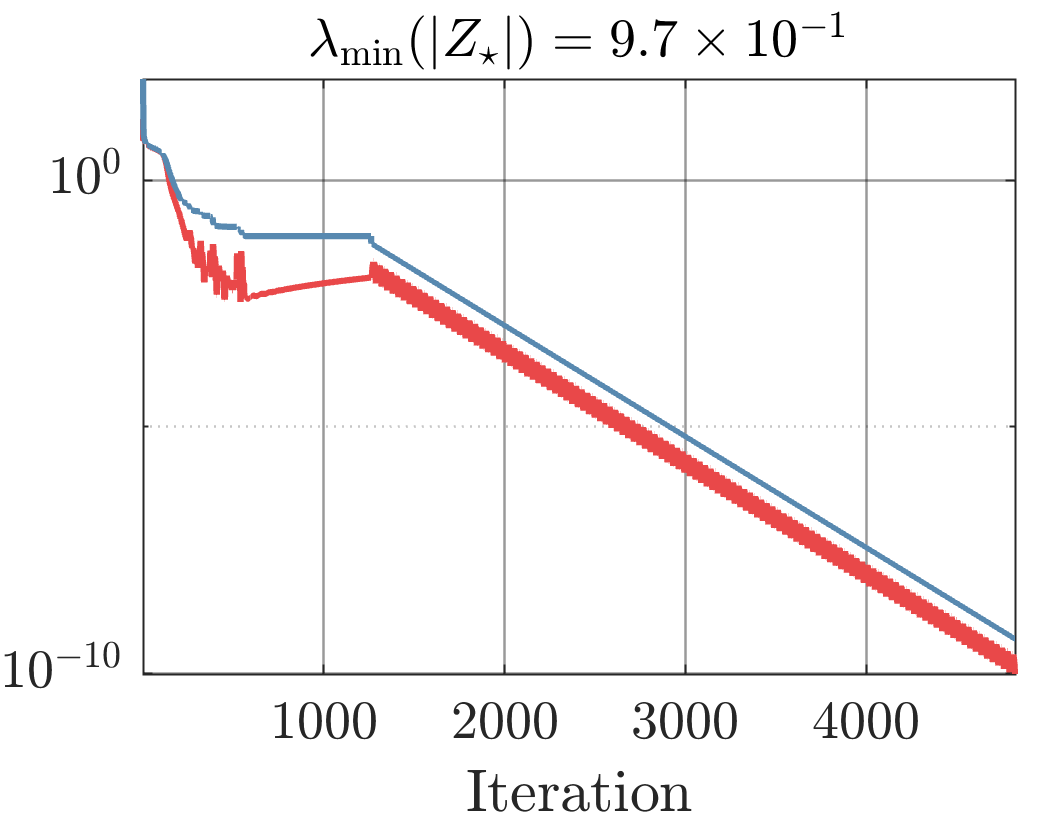}
                \texttt{BQP-r1-20-2}
            \end{minipage}

            \begin{minipage}{0.30\textwidth}
                \centering
                \includegraphics[width=\columnwidth]{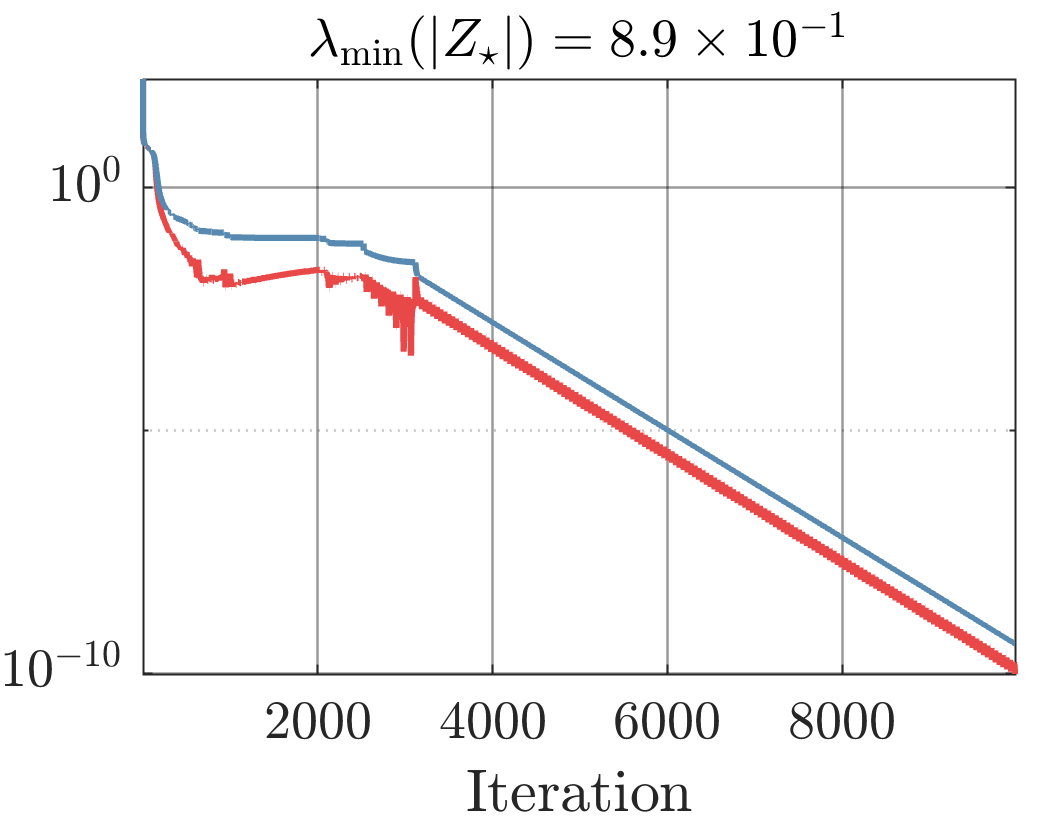}
                \texttt{BQP-r1-30-2}
            \end{minipage}

            \begin{minipage}{0.30\textwidth}
                \centering
                \includegraphics[width=\columnwidth]{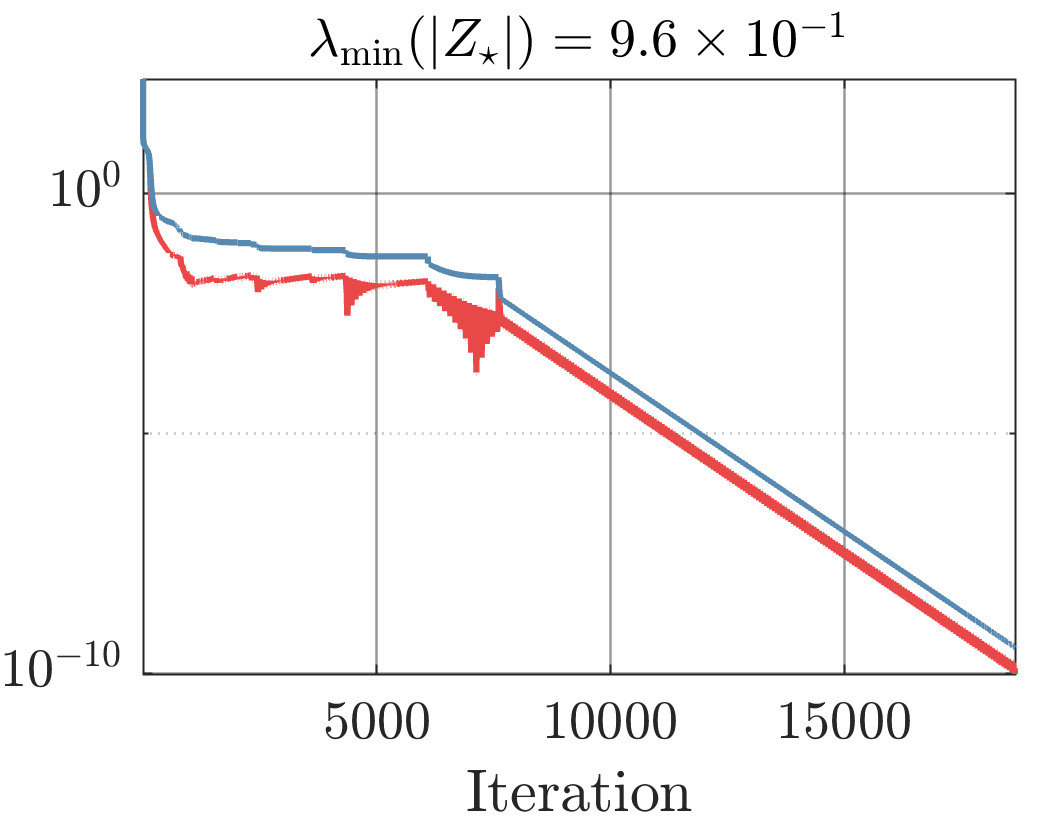}
                \texttt{BQP-r1-40-2}
            \end{minipage}
        \end{tabular}
    \end{minipage}

    \begin{minipage}{\textwidth}
        \centering
        \begin{tabular}{ccc}
            \begin{minipage}{0.30\textwidth}
                \centering
                \includegraphics[width=\columnwidth]{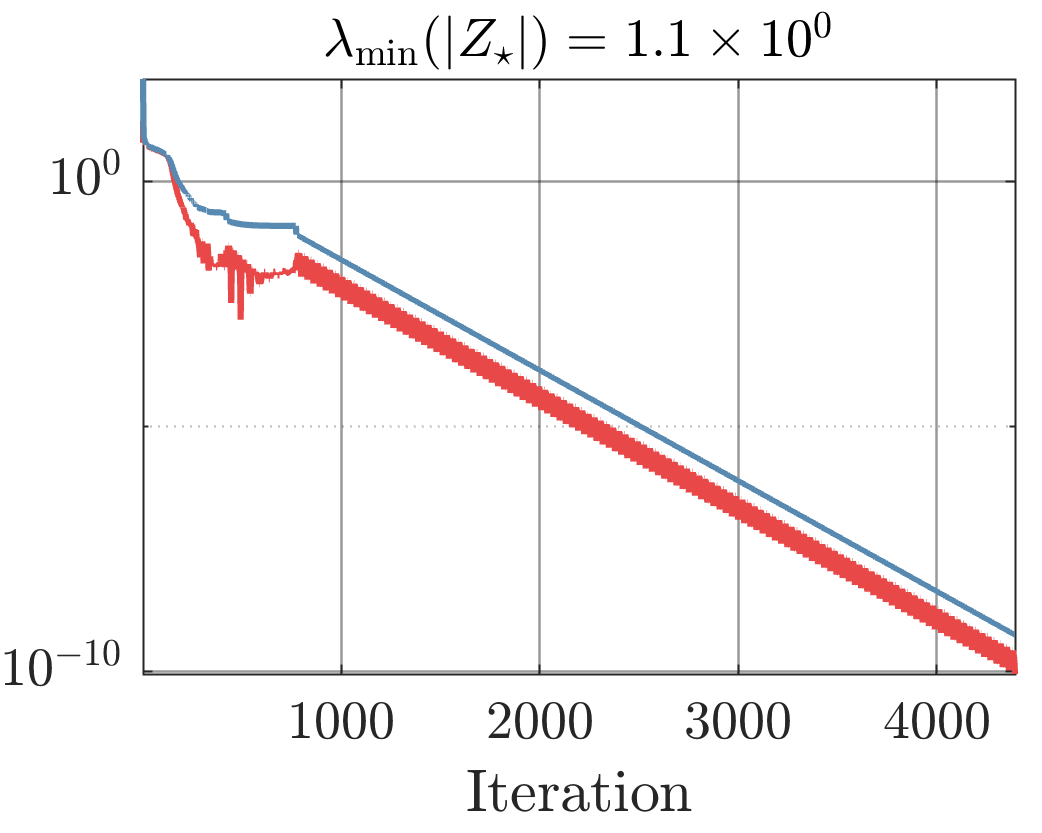}
                \texttt{BQP-r1-20-3}
            \end{minipage}

            \begin{minipage}{0.30\textwidth}
                \centering
                \includegraphics[width=\columnwidth]{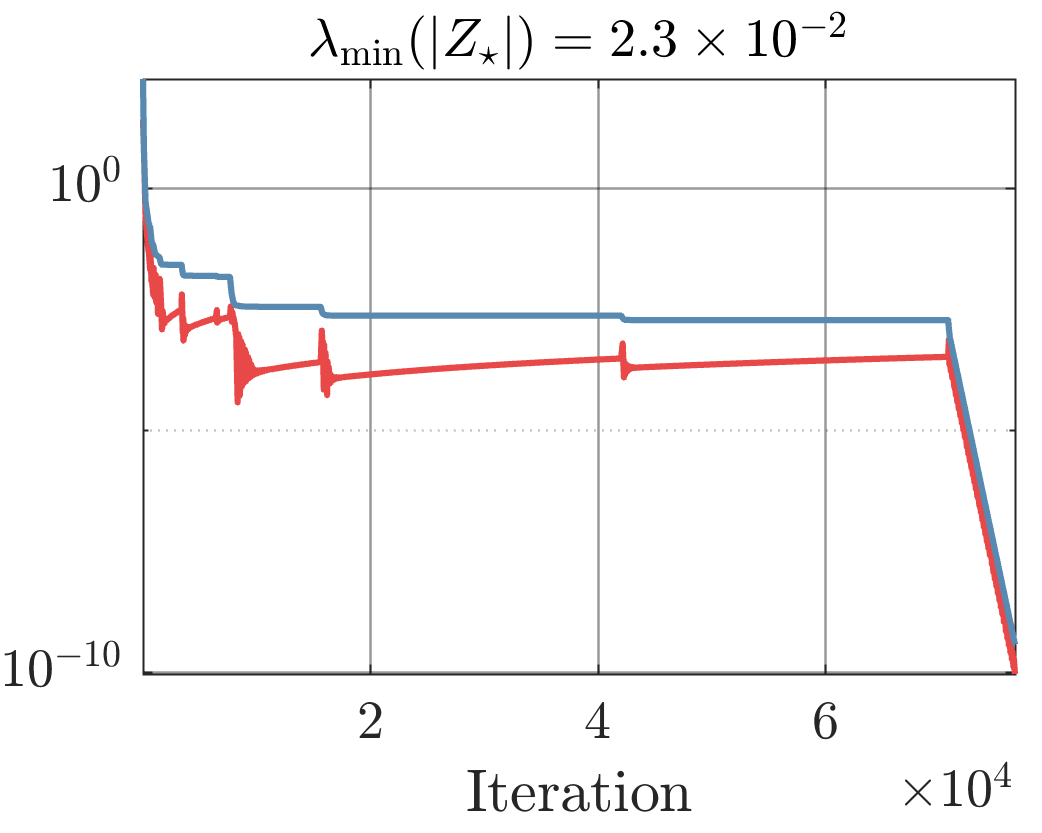}
                \texttt{BQP-r1-30-3}
            \end{minipage}

            \begin{minipage}{0.30\textwidth}
                \centering
                \includegraphics[width=\columnwidth]{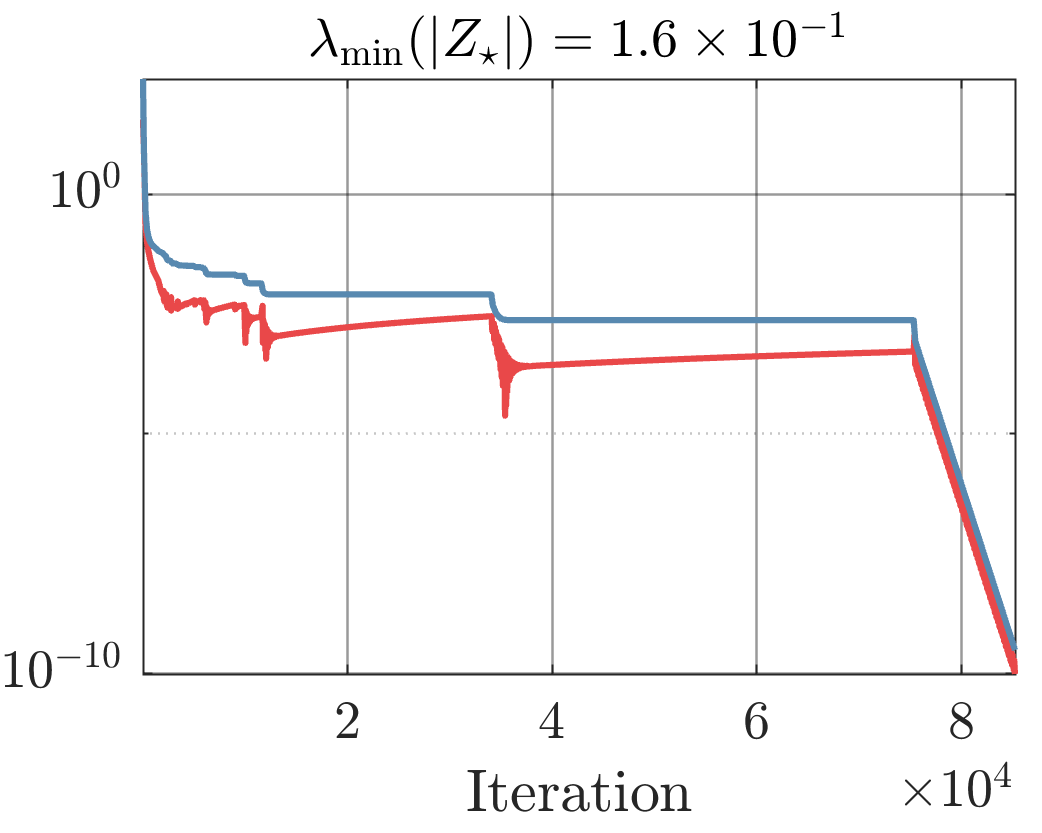}
                \texttt{BQP-r1-40-3}
            \end{minipage}
        \end{tabular}
    \end{minipage}

    \caption{Additional Random BQP problems with $c \sim \calN(0, I_n)$ and random (standard Gaussian) initialization. In all cases, the converging $\Zs$ is nonsingular.}
    \label{app:fig:BQP-r1}
\end{figure}


\begin{figure}[tbp]
    \centering

    \begin{minipage}{\textwidth}
        \centering
        \hspace{5mm} \includegraphics[width=0.35\columnwidth]{figs/legends/legend_rmax_dZ.png}
    \end{minipage}

    \begin{minipage}{\textwidth}
        \centering
        \begin{tabular}{ccc}
            \begin{minipage}{0.30\textwidth}
                \centering
                \includegraphics[width=\columnwidth]{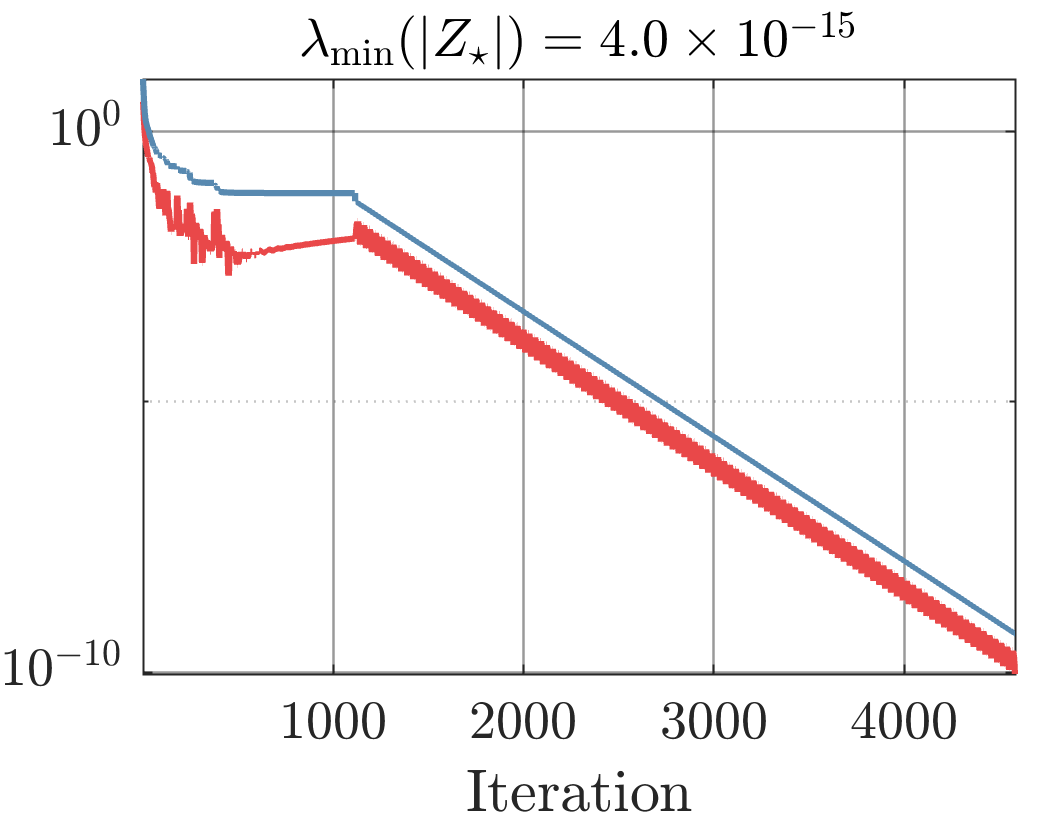}
                \texttt{BQP-r1-20-2}
            \end{minipage}

            \begin{minipage}{0.30\textwidth}
                \centering
                \includegraphics[width=\columnwidth]{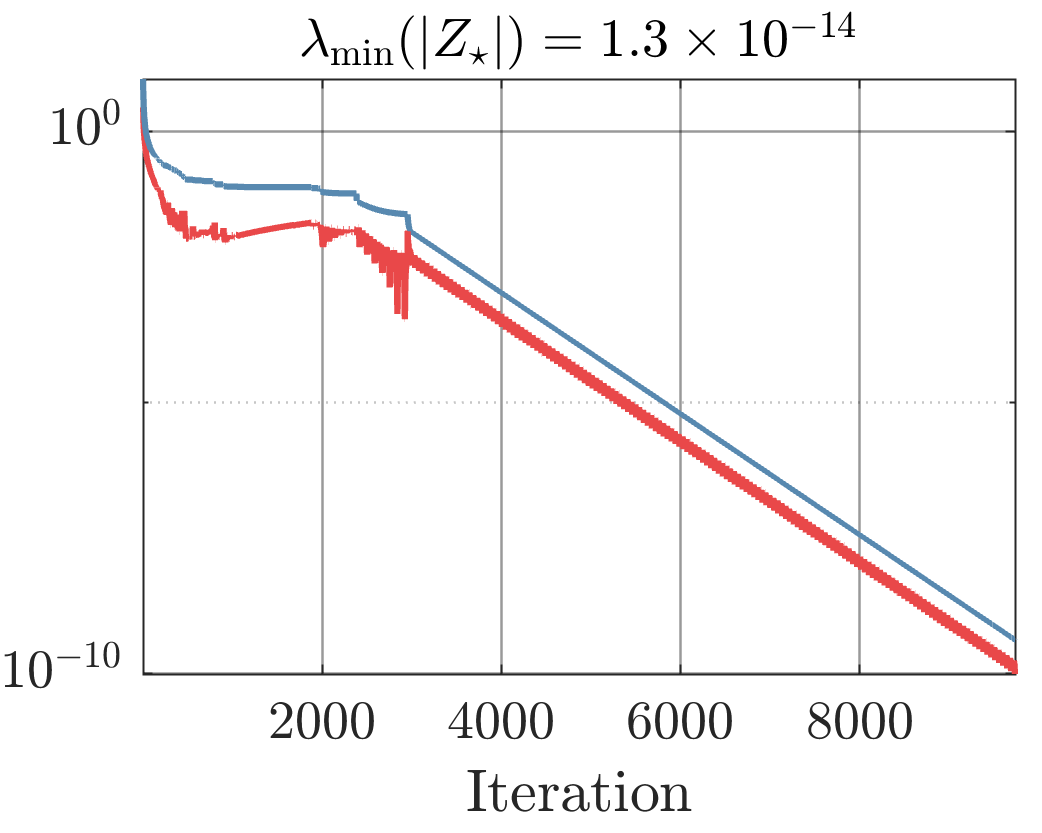}
                \texttt{BQP-r1-30-2}
            \end{minipage}

            \begin{minipage}{0.30\textwidth}
                \centering
                \includegraphics[width=\columnwidth]{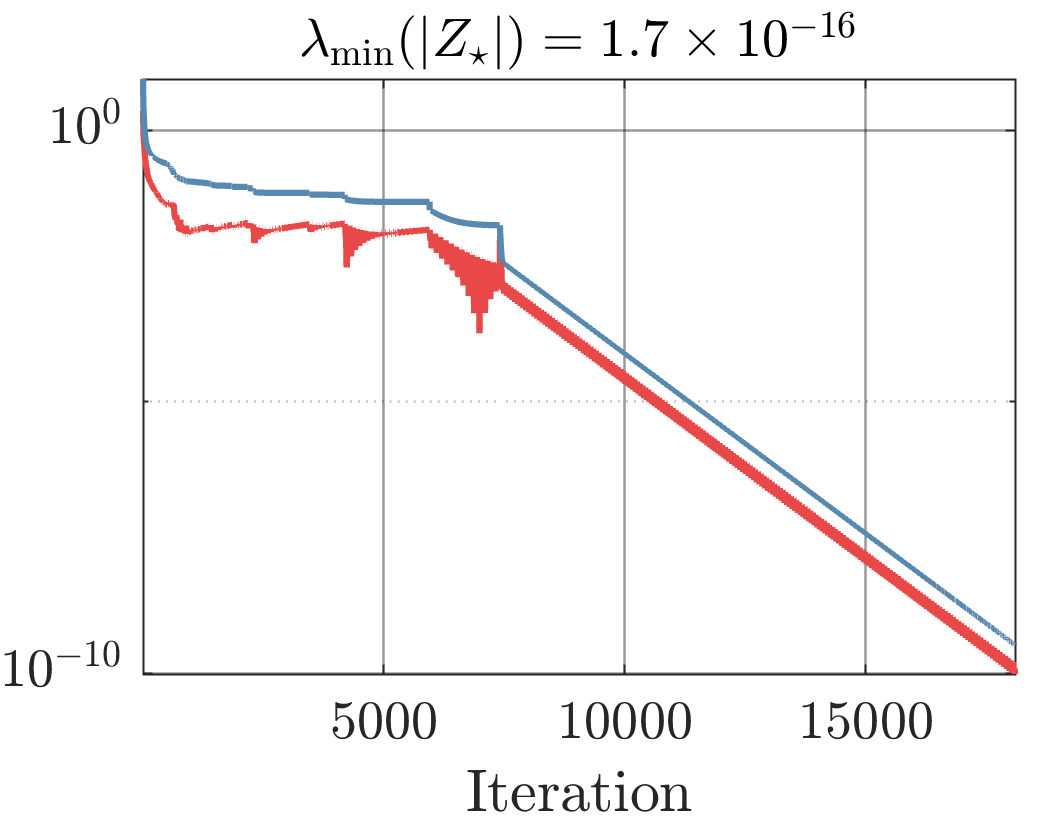}
                \texttt{BQP-r1-40-2}
            \end{minipage}
        \end{tabular}
    \end{minipage}

    \begin{minipage}{\textwidth}
        \centering
        \begin{tabular}{ccc}
            \begin{minipage}{0.30\textwidth}
                \centering
                \includegraphics[width=\columnwidth]{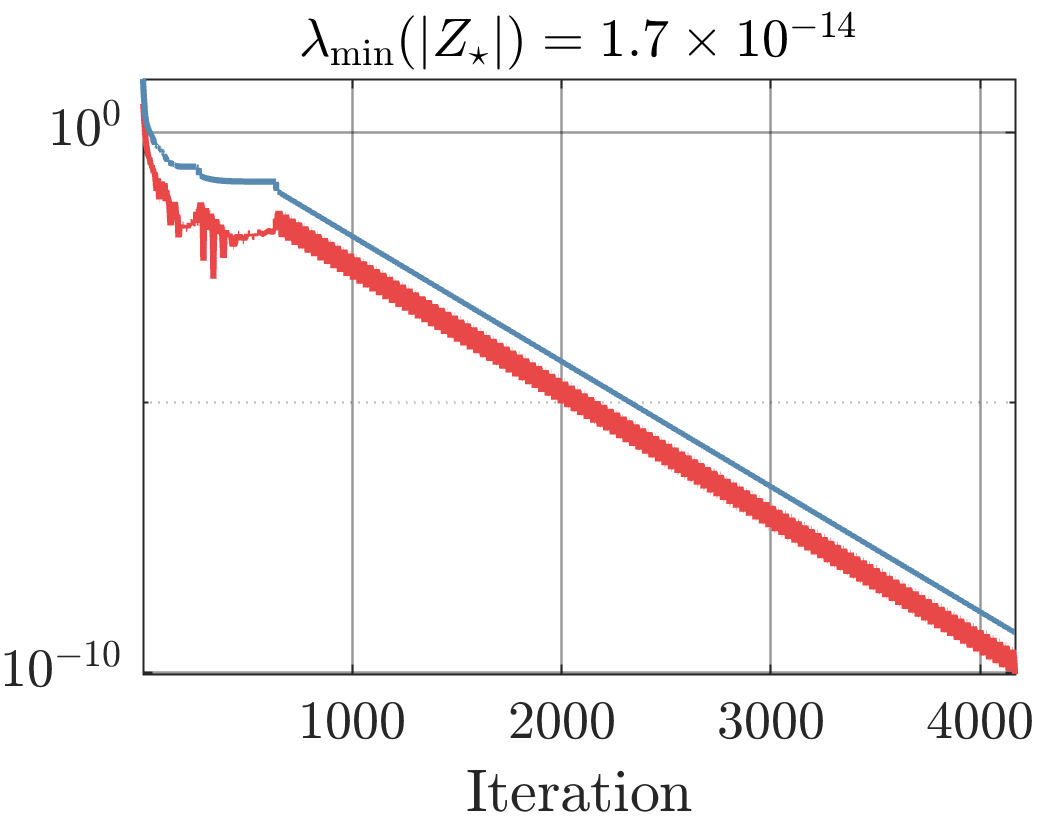}
                \texttt{BQP-r1-20-3}
            \end{minipage}

            \begin{minipage}{0.30\textwidth}
                \centering
                \includegraphics[width=\columnwidth]{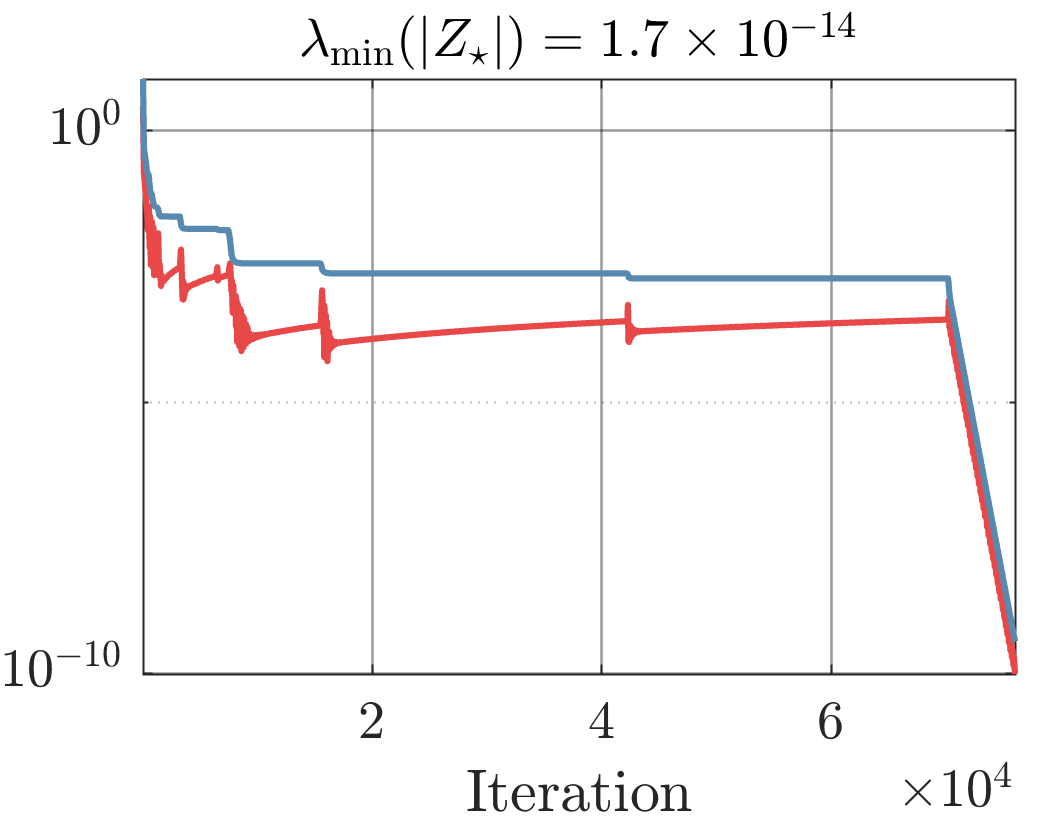}
                \texttt{BQP-r1-30-3}
            \end{minipage}

            \begin{minipage}{0.30\textwidth}
                \centering
                \includegraphics[width=\columnwidth]{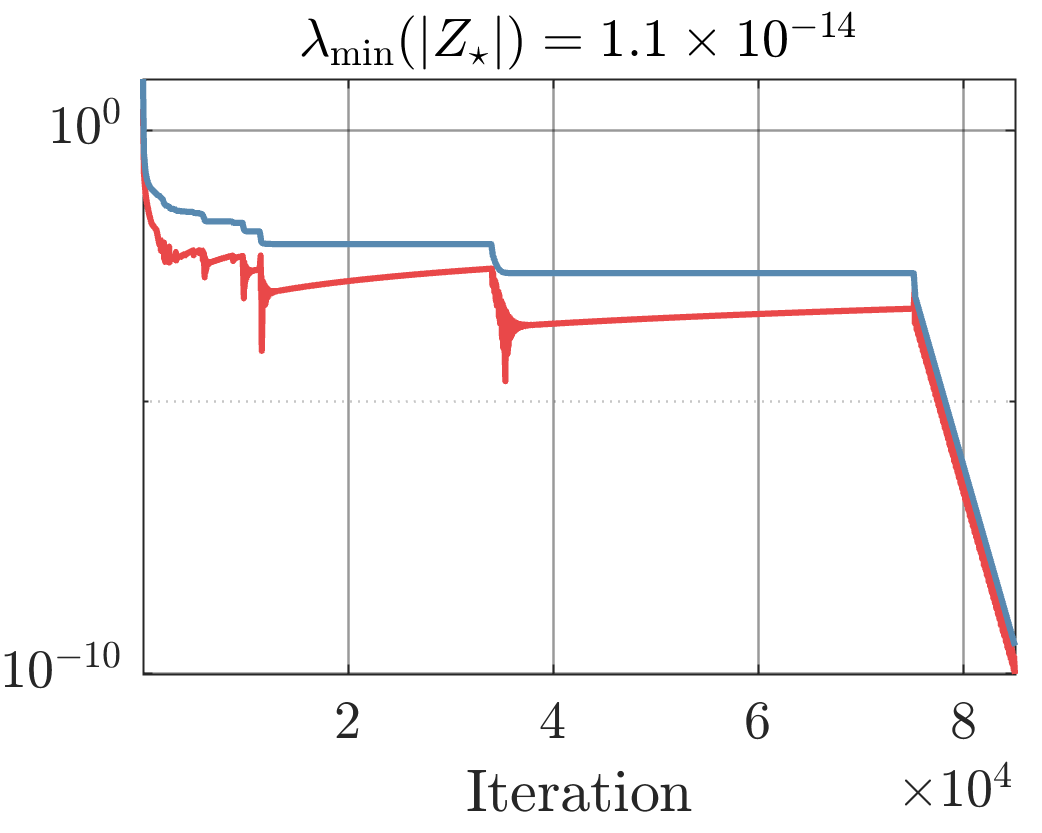}
                \texttt{BQP-r1-40-3}
            \end{minipage}
        \end{tabular}
    \end{minipage}

    \caption{Additional Random BQP problems with $c \sim \calN(0, I_n)$ and all-zeros initialization. In all cases, the converging $\Zs$ is singular. \label{app:fig:BQP-r1-zero}}
\end{figure}


\begin{figure}[tbp]
    \centering

    \begin{minipage}{\textwidth}
        \centering
        \hspace{5mm} \includegraphics[width=0.35\columnwidth]{figs/legends/legend_rmax_dZ.png}
    \end{minipage}

    \begin{minipage}{\textwidth}
        \centering
        \begin{tabular}{ccc}
            \begin{minipage}{0.30\textwidth}
                \centering
                \includegraphics[width=\columnwidth]{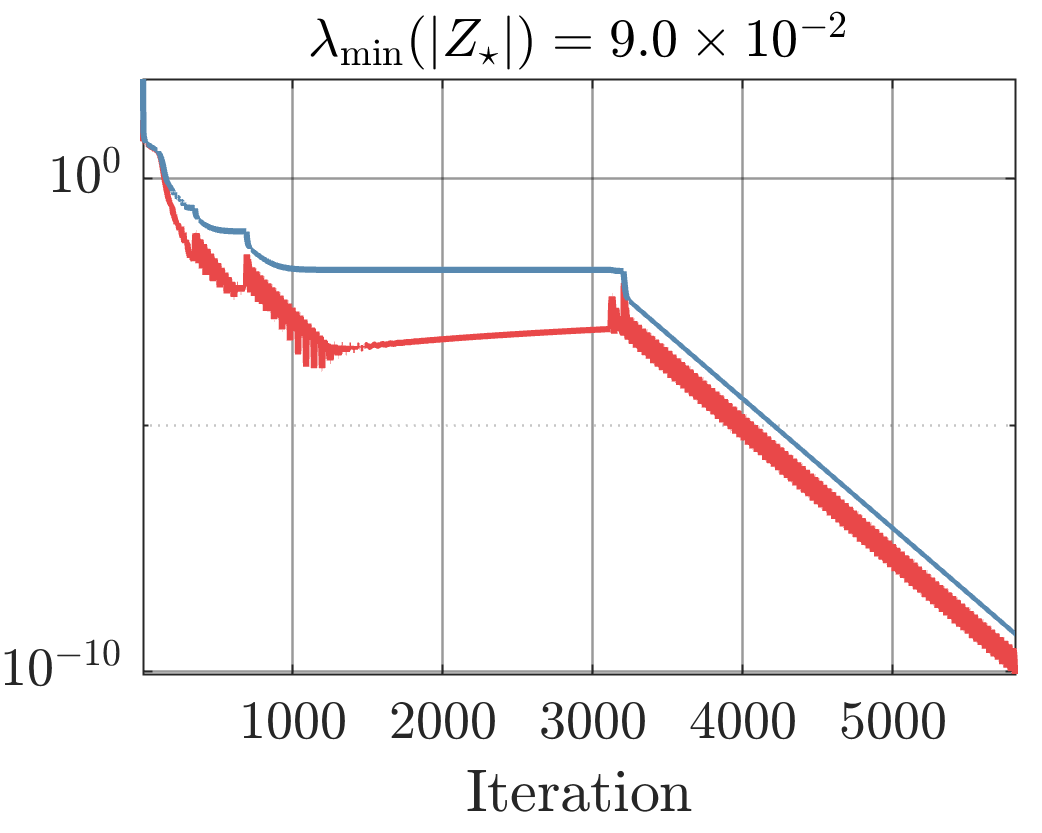}
                \texttt{BQP-r2-20-2}
            \end{minipage}

            \begin{minipage}{0.30\textwidth}
                \centering
                \includegraphics[width=\columnwidth]{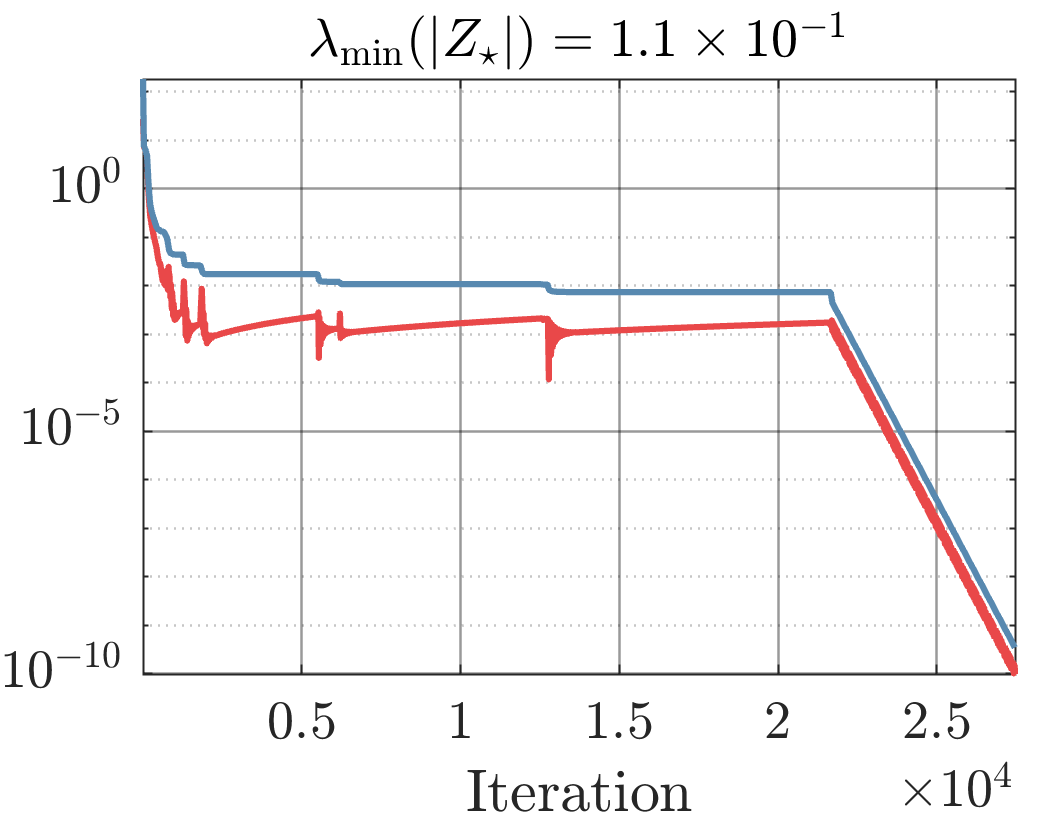}
                \texttt{BQP-r2-30-2}
            \end{minipage}

            \begin{minipage}{0.30\textwidth}
                \centering
                \includegraphics[width=\columnwidth]{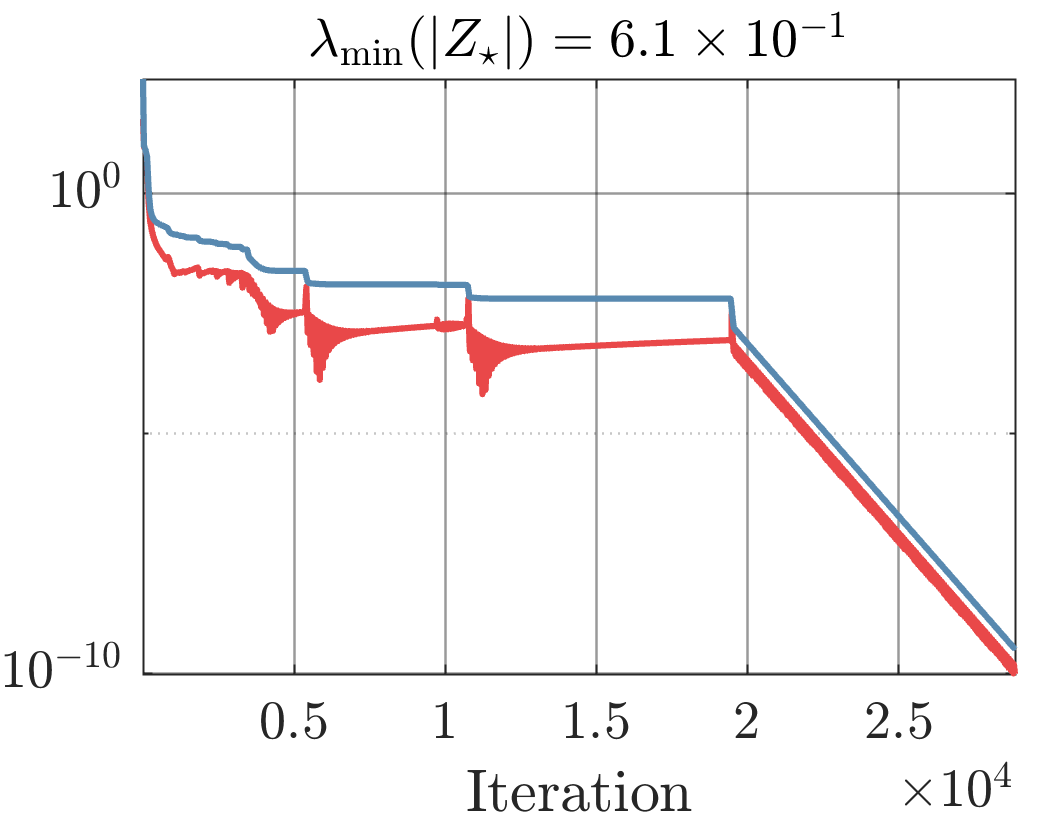}
                \texttt{BQP-r2-40-2}
            \end{minipage}
        \end{tabular}
    \end{minipage}

    \begin{minipage}{\textwidth}
        \centering
        \begin{tabular}{ccc}
            \begin{minipage}{0.30\textwidth}
                \centering
                \includegraphics[width=\columnwidth]{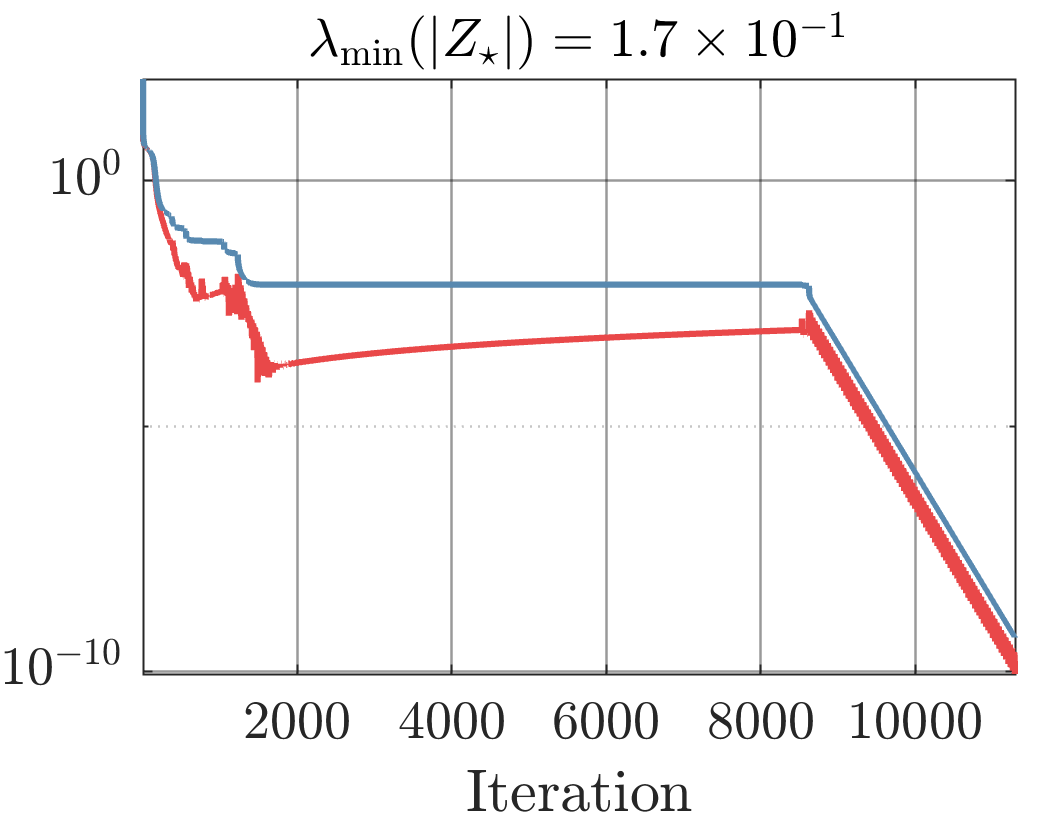}
                \texttt{BQP-r2-20-3}
            \end{minipage}

            \begin{minipage}{0.30\textwidth}
                \centering
                \includegraphics[width=\columnwidth]{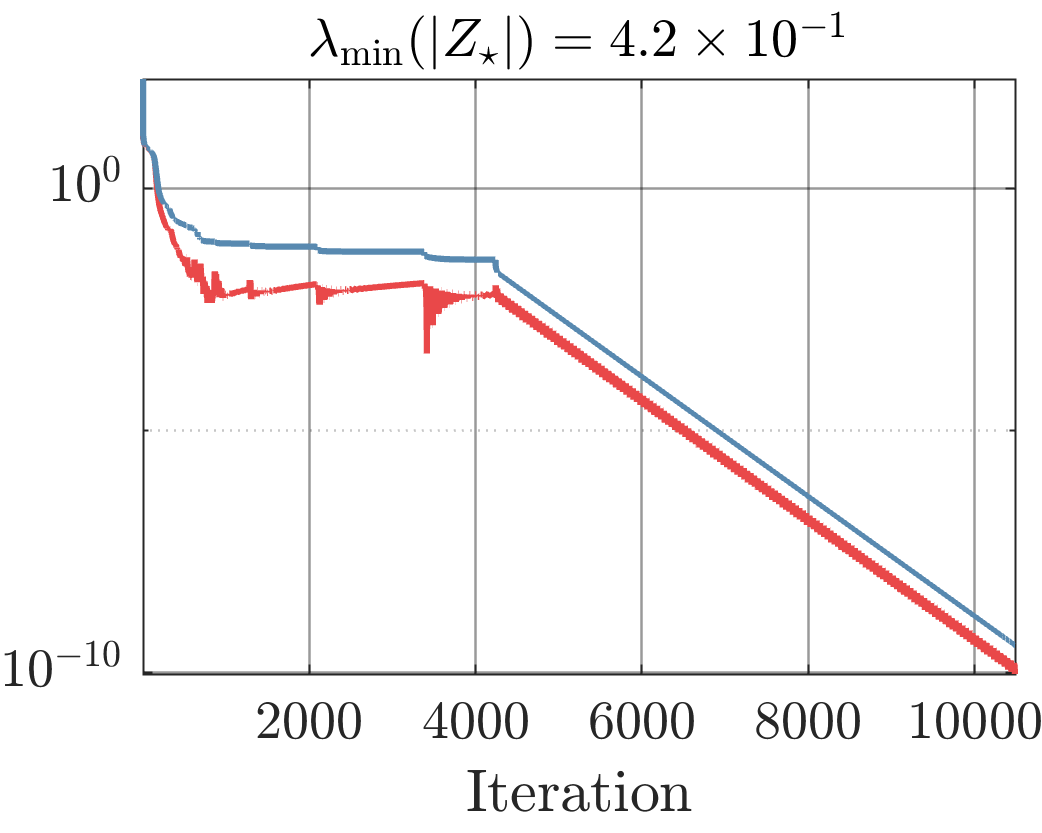}
                \texttt{BQP-r2-30-3}
            \end{minipage}

            \begin{minipage}{0.30\textwidth}
                \centering
                \includegraphics[width=\columnwidth]{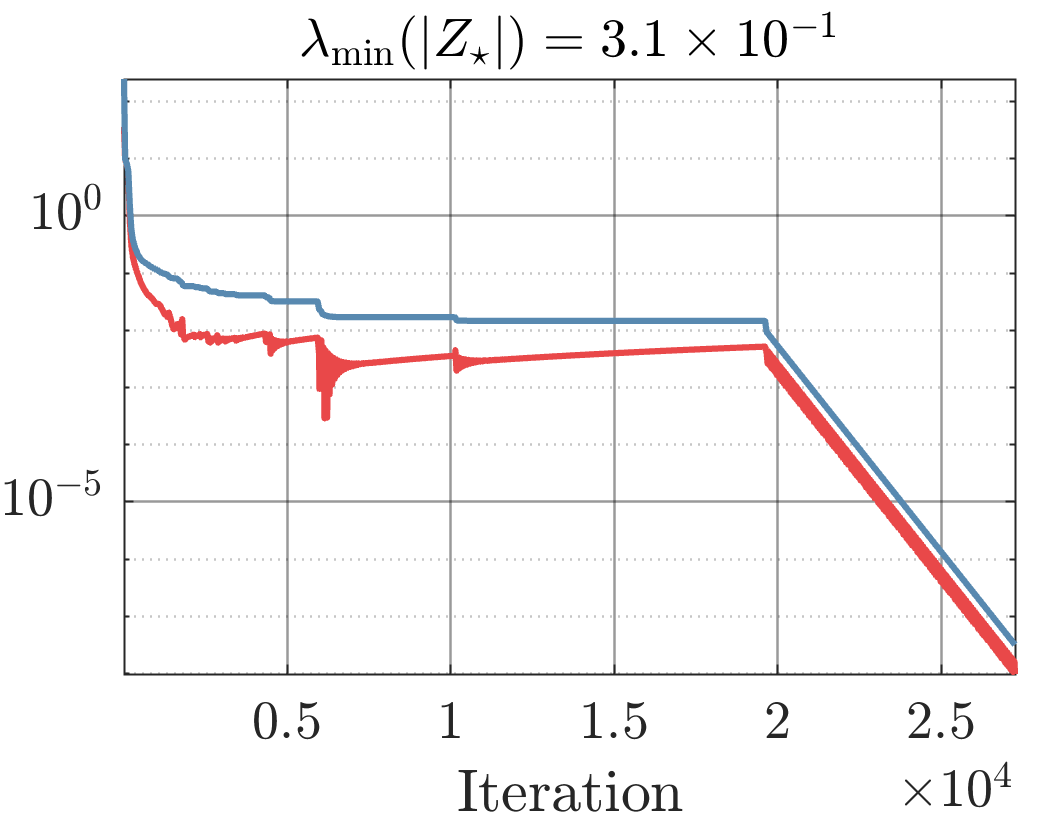}
                \texttt{BQP-r2-40-3}
            \end{minipage}
        \end{tabular}
    \end{minipage}

    \caption{Additional random BQP problems with $c = 0$ and random (standard Gaussian) initialization. In all cases, the converging $\Zs$ is nonsingular. \label{app:fig:BQP-r2}}
\end{figure}

\end{appendices}

\bibliography{ref}

\end{document}